\documentclass[11pt,reqno]{amsart}

\usepackage[T1]{fontenc}

\usepackage{amsmath}								%math
\usepackage{amssymb}
\usepackage{amsthm}
\usepackage{amscd}
\usepackage{amsfonts}
\usepackage{mathtools}
\usepackage{mathptmx}
\usepackage[scaled]{helvet}
\usepackage{microtype}
\usepackage{stmaryrd}
\usepackage[all]{xy}

\usepackage{euler}

\usepackage{extarrows}

\usepackage[colorlinks, citecolor = blue, linkcolor = blue]{hyperref}
\usepackage{color}
\usepackage{tikz}									
\usetikzlibrary{matrix}
\usetikzlibrary{patterns}
\usetikzlibrary{matrix}
\usetikzlibrary{positioning}
\usetikzlibrary{decorations.pathmorphing}
\usetikzlibrary{cd}
\usetikzlibrary{intersections,calc}
\tikzset{dot/.style={circle, fill=black, inner sep=.05cm}}

\usepackage{dsfont}

\usepackage[left=3.5cm,top=3.5cm,right=3.5cm]{geometry}
\usepackage[shortlabels]{enumitem}

\newtheorem{theorem}{Theorem}[section]
\newtheorem{lemma}[theorem]{Lemma}
\newtheorem{proposition}[theorem]{Proposition}
\newtheorem{corollary}[theorem]{Corollary} 
 
\newtheorem*{thmrd}{Theorem \ref{thm:representable diagonal}}

\theoremstyle{definition}
\newtheorem{definition}[theorem]{Definition}
\newtheorem{example}[theorem]{Example}

\newtheorem{convention}[theorem]{Convention}

\newtheorem{remark}[theorem]{Remark}

\newcommand{\acolor}[1]{{\color{purple!60!black} #1}}

\newcommand{\comment}[1]{}

\newcommand{\R}{\mathds{R}}
\newcommand{\Rbar}{\overline{\R}}
\newcommand{\Z}{\mathds{Z}}

\newcommand{\ZZ}{\mathds{Z}}
\newcommand{\Q}{\mathds{Q}}

\newcommand{\N}{\mathds{N}}
\newcommand{\PP}{\mathds{P}}

\newcommand{\T}{\mathds{T}}
\newcommand{\TT}{\mathds{T}}

\newcommand{\mc}{\mathcal}

\newcommand{\sigmabar}{\overline{\sigma}}

\newcommand{\mMbar}{\overline{\mathcal M}}

\newcommand{\calO}{\mathcal{O}}

\newcommand\relint{\mathrm{relint}}

\DeclareMathOperator{\Pic}{Pic}
\DeclareMathOperator{\Div}{Div}

\DeclareMathOperator{\Aut}{Aut}

\DeclareMathOperator{\Isom}{Isom}

\newcommand\alg{{\mathrm{alg}}}
\newcommand\fin{{\mathrm{fin}}}

\DeclareMathOperator{\val}{val}

\DeclareMathOperator{\Trop}{Trop}

\DeclareMathOperator{\id}{id}

\newcommand{\APL}{\mathrm{Rat}}
\newcommand{\Divs}{\Div}
\newcommand{\bL}{{\mathds L}}
\newcommand{\SPic}{{\mathrm{RPic}}}
\newcommand\lcm{\mathrm{lcm}}

\newcommand{\TPL}{TPL}
\newcommand{\mTPL}{\mathrm{TPL}}
\DeclareMathOperator{\Aff}{Aff}
\DeclareMathOperator{\PL}{PL}
\DeclareMathOperator{\Harm}{H}
\newcommand\Mgn[1]{\mathcal M_{#1}}
\newcommand\Mgnbar[1]{\mMbar_{#1}}

\newcommand\MgnMf[1]{\mc M_{#1}^{\mathrm{Mf}}}
\newcommand\Atlass[1]{{\overline {\mc V}_{#1}}}
\newcommand\Atlas[1]{\mc V^{Mf}_{#1}}
\newcommand\GoodAtlas[1]{\mc V^{\mathrm{good}}_{#1}}

\newcommand\br{\mathrm{br}}
\newcommand\src{\mathrm{src}}
\newcommand\WS{{\mathrm{WS}}}
\newcommand\TSM{{\mathrm{TSM}}}
\newcommand\ev{{\mathrm{ev}}}

\renewcommand\injlim\varinjlim

\newcommand{\mysetminusD}{\hbox{\tikz{\draw[line width=0.6pt,line cap=round] (3pt,0) -- (0,6pt);}}}
\newcommand{\mysetminusT}{\mysetminusD}
\newcommand{\mysetminusS}{\hbox{\tikz{\draw[line width=0.45pt,line cap=round] (2pt,0) -- (0,4pt);}}}
\newcommand{\mysetminusSS}{\hbox{\tikz{\draw[line width=0.4pt,line cap=round] (1.5pt,0) -- (0,3pt);}}}
\newcommand{\mysetminus}{\mathbin{\mathchoice{\mysetminusD}{\mysetminusT}{\mysetminusS}{\mysetminusSS}}}
\renewcommand\setminus\mysetminus
\renewcommand\smallsetminus\mysetminus

\title[genusone]{Tropical $\psi$ classes}% in genus one.}

\author{Renzo Cavalieri}
\address{Department of Mathematics, Colorado State University, Fort Collins, Colorado 80523-1874}
\email{\href{mailto:renzo@math.colostate.edu}{renzo@math.colostate.edu}}
\author{Andreas Gross}
\address{Department of Mathematics, Colorado State University, Fort Collins, Colorado 80523-1874}
\email{\href{mailto:andreas.gross@colostate.edu }{andreas.gross@colostate.edu }}
\author{Hannah Markwig}
\address{Eberhard Karls Universit\"at T\"ubingen, Fachbereich Mathematik, Auf der Morgenstelle 10, 72076 T\"ubingen, Germany}
\email{\href{mailto:hannah@math.uni-tuebingen.de}{hannah@math.uni-tuebingen.de}}
\thanks{}


\begin{document}

\begin{abstract}
We introduce a tropical geometric framework that allows us to define $\psi$ classes for moduli spaces of tropical curves of arbitrary genus. We prove correspondence theorems between algebraic and tropical $\psi$ classes for some one-dimensional families of genus-one tropical curves.
\end{abstract}

\maketitle

% \setcounter{tocdepth}{2}
% \tableofcontents

\section{Introduction}
\renewcommand*{\thetheorem}{\Alph{theorem}}

\subsection{Results}
The main goal of this manuscript is to introduce  {\it $\psi$ classes}  for moduli spaces of tropical curves of arbitrary genus. This is achieved in Definition \ref{def:psi}: \emph{$\psi_i$ is the first Chern class of the $i$-th cotangent line bundle on the stack $\Mgnbar{g,n}$ of families of tropical, $n$-marked curves of genus $g$}. We postpone by a few paragraphs  an overview of the technical work  necessary to make such a natural  definition, and begin by discussing some  features of the resulting theory.

First off, $\psi_i$ is a Chern class on a stack, that is the assignment of a (weak)  Chow class on $B$ for any family of tropical curves $\pi: \mc C \to B$, appropriately compatible with base changes (Definition \ref{def:fcc}). As such, any {\it correspondence} statement requires as input a family of curves in algebraic geometry.
We provide two non-trivial correspondence theorems for one-dimensional families of tropical curves of genus one. The first arises from a cycle of well-spaced tropical stable maps to $\T\PP^2$ passing through eight general points.

\begin{theorem}[Proposition \ref{prop:degree of psi1 on pencil}, Proposition \ref{prop:degree of map from pencil to moduli space}]
\label{thm:pcub}
Let $WS_{8}(\T\PP^2,3)$ denote the space of well-spaced tropical stable maps of degree three to the tropical projective plane, with eight marked ends. We denote by $B= \cap_{i=1}^8 ev^{-1}_i(P_i)\subset WS_{8}(\T\PP^2,3)$ the one-dimensional locus of maps where the marked ends are mapped to eight fixed general points in the plane. %\renzo{The standard suggestion is to spell out small numbers and write out large ones. Is eight large or small?} \andreas{I did a quick internet search and it seems to be semi-standard to write out the numbers zero through nine and not anything above. I was asking for consistency's sake and have now changed all 'genus-$1$' to 'genus-one' etc.}
We define  a one-dimensional cycle $\alpha_B$ supported on $B$ and a family of tropical curves $\mc C \to B$, giving rise to a map $g: B \to \Mgnbar{1,17}$. Forgetting the nine non-contracted ends and seven out of the eight marks, one gets eight covering functions $f_i: B \to \Mgnbar{1,\{i\}}$.
For $i = 1, \ldots, 8$, we have
\begin{equation}
    \frac{\deg (f_i^\ast\psi_i \cdot \alpha_B)}{\deg (f_{i\ \ast}(\alpha_B)) } = \frac{1}{24}.
\end{equation}
\end{theorem}
The second correspondence theorem arises from tropical Hurwitz theory.
\begin{theorem}[Proposition \ref{prop:cycle on admissible cover space}, Proposition \ref{prop:degree of psi on admissible covers space}]
\label{thm:ac}
For $d \geq 2$, denote by $B_d$ (a finite cover of) the space of genus-one tropical admissible covers of degree $d$ of a genus-zero tropical curve with ramification profile $((d),(d),(2,1^{d-2}),(2,1^{d-2}))$. 
We define a cycle $\alpha_{B_d}$ supported on $B_d$ and a family of tropical curves giving rise to a map $\pi: B_d \to \Mgnbar{1,1} $. We have:
\begin{equation}
    \frac{\deg (\pi^\ast\psi_1 \cdot \alpha_{B_d})}{\deg (\pi_{\ast}(\alpha_{B_d})) } = \frac{1}{24}.
\end{equation}
\end{theorem}

Theorems \ref{thm:pcub}, \ref{thm:ac} are instances of  the tropical counterpart to the algebraic statement 
\begin{equation}\label{eq:cltr}
    \int_{[\Mgnbar{1,1}^\alg]} \psi_1 = \frac{1}{24}.
\end{equation}
As we will discuss more in depth later in the introduction, for $g>0$ the moduli spaces $\Mgnbar{g,n}$ do not carry a universal family, nor a canonical fundamental cycle.
Not only can one not make \eqref{eq:cltr} into a tropical statement;  in  Example \ref{ex:psimanyvalues} we construct families of genus-one tropical curves giving rise to  maps of equal degree to $\Mgnbar{1,1}$, but with different evaluations for $\psi_1$. This is not surprising since the theory we constructed is entirely combinatorial; the best correspondence statement one may expect would compare the tropical and algebraic $\psi$ classes for any family of tropical curves which arises as the tropicalization of an algebraic family of curves. While we do not prove such a statement in full generality here, Theorems \ref{thm:pcub} and \ref{thm:ac} provide some compelling initial evidence. 

A key feature of $\psi$ classes in algebraic geometry is their behavior with respect to pull-back and push-forward via tautological forgetful morphisms. Our next result recovers similar statements for the tropical theory.

\begin{theorem}[Theorem \ref{thm:pull-back of psi class}, Theorem \ref{thm:germ of dilaton}] \label{thm:C} Let $g,n \in \Z_{\geq 0}$ such that $2g-2+n >0$, and let $\pi_\star: \Mgnbar{g,n\sqcup \{\star\}} \to \Mgnbar{g,n}$ denote the morphism forgetting the end marked with $\star$. Then:
\begin{enumerate}
    \item for $i = 1, \ldots, n$,
    \begin{equation*}
     \psi_i = \pi_\star^\ast \psi_i+
    D_{i,\star},   
    \end{equation*}
    where $D_{i,\star}$ denotes the divisor of curves where the $i$- and $\star$-ends are are part of a tripod whose edges are all infinite. 
    \item
    \[
\pi_*(\psi_\star\cdot[\MgnMf{g,n\sqcup\{\star\}}])=(2g-2+n)\cdot [\MgnMf{g,n}] \ , 
\]
where $[\MgnMf{g,n}]$ denotes the cycle supported with  weight one on the cones of tropical curves with all vertices of genus zero.
\end{enumerate}
\end{theorem}
Finally, our theory agrees with
all previously made statements for moduli spaces of tropical curves of genus zero.

\begin{theorem}[Theorem \ref{thm:Mg0n is what it's supposed to be}, Corollary \ref{cor:psi class in genus 0}]
The  tropicalization  of the forgetful morphism %forgetting the $\star$-mark
\begin{equation*}
  \Trop(\Mgnbar{0,n \sqcup \{\star\}}^\alg)\to \Trop(\Mgnbar{0,n}^\alg)  
\end{equation*}
is a family of tropical stable curves inducing an isomorphism between $\Trop(\Mgnbar{0,n}^\alg)$ and $\Mgnbar{0,n}$.

The class $\psi_i$ is represented by the cycle on $\Mgnbar{0,n}$ taking value one on all points corresponding to curves where the $i$-marked leg is adjacent to a four-valent vertex, and zero everywhere else.
\end{theorem} 

\vspace{0.2cm}
We now take a few steps back to discuss some of the ingredients used to set up the theory. The first step is to define a notion of families of tropical curves such that the corresponding moduli stack is both reasonably well-behaved and with enough structure to do intersection theory on it.

We define a {\it tropical space $X$} (Definition \ref{def:tropsp}) to be a locally polyhedral space endowed with an {\it affine structure $Aff_X$}, a subsheaf of the sheaf of piecewise (affine) linear functions on $X$. The sections of $Aff_X$ are called \textit{affine functions}.
A {\it tropical curve} is then a tropical space of dimension one, with the additional structure of a genus function (Definition \ref{def:tropcurve}). While at this point there is still a fair amount of freedom in the affine structures allowed for tropical curves, we make the concept restrictive for \textit{smooth and nodal curves}: germs of affine functions are given by harmonic functions for genus-zero points, and are constant for  points of positive genus. %\andreas{in fact, they are always the harmonic functions in genus zero, even in the infinite smooth and node case. We use this when saying that the underlying \TPL{}-curve determines the tropical curve}.
%\renzo{I agree for the infinite genus zero points at ends, but nodes at infinity as well?} \andreas{It's because we define harmonic functions to have finite values}
We define a \textit{family of tropical curves} (Definition \ref{def:family}) to be a morphism of tropical spaces such that the fibers are nodal tropical curves, with an additional compatibility between the affine structures of base and total space. Formally, this compatibility is stated as requiring the exactness of the sequence of sheaves \eqref{equ:exact sequence}; in simple terms, we are stipulating that any (germ of an) affine function on a fiber should extend to an open thickening of the fiber, and that any (germ of an) affine function that restricts constantly on fibers is the pull-back of an affine function on the base.

The  moduli stack $\Mgnbar{g,n}$ (Definition \ref{def:moduli}) associated to families of tropical curves over tropical spaces has some desirable properties.

\begin{thmrd}
For every pair $g,n\in \Z_{\geq 0}$ of non-negative integers with $2g-2+n>0$, the diagonal of the stack $\Mgnbar{g,n}$ is representable.
\end{thmrd}

Unfortunately, $\Mgnbar{g,n}$ is not a geometric stack: one cannot construct an \textit{atlas}, i.e.\ an \'etale surjective morphism from a tropical space. In fact, one can get pretty close: let $\Atlass{g,n}$ denote the extended cone complex parameterizing tropical curves together with a {\it cycle rigidification}. One can construct piecewise linear functions on $\Atlass{g,n}$ through {\it cross ratios} on tropical curves (Definition \ref{def:cycleriggraph}): informally, via integration of a one-form along a directed path. Endowing $\Atlass{g,n}$ with the affine structure generated by cross ratios, one obtains that the forgetful morphism $\pi_\star: \Atlass{g, n\sqcup \{\star\}}\to \Atlass{g,n}$ is a morphism of tropical spaces and even that the exact sequence \eqref{equ:exact sequence} holds. However, $\pi_\star$ is not a family of tropical curves as the fibers do not have enough functions: consider  $x \in \Atlass{g, n\sqcup \{\star\}}$ corresponding to a point
in a fiber of $\pi_\star$ which is either on a leg  adjacent to a vertex of positive genus, or on an edge adjacent to two vertices of positive genus; there are no non-trivial cross ratios involving the leg marked with $\star$ (which is incident to $x$), and therefore no  non-constant functions at $x$. This observation allows one to identify a natural open substack of $\Mgnbar{g,n}$ for which the above rigidification does provide an atlas: we denote by $\MgnMf{g,n}$ the moduli stack of families of \textit{Mumford curves}, i.e.\ tropical curves with no points of positive genus.

\begin{theorem}[Theorem \ref{thm:existance of atlas}, Theorem \ref{thm:universal family}] \label{thm:E}
The morphism $\Atlas{g,n}\to \MgnMf{g,n}$   defined by the family $\pi_\star\colon\Atlas{g,n\sqcup\{\star\}}\to \Atlas{g,n}$ is a covering of $\MgnMf{g,n}$. Further, the forgetful morphism $\pi_\star \colon \MgnMf{g,n\sqcup\{\star\}}\to \MgnMf{g,n}$ represents the universal family.
\end{theorem}

Theorem \ref{thm:E}  essentially transfers to the Mumford locus of $\Mgnbar{g,n}$ any statment about moduli spaces of rational tropical curves. In fact, one may canonically construct a family of tropical curves over a locus called $\GoodAtlas{g,n}$, giving rise to a map onto a  larger open substack than $\MgnMf{g,n}$, as described in Definition \ref{def:good} and Proposition \ref{prop:family of tropical curve over atlas}.%\andreas{the map $\GoodAtlas{g,n}\to \Mgnbar{g,n}$ is not 'etale, to we may not want to call it an atlas}  

The intuitive meaning of the lack of an atlas for $\Mgnbar{g,n}$ is that the affine structure on the total space of a family of tropical curves is not determined, even up to automorphisms, by the fibers. This means that one cannot define a fundamental cycle on $\Mgnbar{g,n}$;  in Remark \ref{rem:fundcl}  this issue is illustrated for $\Mgnbar{1,1}$, where the obvious candidate for a fundamental cycle is the constant function  $1 :|\Mgnbar{1,1}|\to \ZZ$; the pull-back of $1$ to the base of the families $\mc C_a$ from Example \ref{ex:tropfam} is balanced, hence a tropical cycle on $\R$. However,  when $a\not=b$ the fiber product
$\mc C_a \times_{\Mgnbar{1,1}} \mc C_b$ gives a one dimensional family of tropical curves such that the pull-back of $1$ is not balanced. While the absence of a fundamental cycle on  $\Mgnbar{g,n}$ prevents one from  globally integrating, one may still do intersection theory on the stack by working with families whose bases admit a fundamental cycle.

\subsection{Context, Connections and Considerations}

This work is at the confluence of three current areas of mathematical research: tautological classes on moduli spaces of curves, tropical intersection theory, and tropical moduli spaces of curves. 

Tautological classes were introduced by Mumford in \cite{m:taegotmsoc}, as a set of Chow classes on the moduli spaces of curves naturally arising from the geometry of the curves parameterized. The notion of $\psi$ classes stems from the observation that cotangent spaces of curves at a marked point naturally organize themselves in families to produce a line bundle on the moduli space of curves, of which $\psi$ is the first Chern class.
The intersection theory of $\psi$ classes exhibits a rich combinatorial structure \cite{ac:cag}, culminating in  {\it Witten's conjecture/Kontsevich's theorem} \cite{witten,k:loc}: the generating function for intersection numbers of $\psi$ classes satisfies a classical integrable hierarchy. 
From the perspective of tropical geometry, $\psi$ classes  describe the first Chern class of the normal bundles to boundary divisors, and are therefore intimately related to the  edge coordinates on tropical moduli spaces of curves.

From its inception, one of the applications of tropical geometry has been to combinatorialize classical enumerative geometric problems. To this end, Mikhalkin \cite{MikhalkinICM} sketched the foundations of tropical intersection theory, introducing the notions of {\it tropical cycles} and their {\it stable intersection}. Allerman and Rau \cite{AllermannRau} gave an alternative definition of intersection products by pulling-back Cartier divisors. Katz \cite{Katz} showed the agreement of the two constructions by relating both to toric intersection theory of \cite{FS}. In all these cases, it is essential that the tropical space on which one does intersection theory is  embedded in a vector space, which provides a global notion of integral linear functions. 
Shaw \cite{Shaw} generalizes the intersection product to {\it locally matroidal fans}, and the second author \cite{GrossToroidal} to {\it weakly embedded cone complexes}, arising from tropicalizations of cycles on toroidal embeddings.

It is hard to pinpoint the precise moment in which  moduli spaces of tropical rational  curves are born. While \cite{MikhalkinICM} introduces the name and notation which we currently use, the space of phylogenetic trees studied in \cite{billeraetal} and the tropical Grassmannian of \cite{SpeyerSturmfels} lay the foundation for the combinatorial analysis  of  $\Mgn{0,n}$ as a balanced, rational, polyhedral fan. Work of Kapranov \cite{kapranovchowquot} and Tevelev \cite{Tev} exhibit  $\Mgnbar{0,n}^\alg$ as a tropical compactification, and Gibney-Maclagan \cite{GM} study its ideal  in the Cox ring of the toric variety whose fan is $\Mgn{0,n}$. Fran\c{c}ois and Hampe \cite{FrancoisHampe} show that $\Mgn{0,n}$ represents a moduli functor over smooth tropical spaces.
The embedding of $\Mgn{0,n}^\alg$ into a torus induced by the Pl\"ucker embedding of $Gr(2,n)$ is the driving force for all these results. In higher genus there is no such embedding, so tropicalization  takes a different perspective: \cite{ACP} relates the edge lengths of tropical curves to the smoothing parameters of the corresponding nodes of algebraic curves; tropicalization is then a deformation retraction of the analytic moduli space $\Mgn{g,n}^{an}$ onto a skeleton induced by the toroidal structure of the boundary and isomorphic, as a cone complex, to the moduli space of tropical curves $\Mgn{g,n}$. Another perspective on tropicalization uses logarithmic geometry \cite{CCUW}, where it is also shown that $\Mgn{g,n}$ represents a functor over the category of rational polyhedral cone complexes.

The notion of $\psi$ classes on moduli spaces of rational tropical curves was introduced by Mikhalkin in \cite{MikhalkinModuli}  and later investigated by Kerber and Markwig \cite{KerberMarkwig}, who proved a correspondence theorem stating the equality of intersection numbers of algebraic and tropical $\psi$ classes. Katz \cite{Katz} later gave a non-computational proof of such equality  by combining the connection of tropical and toric intersection theory with the fact that $\Mgn{0,n}$ gives a tropical compactification.
Correspondence theorems for intersection numbers of $\psi$ classes on moduli spaces of higher genus tropical curves appear in Jin's Phd thesis \cite{jinthesis}: higher genus curves produce \'etale covers of genus zero (orbifold) twisted curves, giving rise to morphisms among the corresponding moduli spaces.  The equality of algebraic and tropical intersection numbers is shown using the algebraic degree of the branch morphism  as a geometric input datum and doing tropical intersection theory on the  moduli spaces of genus zero twisted curves. 

In the current work, the construction of $\psi$ classes lies entirely in the tropical world. The main technical issue to overcome is how to make the integral lattices of adjacent cones of the extended cone complex $\Mgnbar{g,n}$ {\it communicate} with each other. The solution we propose is that families of tropical curves must be endowed with such information, in the form of a sheaf of functions  to be considered affine. Germs of functions at points of faces then provide the appropriate transition data among the integral lattices of  adjacent cones. Once one knows how to transition affine functions across faces, it is  possible to define a notion of balancing, which is the key tool for tropical intersection theory. One may also define sheaves $Aff_{\mc C}(ks)$ of affine functions with prescribed order  along a linear section, which are torsors over  affine functions. This allows to sidestep the technical issue of making sense of the notions of tangent bundles or relative dualizing sheaves in the category of tropical spaces: the \textit{$i$-th cotangent line bundle}  of a family of tropical curves $\mc C$ is defined to be $Aff_{\mc C}(-s_i)$, drawing from the algebraic identification of the $i$-th cotangent line bundle with the conormal bundle to the $i$-th section.

Given the absence of any algebraic input, it is not surprising that one obtains a combinatorial theory which is broader than the algebraic theory. It  appears that when a family of tropical curves arises from an algebraic one, then the combinatorial theory agrees with the algebraic one. The computation of the degree of $\psi$ on $\Mgnbar{1,1}$ we make in Section \ref{sec:71} is very much parallel to its classical counterpart \cite[Section 3.13]{v:mscgw}: a pencil of plane cubics has nine base points, and hence it provides a family of genus one curves with nine sections, with total space  $Bl_{p_1, \ldots,p_9 }\PP^2$. The $\psi$ classes on this family are dual to the self-intersections of the exceptional divisors, and hence have degree one. Since the family has twelve rational fibers, the pencil gives  a degree-twelve covering of $\Mgnbar{1,1}^\alg$. We consider tropical stable maps instead of cubic curves to obtain a covering of $\Mgnbar{1,1}$, as it would be impossible for curves with very large $j$-invariant to satisfy the point constraints without contracting any edge. Well-spacedness, which is also a realizability condition \cite{SpeyerUniformizing,RanganathanSantosParkerWiseI}, ensures that the family is pure-dimensional. After that, the proof is parallel to the classical one: the degree of the covering is twelve, as a consequence of the count of rational curves in the family, or from \cite{KerberMarkwig}. The  class $\psi_i$ is supported on the unique curve where the $i$-th leg is incident to a four-valent vertex, and an explicit computation shows that the multiplicity is one.

The intersection theory of $\psi$ classes in algebraic geometry is  controlled by the {\it Virasoro constraints} \cite{v:mscgw}: an infinite sequence of recursive relations which reconstruct all intersection numbers from the initial condition $\int_{\Mgnbar{0,3}^\alg} 1 = 1$. The first two relations, known as the {\it string} and {\it dilaton} equations, follow from the pull-back and push-forward properties of $\psi$ classes along forgetful morphisms. Theorem \ref{thm:C} should then be interpreted as the analogous properties (i.e.\ \textit{germs} of string and dilaton) holding in the tropical theory.

\subsection{Future directions} 

We consider this work very much as a beginning, rather than a story close to conclusion, and discuss here some of the  avenues of investigation that we intend to pursue.

As mentioned earlier, the ultimate goal would be to prove a general correspondence statement, comparing tropical and algebraic $\psi$ classes for families related via tropicalization.

A first natural step towards such a statement consists in verifying it for a larger class of examples. After some preliminary investigation, we feel confident that the correspondence statement of Section \ref{sec:71} will naturally extend to one-dimensional families of tropical curves obtained by resolving  pencils of genus one curves on smooth toric surfaces using well-spaced tropical stable maps. We have also analyzed some one-dimensional families of well-spaced tropical stable maps to tropical $\PP^1\times \PP^1$ of bidegrees $(2,3)$ and $(2,4)$, where a correspondence statement depends on establishing a comparison lemma analogous to Theorem $\ref{thm:pull-back of psi class}$ for the forgetful morphism from the moduli space of maps to the moduli space of curves. A rather general statement which seems  reasonable given the current technology is a correspondence theorem for stationary genus one descendant invariants with one $\psi$-insertion.

Tropical stable maps give a natural way to construct families of tropical curves; in positive genus superabundance causes loci of tropical stable maps subject to geometric constraints to  not be equidimensional. In genus one, imposing the condition of well-spacedness restores equidimensionality and allows to naturally define intersection cycles of tropical stable maps. One should think of well-spacedness as identifying a {\it virtual fundamental cycle} very much analogously to the case of algebraic Gromov-Witten theory. For genus strictly greater than one, such a tool is at present not available.

One may construct families of tropical curves through tropical admissible covers, generalizing the construction in Section \ref{sec:72}. At present, this seems the most direct avenue to seek correspondence statements for families of arbitrary genus.

Eventually, it would be desirable to gain a complete understanding of the intersection theory of tropical $\psi$ classes, including the description of positive dimensional cycles, and of intersection cycles of multiple $\psi$ classes. While a purely combinatorial analysis seems rather laborious, recent developments in  logarithmic geometry \cite{LogarithmicTropicalization,GrossToroidal} are offering a conceptual approach to the coordinate-free tropicalization of cycles. While currently the main technical obstacle is that  for $g\geq 2$, there is no sufficiently explicit description of a cycle representing the algebraic $\psi$ class, we are hopeful that in the not too distant future it will be possible to answer in the positive the fundamental question: {\it are tropical $\psi$ classes the tropicalization of algebraic $\psi$ classes?}

\subsection{Structure of the paper}
Sections \ref{sec-tropicalspaces} through \ref{sec:intersection theory} are devoted to establishing foundations.  Section \ref{sec-tropicalspaces} introduces the categories of tropical piecewise linear (\TPL{}) spaces and tropical spaces.
Section \ref{sec:3} develops the notion of families of tropical curves and defines the corresponding moduli stack $\Mgnbar{g,n}$. This stack has representable diagonal.

In Section \ref{sec:mumford} we show that the stack $\Mgnbar{g,n}$ is not geometric: it does not admit an atlas. We then restrict our attention to an open substack which is geometric: it parameterizes families of Mumford (also called explicit) curves. We show that $\MgnMf{g,n}$ is essentially as well behaved as spaces of rational tropical curves.

Section \ref{sec:morphisms} establishes that the natural forgetful and section morphisms in the category of \TPL{}-spaces in fact are morphisms in the category of tropical curves. The key technical tool here is to define the notion of the \textit{stabilization} of a family of tropical curves when forgetting a mark.

In Section \ref{sec:intersection theory}, the necessary concepts of tropical intersection theory are recalled and adapted to the context necessary to define tropical $\psi$ classes. Basic properties of $\psi$ classes are then investigated.

Finally, Section \ref{sec:classes} contains two correspondence statements for tropical $\psi$ classes in genus one. Section \ref{sec:71} studies a one-dimensional family of well-spaced tropical stable maps of degree three to the tropical projective plane incident to  eight general  points. While the section has been written in a rather concise way to highlight the key ideas of the computation, a complete combinatorial description of the family is given in Appendix \ref{app:main}. Section \ref{sec:72} analyzes a family of one-dimensional spaces of of genus one admissible covers.

\subsection{Glossary of Notation}

This paper lives almost in its entirety in the tropical world. For this reason, we chose to omit the superscript {\it trop} from our notation; we have instead added the superscript {\it alg} anytime an algebro-geometric moduli space makes an appearance. We acknowledge this is not standard, but mantain it is a sound choice for this work. Here is a list of some of the notation used in the paper.
$$
\begin{array}{cl}
\T &  \mbox{[page \pageref{pr:t}]; the tropical affine line $\R \cup \{\infty\}.$}\\
\Rbar & \mbox{[page \pageref{pr:pline}]; the tropical projective line $\R \cup \{\pm \infty\}.$} \\
PL_X & \mbox{Definition \ref{def:tpl}; sheaf of piecewise linear functions on $X$.} \\
Aff_X & \mbox{Definition \ref{def:tropsp}; sheaf of  linear functions on $X$.} \\
Aff_{\mc C}(ks) & \mbox{Definition \ref{def:funwipo}; sheaf of   functions with  prescribed order along a section.} \\
\Omega^1_X & \mbox{Definition \ref{def:tropsp}; sheaf of  one forms on $X$. }\\
H_{\mc C} & \mbox{Definition \ref{def:harm}; sheaf of fiberwise harmonic functions.}\\
\Mgnbar{g,n}^{TPL} & \mbox{Definition \ref{def:tplstack}; stack of families of TPL curves. }\\
\Mgnbar{g,n} & \mbox{Definition \ref{def:moduli}; stack of families of tropical curves. }\\
\MgnMf{g,n} & \mbox{Definition \ref{def:mumfcusp}; open substack of families of Mumford curves.}\\
\Atlass{g,n} & \mbox{Definition \ref{def:atlas};  cone complex of  cycle rigidified curves}. \\
\GoodAtlas{g,n}& \mbox{Definition \ref{def:good}; restriction of $\Atlass{g,n}$ to {\it good} curves.}\\
\Atlas{g,n}& \mbox{Definition \ref{def:good}; restriction of $\Atlass{g,n}$ to Mumford curves.}\\
\end{array}
$$
We conclude by warning the reader of a slight abuse of notation we chose to adopt: it is common to denote by $\Mgnbar{g,n}$ the moduli space of families of curves with markings labeled by elements of the finite set $n = \{1, \ldots, n\}.$ When we have an additional mark playing a distinguished role we call it $\star$ and denote the indexing set $n\sqcup \{\star\}$ rather than $[n]\sqcup \{\star\}$.

\subsection{Acknowledgements}
R.C. is grateful for the support from the Simons collaboration grant 420720. H.M. is grateful for the support of the DFG collaborative research center SFB 1489.
We would like to thank Dhruv Ranganathan for answering our questions about well-spacedness, and Martin Ulirsch and Dimitry Zakharov for conversations and comments on the manuscript.
%\section{$\psi$ classes on $M_{1,1}^{trop}$}
%\section{String and Dilaton}
%\section{Intersection theory of $\psi$ classes.}

\renewcommand*{\thetheorem}{\arabic{section}.\arabic{theorem}}

\section{Tropical Spaces}
\label{sec-tropicalspaces}
This section introduces the category of spaces containing  the families of tropical stable curves we consider.
%The underlying spaces of our families of tropical stable curves will be what we call tropical spaces, which are
What we call tropical spaces  are generalizations of (abstract) tropical varieties, previously existing notions of tropical spaces, rational polyhedral spaces, and affine manifolds with singularities  \cite{MikhalkinICM,BIMS,BIMS,AllermannRau,MZeigenwave,Lefschetz,GrossSiebert,IKMZ,MR,Cartwright,ShawSurfaces}. 

\subsection{Polyhedral sets, \TPL{}-spaces, and affine structures}

A \emph{(rational) polyhedron} in $\R^n$ is a finite intersection of half spaces 
%\begin{equation*}
$
\{x\in \R^n \mid \langle m,x\rangle \leq a \}$,
%\end{equation*}
where $m\in (\Z^n)^\vee$ and $a\in\R$. Let \label{pr:t} $\T=\R\cup\{\infty\}$, topologized to be homeomorphic to a half-line. The space $\T^n$ has a natural stratification
\begin{equation}
\T^n=\bigsqcup_{I\subseteq\{1,\ldots,n\}} \T^n_I \ ,
\end{equation}
where the stratum
\begin{equation}
\T^n_I=\{(x_i)\in \T^n\mid x_i=\infty \text{ if and only if }i\in I\} 
\end{equation}
can be identified with $\R^{n-|I|}$. A (rational) polyhedron in $\T^n$ is the closure of a polyhedron in $\T^n_I$ for some subset $I\subseteq \{1,\ldots, n\}$. If $P\subseteq \T^n_I$ is a polyhedron, then the \emph{finite faces} of the polyhedron $\overline P$ in $\T^n$ are the closures of the faces of $P$, and the \emph{infinite faces} of $\overline P$ are the closures of the nonempty sets of the form $\overline P\cap \T^n_J$ for some $I\subsetneq J\subseteq \{1,\ldots, n\}$. The complement in $P$ of the union of its proper faces (finite or infinite) %of a polyhedron $P$ in $\T^n$ 
is called the \emph{relative interior of $P$}, denoted $\relint(P)$.
A \emph{polyhedral set} in $\T^n$ is a finite union of polyhedra. 

 Let $\Rbar=\R\cup\{\pm\infty\}$ \label{pr:pline}, topologized to be homeomorphic to a closed interval. A \emph{potentially infinite (integral) affine (linear) function} on an open subset $U$ of a polyhedral set in $\T^n$ is a continuous function $f\colon U \to \Rbar$ such that 
 %$f^{-1}\R$ is dense in $U$ and such that 
 $f$ is locally a restriction of a function of the form
\begin{equation*}
x\mapsto \langle m,x\rangle +a \ ,
\end{equation*}
where $m\in (\Z^n)^\vee$ and $a\in \Rbar$. 
We follow standard conventions to extend operations of addition and multiplication to include  $\pm \infty $, with only $\infty-\infty$ being undefined.
%$x\pm\infty=\pm\infty$ for all $x\in\R\cup\{\pm \infty\}$,  $0\cdot \pm\infty =0$, $k\cdot \pm\infty =\pm\infty$ for $k>0$, $k\cdot \pm\infty =\mp\infty$ for $k<0$, and $\infty-\infty$ is undefined. 
In particular, if $f(x)=\langle m,x\rangle +a$ and $x\in U\cap \T^n_I$, then either $m_i\geq 0$ for all $i\in I$ or $m_i\leq 0$ for all $i\in I$. 

A function is \emph{(integral) affine (linear)} if it is a potentially infinite affine function and has values in $\R$. Being affine is a local condition, so we obtain a sheaf of Abelian groups $\Aff_X$ of affine functions on an open subset $U$ of a polyhedral set in $\T^n$.

% and let $x\in U$. A \emph{local face structure} for $U$ at $x$ is a finite set $\Sigma$ of polyhedra in $\Rbar^n$ such that 
% \begin{enumerate}
% \item $|\Sigma|=\bigcup_{\sigma\in\Sigma}\sigma$ is a neighborhood of $x$ in $U$.
% \item if $\sigma\in\Sigma$, and $\tau$ is a face of $\sigma$, then $\tau\in\Sigma$, and
% \item if $\sigma,\tau\in\Sigma$, then $\sigma\cap \tau$ is either empty or a face of both $\sigma$ and $\tau$.
%  \end{enumerate}

Let $U\subseteq \T^n$ be an open subset of a polyhedral set. A closed subset $A\subseteq U$ of $U$ is \emph{locally polyhedral in $U$} if for every $x\in U$ there exists a polyhedral set $P$ in $\T^n$ and an open neighborhood $V$ of $x$ in $U$ such that $V\cap A= V\cap P$.

A \emph{piecewise (integral, affine) linear function} on an open subset $U$ of a polyhedral set in $\T^n$ is a continuous function $f\colon U\to \Rbar$ such that for every point $x\in U$ there exist polyhedra $P_1,\ldots,P_k$ in $\T^n$ with $U\cap \bigcup_{i=1}^k P_i$  a neighborhood of $x$ in $U$ and $f\vert_{U\cap P_i}$ a potentially infinite affine function on $U\cap P_i$ for all $1\leq i\leq k$.

Being piecewise linear is a local condition on $f$, so we obtain a sheaf $\PL_U$ of piecewise linear functions on $U$. Because piecewise linear function can have infinite values, $\PL_U$ is not a sheaf of Ablian groups. However, the subsheaf $\PL^\fin_U$ consisinting of all piecewise linear functions with finite values is a sheaf of Abelian groups. By definition, the sheaf $\Aff_U$ is a subsheaf of $\PL^\fin_U$.

\begin{definition}\label{def:tpl}

A \textbf{tropical piecewise linear space}, or \TPL{}-space for short, is a pair $(X,\PL_X)$ consisting of a paracompact Hausdorff topological space $X$ and a sheaf $\PL_X$ of $\Rbar$-valued continuous functions on $X$, the piecewise (integral, affine) linear functions, such that for every point $x\in X$ there exists a neighborhood $U$ of $x$, an open subset $V$ of a polyhedral set in $\T^n$ for some $n\in\N$, and a homeomorphism $\phi\colon U\to V$ that induces an isomorphism $\phi^{-1}\PL_V\to \PL_U$ via precomposition with $\phi$. The datum $U\xrightarrow{\phi} V$ is called a \textbf{chart} at $x$.

We denote by $\PL^\fin_X$ the sheaf of Abelian groups on $X$ whose sections are the piecewise linear function on $X$ with finite values.

A \textbf{morphism} between two \TPL{}-spaces $X$ and $Y$ is a continuous map $X\to Y$ that pulls back piecewise linear functions to piecewise linear functions. 
\end{definition}

\begin{example}
If $\Sigma$ is a cone complex or an extended cone complex \cite[\S 2]{ACP}, one obtains a sheaf $\PL_{\Sigma}$ of $\Rbar$-valued functions on $\Sigma$ by defining $\Gamma(U,\PL_\Sigma)$ for an open subset $U\subseteq \Sigma$ as the set of continuous functions $\phi:U\to \Rbar$ such that the restriction $\phi\vert_{\sigma\cap U}$ is piecewise linear  for all $\sigma\in \Sigma$. Then the pair $(\Sigma,\PL_\Sigma)$ is a \TPL{}-space.
\end{example}

\begin{remark}
Constant functions are  piecewise linear, so $\PL_X$  contains the constant sheaf $\R_X$ %associated to $\R$ 
as a subsheaf.
\end{remark}

\begin{definition}
\label{def:finite point}
A point $x$ of a \TPL{}-space $X$ is \textbf{finite} if for every function $f\in \Gamma(U,\PL_X)$ defined on an open neighborhood $U$ of $x$, the equality $f(x)=\infty$ implies that $f$ is constant on a neighborhood of $x$. The \textbf{irreducible components} of a \TPL{}-space $X$ are the closures in $X$ of the connected components of the set of finite points of $X$.
%\begin{equation}
%X\setminus \{x\in X\mid x\text{ is not finite}\} \ .
%\end{equation}

\end{definition}

\begin{definition}
A closed subset $A$ of a \TPL{}-space $X$ is called \textbf{locally polyhedral} if for every chart $X\supseteq U\xrightarrow{\phi}V\subseteq \T^n$ the image $\phi(A\cap U)$ is locally polyhedral in $V$ as defined above. 
%A sheaf $\mF$ of Abelian groups on $X$ is \textbf{constructible} if the stalk $\mF_x$ is a finitely generated free Abelian group for every $x\in X$, and there exists a filtration $\emptyset=A_0\subseteq A_1\subseteq \ldots\subseteq A_k=X$ of locally polyhedral subsets of $X$ such that $\mF\vert_{A_i\smallsetminus A_{i-1}}$  is locally constant for all $1\leq i\leq k$.
\end{definition}

% Let $U$ be an open subset of a polyhedral set in $\Rbar^n$, and let $\mF$ be a sheaf of Abelian groups on $U$. We say that $\mF$ is constructible, if at every $x\in U$ there exists a local face structure $\Sigma$ at $x$ such that $\mF|_{\relint(\sigma)}$ is constant of finite rank for every $\sigma\in \Sigma$. A sheaf $\mF$ of Abelian groups on a \TPL{} space $X$ is called constructible if at every $x\in X$ there exists a chart $U\xrightarrow{\phi}V$ such that $\phi_*\mF$ is constructible on $V$.  

\begin{definition}
Let $X$ be a \TPL{}- space.  A \textbf{polyhedron} in $X$ is a locally polyhedral subset $P$ of $X$, together with an affine structure $\Aff_P$ on $P$ %such that $(P,\Aff_P)$ 
such that there exists a chart $X\supseteq U\xrightarrow{\phi} V\subseteq \T^n$ with $P\subseteq U$ that maps $P$ isomorphically (as a tropical space) onto a polyhedron in $\T^n$.
\end{definition}

\begin{definition}
Let $X$ be a \TPL{}-space and let $x\in X$. A \textbf{local face structure} of $X$ at $x$ is a collection $\Sigma$ of polyhedra in $X$ 
\begin{enumerate}
    \item for every $\sigma\in \Sigma$ we have $x\in \sigma$ , 
    \item if $\tau$ is a face of $\sigma\in\Sigma$ such that $x\in \tau$, then $\tau\in \Sigma$, 
    \item for all $\sigma,\sigma'\in \Sigma$ we have $\sigma\cap \sigma'\in \Sigma$, 
    \item the set $\vert \Sigma\vert\coloneqq \bigcup_{\sigma\in\Sigma}\sigma$ is a neighborhood for $x$, and
    \item there exists a chart $U\xrightarrow{\phi}V\subseteq \T^n$ such that $\vert \Sigma\vert\subseteq U$ and such that $\phi$ maps every $\sigma\in\Sigma$ isomorphically (as a tropical space) onto a polyhedron in $\T^n$.
\end{enumerate}
% If $X$ is a tropical space, we also ask that
% \begin{enumerate}
%     \item [(6)] The inclusions $\sigma\hookrightarrow X$ are linear for all $\sigma\in\Sigma$. 
% \end{enumerate}
\end{definition}

\begin{definition} \label{def:tropsp}
An \textbf{affine structure} on a \TPL{}-space $X$ is a subsheaf of Abelian groups $\Aff_X$ of $\PL^\fin_X$ containing the constant sheaf $\R_X$.
%such that $\Omega^1_X\coloneqq \Aff_X/\R_X$ is constructible. 
A \TPL{}-space that is equipped with an affine structure is called a \textbf{tropical space}.
We call the sections of $\Aff_X$ \textbf{affine (linear) functions} on $X$ and the sections of $\Omega^1_X\coloneqq \Aff_X/\R_X$ \textbf{tropical one-forms on $X$}. 
\end{definition}

\begin{example}
If $P$ is a polyhedral set in $\T^n$, then $P$ together with the inclusion $\Aff_P\to \PL_P$ is a tropical space. 
\end{example}

\begin{example}
\label{ex:minimal and maximal affine structures}
Every \TPL{}-space $X$ has a minimal affine structure given by $\Aff_X=\R_X$ and a maximal affine structure given by $\Aff_X=\PL_X^\fin$.
\end{example}

\begin{example}
Let $\Sigma$ be a fan in $\R^n$. Then it is a polyhedral set in $\R^n$ and thus has an induced affine structure, making it a tropical space. Let $\overline\Sigma$ denote the tropical toric variety associated to $\Sigma$ \cite{KajiwaraTropicalToric,PayneLimit}. One obtains an affine structure on $\overline\Sigma$ by defining $\Gamma(U,\Aff_{\overline\Sigma})$ for an open subset $U\subseteq \overline\Sigma$ as the group of all piecewise linear functions $\phi\in \Gamma(U,\PL_{\overline\Sigma})$ that have finite values everywhere and satisfy $\phi\vert_{U\cap\Sigma}\in \Gamma(U\cap\Sigma, \Aff_\Sigma)$. We can thus consider tropical toric varieties as tropical spaces in a natural way. This construction generalizes to the weakly embedded extended cone complexes of \cite{GrossToroidal}. If $\Sigma$ is the fan of $\PP^n$ considered with its natural torus action, we denote the tropical toric variety $\overline \Sigma$ by $\T\PP^n$, the tropical projective $n$-space. Note that the underlying space of $\T\PP^1$ is precisely $\Rbar$.
\end{example}

{
\begin{remark}
This definition of affine structures is very general. As seen in Example \ref{ex:minimal and maximal affine structures}, the sheaf of affine functions can be arbitrarily close to the (very large) sheaf of all finite piecewise linear functions. In particular, affine structures do not necessarily admit a combinatorial description. In almost all examples that we consider, including all moduli spaces appearing in this paper, the sheaf of tropical one-forms is constructible in a suitable sense and therefore the affine structure can in fact be described combinatorially on the spaces of interest.
\end{remark}
}

\begin{definition}
A \textbf{morphism} between two tropical spaces $X$ and $Y$ is a morphism $f\colon X\to Y$ of \TPL{}-spaces such that the pull-back $\PL_Y\to f_*\PL_X$ maps $\Aff_Y$ into $f_*\Aff_X$. Sometimes we will call such a morphism a \textbf{linear morphism} from $X$ to $Y$ to stress that it respects the affine (linear) structure.
\end{definition}

If $A$ is a locally polyhedral subset of a \TPL{}-space $X$, then the restrictions of the functions in $\PL_X$ to $A$ define a sheaf $\PL_A$ on $A$, and $(A,\PL_A)$ is a $\TPL{}$-space. If $\Aff_X$ is an affine structure on $X$, then the restrictions of the functions in $\Aff_X$ to $A$ define an affine structure $\Aff_A$ on $A$.

\subsection{Fiber products of \TPL{}-spaces}

\begin{definition}
A morphism $f\colon X\to Y$ of \TPL{}-spaces is a \textbf{closed immersion} if $f$ maps $X$ homeomorphically onto a closed subset of $Y$, and the morphism 
%\begin{equation}
$f^{-1}\PL_Y\to \PL_X$
%\end{equation}
is surjective. 
\end{definition}

Every locally polyhedral  subset $Z$ of a \TPL{}-space $X$ has a unique \TPL{}-structure, the \emph{induced \TPL{}-structure}, making the inclusion $\iota\colon Z\to X$ a closed immersion. Namely, one defines $\PL_Z$ as the sheaf consisting of all functions on $Z$ that are locally restrictions of functions in $\PL_X$. It follows directly from the definitions that any closed immersion $f\colon Z\to X$ induces an isomorphism $Z\to f(Z)$, where $f(Z)$ is equipped with the induced \TPL{}-structure. 
Since the underlying maps of closed immersions are injective, they are monomorphisms in the category of \TPL{}-spaces. Moreover, a closed immersion $f\colon Z\to X$ satisfies the following universal property: a morphism of \TPL{}-spaces from $W$ to $X$ whose image is contained in $f(Z)$ determines a  morphism %of  \TPL{}-spaces 
from $W$ to $Z$ making the following diagram commutative:
$$
\xymatrix{& &W\ar[d] \ar@{.>}[dll]_{\exists \,!}  \\
Z \ar[rr]^f& &\hspace{0.5cm}f(Z) \subseteq X}
$$

%\renzo{I rewrote this line to make it sound more like a universal property (i.e. that there is a non-trivial completion of a diagram happening). Here is the old sentence: if $f\colon Z\to X$ is a closed immersion, then the morphisms of \TPL{}-spaces from any \TPL{}-space $W$ to $Z$ are precisely the morphisms of \TPL{}-spaces from $W$ to $X$ whose image is contained in $f(Z)$.}

\begin{proposition}
\label{prop:tpl-fiber products}
Let $p\colon X\to S$ and $q\colon Y\to S$ be morphisms of \TPL{}-spaces. Then the fiber product $X\times^{\TPL{}}_S Y$ in the category of \TPL{}-spaces exists. Furthermore, the underlying topological space of $X\times^{\TPL{}}_S Y$ coincides with the fiber product of the underlying topological spaces of $X$ and $Y$ over the underlying topological space of $S$.
\end{proposition}

\begin{proof}
 Given open covers $\{X_i\}_{i\in I}$, $\{Y_i\}_{i\in I}$ and $\{S_i\}_{i\in I}$ of $X$, $Y$, and $S$, respectively, such that $p(X_i)\cup q(Y_i)\subseteq S_i$ for all $i\in I$, the existence of $X\times^{\TPL{}}_S Y$ is equivalent to the existence of $X_i\times^{\TPL{}}_{S_i} Y_i$ for all $i\in I$. We may thus assume that there exist locally polyhedral sets $P\subseteq \T^m$, $Q\subseteq \T^n$, and $R\subseteq \T^k$ for some $m,n,k\in \N$ such that $X$ is an open subset of $P$, $Y$ is an open subset of $Q$, and $S$ is an open subset of $R$. After further shrinking $X$, $P$, $Y$, and $Q$, we may assume that $p$ and $q$ extends to morphisms $p'\colon P\to R$ and $q'\colon Q\to R$. It then suffices to show that the fiber product $P\times^{\TPL{}}_R Q$ exists. After replacing $P$ and $Q$ by their graphs with respect to $p'$ and $q'$, respectively, we may assume that $p$ is the restriction of the projection $\T^{m+k}\to \T^k$, and $q$ is the restriction of the projection $\T^{n+k}\to \T^k$. It follows immediately that the set-theoretic fiber product $P\times^{\mathrm{set}}_R Q$ is a locally polyhedral subset of $\T^{m+n+k}$. Since $\T^{m+n+k}$ is the fiber product $\T^{m+k}\times^{\TPL{}}_{\T^k} \T^{n+k}$ in the category of \TPL{}-spaces, the locally polyhedral subset $P\times^{\mathrm{set}}_R Q$ of $\T^{m+n+k}$, equipped with the induced \TPL{}-structure, is a fiber product in the category of \TPL{}-spaces. 

With the same reductions, the statement about the topology on fiber products on \TPL{}-spaces is reduced to the fact that  the topology of $\T^{m+n+k}$ coincides with the topology on the fiber product $\T^{m+k}\times^{\mathrm{top}}_{\T^k}\T^{n+k}$ taken in the category of topological spaces.
\end{proof}

\begin{remark}
The $\PL$-functions on a fiber product are in general not the sums of pull-backs of functions from the factors.
Indeed, if $p,q\colon \R^2\to \R$ are the two projections, then the morphism 
\begin{equation*}
p^{-1}\PL_\R\oplus ~ q^{-1}\PL_\R\to \PL_{\R^2}
\end{equation*}
 of sheaves is not an epimorphism. For example, the global $\PL$-function
\begin{equation*}
\R^2\to \Rbar, \;\; (x,y)\mapsto\min\{x,y\}
\end{equation*}
is not in the image, and neither are its germs at any point on the diagonal.
\end{remark}

\begin{convention}
\label{conv:TPL site}
We view the category of \TPL{}-spaces as a site by equipping it with the Grothendieck topology in which the coverings are the jointly surjective collections of  local isomorphisms.
\end{convention}

\begin{proposition}
\label{prop:TPL-base change of immersion is immersion}
Let $f\colon X\to Y$ be a morphism of \TPL{}-spaces, and let $\iota\colon Z\to Y$ be a closed immersion. Then the induced morphism $Z\times^{\TPL{}}_Y X\to X$ is a closed immersion. 
\end{proposition}

\begin{proof}
The preimage $f^{-1}(\iota(Z))$ is a locally polyhedral subset of $X$. If we equip it with the induced \TPL{}-structure, then the universal property of closed immersions implies  $f^{-1}(\iota(Z))=Z\times^{\TPL{}}_Y X$.
\end{proof}

\subsection{Fiber products of tropical spaces}

\begin{proposition}
\label{prop:fiber products of tropical spaces}
Let $p\colon X\to S$ and $q\colon Y\to S$ be morphisms of tropical spaces. Then the fiber product $X\times_S Y$ in the category of tropical spaces exists. Furthermore, the underlying \TPL{}-space of  $X\times_S Y$ coincides with $X\times^{\TPL{}}_S Y$ and the morphism of sheaves
\begin{equation}
q'^{-1}\Aff_X\oplus ~ p'^{-1}\Aff_Y\to \Aff_{X\times_S Y} \ ,
\end{equation}
where $p'\colon X\times_S Y\to Y$ and $q'\colon X\times_S Y\to X$ are the projection maps, is surjective.
\end{proposition}

\begin{proof}
Let $W=X\times^{\TPL{}}_S Y$ and let $p'\colon W\to Y$ and $q'\colon W\to X$ denote the projections (which are morphisms of \TPL{}-spaces). Define the affine structure $\Aff_W$ on $W$ as the image of the morphism of sheaves $q'^{-1}\Aff_X\oplus ~ p'^{-1}\Aff_Y\to \PL_W$. Then for every tropical space $Z$, a map $f\colon Z\to W$ is a morphism of tropical spaces if and only if it is a morphism \TPL{}-spaces and the composites $p'\circ f$ and $q'\circ f$ are morphisms of tropical spaces. It follows from this and the universal property of $W$ as a fiber product in the category of \TPL{}-spaces that $W$ is also the fiber product of $X$ and $Y$ over $S$ in the category of tropical spaces. By construction, it has all of the asserted properties.
\end{proof}

\begin{convention}
\label{conv:tropical site}
We view the category of tropical spaces as a site by equipping it with the Grothendieck topology in which the coverings are the jointly surjective collections of  local isomorphisms.
\end{convention}

\begin{definition}
A morphism $f\colon Z\to X$ of tropical spaces is a \textbf{closed immersion} if its underlying morphism of \TPL{}-spaces is a closed immersion and the morphism of sheaves
%\begin{equation}
$
f^{-1}\Aff_X\to \Aff_Z$
%\end{equation}
is an epimorphism.
\end{definition}

Given a locally polyhedral subset $Z$ of a tropical space $X$, there is a unique way to make $Z$ into a tropical space such that the inclusion $Z\to X$ is a closed immersion of tropical spaces. 
Similar as for \TPL{}-spaces, this defines a bijection between isomorphism classes of closed immersions into a given tropical space $X$ and locally polyhedral subsets of $X$. 

%\begin{definition}
%A morphism $f\colon X\to Y$ of tropical spaces is \textbf{proper}, or \textbf{open} if the underlying morphism of \TPL{}-spaces is proper, open, or a submersion, respectively.
%\end{definition}

\begin{proposition}
Let $f\colon X\to Y$ and $g\colon Y'\to Y$ be a morphisms of tropical spaces. If $f$ is a closed immersion, then the induced morphism $X\times_Y Y'\to Y'$ has the same property.
\end{proposition}

\begin{proof}
This follows directly from Proposition \ref{prop:TPL-base change of immersion is immersion} and Proposition \ref{prop:fiber products of tropical spaces}.
\end{proof}

\section{The stack of tropical stable curves}\label{sec:3}

Abstract tropical curves and their moduli spaces have been studied from various perspectives \cite{MZJacobians,MikhalkinICM,MikhalkinModuli,GathmannKerberMarkwig,ACP,GM,Tev,CapComparing,CapICM}. Here, we define tropical curves and their families suiting the context of Section \ref{sec-tropicalspaces}.

\subsection{Tropical curves}

\begin{definition}
A \textbf{\TPL{}-curve} is a purely one-dimensional \TPL{}-space $\Gamma$, together with a \textbf{genus function} $\gamma_\Gamma\colon \Gamma\to \Z^{\geq 0}$ with finite support. The \textbf{genus} of a compact \TPL{}-curve $\Gamma$ is given by $
%\dim H^1(\Gamma;\Q)
h^1(\Gamma)
+\sum_{p\in \Gamma}\gamma_\Gamma(p)$.
\end{definition}

The underlying space of a \TPL{}-curve $\Gamma$  is a topological graph. %For every point $p\in \Gamma$  then every local face structure $\Sigma$ at $p$ has the same number of one-dimensional strata incident to $p$. 
%More precisely, if $\Sigma(1)$ denotes the set of one-dimensional strata of $\Sigma$, and $T_p\Gamma\coloneqq \injlim_{\Delta} \Delta(1)$\renzo{direct limit overkill?}\andreas{how about: ``For every point $p\in \Gamma$ we define the set $T_p\Gamma$ of \emph{direction vectors} at $p$ as the set of germs of immersions $(\R_{\geq 0},0)\to (\Gamma,p)$'' ?} is the direct limit over all local face structures $\Delta$ at $p$, then the natural map $\Sigma(1)\to T_p\Gamma$ is a bijection. 
For every point $p\in \Gamma$ we define the set $T_p\Gamma$ of \emph{direction vectors} at $p$ as the set of germs of immersions $(\R_{\geq 0},0)\to (\Gamma,p)$. One may describe $T_p\Gamma$  concretely by picking a local face structure $\Sigma$ at $p$, and
identifying $T_p\Gamma$ with the set of one-dimensional strata of $\Sigma$.
We call an element of $T_p\Gamma$ a \emph{direction at $p$} and the number of directions at $p$ the \emph{valence of $\Gamma$ at $p$}, denoted  $val(p)$.

Given a function $m\in \Gamma(U,\PL_\Gamma^\fin)$ defined on a neighborhood $U$ of $p$ and a direction $v\in T_p\Gamma$ represented by a germ $(\R_{\geq 0},0)\to (\Gamma,p)$, it makes sense to take the slope of $m$ with respect to $v$ by pulling $m$ back to a neighborhood of zero in $\R_{\geq 0}$. We denote this slope by $d_v m$ and call it the \emph{slope of $m$ at $p$  along $v$}. We say that a function $m\in \Gamma(U,\PL_\Gamma^\fin)$ on an open subset $U$ of a \TPL{}-curve $\Gamma$ is \emph{harmonic} if at every $p\in U$ the sum of all slopes at $p$ is zero, that is if
\begin{equation}
\sum_{v\in T_p\Gamma}d_v m=0 \ .
\end{equation}
Since harmonicity is a local condition, harmonic functions on $\Gamma$ give a subsheaf of Abelian groups $\Harm_\Gamma$ of $\PL_\Gamma^\fin$.

%Let $d>0$ be an integer and $r>0$ be a positive real number. For $d=1$, we define $C_{r,d}=[0,r)$, whereas for $d\geq 2$ We define $C_{r,d}$ as the union in $\R^{d-1}$ of the open rays $\{te_i \mid 0\leq t< r\}$ for $0\leq i\leq d-1$, where $e_i$ is the $i$-th standard basis vector for $i>0$, and $e_0=-\sum_{i=1}^{d-1}e_i$. 
%%It is an open subset of a polyhedral set in $\Rbar^{n-1}$, so it is a tropical space. 
%We will always consider $C_{n,r}$ as tropical spaces, with both the tropical piecewise linear and the affine structure being induced from the ambient spaces.

\begin{definition}\label{def:tropcurve}
A \textbf{tropical curve} is a \TPL{}-curve equipped with an affine structure. The genus of a tropical curve is the genus of the underlying \TPL{}-curve. 
A point $p$ of a tropical curve $\Gamma$ is said to be \textbf{smooth} if one of the following conditions hold:
\begin{enumerate}
\item  $\gamma_\Gamma(p)>0$, $p$ is finite, and $\Aff_{\Gamma,p} = \R$;
\item  $\gamma_\Gamma(p)=0$, $p$ is  finite, and $\Aff_{\Gamma,p}=\Harm_{\Gamma,p}$;
\item  $\gamma_\Gamma(p)=0$, $p$ is infinite with $\val(p)=1$, and $\Aff_{\Gamma,p}= \R ~(=\Harm_{\Gamma,p})$.
\end{enumerate}
We say $\Gamma$ is smooth if all of its points are smooth. A point $p\in\Gamma$ is a \textbf{node at infinity} if $\gamma_\Gamma(p)=0$, $\val(p)=2$, and $\PL_{\Gamma,p}^\fin=\R$.
%and every piecewise-linear function $m$ defined in a neighborhood of $p$ with $m(p)\neq \infty$ is constant in a neighborhood of $p$. 
We say that $\Gamma$ is nodal if every point in $\Gamma$ is either smooth or a node at infinity. On a nodal tropical curve, the affine structure is uniquely determined by the underlying \TPL{}-curve. We say that a \TPL{}-curve is smooth (nodal) if it is the underlying \TPL{}-curve of a smooth (nodal) tropical curve.
\end{definition}

\begin{remark}
For $d\geq 2$, define $C_{r,d}$ as the union in $\R^{d-1}$ of the open rays $\{te_i \mid 0\leq t< r\}$ for $0\leq i\leq d-1$, where $e_i$ is the $i$-th standard basis vector for $i>0$, and $e_0=-\sum_{i=1}^{d-1}e_i$; then, a point $p$  with $\gamma_\Gamma(p)=0$ of a tropical curve $\Gamma$ is smooth if and only if $p$ has a neighborhood isomorphic to either $C_{r,d}$ for some $r \in \R_{>0}$ and $d\geq 2$, or to $(s,\infty]$ for some $s\in \R$. A point of $\Gamma$ is a node at infinity if an only if it is not smooth and has a neighborhood isomorphic to $(r,\infty]\times \{\infty\} \cup \{\infty\} \times (r,\infty]\subseteq \T^2$ for some $r\in \R$.
\end{remark}
%
%
%\begin{definition}
%A \emph{tropical curve} is a tropical space $\Gamma$ whose underlying \TPL{}-space is a \TPL{}-curve, together with a function $\gamma\colon \Gamma\to \N$ with finite support. A point $p$ of a tropical curve $\Gamma$ is said to be \emph{smooth} if there exists a positive real number $r>0$ and a neighborhood $U$ of $p$, such that either $U$ is isomorphic to  $(r,\infty]$, $U$ is isomorphic to  $C_{r,d}$ for some $d\in \N_{\geq 2}$, or $U$ is isomorphic to $C_{r,1}$ and $g(x)> 0$. We say $\Gamma$ is smooth if all of its points are smooth. A point in $x$ is a \emph{node} if it is not smooth, we have $g_{\Gamma,x}=0$, and $x$ has a neighborhood that is isomorphic to the polyhedral set $(r,\infty]\times \{\infty\} \cup \{\infty\} \times (r,\infty]\subseteq \Rbar^2$ for some $r\in \R$. We say that $\Gamma$ is nodal if every point in $\Gamma$ is either smooth or a node. The \emph{genus} of a compact tropical curve $\Gamma$ is given by $\dim H^1(\Gamma;\Z)+\sum_{x\in \Gamma}g(x)$.
%\end{definition}

\begin{definition}
Let $I$ be a finite set. An  \textbf{$I$-marked semi-stable}  tropical (resp.\ \TPL{}) curve is a tuple $(\Gamma,(s_i)_{i\in I})$ such that $\Gamma$ is a compact connected nodal tropical (resp.\ \TPL{}) curve, and $i\mapsto s_i$ is a bijection from $I$ to the set of infinite smooth points of $\Gamma$. An $I$-marked semi-stable tropical (resp.\ \TPL{}) curve is called \textbf{stable} if every irreducible component of $\Gamma$ of genus zero contains at least three special points (marked points or nodes at infinity), and every irreducible component of $\Gamma$ of genus one contains at least one special point.
\end{definition}

For the purposes of this paper a graph $G$ is a tuple 
\[
G=(V(G),F(G),L(G)\subseteq E_\infty(G)\subseteq E(G),v_G,e_G,\gamma_G) \ ,
\]
where $V(G)$, $F(G)$, $E(G)$, are finite sets, the sets of vertices, flags, and edges, respectively; $v_G\colon F(G)\to V(G)$ is a map assigning to every flag its adjacent vertex, $e_G\colon F(G)\to E(G)$ is a surjective map with fibers of cardinality at most two, and $\gamma_G\colon V(G)\to \N$ is a map assigning to every vertex its genus. The set $L(G)$ of legs of $G$ is the subset of $E(G)$ consisting of edges whose fiber under $e_G$ has cardinality one. The subset $E_\infty(G)$ that comes with the datum $G$ is the set of unbounded edges of $G$, and we always require that $L(G)\subseteq E_\infty(G)$. We denote the set of bounded edges of $G$ by $E_b(G)=E(G)\setminus E_\infty(G)$. For every vertex $v$ we denote by $T_v G=v_G^{-1}\{v\}$ its set of adjacent flags and by  $\val(v)=\#T_v G$ the valence of $v$. A graph is \textit{stable} if it is connected and for every vertex $v\in V(G)$ one has
\begin{equation}
\val(v)+2\gamma_G(v)\geq 3 \ .
\end{equation}
The {\it genus} $g(G)$ of $G$ is given by 
\begin{equation}
g(G)=\#E(G)-\#L(G)-\#V(G)+1+\sum_{v\in V(G)}\gamma_G(v) \ .
\end{equation}
For any finite set $I$, an $I$-marked graph is a graph $G$ together with a bijection $I\to L(G)$. {In the geometric realization of a graph, we assume the legs to be compact segments, but we do not consider the lone endpoint to be a vertex.}

\begin{definition}
Let $\Gamma$ be an $I$-marked stable \TPL{}-curve. The \textbf{combinatorial type of $\Gamma$} is the unique stable graph $G$ whose geometric realization is homeomorphic (as $I$-marked weighted spaces) to the underlying topological graph of $\Gamma$ in such a way that the unbounded edges of $G$ are precisely those edges that contain an infinite point of $\Gamma$. In particular, the set of vertices of $G$ consists of the  points of $\Gamma$ which have strictly positive genus or valence strictly greater than two.

Given a stable graph $G$ and a bounded edge $e\in E_b(G)$, one can \emph{specialize} $e$ in two ways. Namely, one can \emph{stretch} $e$, which  adds $e$ to $E_\infty(G)$, or one can \emph{contract} $e$, which is achieved by identifying the endpoints of $e$ if it is not a loop, and by deleting $e$ and increasing the genus of its adjacent vertex by one if it is a loop. Any graph obtained from $G$ by a sequence of edge specializations is called a specialization of $G$. A \emph{morphism} between two stable graphs is given by a sequence of graph specializations and graph isomorphisms. %After imposing the obvious commutators on the free category thus created, one can write any morphism of stable graphs as the composite of one specialization and one isomorphism. %\andreas{maybe someone can express this in a less technical way}. 

%whose set of vertices is given by
%\begin{equation}
%V(\ct)=\{p\in\Gamma\mid \val(p)>0\}\cup \{p\in \Gamma\mid \gamma_\Gamma(p)>0\} \ ,
%\end{equation}

%whose set of flags is given by
%\begin{equation}
%F(\ct)=\bigsqcup_{p\in V(\ct)} T_p\Gamma \ ,
%\end{equation}
%and whose set of edges $E(\ct)$ are the connected %components of $\Gamma\setminus V(\ct)$.
%The map $v_G\colon F(G)\to V(G)$ is defined by mapping $T_p\Gamma$ to $p$, and $e_g\colon F(G)\to E(G)$ is defined by mapping a flag represented by a polytope $\sigma\subseteq\Gamma$ that is part of a local face structure at a vertex to the unique edge containing $\relint(\sigma)$. Finally, $\gamma_G$ is the restriction of $\gamma_\Gamma$ to $V(G)$, and the marking $I\to L(G)$ maps $i$ to the unique edge containing $x_i$.

\end{definition}

\subsection{Families of tropical curves}

The starting point for defining families of tropical curves is the cone over an $I$-marked stable graph $G$, as introduced in \cite[Definition 4.1]{CCUW}. In our case, we choose the base to equal the extended cone $\overline\sigma_G =\Rbar_{\geq 0}^{E_b(G)}$. We recall the construction in this context.

%Let $I$ be a finite set and let $G$ be an $I$-marked stable graph. Let $\sigma_G= \R_{\geq 0}^{E(G)}$. 

For every edge $e\in E_b(G)$, denote by $l_e$ the coordinate on $\sigmabar_G$ corresponding to $e$. 
For every $v\in V(G)$, define $\sigmabar_{G,v}=\sigmabar_G$. For every bounded edge $e\in E_b(G)$, define $\sigmabar_{G,e}$ as the extended cone of
\begin{equation}
\sigma_{G,e}=\{(x,y)\in \R_{\geq 0}^{E_b(G)} \times \R_{\geq 0} \mid y\leq l_e(x) \} \ .
\end{equation}
% The cone $\sigmabar_{G,e}$ has two faces that are naturally isomorphic to $\sigmabar_G$, namely the subsets
If $v_1$ and $v_2$ are the vertices of $e$, then the two faces of $\sigmabar_{G,e}$ given by the extended cones of the faces
%\begin{equation}
$
\{(x,0)\in \sigma_{G,e}\}
$
%\end{equation}
and
%\begin{equation}
$
\{(x,y)\in \sigma_{G,e}\mid y=l_e(x)\} 
$
of $\sigma_{G,e}$
%\end{equation}
can be naturally identified with $\sigma_{G,v_1}$ and $\sigma_{G,v_2}$, respectively.
For every leg $l\in L(G)$ define
\begin{equation}
\sigmabar_{G,l}=\sigmabar_G\times \Rbar_{\geq 0} \ .
\end{equation}
The cone $\sigmabar_{G,v_G(l)}\cong \sigmabar_G\times \{0\}$ is naturally identified with a face of $\sigmabar_{G,l}$.
Finally, for every $e\in E_\infty(G)\setminus L(G)$, we define $\overline\sigma_{G,e}= \overline\sigma_G\times (\Rbar_{\geq 0} \times\{\infty\}\cup \{\infty\}\times \Rbar_{\geq 0})$. If $v_1$ and $v_2$ are the vertices of $e$, then the two subsets $\overline\sigma_G\times\{(0,\infty)\}$ and $\overline\sigma_G\times\{(\infty,0)\}$ of $\overline\sigma_{G,e}$ can be naturally identified with $\sigma_{G,v_1}$ and $\sigma_{G,v_2}$.

Denote by $\pi_G\colon\mc C_G\to \overline\sigma_G$ the \TPL{}-space over $\overline \sigma_G$ obtained by gluing the cones $\sigmabar_{G,v}, \sigmabar_{G,e}, \sigmabar_{G,l}$
%for every vertex $v\in V(G)$ the cone $\sigmabar_{G,v}$, and for every edge $e\in E(G)$ the cone ${\sigmabar_{G,e}}$
according to how edges and vertices are glued in $G$. Define a function $\gamma_{\mc C_G}\colon \mc C_G \to \N$ by 
\begin{equation}
\gamma_{\mc C_G}(x)=\sum_{v\in V(G)} \gamma_G(v)\cdot\mathbf{1}_{\overline{\sigma}_{G,v}}(x) + \dim(\injlim_{x\in U}H_1(U\cap \pi_G^{-1}\relint(\sigma_G) ;\Q
)) \ ,
\end{equation}
where $\mathbf 1_{\overline{\sigma}_{G,v}}$ denotes the indicator function associated to $\overline{\sigma}_{G,v}$ and the direct limit is taken over all neighborhoods $U$ of $x$.

By construction, the fiber of $\pi_G\colon\mc C_G\to \overline\sigma_G$ over a point $x\in \relint(\sigma_G)$, together with the restriction of $\gamma_{\mc C_G}$, is a stable \TPL{}-curve whose combinatorial type can be identified with $G$ in such a way that the length of the edge of $\pi_G^{-1}\{x\}$ corresponding to $e\in E_b(G)$ is given by $l_e(x)$. 

{
Any morphism $G\to G'$ of stable graphs induces a morphism $\overline \sigma_G\to \overline \sigma_{G'}$ of extended cones that identifies $\overline\sigma_G$ with a (potentially infinite) face of $\overline\sigma_{G'}$. By construction, the \TPL{}-space $\mc C_G$ can be naturally identified with the \TPL{}-space $\mc C_{G'}\times^{\text{\TPL{}}}_{\overline \sigma_{G'}}\overline\sigma_G$ in a way that respects the projection to $\overline \sigma_G$.
}
%\rcolor{Either add a concrete example or a specific reference to CCUW.} \andreas{Figure 9 in CCUW seems to fit best, altough it's not quite the picture we want.}

%\begin{definition}
%Let $B$ be a tropical space. A \emph{family of $n$-marked tropical stable curves of genus $g$ over $B$} $(\pi\colon \mc C\to B, (s_i)_{1\leq i\leq n}, g)$, consists of a proper submersion $\pi\colon\mathcal C\to B$, $n$-sections $s_1,\ldots, s_n$ of $\pi$, and a function $g\colon \mc C\to \N$ such that
%\begin{enumerate}
%\item for every $b\in B$ there exists a local face structure $\Sigma$ of $B$ at $b$ such that for every $\sigma\in \Sigma$ there exists a combinatorial type $C$ of $n$-marked genus-$g$ tropical curves and a linear morphism $f_\sigma\colon \sigma\to  \overline\sigma_C$ such that, if $\iota_\sigma\colon \sigma\to B$ denotes the inclusion, there exists an isomorphism
%\begin{equation}
%\iota_\sigma^*\mc C\cong f_\sigma^* \mc C_C \ .
%\end{equation}
%\item
%For every $b\in  B$
%\begin{equation}
%\label{equ:exact sequence}
%0\to \Omega^1_{B,b}\to \Omega^1_{\mathcal C}|_{\mathcal C_b} \to \Omega^1_{\mathcal C_b}\to 0
%\end{equation}
%is an exact sequence of sheaves on $\mc C_b$.
%\end{enumerate}
%\end{definition}

\begin{definition}
Let $B$ be a \TPL{}-space and $n$ a non-negative integer. A family of \textbf{$n$-marked genus-$g$ stable \TPL{}-curves over $B$} is a tuple $(\pi\colon \mc C\to B, (s_i)_{i\in [n]}, \gamma_{\mc C}$) consisting of a morphism $\pi\colon\mc C\to B$, a section  $s_i$ of $\pi$ for every $i\in [n]$, and a \textbf{genus function}  $\gamma_{\mc C}\colon \mc C\to \N$ such that
\begin{enumerate}
\item for every $b\in B$, the fiber $\mc C_b\coloneqq \{b\}\times^{\text{\TPL{}}}_B\mc C$, together with the restrictions of the sections and the genus function, is an $n$-marked genus-$g$ stable \TPL{}-curve.
\item for every $b\in B$ there exists a local face structure $\Sigma$ of $B$ at $b$ such that for every $\sigma\in \Sigma$ there exists an $n$-marked genus-$g$ stable graph $G_\sigma$, a linear morphism $f_\sigma\colon \sigma\to  \sigmabar_{G_\sigma}$ with
\begin{equation}
f_\sigma(\relint(\sigma))\subseteq \relint(\sigma_{G_\sigma}) \ ,
\end{equation}  
and  an isomorphism
\begin{equation}
\chi_\sigma\colon \mc C\times^{\text{\TPL{}}}_B\sigma \xrightarrow\cong \mc C_{G_\sigma}\times^{\text{\TPL{}}}_{\sigmabar_{G_\sigma}} \sigma 
\end{equation}
respecting the induced genus functions and the induced sections.
\end{enumerate}
%When $I = [n]$, an $I$-marked curve is commonly referred to as an \textbf{$n$-marked curve}.
The datum $(\Sigma, (G_\sigma,f_\sigma,\chi_\sigma)_{\sigma\in\Sigma})$ is called a \textbf{\TPL{}-trivialization} of the family $\mc C\to B$ at $b$. %\andreas{I added the isomorphism to the datum of the trivialization.}
\end{definition}

%
%\begin{definition}
%Let $B$ be a tropical space. A \emph{family of $n$-marked tropical stable curves of genus $g$ over $B$} $(\pi\colon \mc C\to B, (s_i)_{1\leq i\leq n}, g)$, consists of a proper submersion $\pi\colon\mathcal C\to B$, $n$-sections $s_1,\ldots, s_n$ of $\pi$, and a function $g\colon \mc C\to \N$ such that for every $b\in B$ the fiber $\mathcal C_b=\pi^{-1}\{b\}$ defines an $n$-marked tropical stable curve $(\mc C_b, (s_i(b))_{1\leq i\leq b},g\vert_{\mc C_b})$ of genus $g$, and such that 
%\begin{equation}
%\label{equ:exact sequence}
%0\to \Omega^1_{B,b}\to \Omega^1_{\mathcal C}|_{\mathcal C_b} \to \Omega^1_{\mathcal C_b}\to 0
%\end{equation}
%is exact. Furthermore, there should exists an injection of $\Omega^1_{\mc C/B}\coloneqq \Omega^1_{\mc C}/\pi^{-1}\Omega^1_B$ into a sheaf $\mc F$, such that for every $x\in \mc C$ the stalk of $\mc G\coloneqq \mc F/\Omega^1_{\mc C/B}$ is free of rank $g(x)$, and such that $\pi_*\mc G$ is locally free of rank $g$. \andreas{On which choices does $\pi_*\mc G$ depend? It would be a good candidate for a `tropical Hodge bundle'}.
%\end{definition}

It  follows from the definition that if $\pi\colon \mc C\to B$ is a family of $n$-marked genus-$g$ stable \TPL{}-curves and $B'\to B$ is a morphism of \TPL{}-spaces, then $\mc C\times_B B'\to B'$, with the induced sections and genus function, is a family of $n$-marked genus-$g$ stable \TPL{}-curves again. Therefore, the following definition makes sense:

\begin{definition}\label{def:tplstack}
We denote by $\Mgnbar{g,n}^{\text{\TPL{}}}$ the category fibered in groupoids over the category of $\TPL{}$-spaces with 
\begin{equation*}
\Mgnbar{g,n}^{\text{\TPL{}}}(B) = \{\text{families of } n\text{-marked genus-}g \text{ stable \TPL{}-curves over } B \} 
\end{equation*}
for any \TPL{}-space $B$. Considering the category of \TPL{}-spaces as a site, as outlined in Convention \ref{conv:TPL site}, the fibered category $\Mgnbar{g,n}^{\text{\TPL{}}}$  is a stack.
\end{definition}

Since the affine functions on smooth tropical curves are all harmonic, the following definition and lemma will be of great importance to us:

\begin{definition} \label{def:harm}
Let $\pi\colon \mc C\to B$ be a family of stable \TPL{}-curves. The \textbf{sheaf of fiberwise harmonic functions}, denoted by $\Harm_{\mc C}$, is the subsheaf of $\PL_{\mc C}^\fin$ whose sections, when restricted to any fiber, are harmonic.
\end{definition}

\begin{lemma}
\label{lem:exact sequence for harmonic functions}

Let $\pi\colon \mc C\to B$ be a family of stable \TPL{}-curves, let $x\in \mc C$ be a point with $\gamma_{\mc C}(x)=0$, and let $b=\pi(x)$. Then the sequence
\begin{equation} \label{eq:harmes}
0\to \PL_{B,b}^\fin\to \Harm_{\mc C,x} \to\raisebox{.7ex}{$\Harm_{\mc C_b,x}$}\Big/\raisebox{-.7ex}{$\R$}\to 0 
\end{equation}
is exact. 
%\renzo{[...] Where are we using the $g=0$ hypothesis here...}\andreas{Without that assumption we would not get a neighborhood of $x$ with contractible fibers. Not only that, but without the assumption the statement is actually false.}
\end{lemma}

\begin{proof}
 It is clear that \eqref{eq:harmes} is a complex and that it is exact on the left. To see exactness in the middle, given a function $\phi$ on an open neighborhood $W$ of $x$ such that $\phi\vert_{W_b}$ is constant, we show that $\phi$ must be constant along fibers for all points $c$ in a neighborhood of $b$.
Using a \TPL{}-trivialization at $b$, one may construct for every $d\in T_x\mc C_b$ a  local section $s_d$ of $\pi$ such that that $s_d(b)$ is a point arbitrarily close to $b$ in the direction given by $d$, and  $s_d$ does not intersect the locus of vertices of the fibers of $\mc C$. We denote by $C$ a connected neighborhood of $b$ contained in the common locus of definition of the sections $s_d$. 

Let $U$ be the connected component of $\left(W\cap \pi^{-1}(C) \right)\setminus \bigcup_{d\in T_x\mc C_b}s_d(B)$ containing $x$. After potentially shrinking $C$, the fiber $\overline U_c$ of a point $c\in C$ is a tree whose leaves are precisely the points $s_d(c)$, $d\in T_x\mc C_b$. Therefore, the restriction $\phi\vert_{U_c}$ is uniquely determined by the slopes of $\phi\vert_{W_c}$ at the points $s_d(c)$, $d\in T_x\mc C_b$, and by its value at $s_{d_0}(c)$ for some fixed $d_0\in T_x\mc C_b$.
For any $d\in T_x\mc C_b$, the slope of $\phi\vert_{W_c}$ at $s_d(c)$ is the same for all $c\in C$ since $C$ is connected and slopes take values in a discrete set. As the slope of $\phi\vert_{W_b}$ is zero everywhere, it follows that $\phi$ is constant on all fibers in $C$, and therefore that $\phi=\pi^*(s_{d_0}^*\phi)$ on $U$. This proves the exactness in the middle.

With all notation as above, we observe that given a harmonic function $\xi$ on $W_b$, we can define a fiberwise harmonic function $\tilde{\xi}$ on $\overline U$ by defining it on a fiber $\overline U_c$ as the unique harmonic function with $\tilde{\xi}(s_{d_0}(c))=\xi(s_{d_0}(b))$ and slope at $s_d(c)$ equal to the slope of $\xi$ at $s_d(b)$ for all $d\in T_x\mc C_b$. By construction, one then has $\tilde{\xi}\vert_{U_b}=\xi\vert_{U_b}$, which proves exactness on the right.
\end{proof}

\begin{definition}\label{def:family}
Let $B$ be a tropical space. A \textbf{family of $n$-marked  genus-$g$ stable tropical curves over $B$} is a family $\pi\colon\mc C\to B$ of $n$-marked genus-$g$ stable \TPL{}-curves, together with an affine structure on $\mc C$ making $\pi$ linear, and such that for every $b\in B$ the fiber $\mc C_b= \{b\}\times_B \mc C$ is a stable tropical curve; further we require that the sequence
\begin{equation}
\label{equ:exact sequence}
0\to \Omega^1_{B,b}\to \Omega^1_{\mc C}|_{\mc C_b} \to \Omega^1_{\mc C_b}\to 0 
\end{equation}
is an exact sequence of sheaves on $\mc C_b$. 
%As before, when we use the terminology \textbf{$n$-marked}, we are specifying $I = [n]$.  
 \end{definition}

\begin{remark}
The exactness of the sequence displayed in \eqref{equ:exact sequence} is equivalent to asking that every affine function on $\mc C$ whose restriction to the fiber over $b$ is constant is locally a pull-back of an affine function defined on a neighborhood of $b$ in $B$.
\end{remark}

{
\begin{remark}
As a consequence of Lemma \ref{lem:exact sequence for harmonic functions}, every family of stable \TPL{}-curves $\pi\colon\mc C\to B$ can be  equipped with affine structures on base and total space to yield a family of stable tropical curves. More precisely, the lemma proves that if $\widetilde\Harm_{\mc C}$ denotes the affine structure on $\mc C$ with stalks
\begin{equation*}
\widetilde\Harm_{\mc C,x} = 
\begin{cases}
\Harm_{\mc C,x} & \text{ if } \gamma_{\mc C}(x)=0 \\
\PL^\fin_{B,\pi(x)} & \text{ if }\gamma_{\mc C}(x)>0 \,
\end{cases}    
\end{equation*}
then $(\mc C,\widetilde\Harm_{\mc C})\to (B, \PL^\fin_B)$ is a family of stable tropical curves. Families of stable tropical curves obtained in this way are not particularly useful from the point of view of tropical intersection theory; there are "too many" affine functions on the base to have positive-dimensional tropical cycles on $(B,\PL^\fin_B)$ {(see \eqref{def:balancing} for the definition of balancing used to define tropical cycles on tropical spaces)}.%\renzo{I think this last sentence needs to be clarified, otherwise it is going to upset the reader. I suppose we are trying to say that these are some kind of maximal structures that have too many functions to produce any meaningful intersection theory? Can you add a clarifying sentence? \acolor{See if it's less upsetting now. I didn't write this in the first place because we only define tropical cycles on tropical spaces much later}}
\end{remark}
}

%It turns out that every family of stable \TPL{}-curves can be equipped with natural affine structures on base and total space that makes it a family of stable tropical curves. This is the content of the following lemma:

\begin{example} \label{ex:tropfam}
For any integer $a$, we describe a family $\mc C^a$ of elliptic curves over $\T\PP^1 = \Rbar$; refer to Figure \ref{fig:famell} to keep up with the notation. As a cone complex, this is the union of four extended plane  orthants $[0, \infty]\times [0, \infty]$, denoted  $L_{\pm}, T_{\pm}$. We denote by $(u_\pm, v_\pm)$ (integral) coordinate functions for $L_\pm$, and by $(x_\pm, y_\pm)$  coordinate functions for $T_\pm$; we denote by $x$ a coordinate on the base $\T\PP^1$.

\begin{figure}[tb]
\begin{center}
\begin{tikzpicture}[scale=.4]

\fill[ fill = gray!30] (0,0) -- (0,8.5)--(8.5,8.5) -- (8.5,0) --(0,0);
\fill[ fill = gray!30] (0,0) -- (0,8.5)--(-8.5,8.5) -- (-8.5,-8.5) --(0,0);
\draw[very thick] (8.5,0)--($(8.5,0)!.8!(8.5,8.5)$);
\draw[dashed,very thick] ($(8.5,0)!.8!(8.5,8.5)$) --(8.5,8.5);
\draw[very thick] (0,0)--($(0,0)!.8!(0,8.5)$);
\draw[dashed,very thick] ($(0,0)!.8!(0,8.5)$) --(0,8.5);
\draw[very thick] (-8.5,-8.5)--($(-8.5,-8.5)!.8!(-8.5,8.5)$);
\draw[dashed,very thick] ($(-8.5,-8.5)!.8!(-8.5,8.5)$) --(-8.5,8.5);

\draw[very thick, black] (-8.5,8.5) -- (8.5,8.5) ;
% \draw[->] (-9,0)-- (10,0);
% \draw[->] (0,-8)-- (0,9);
\draw[very thick, fill = white] (0,0)-- (8.5,0)--($.97*(8.5,-4.05)$) --(0,0);
\fill[thick, fill = white] (8.5,-2) ellipse (1 and 2);
\draw[very thick] (8.5,0) arc (90:-90:1 and 2);
\draw[very thick,dashed] (8.5,0) arc (90:270:1 and 2);

\draw[very thick, fill = white] (0,0)-- (-8.5,-8.5)--($.95*(-8.2,-12.5)$) --(0,0);
\draw[very thick, fill = white] (-8.5,-10.5) ellipse (1 and 2);
 
 \begin{scope}
 \clip (0,0) circle (4.8cm);
 \foreach \i in {-8,...,8}
      \foreach \j in {0,...,8}{
        \draw[fill = black] (\i,\j) circle(1pt);};

 \foreach \i in {-8,...,0}
      \foreach \j in {\i,...,0}{
        \draw[fill = black] (\i,\j) circle(1pt);};
 \end{scope}
 
\draw[very thick] (-8.5,-8.5) -- (0,0) -- (8.5,0) ;

\node at(5,5) {$T_+$};
\node at(-5,5) {$T_-$};
\node at(5,-1.5) {$L_+$};
\node at(-5,-6.3) {$L_-$};
\node at(9.5,8.5) {$s_1$};
% \node[red] at(9.5,0.5) {$H_0$};
\node at(0.5,-0.8) {$v$};

\end{tikzpicture}

\caption{A neighborhood of the fiber $\mc C^a_0$ of the family $\mc C^a$ in the case $a=1$. Using the affine structure, the finite part of $T_+\cup T_-$, indicated in gray, can be embedded into the plane. This embedding depends on $a$, as illustrated by the integral lattice depicted in a neighborhood of $v$.
%The light gray part shows the local embedding dictated by the sheaf of affine functions. In this case, we have $a = 1$. In red we see the loci $H_0, H_\infty$.
}
\label{fig:famell}
\end{center}
\end{figure}
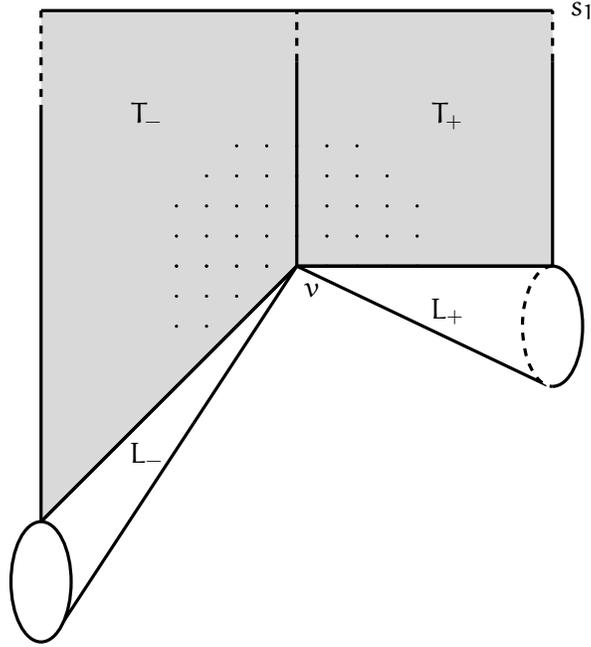

The finite one dimensional faces of  $L_-$ are identified, and glued with the horizontal axis of $T_{-}$ (and similarly for the cones labeled by $+$). The finite vertical axes of  $T_-$ and $T_+$ are identified.
We define a genus function $\gamma_{\mc C^a}: \mc C^a \to \ZZ$ to have value one at the point $v$ of $\mc C^a$ obtained  gluing the four $(0,0)$-vertices  of the extended cones, and zero elsewhere.

The map  $\pi: \mc C^a \to \T\PP^1 = [-\infty, \infty]$
 defined by
\begin{align}
    \pi_{|T_-}(x_-, y_-) &=  -x_-, &
     \pi_{|T_+}(x_+, y_+) &= x_+, \nonumber \\
      \pi_{|L_-}(u_-, v_-) &= -u_--v_-,&
        \pi_{|L_+}(u_+, v_+) &= u_++v_+,
\end{align}
 has  tropical elliptic curves for fibers. 

To make the extended cone complex $\mc C^a$ into a tropical space, we endow it with a sheaf of affine linear functions. Let ${\widetilde{\Aff}_{\mc C^a}}$ denote the shubsheaf of $\PL^\fin_{\mc C^a}$ consisting of all piecewise linear function whose restriction to any fiber is harmonic and whose pull-backs to $T_{\pm}$ and $L_{\pm}$ are affine, where we consider $[0,\infty]\times [0,\infty]$ with its standard affine structure. One quickly checks that with this affine structure, the projection $\pi$ is linear, all fibers of $\pi$ are stable tropical curves, and the sequence \eqref{equ:exact sequence} is exact away from the central fiber $\mc C^a_0$. The sequence is not exact on the central fiber because the horizontal slopes of a function in ${\widetilde{\Aff}_{\mc C^a}}$  defined in a neighborhood of a point $x\in \mc C^a_0$ in the positive (on $T_+$) and negative (on $T_-$) directions are independent. To fix this, one needs to replace ${\widetilde{\Aff}_{\mc C^a}}$ by a subsheaf $\Aff_{\mc C^a}$ in which the transition from $T_+$ to $T_-$ is specified:
\begin{equation}\label{tf}
x_+ = -x_- \ \ \ \ \ \ \  y_+= y_- - ax_-,
\end{equation}
for some $a\in \Z$; different values of $a$ yield different affine structures on the same underlying \TPL{}-space, that may be interpreted as different local embeddings of $T_-\cup T_+$ into $\T^2$: the value $a$ measures the difference of the slopes of the images of the horizontal axes, as shown in Figure \ref{fig:famell}. %\andreas{I made this a little more concise. Old version is commented out.}

% Since any neighborhood of $x$ intersects both $T_-$ and $T_+$, germs of affine functions can be presented in two different ways. The datum on how to transition from one presentation to the other must be specified:
% \begin{equation}\label{tf}
% x_+ = -x_- \ \ \ \ \ \ \  y_+= y_- - ax_-,
% \end{equation}
% for some $a\in \Z$; different values of $a$ yield different affine structures on $\mc C$, that may be interpreted as different local embeddings of $T_-\cup T_+$ into $\T^2$: $a$ measures the difference of the slopes of the images of the horizontal axes, as shown in Figure \ref{fig:famell}.

% From the description of the stalks of $\Aff_{\mc C^a}$, it is clear that the sequence \eqref{equ:exact sequence} is exact, thus making $\pi: \mc C^a \to \T\PP^1$ into a family of tropical curves. 
\end{example}

\begin{lemma}
\label{lem:pull-back of family is family}
Let $\pi \colon \mc C\to B$ be a family of $I$-marked genus-$g$ stable tropical curves and let $f\colon B'\to B$ be a morphism of tropical spaces. If $\mc C'=B'\times_B \mc C$ and $\pi'\colon \mc C'\to B'$ is the projection, then $\pi'$, together with the induced sections and genus function, is a family of $I$-marked genus-$g$ stable tropical curves.
\end{lemma}

\begin{proof}
It suffices to prove for every $b\in B'$ the exactness of the sequence displayed in \eqref{equ:exact sequence}. Let $x\in \mc C'_b$ and let $\phi$ be an affine function defined in a neighborhood of $x$ in $\mc C'$ such that $\phi$ is constant on a neighborhood of $x$ in $\mc C'_{b}$. Let $f'\colon \mc C'\to \mc C$ denote the projection. By Proposition \ref{prop:fiber products of tropical spaces}, we may assume that $\phi= f'\phantom{}^*\varphi$ for an affine function $\varphi$ defined in a neighborhood of $f'(x)$ in $\mc C$. By assumption, $\varphi$ is constant on a neighborhood of $f'(x)$ in $\mc C_{f(b)}$. Since   $\pi\colon \mc C\to B$ is a family of tropical curves, it follows that $\varphi=\pi^*\chi$ for some affine function $\chi$ defined on a neighborhood of $f(b)$ in $B$. It follows that $\phi=\pi'\phantom{}^*(f^*\chi)$ in a neighborhood of $x$, finishing the proof.
\end{proof}

\begin{example}
Let $\Z\ltimes\Z$ act on $\R^2$ via $(a,b).(x,y)=(a+x,(-1)^a(b+y))$ and let $C=\R^2/\Z\ltimes\Z$ be the quotient (homeomorphic to a Klein bottle), equipped with the induced affine structure. Let $B=\R/\Z$, also equipped with the induced affine structure, let $\pi\colon C\to B$ denote the morphism with  $\pi(\overline{(x,y)})=\overline x$, and let $s\colon B\to C$ denote the section of $\pi$ with $s(\overline x)=\overline{(x,0)}$. Attaching to $C$ a copy of $\Rbar_{\geq 0}\times B$ along $s(B)$ we obtain an isotrivial family $\mc C\to B$ of one-marked genus one stable \TPL{}-curves over $B$. Define an affine structure on $\mc C$ by declaring a piecewise linear function $\phi$ on $\mc C$ to be affine if it is harmonic on all fibers,  $\phi\vert_{C\setminus s(B)}$ is affine on $C$, and $\phi\vert_{\Rbar_{>0}\times B}$ is affine on $\Rbar_{\geq 0}\times B$. One checks that with this affine structure on the total space, $\mc C\to B$ is an isotrivial family of one-marked genus one stable tropical curves. Because $C$ is not homeomorphic to a trivial $S^1$-bundle over $B$, the family $\mc C\to B$ is not isomorphic to a trivial family, neither as a family of \TPL{}-curves, nor as a family of tropical cruves.  However, the pull-back  of $\mc C\to B$ along a connected double cover $B'\to B$ is isomorphic to a trivial family.
\end{example}

With Lemma \ref{lem:pull-back of family is family} established, we can make the following definition:

\begin{definition}\label{def:moduli}
We denote by $\Mgnbar{g,n}$ the category fibered in groupoids over the category of tropical spaces with 
\begin{equation*}
\Mgnbar{g,n}(B) = \{\text{families of } n\text{-marked genus-}g\text{ stable tropical curves over } B \} 
\end{equation*}
for any tropical space $B$. Considering the category of tropical spaces as a site, as outlined in Convention \ref{conv:tropical site}, the fibered category $\Mgnbar{g,n}$ is a stack.
\end{definition}

\subsection{Representability of the diagonal}

Recall that for any category $\mc M$ fibered in groupoids over a base category $\mc S$ with finite products, the following three conditions are equivalent:
\begin{enumerate}
\item The diagonal $\Delta_{\mc M}\colon \mc M\to \mc M\times_{\mc S}\mc M$ is representable.
\item For any two objects $U$ and $V$ of $\mc S$ and morphisms $U\to \mc M$ and $V\to \mc M$, the fiber product $U\times_{\mc M} V$ is representable by an object in $\mc S$.
\item For any object $U$ of $\mc S$ and any two elements $\xi_1$ and $\xi_2$ of $\mc M(U)$, the presheaf $\mathrm{Isom}_U(\xi_1,\xi_2)$
on $(\mc S/U)$ that assigns to $T\xrightarrow f U$ the set of isomorphisms in $\mc M(T)$ from $f^*\xi_1$ to $f^*\xi_2$ is representable.
\end{enumerate}

\begin{proposition}
\label{prop:representability of diagonal TPL-version}
For every pair $g,n \in \Z_{\geq 0}$ of nonnegative integers with $2g-2+n>0$, the stack $\Mgnbar{g,n}^{\text{\TPL{}}}$ over the category of \TPL{}-spaces has representable diagonal.
\end{proposition}

\begin{proof}

We need to show that given two families $\mc C_1\to B$ and $\mc C_2\to  B$ of  $n$-marked genus-$g$ stable \TPL{}-curves over a base $B$, the functor $\mathrm{Isom}_B(\mc C_1,\mc C_2)$ is represented by a \TPL{}-space over $B$. Let $b\in B$, and for $i\in \{1,2\}$ let $(\Sigma_i,(G_{\sigma,i},f_{\sigma,i},\chi_{\sigma,i}))_{\sigma\in \Sigma_i})$ be a \TPL{}-trivialization of $\mc C_i\to B$ at $b$. After taking a common refinement of the local face structures $\Sigma_1$ and $\Sigma_2$ at $b$, we may assume that $\Sigma_1=\Sigma_2$, and we denote this local face structure by $\Sigma$. Since \TPL{}-spaces can be glued along open subsets, we may assume that $B=\vert \Sigma\vert$.

Consider the set $\Delta$ consisting of all pairs $(\sigma,\phi)$, where $\sigma\in\Sigma$ and $\phi\colon G_{\sigma,1}\to G_{\sigma,2}$ is an isomorphism of marked graphs. If $\tau$ is a face of $\sigma$, then $G_{\tau,1}$ and $G_{\tau,2}$ are obtained from $G_{\sigma,1}$ and $G_{\sigma,2}$, respectively, via edge specializations. An isomorphism $\phi\colon G_{\sigma,1}\to G_{\sigma,2}$ induces an isomorphism $G_{\tau,1}\to G_{\tau,2}$ if and only if $\phi$ maps contracted (resp.\ stretched) edges to contracted (resp.\ stretched) edges. We define a partial order $\preceq$ on $\Delta$ by stipulating that $(\tau,\varphi)\preceq (\sigma,\phi)$ if and only if $\tau$ is a face of $\sigma$ and $\varphi$ is induced by $\phi$. For every $\delta = (\tau,\varphi)\in \Delta$, define $\sigma_{\delta}\coloneqq \tau$. Let $T$ be the \TPL{}-space obtained by gluing the cones $\sigma_\delta$ for every $\delta\in \Delta$  according to the partial order on $\Delta$, that is we define $T$ as
\begin{equation}
T=\injlim_{\delta\in\Delta} \sigma_\delta \ .
\end{equation}
For every element $\delta=(\tau,\varphi)$ of $\Delta$, the isomorphism $\varphi$ defines an isomorphism $\varphi_G\colon\overline\sigma_{G_{\tau,1}}\to \overline\sigma_{G_{\tau,2}}$. The set
\begin{equation}
P_\delta=\{x\in \tau\mid \varphi_G(f_{\tau,1}(x))=f_{\tau,2}(x)\}
\end{equation}
is a polyhedron in $\tau=\sigma_\delta$. Let
\begin{equation}
I=\bigcup_{\delta\in \Delta} P_\delta \subseteq T \ ,
\end{equation}
and let $g\colon I\to B$ denote the natural morphism of \TPL{}-spaces induced by the projection $T\to \vert\Sigma\vert$.  By construction, the isomorphisms $\chi_{\sigma,i}$ induce an isomorphism $\chi\colon g^*\mc C_1\to g^*\mc C_2$ of \TPL{}-spaces such that for $\delta=(\tau,\varphi)\in \Delta$ and $x\in \relint(\delta)$ the combinatorial type of the restriction of $\chi$ to the fiber over $x$ is given by $\varphi$.

To finish the proof, we need to show that $(I\xrightarrow g B,\chi)$ represents $\mathrm{Isom}_B(\mc C_1,\mc C_2)$. Let $B'\xrightarrow h B$ be a morphism of \TPL{}-spaces, and let $\zeta\colon h^*\mc C_1\to h^*\mc C_2$ be an isomorphism of families of stable \TPL{}-curves. Let $b'\in B'$ and let $\Theta$ be a local face structure at $b'$ such that $h$ maps the polyhedra of $\Theta$ linearly into polyhedra of $\Sigma$. For $\theta\in \Theta$, let $\sigma\in \Sigma$ be the minimal polyhedron in $\Sigma$ containing $h(\theta)$. Then for every $x\in \relint(\theta)$, the fiber $(h^*\mc C_i)_x$ has combinatorial type $G_{\sigma,i}$. Furthermore, the isomorphism $(h^*\mc C_1)_x\to (h^*\mc C_2)_x$ induced by $\zeta$ is induced by the same isomorphism $\zeta_\theta\colon G_{\sigma,1}\to G_{\sigma,2}$ of stable graphs for all $x\in \relint(\theta)$. It follows that the lift $k_\theta\colon\theta\to P_{(\sigma,\zeta_\theta)}$ of $\theta\xrightarrow{h\vert_\theta} \sigma$ is the unique lift of $\theta\xrightarrow{h\vert_\theta} B$ to $I$ via which $\chi$ induces the restriction $\zeta_\theta\colon (h^*\mc C_1)_\theta\to (h^*\mc C_2)_\theta$ of $\zeta$ to $\theta$. By the uniqueness of $k_\theta$, the morphisms $k_\theta$ glue to a unique lift $k\colon B'\to I$ of $h$ via which $\chi$ induces $\zeta$.
\end{proof}

\begin{theorem}
\label{thm:representable diagonal}
For every pair $g,n\in \Z_{\geq 0}$ of non-negative integers with $2g-2+n>0$, the diagonal of the stack $\Mgnbar{g,n}$ is representable.
\end{theorem}

\begin{proof}
We need to show that given two families $\mc C_1\to B$ and $\mc C_2\to  B$ of  $n$-marked genus-$g$ stable tropical curves over a base $B$, the functor $\mathrm{Isom}_B(\mc C_1,\mc C_2)$ is represented by a tropical space over $B$. Since families of stable tropical curves are families of stable \TPL{}-curves, we can apply Proposition \ref{prop:representability of diagonal TPL-version} and obtain a morphism of \TPL{}-spaces $J\xrightarrow {g} B$ and a universal isomorphism $\chi\colon g^*\mc C_1\to g^*\mc C_2$. We equip $J$ with the minimal affine structure such that $g$ becomes linear, that is we let $\Aff_J$ consist of all functions that are locally pull-backs of affine functions on $B$. Denote $g^*\mc C_i$ by $\mc C'_i$, where the pull-back is taken in the category of tropical spaces, and let $\pi'_i\colon \mc C'_i\to J$ be the projections. Let $y\in \mc C_1'$, and let $l'$ be an affine function on $\mc C_2'$ defined on a neighborhood of $\chi(y)$. Since the restriction $(\mc C_1')_{\pi_1'(y)} \to (\mc C_2')_{\pi_1'(y)}$ of $\chi$ to the fiber over $\pi_1'(y)$ is linear, there exists an affine function $l$ on a neighborhood of $y$ in $\mc C'_1$ such that the germs of the restrictions of $l$ and $\chi^*l'$ to $(\mc C_1')_{\pi_1'(y)}$ at $y$ coincide. Since both $l$ and $\chi^*l'$ are harmonic on all fibers, by Lemma \ref{lem:exact sequence for harmonic functions} there exists a piecewise linear function $m$ on a neighborhood of $\pi'_1(y)$ such that $\chi^*l'-l$ coincides with $\pi'_1\phantom{}^*m$ on a neighborhood of $y$. Let $\Aff_I$ be the subsheaf of $\PL_{J}^\fin$ generated by $\Aff_J$ and all piecewise linear functions $m$ obtained this way, let $I$ be the tropical space obtained by replacing $\Aff_J$ by $\Aff_I$, and let $f\colon I\to B$ be the linear morphism induced by $g$. Then by construction, $\chi$ induces an isomorphism $\varphi\colon f^*\mc C_1\to f*\mc C_2$ of families of stable tropical curves. It follows directly from the universal property of $J$ and $\chi$ and the construction of $\Aff_I$ that $(I,\varphi)$ represents $\mathrm{Isom}_B(\mc C_1,\mc C_2)$.
\end{proof}

{
\begin{example}
\label{ex:diagonal is representable}
Let $a,b\in \Z$. Example \ref{ex:tropfam} yields two families $\mc C^a$ and $\mc C^b$ over $\T\PP^1$ of one-marked genus-one stable tropical curves. Let $I= \Isom_{\T\PP^1}(\mc C^a,\mc C^b)$,  $f\colon I\to \T\PP^1$  the structure morphism, and  $\chi\colon f^*\mc C^a\to f^*\mc C^b$  the universal isomorphism. For every $x\in \T\PP^1\setminus \{0\}$, there are  two isomorphisms $\mc C^a_x\to \mc C^b_x$, whereas there exists  one isomorphism $\mc C^a_0\to \mc C^b_0$. According to the proof of Proposition \ref{prop:representability of diagonal TPL-version}, the underlying \TPL{}-space of $I$ is given by two copies of $\T\PP^1$ glued together at their respective origins. Let $0'$ denote the unique preimage under $f$ of $0$. From the description of the affine structures on $\mc C^a$ and $\mc C^b$ in Example \ref{ex:tropfam} and from the description of the affine structure on $I$ in the proof of Theorem \ref{thm:representable diagonal} one concludes that the restriction $I\setminus \{0'\}\to \T\PP^1$ of $f$ is a local isomorphism. To determine the affine structure at $0'$ denote by $\phi_c$, for $c\in\{a,b\}$, the function on a neighborhood in $\mc C^c$ of the finite genus-zero points of $\mc C^c_0$ whose restriction to $T_+$ is given by $y_+$, notation as in Example \ref{ex:tropfam}. Denoting by $f^*\phi_c$ the pull-back of $\phi_c$ to $f^*\mc C^c$, the two functions $f^*\phi_a$ and $\chi^*f^*\phi_b$ are both defined on a neighborhood in $f^*\mc C^a$ of the finite genus-zero points of the central fiber $(f^*\mc C^a)_{0'}$, and their restrictions to the central fiber have the same slope. Therefore, the difference $f^*\phi_a-\chi^*f^*\phi_b$ determines an affine function $m$ on a neighborhood of $0'$ in $I$. The horizontal slopes of $\phi_c$ are $0$ over $\Rbar_{\geq 0}$ and $-c$ over $\Rbar_{\leq 0}$. It follows that $m$ has slope $0$  on the two rays of $I$ mapping to $\Rbar_{\geq 0}$ and slope $b-a$ on the two rays of $I$ mapping to $\Rbar_{\leq 0}$. The pair of functions $(x,m)$ defines a  map $g:I\to S \subset \T\PP^1\times \T\PP^1$, with
\begin{equation}
    S= \Rbar_{\geq 0} 
    \begin{pmatrix}
    1\\0
    \end{pmatrix}
    \cup 
    \Rbar_{\geq 0}
    \begin{pmatrix}
    -1 \\ b-a
    \end{pmatrix}
\end{equation}
such that $f$ is given by the composite of $g$ with the projection to the first coordinate, $g$ is a double cover away from $0'$, and the affine functions on $I$ are  the pull-backs via $g$ of the affine functions on $S$. %Informally, we see that $I$ is obtained by doubling the rays of $S$.
\end{example}
}

%\subsection{The Mumford-locus}\renzo{Hannah: does the tropical community use any term to denote tropical curves like this? The term Mumford-curve comes from number theory.}

\subsection{Affine functions on families of curves}

{
In order to describe the affine structures of the moduli spaces $\Mgnbar{g,n}$, we need to have good control over the affine structures on the total spaces of families of stable tropical curves. By definition,  affine functions on the total space of a family are fiberwise harmonic. In what follows, we investigate under which conditions a fiberwise harmonic function is affine. 
}

\begin{lemma}
\label{lem:local structure at interior of edge}
Let $\pi\colon\mc C\to B$ be a family of stable tropical curves, and let $x\in \mc C$ be a point that is contained in an open edge of $\mc C_{\pi(x)}$. Then there exists an $\epsilon>0$ and an open neighborhood $U$ of $\pi(x)$ such that $x$ has a neighborhood in $\mc C$ which is isomorphic to $U \times (0,\epsilon)$ over $B$. 
\end{lemma}

\begin{proof}
Using a \TPL{}-trivialization, we find a neighborhood $W$ of $\pi(x)$ and a $\delta>0$ such that $x$ has an open neighborhood $V$ in $\mc C$ allowing an isomorphism $f\colon V\to W \times (0,\delta)$ of \TPL{}-spaces over $B$. Equipping $W \times (0,\delta)$ with the affine structure induced by $f$, call it $\Aff'$, the map $f$ becomes an isomorphism of tropical spaces over $B$. Since $\pi\colon \mc C\to B$ is a family of tropical curves, there exists an integral affine function $\phi$ at $x$ that has slope one on $(0,\delta)\cong \{\pi(x)\} \times (0,\delta)$. After potentially shrinking $V$, $W$ and $\delta$, we can assume that $\phi$ is defined all of $W \times (0,\delta)$. Because $\phi$ is harmonic on all fibers, it follows that $\phi$ has slope one on all fibers $\{w\} \times (0,\delta)$ for $w\in W$. Therefore, $\phi$ induces a linear map
\begin{equation}
W \times (0,\delta)\xrightarrow{\id_W\times \phi} W \times \R
\end{equation}
that is an isomorphism of \TPL{}-spaces onto an open subset $V'$ of $W\times \R$. Since the affine structure $\Aff'$ is generated by $\phi$ and the pull-backs of affine function from $W$, it follows that $\id_W \times\phi$  is an isomorphism of tropical spaces onto $V'$. The proof is concluded by observing that  $(\pi(x),\phi(f(x)))$ has a neighborhood in $V'$  isomorphic to $U\times (0,\epsilon)$ for some  $0<\epsilon<\delta$ and some  neighborhood $U$ of $\pi(x)$ in $W$.
\end{proof}

\begin{lemma}
\label{lem:affineness is a clopen condition on fibers}
Let $\pi\colon \mc C\to B$ be a family of stable tropical curves, let $U\subseteq \mc C$ be an open subset such that $\gamma_{\mc C}\vert_U=0$, let $\phi\in \Gamma(U,\Harm_{\mc C})$, and let $b\in B$. Then the set
\begin{equation}
V=\{x\in \mc C_b\cap U\mid \phi_x\in \Aff_{\mc C, x}\}
\end{equation}
is open and closed in $\mc C_b\cap U$.
\end{lemma}

\begin{proof}
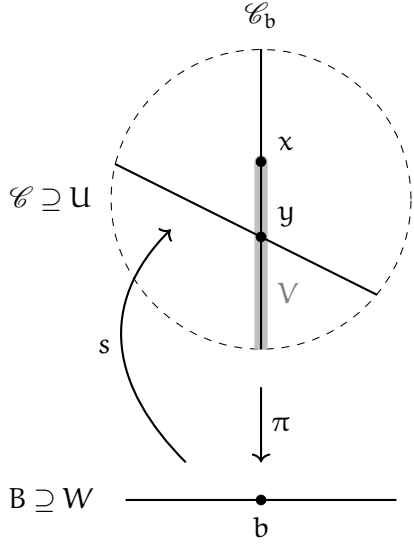
\begin{figure}
    \centering
    \begin{tikzpicture}[thick]
    \begin{scope}
        \node at (-2.8,0) {$\mc C \supseteq U$};
        \node at (0,2) [label= {above:$\mc C_b$}] {};
    
        \clip(0,0) circle (2);
        \coordinate (s) at (-3,1);
        \coordinate (x) at (0,.5) {};
        \draw [dashed] (0,0) circle (2);
        \draw [line width=5pt, line cap=round, gray!50](x)--
            node [pos=.7,right] {\color{gray} $V$}
        (0,-2);
        \draw (0,2)--(0,-2);
        \node [fill=black, circle, inner sep=.05cm, label=10:$x$] at (x) {};
        \node [fill=black, circle, inner sep=.05cm, label=10:$y$] (y) at (0,-.5){};
        \draw (s)--($(s)!2!(y)$);
    \end{scope}
    
    \draw [->] (0,-2.5)--
        node [right] {$\pi$}
    (0,-3.5);
    
    \draw[->]   (-1,-3.5) ..
     node [left]{$s$}
    controls (-2,-2.5) and  ($(s)!.6!(y) -(1,1.5)$) .. ($(s)!.6!(y)-(0,.5)$);
    
    \begin{scope}[yshift=-4cm]   
        \node at (-2.8,0) {$B\supseteq W$};
        \draw (-1.8,0)--(1.8,0);
        \node [fill=black, circle, inner sep=.05cm, label=below:$b$] at (0,0) {};   
    \end{scope}

    \end{tikzpicture}
    \caption{By what we prove in Lemma \ref{lem:affineness is a clopen condition on fibers}, the set $V$ does in fact contain all of $\mc C_b\cap U$ in the situation illustrated here. We draw $x$  in the boundary of $V$ only to better visualize the idea of the proof.}
    \label{fig:clopen}
\end{figure}

Refer to Figure \ref{fig:clopen} throughout this proof.
It is clear that $V$ is open in $\mc C_b \cap U$. To show that it is closed, let $x\in \overline V\cap U$. By
Definition \ref{def:tropcurve}, since $\phi$ restricts to a harmonic function on $\mc C_b$, it defines a one-form on $\mc C_b$; by the short exact sequence of cotangent sheaves \eqref{equ:exact sequence}, such a form admits a lift, that is there exists $m_x\in \Aff_{\mc C,x}$ such that the restrictions of $m$ and $\phi$ to $\mc C_b$ agree in a neighborhood of $x$ in $\mc C_b$. By Lemma \ref{lem:exact sequence for harmonic functions}, there exits $\varphi_b\in \PL_{B,b}^\fin$ such that $\phi_x= m_x+ (\pi^*\varphi)_x$. We need to show that $\varphi_b\in \Aff_{B,b}$. Since $x\in \overline V$, there exists a point $y\in V$ that is contained in the interior of an edge of $\mc C_b$ and such that $m$ is defined at $y$ and $\phi_y-m_y=(\pi^*\varphi)_y$. By Lemma \ref{lem:local structure at interior of edge}, there exists a linear section $s\colon W\to U$ of $\pi$ in a neighborhood $W$ ob $b$ such that $s(b)=y$. We see that
\begin{equation}
\varphi_b=s^*(\pi^*\varphi)_y=s^*\phi_y-s^*m_y\in \Aff_{B,b}. 
\end{equation}
%which implies that
%\begin{equation}
%\phi_x=(\pi^*\varphi)_x+m_x\in \Aff_{B,b} \ .
%\end{equation}
\end{proof}

\begin{lemma}
\label{lem:characterization of affineness in terms of section and harmonicity}
Let $\pi\colon \mc C\to B$ be a family of stable tropical curves and let $s\colon B\to \mc C$ be a linear section of $\pi$ such that $s(c)$ is a smooth point of $\mc C_c$ for all $c\in B$. Moreover, let $b\in B$ such that $\gamma_{\mc C}(s(b))=0$, and let $\phi\in \Harm_{\mc C,s(b)}$. Then $\phi\in \Aff_{\mc C, s(b)}$ if and only if $s^*\phi\in \Aff_{B,b}$.
\end{lemma}

\begin{proof}
The only if part is clear, so assume that $s^*\phi\in \Aff_{B,b}$. By the exact sequence for cotangent sheaves \eqref{equ:exact sequence}, there exists $m\in \Aff_{\mc C,s(b)}$ such that the restrictions $\phi\vert_{\mc C_b}$ and $m\vert_{\mc C_b}$ coincide. By Lemma \ref{lem:exact sequence for harmonic functions}, there exists a piecewise  linear function $\varphi\in \PL_{B,b}$ such that $\phi -m =\pi^*\varphi$. Since
\begin{equation}
\varphi=s^*\pi^*\varphi=s^*\phi-s^*m\in \Aff_{B,b} \ ,
\end{equation}
it follows that $\phi\in \Aff_{\mc C,s(b)}$.
\end{proof}

\begin{definition}\label{def:funwipo}
Let $\pi\colon \mc C\to B$ be a family of stable tropical curves and let $s\colon B\to \mc C$ be a linear section of $\pi$ such that $s(b)$ is a smooth genus-zero point of $\mc C_b$ for all $b\in B$. For an integer $k\in \Z$, we define the \textbf{sheaf of functions with prescribed order $k$ along $s$}, denoted $\Aff_{\mc C}(ks)$:  its sections over an open subset $U\subseteq \mc C$ are functions $m\in \Gamma(U,\PL_{\mc C})$, possibly with value $\pm\infty$ on $s(B)$, that are affine away from the support of $s$, that is
\begin{equation}
m\vert_{U\setminus s(B)}\in \Gamma(U\setminus s(B), \Aff_{\mc C}) \ ,
\end{equation}
 and such that for all $b\in s^{-1}U$ we have 
\begin{equation}
\sum_{v\in T_{s(b)}\mc C_b} d_v (m\vert_{\mc C_b}) = k \ .
\end{equation}
Since the constant sheaf $\R$ acts on $\Aff_{\mc C}(ks)$ additively, we can take the quotient $\Omega^1_{\mc C}(ks)\coloneqq \Aff_{\mc C}(ks)/\R$, the 
\textbf{sheaf of tropical one-forms with prescribed order $k$ along $s$}.
\end{definition}

If $r\in \Gamma(U,\Aff_{\mc C}(k s))$ and $m\in \Gamma(U,\Aff_{\mc C})$, then $m+r\in \Gamma(U,\Aff_{\mc C}(k s))$. This endows $\Aff_{\mc C}(ks)$ with an action by $\Aff_{\mc C}$. The following two propositions shows that  $\Aff_{\mc C}(ks)$ is in fact an $\Aff_{\mc C}$-torsor. 

%\begin{definition}
%Let $\pi\colon \mc C\to B$ be a family of stable tropical curves, let $s\colon B\to \mc C$ be a linear section of $\pi$ such that $s(b)$ is a smooth genus-zero point of $\mc C_b$ for all $b\in B$, and let $k\in \Z$. Then we define $\Omega^1_{\mc C}(ks)$ as the quotient of 
%$\Aff_{\mc C}(ks)$ by $\R$. 
%\end{definition}

\begin{proposition}
\label{prop:aff(ks) is a torsor}
Let $\pi\colon \mc C\to B$ be a family of stable tropical curves, let $s\colon B\to \mc C$ be a linear section of $\pi$ such that $s(b)$ is a smooth genus-zero point of $\mc C_b$ for all $b\in B$, and let $k\in \Z$. Moreover, let $U\subseteq \mc C$ be an open subset and let $l,l'\in \Gamma(U,\Aff_{\mc C}(ks))$. Then there exists a unique $m\in \Gamma(U,\Aff_{\mc C})$ such that $l'=l+m$.
\end{proposition}

\begin{figure}
    \centering
    \begin{tikzpicture}[thick]
        \begin{scope}
            \node at (-1,2) {$U$};
            \coordinate (tl) at (0,0);
            \coordinate (tr) at (4,-.5);
            \coordinate (ul) at (0,4.5);
            \coordinate (ur) at (4,5);
            \coordinate (s) at ($(tl)!.2!(ul)$);
            \coordinate (t2l) at ($(tl)!.6!(ul)$);
            \coordinate (t2r) at ($(tr)!.6!(ur)$);
        
             %%% V filling
            \fill[gray!30] (t2l)--(t2r)--(ur)--(ul)--cycle;

            %%% t
             \draw (t2l)--
                node[at start, left]{$t$}
             ($(t2l)!.8!(t2r)$);
               
             \draw [dashed] ($(t2l)!.8!(t2r)$)--
                node[at end, inner sep=.05cm, circle, fill=black, label=right:$y$]{}
             (t2r);

             %%% s
             \draw (s)--
                node[at start,left]{$s$}
             ($(s)!.8!(ur)$);
             \draw [dashed] ($(s)!.8!(ur)$)--
                node[at end, inner sep=.05cm, circle, fill=black, label={right:$x\in s(N)$}]{}
             (ur);
             
             %%% u
             \draw (ul)--
                node[at start,left]{$u$}
             ($(ul)!.8!(ur)$);
             \draw [dashed] ($(ul)!.8!(ur)$)--(ur);
             
             %%% vertical
             %\draw (0,0)--(0,3.5);
             %\draw[dashed] (0,3.5)--(0,4.5);
             
             \draw ($(tr)!.2!(ur)$)--
             ($(tr)!.8!(ur)$);
             \draw[dashed] ($(tr)!.8!(ur)$)--(ur);

             %%%% arrow
             \draw [->] (2,0)--
                node[auto]{$\pi$}
             (2,-1);
        \end{scope}
        
        \begin{scope}[yshift=-1.5cm]
            \node at (-1,0) {$W$};
            \draw (0,0)--(3,0);
            \draw[dashed] (3,0)--
                node[at end, inner sep=.05cm, circle, fill=black, label=right:$b\in N$]{}
            (4,0);
        \end{scope}
    \end{tikzpicture}    
    \caption[]{In the situation depicted here, we have $N=\{b\}$. The affine function $m = l-l'$ is constant along all fibers
    in the gray region, and therefore extends across the section $u$ at infinity.}
    \label{fig:tors}
\end{figure}
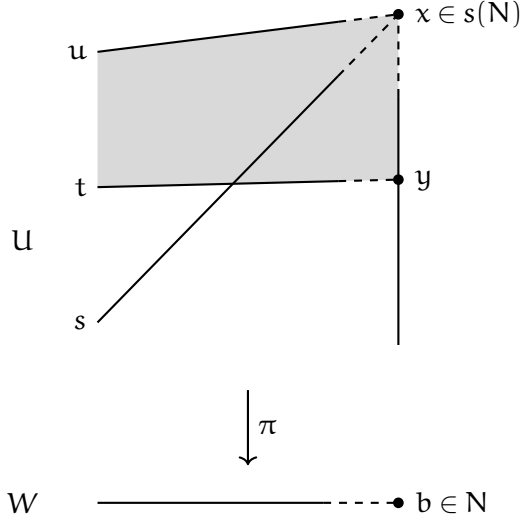
\begin{proof}
%This statement is a simple consequence of Lemma \ref{lem:affineness is a clopen condition on fibers} for all finite points in the support of $s$. 
%For the points $b$ such that $s(b)$ is an infinite, smooth point, the proof shows that the affine function $m$ must be constant in a neighborhood of $s(b)$ 
%For points $b$ such that $s(b)$ is an infinite, smooth point, refer to Figure \ref{fig:tors}.
We refer to Figure \ref{fig:tors} to illustrate the proof. Let
\begin{equation}
N=\{b\in B\mid s(b)\text{ is infinite in }\mc C_b\} \ .
\end{equation}
The set $s(N)$ closed in $U$ and  $V\coloneqq U\setminus s(N)$ is a dense open subset of $U$. The restrictions of $l$ and $l'$ to $V$ have finite values everywhere, so we can take their difference $m\coloneqq l\vert_{V}-l'\vert_{V}$.
By definition of $\Aff_{\mc C}(ks)$ and Lemma \ref{lem:affineness is a clopen condition on fibers}, the function $m$ is affine.  It suffices to prove that $m$ extends to a function in $\Gamma(U, \Aff_{\mc C})$. Since such an extension is unique by the denseness of $V$, we can do this locally  at a point in $x\in s(N)$, that is we can shrink $U$ to arbitrary small neighborhoods of $x$. In particular, we may assume that $U$ has connected fibers. Let $b=\pi(x)$, and let $y\in \mc C_b \cap V$ be a point in the interior of the leg that contains $x$. By Lemma \ref{lem:local structure at interior of edge}, there exists a neighborhood $W$ of $b$ and a section $t\colon W\to V$ with $t(b)=y$. After potentially shrinking $U$ and $W$ we may assume that $\pi(U)=W$ and that there exists a section $u\colon W\to U$ with $u(b)=x$ and such that $u(b')$ is an infinite point of $\mc C_{b'}$ on the same leg as $t(b')$ for all $b'\in W$. Let $b'\in W$. Since $l\vert_{\mc C_{b'}}$ and $l'\vert_{\mc C_{b'}}$ have the same slope in the unique fiber direction leaving from $u(b')$, the function $m\vert_{\mc C_{b'}}$ has slope zero in that direction. It follows that $m\vert_{\mc C_{b'}}$ is constant on the edge segment between $t(b')$ and $u(b')$. Since this is true for every $b'\in W$, we see that $m=((t\circ\pi)^*m)\vert_ V$ on the connected component of $x$ in $U\setminus t(W)$. 
\end{proof}

\begin{proposition}
\label{prop:aff(ks) is everywhere inhabited}
Let $\pi\colon \mc C\to B$ be a family of stable tropical curves, let $s\colon B\to \mc C$ be a linear section of $\pi$ such that $s(c)$ is a smooth point of $\mc C_{c}$ for all $c\in B$, and let $k\in \Z$. Then for every $b\in B$ and $x\in \mc C_b$ the map
\begin{equation}
\Aff_{\mc C}(ks)_x\to \Aff_{\mc C_{b}}(k\cdot s(b))_x 
\end{equation}
induced by the restriction of functions is surjective.
\end{proposition}

\begin{proof}
If $x\neq s(b)$ the assertion follows from the short exact sequence of the cotangent sheaves \eqref{equ:exact sequence}, so we can assume that $x=s(b)$. Let $\phi\in \Aff_{\mc C_{b}}(ks(b))_x$. First assume that $x$ is an infinite point in $\mc C_b$ and refer to Figure \ref{fig:inf}.
\begin{figure}
    \centering

    \begin{tikzpicture}[thick]
        \begin{scope}
            \node at (-1,2) {$U$};
            \coordinate (tl) at (0,0);
            \coordinate (tr) at (4,-.5);
            \coordinate (ul) at (0,4.5);
            \coordinate (ur) at (4,5);
            \coordinate (s) at ($(tl)!.2!(ul)$);
        
             %%% V filling
            \fill[gray!30] (tl)--(tr)--(ur)--(s)--cycle;
            \node at (3.5,.5)[ pin={[pin distance= 1cm]right:$V$}] {};
            
            %%% t
             \draw (tl)--
                node[at start, left]{$t$}
                node[pin=225:$\phi(t(b))$]{}
             ($(tl)!.8!(tr)$);
             \draw [dashed] ($(tl)!.8!(tr)$)--
                 node[at end, inner sep=.05cm, circle, fill=black, label=below:$t(b)$]{}
             (tr);

             %%% s
             \draw (s)--
                node[at start,left]{$s$}
             ($(s)!.8!(ur)$);
             \draw [dashed] ($(s)!.8!(ur)$)--
                node[at end, inner sep=.05cm, circle, fill=black, label={above:$x$}]{}
             (ur);
             
             %%% u
             \draw (ul)--
                node[at start,left]{$u$}
             ($(ul)!.8!(ur)$);
             \draw [dashed] ($(ul)!.8!(ur)$)--(ur);
             
             %%% vertical
             %\draw (0,0)--(0,3.5);
             %\draw[dashed] (0,3.5)--(0,4.5);
             
             \draw (tr)--
                node[pin=45:$\phi$]{}
             ($(tr)!.8!(ur)$);
             \draw[dashed] ($(tr)!.8!(ur)$)--(ur);
             
             %%% function values
             \node at (2.5,1.7) {$\varphi$};
             \node at (1.2,3.5) {$(s\circ\pi)^*\varphi$};

             %%%% arrow
             \draw [->] (2,-.5)--
                node[auto]{$\pi$}
             (2,-2);
        \end{scope}
        
        \begin{scope}[yshift=-2.5cm]
            \node at (-1,0) {$B$};
            \draw (0,0)--(3.2,0);
            \draw[dashed] (3.2,0)--
                node[at end, inner sep=.05cm, circle, fill=black, label=below:$b$]{}
            (4,0);
        \end{scope}
    \end{tikzpicture}
    
    \caption{Extending the function $\phi$ to a neighborhood of $\mc C_b$: it is extended constantly along the section $t$, by a function $\varphi$ with the appropriate vertical slope between $t$ and $s$, and constantly above $s$. }

    \label{fig:inf}
\end{figure}

 Let $u\colon B\to \mc C$ be the section of $\pi$ such that $u(b)=x$ and such that $u(c)$ is an infinite point in $\mc C_{c}$ for all $c\in B$. After potentially shrinking $B$, Lemma \ref{lem:local structure at interior of edge} guarantees the existence of a section $t\colon B\to \mc C$ such that for all $c\in B$, the point $t(c)$ is contained in the interior of the edge of $\mc C_c$ adjacent to $u(c)$; further we can request that $t(c)\neq s(c)$ for all $c\in B$.
 Let $U$ be the connected component of $\mc C\setminus t(B)$ containing $x$.   Let $\varphi$ be the function on a small neighborhood of  $\overline U$ that has constant value $\phi(t(b))$ on $t(B)$, value $\phi(x)$ on $u(B)$, and whose slope on $U\cap \mc C_c$ in the fiber direction is $k$ for all $c\in B$ (use a \TPL{}-trivialization to see that $\varphi$ is piecewise linear). By Lemma \ref{lem:affineness is a clopen condition on fibers} and Lemma \ref{lem:characterization of affineness in terms of section and harmonicity}, $\varphi$ is affine on  $U\setminus u(B)$. Let $V$ be the connected component of $U\setminus s(B)$ that contains the edge segment between $s(b)$ and $t(b)$ in $\mc C_b$. Then the function 
\begin{equation}
U\to \Rbar,\;\; y\mapsto \begin{cases}
\varphi(y) &,\; y\in V \\
(s\circ\pi)^* \varphi& ,\; y\notin V
\end{cases}
\end{equation}
defines an element in $\Aff_{\mc C}(ks)_x$ that lifts $\phi$.

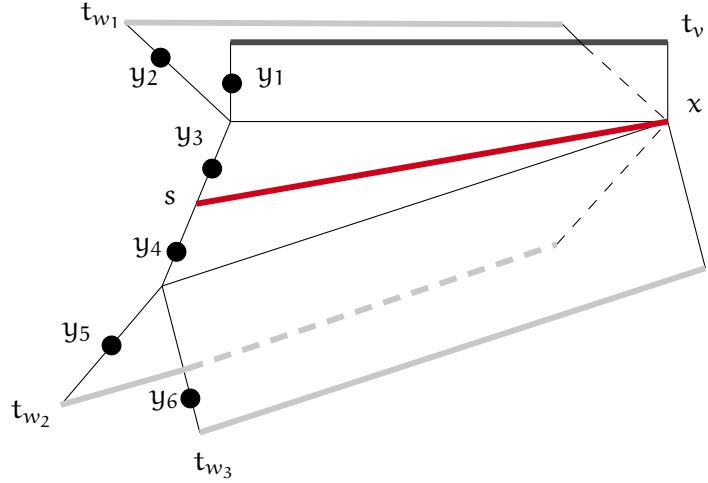
\begin{figure}[tb]
    \centering

\begin{tikzpicture}[x=0.75pt,y=0.75pt,yscale=-1,xscale=1]

%Shape: Rectangle [id:dp16409826320079413] 
\draw   (184,47) -- (404.5,47) -- (404.5,87) -- (184,87) -- cycle ;
%Straight Lines [id:da2375294483209558] 
\draw    (149.5,170) -- (184,87) ;
%Straight Lines [id:da06913992936074043] 
\draw    (98.5,230) -- (149.5,170) ;
%Straight Lines [id:da5381873635608085] 
\draw    (149.5,170) -- (168.5,244) ;
%Straight Lines [id:da6776081768013928] 
\draw    (404.5,87) -- (423.5,161) ;
%Straight Lines [id:da6989360466893704] 
\draw  [dash pattern={on 4.5pt off 4.5pt}]  (348.5,149) -- (404.5,87) ;
%Straight Lines [id:da9696844662989148] 
\draw    (149.5,170) -- (404.5,87) ;
%Straight Lines [id:da23197592479470464] 
\draw [color={rgb, 255:red, 200; green, 200; blue, 200 }  ,draw opacity=1 ][line width=2.25]  [dash pattern={on 6.75pt off 4.5pt}]  (161.5,212) -- (348.5,149) ;
%Straight Lines [id:da24093890361644377] 
\draw [color={rgb, 255:red, 200; green, 200; blue, 200 }  ,draw opacity=1 ][line width=2.25]    (168.5,244) -- (251.95,216.84) -- (423.5,161) ;
%Straight Lines [id:da6642791441797755] 
\draw [color={rgb, 255:red, 200; green, 200; blue, 200 }  ,draw opacity=1 ][line width=2.25]    (98.5,230) -- (161.5,212) ;
%Straight Lines [id:da7802900195764662] 
\draw [color={rgb, 255:red, 208; green, 2; blue, 27 }  ,draw opacity=1 ][fill={rgb, 255:red, 208; green, 2; blue, 27 }  ,fill opacity=1 ][line width=2.25]    (166.75,128.5) -- (404.5,87) ;
%Straight Lines [id:da8568665031338585] 
\draw [color={rgb, 255:red, 74; green, 74; blue, 74 }  ,draw opacity=1 ][line width=2.25]    (184,47) -- (404.5,47) ;
%Straight Lines [id:da24336900980409815] 
\draw    (130.5,37) -- (184,87) ;
%Straight Lines [id:da09271390666712032] 
\draw  [dash pattern={on 4.5pt off 4.5pt}]  (362.5,49) -- (404.5,87) ;
%Straight Lines [id:da6940736303670845] 
\draw [color={rgb, 255:red, 200; green, 200; blue, 200 }  ,draw opacity=1 ][line width=2.25]    (130.5,37) -- (351.5,38) ;
%Straight Lines [id:da12457226257959664] 
\draw    (351.5,38) -- (362.5,49) ;
%Shape: Circle [id:dp3627391455743172] 
\draw  [color={rgb, 255:red, 0; green, 0; blue, 0 }  ,draw opacity=1 ][fill={rgb, 255:red, 0; green, 0; blue, 0 }  ,fill opacity=1 ] (180,67.75) .. controls (180,65.13) and (182.13,63) .. (184.75,63) .. controls (187.37,63) and (189.5,65.13) .. (189.5,67.75) .. controls (189.5,70.37) and (187.37,72.5) .. (184.75,72.5) .. controls (182.13,72.5) and (180,70.37) .. (180,67.75) -- cycle ;
%Shape: Circle [id:dp13410140087940592] 
\draw  [color={rgb, 255:red, 0; green, 0; blue, 0 }  ,draw opacity=1 ][fill={rgb, 255:red, 0; green, 0; blue, 0 }  ,fill opacity=1 ] (144,54.75) .. controls (144,52.13) and (146.13,50) .. (148.75,50) .. controls (151.37,50) and (153.5,52.13) .. (153.5,54.75) .. controls (153.5,57.37) and (151.37,59.5) .. (148.75,59.5) .. controls (146.13,59.5) and (144,57.37) .. (144,54.75) -- cycle ;
%Shape: Circle [id:dp8116722305084498] 
\draw  [color={rgb, 255:red, 0; green, 0; blue, 0 }  ,draw opacity=1 ][fill={rgb, 255:red, 0; green, 0; blue, 0 }  ,fill opacity=1 ] (170,110.75) .. controls (170,108.13) and (172.13,106) .. (174.75,106) .. controls (177.37,106) and (179.5,108.13) .. (179.5,110.75) .. controls (179.5,113.37) and (177.37,115.5) .. (174.75,115.5) .. controls (172.13,115.5) and (170,113.37) .. (170,110.75) -- cycle ;
%Shape: Circle [id:dp8097868670410762] 
\draw  [color={rgb, 255:red, 0; green, 0; blue, 0 }  ,draw opacity=1 ][fill={rgb, 255:red, 0; green, 0; blue, 0 }  ,fill opacity=1 ] (152,152.75) .. controls (152,150.13) and (154.13,148) .. (156.75,148) .. controls (159.37,148) and (161.5,150.13) .. (161.5,152.75) .. controls (161.5,155.37) and (159.37,157.5) .. (156.75,157.5) .. controls (154.13,157.5) and (152,155.37) .. (152,152.75) -- cycle ;
%Shape: Circle [id:dp5015438796016227] 
\draw  [color={rgb, 255:red, 0; green, 0; blue, 0 }  ,draw opacity=1 ][fill={rgb, 255:red, 0; green, 0; blue, 0 }  ,fill opacity=1 ] (159,226.75) .. controls (159,224.13) and (161.13,222) .. (163.75,222) .. controls (166.37,222) and (168.5,224.13) .. (168.5,226.75) .. controls (168.5,229.37) and (166.37,231.5) .. (163.75,231.5) .. controls (161.13,231.5) and (159,229.37) .. (159,226.75) -- cycle ;
%Shape: Circle [id:dp19935978214938976] 
\draw  [color={rgb, 255:red, 0; green, 0; blue, 0 }  ,draw opacity=1 ][fill={rgb, 255:red, 0; green, 0; blue, 0 }  ,fill opacity=1 ] (119.25,200) .. controls (119.25,197.38) and (121.38,195.25) .. (124,195.25) .. controls (126.62,195.25) and (128.75,197.38) .. (128.75,200) .. controls (128.75,202.62) and (126.62,204.75) .. (124,204.75) .. controls (121.38,204.75) and (119.25,202.62) .. (119.25,200) -- cycle ;

% Text Node
\draw (413,74.4) node [anchor=north west][inner sep=0.75pt]    {$x$};
% Text Node
\draw (149,120.4) node [anchor=north west][inner sep=0.75pt]    {$s$};
% Text Node
\draw (107,25.4) node [anchor=north west][inner sep=0.75pt]    {$t_{{w}_{1}}$};
% Text Node
\draw (72,226.4) node [anchor=north west][inner sep=0.75pt]    {$t_{{w}_{2}}$};
% Text Node
\draw (164,252.4) node [anchor=north west][inner sep=0.75pt]    {$t_{{w}_{3}}$};
% Text Node
\draw (410,31.4) node [anchor=north west][inner sep=0.75pt]    {$t_{v}$};
% Text Node
\draw (196,58.4) node [anchor=north west][inner sep=0.75pt]    {$y_{1}$};
% Text Node
\draw (132,58.4) node [anchor=north west][inner sep=0.75pt]    {$y_{2}$};
% Text Node
\draw (155,88.4) node [anchor=north west][inner sep=0.75pt]    {$y_{3}$};
% Text Node
\draw (134,144.4) node [anchor=north west][inner sep=0.75pt]    {$y_{4}$};
% Text Node
\draw (98,187.4) node [anchor=north west][inner sep=0.75pt]    {$y_{5}$};
% Text Node
\draw (142,221.4) node [anchor=north west][inner sep=0.75pt]    {$y_{6}$};

\end{tikzpicture}
     \caption[]{The function $\varphi$ extending $\phi$ may be described informally as follows: it is constantly equal to zero along the section $t_v$; along a given fiber $\mc C_c$, it has vertical slope one until either it hits the section $s$, or it hits a vertex; in the latter case, it continues with slope one towards the section $s$, and it is extended  constantly along other directions.
    According to \eqref{eq:defplfext}, $\varphi$ is defined at the points $y_i$ as follows:
     \begin{eqnarray}
     \varphi(y_1) & = & \chi_{w_1}(y_1) = \chi_{w_2}(y_1) = \chi_{w_3}(y_1), \nonumber\\
      \varphi(y_2) & = & \chi_{w_1}(s(c)),  \nonumber\\
      \varphi(y_3) & = &  \chi_{w_2}(y_3) = \chi_{w_3}(y_3) ,\nonumber\\
         \varphi(y_4) =  \varphi(y_5) =   \varphi(y_6)  & = &  \chi_{w_2}(s(c)) = \chi_{w_3}(s(c)). \nonumber 
     \end{eqnarray}
    }
    \label{fig:tripod}
\end{figure}
Now assume that $x$ is not an infinite point in $\mc C_b$, and refer to Figure \ref{fig:tripod}. Since every element in $\Aff_{\mc C_b}(ks(b))_x$ is a linear combination of elements in $\Aff_{\mc C_b}(s(b))_x$, we may assume that $\phi\in \Aff_{\mc C_b}(s(b))_x$.  In fact, we may assume that there exists a direction $v\in T_x \mc C_b$ such that $d_v\phi=1$ and $d_{w}\phi=0$ for all $w\in T_x\mc C_b\setminus\{v\}$. Let $w\in T_x\mc C_b$. By Lemma \ref{lem:local structure at interior of edge}, there exists, after potentially shrinking $B$, a linear section $t_w\colon B\to \mc C$ of $\pi$ such that $t_w(b)$ is on the edge adjacent to $x$ in the direction specified by $w$, and such that all values of $t_w$ are in the interior of edges of their fibers. We may assume that all these sections are disjoint from $s$, and there exists a neighborhood $U$ of $x$ with simply connected fibers such that $\gamma_{\mc C}\vert_U=0$ and such that for every $c\in B$, every non-compact edge of $U\cap \mc C_c$ contains one of the points $t_w(c)$. For every $c\in B$ and $w\in T_x\mc C_b$ with $w\neq v$, choose a one-form on $U\cap \mc C_c$ whose support contains the path  $[t_w(c), t_v(c)]$ and has slope one in the direction from $t_w(c)$ to $t_v(c)$. There exists a unique affine function $\chi_{w,c}$ on $U\cap \mc C_c$  that lifts this one-form and satisfies $\chi_{w,c}(t_v(c))=0$. %Informally, this function measures the negative distance from a point $y\in \mc C_c$ to the point $t_v(c)$. 
Let 
\begin{equation}
\chi_w\colon U\to \R,\;\;y\mapsto \chi_{w,\pi(y)}(y) \ .
\end{equation}
By construction, $\chi_w$ is fiberwise harmonic and satisfies $t_v^*\chi_w=0$. Therefore, $\chi_w$ is affine by Lemma \ref{lem:affineness is a clopen condition on fibers} and Lemma \ref{lem:characterization of affineness in terms of section and harmonicity}. Define $\tilde{U}$ to be the connected component of $U\smallsetminus \bigcup_{w\in T_x\mc C_b} t_w(B)$ containing $x$, and construct a function $\varphi\colon \tilde{U}\to \R$ as follows:
\begin{equation} \label{eq:defplfext}
    \varphi(y) = \left\{
\begin{array}{cl}
\chi_w(y), & \mbox{if  $s(c)\in [t_w(c), t_v(c)]$}\\ & \mbox{and $y\in[s(c),t_v(c)]$.}\\
\chi_w(s(c)), & \mbox{else, with $w$ such that $y \in [t_w(c), t_v(c)]$.}
\end{array}
    \right.
\end{equation}

%Let $\varphi\colon U\to \R$ be the function the assigns to $y\in U$ the negative of the distance to $t(\pi(y))$ in $U\cap \mc C_{\pi(y)}$ of the point on the direct path between $t(\pi(y))$ and $s(\pi(y))$ that is closes to $y$, except it that point coincides with $t(\pi(y))$, in which case $\varphi(y)$ is the distance from $y$ to $t(\pi(y))$. 

Using a \TPL{}-trivialization one sees that $\varphi$ is piecewise integral linear, and hence 
\begin{equation}
\varphi\vert_{\tilde{U}\setminus s(B)}\in \Gamma(\tilde{U}\setminus s(B),\Harm_{\mc C}) \ .
\end{equation} 
By construction, we have $t_v^*\varphi=0$ and 
\begin{equation}
t_w^*\varphi= s^*\chi_w 
\end{equation}
for $w\in T_x\mc C_b$ with $w\neq v$. Therefore, we have
\begin{equation}
\varphi\vert_{\tilde{U}\setminus s(B)}\in \Gamma(\tilde{U}\setminus s(B),\Aff_{\mc C})
\end{equation}
by Lemma \ref{lem:affineness is a clopen condition on fibers} and Lemma \ref{lem:characterization of affineness in terms of section and harmonicity} and hence 
\begin{equation}
\varphi\in \Gamma(\tilde{U},\Aff_{\mc C}(s)) \ .
\end{equation}
Since $\varphi$ restricts to $\phi$ on $\mc C_b$, this finishes the proof.
\end{proof}

\section{In pursuit of an atlas} %for $\Mgnbar{g,n}$}% of $\MgnMf{g,n}$} 
\label{sec:mumford}
%\renzo{I like it! I propose not specifying for what the atlas is, so that it does not  give the impression of false advertisement}

{In this section we attempt to construct an atlas of $\Mgnbar{g,n}$, that is  a local isomorphism $B\to \Mgnbar{g,n}$, where $B$ is a tropical space.  First we eliminate non-trivial automorphisms by considering the space $\Atlass{g,n}$ of {\it cycle-rigidified} curves and then consider the universal \TPL{}-curve over $\Atlass{g,n}$. This has similarities to the construction of tropical Teichm\"uller space in \cite{Ulirsch}. The major challenge is to equip the rigidification and its universal family with affine structures to obtain a family of stable tropical curves. This essentially amounts to recovering the affine structure of $\Mgnbar{g,n}$ from its functorial description. Both the fact that the isotropy groups at curves with higher-genus vertices are too small and the fact that the affine structure on the total space of a family is not determined by the underlying \TPL{}-family (see Example \ref{ex:tropfam}) make it impossible to obtain an atlas for all of $\Mgnbar{g,n}$, but we can obtain a natural surjection onto the locus of \emph{good} curves and a local isomorphism onto the locus of \emph{Mumford} curves.
}

%\rcolor{In this section we restrict our attention to an open sub-stack of $\Mgnbar{g,n}$ parameterizing curves with no genus concentrated at vertices (also called explicit tropical curves \cite{BBM}). We show that this is an geometric stack, i.e. we construct an atlas covering it. }

\subsection{The rigidification  $\Atlass{g,n}$ of $\Mgnbar{g,n}$}

\begin{definition}
Let $G$ by an $n$-marked stable graph of genus $g$. An \textbf{oriented cycle basis} for $G$ is  a collection $c^G_1,\ldots, c^G_g\in H_1(G;\Z)$ that generates $H_1(G;\Z)$ and such that each $c^G_i$ is either zero, or corresponds to a primitive cycle in $G$. We  call a stable graph  equipped with an oriented cycle basis a \textbf{cycle-rigidified stable graph}. A \textbf{morphism} between two cycle-rigidified stable graphs $G$ and $G'$ of the same genus $g$ is a morphism  $f\colon G\to G'$ of graphs such that $f_*c^G_i=c^{G'}_i$ for all $1\leq i\leq g$. 
\end{definition}

If $G$ is a cycle-rigidified stable graph of genus $g$ and $G'$ is obtained as a specialization of $G$, then the cycles $c^G_1,\ldots,c^G_g$ induce and oriented cycle basis $c^{G'}_1,\ldots,c^{G'}_{g}$ for $G'$.

\begin{lemma}
\label{lem:ridig graphs are rigid}
Let $G$ be a cycle-rigidified stable graph, and let $f\colon G\to G$ be an automorphism of $G$. Then $f$ is the identity.
\end{lemma}

%The statement of Lemma \ref{lem:ridig graphs are rigid} is well-known. For the readers convenience, we provide a proof, closely following the proof of \cite[Lemma 1]{Zimmermann96}.
For a proof of  Lemma \ref{lem:ridig graphs are rigid}, see, for example, \cite[Lemma 1]{Zimmermann96}.

%\begin{proof}
%Since $f$ respects the generators $c^G_1,\ldots, c^G_{g(G)}$ for $H_1(G,\Q)$, it induces the identity on $H_1(G,\Q)$\renzo{I don't think it is a good idea to call $g(G)$ the genus of the graph and g the 1st Betti number of the graph. Also, I am not sure what is the use of having this proof in the paper. I would remove it. Finally, it seems an overly high tech proof for a simple fact... really it is the combination of the fact that there are no nontrivial automorphisms of labeled trees and that complement of the union of the $c_i'$s is a disjoint union of labeled trees}. Therefore, the Lefschetz number of $f$ is $1-g$, where $g=\dim_\Q H_1(G,\Q)$. By the Hopf Trace formula, the Lefschetz number equals the Euler characteristic of the fixed point locus $\Lambda$ of $f$. It follows that $\Lambda$ is connected and has first Betti number $g$ (in the case $g=1$, $\Lambda$ could be potentially be empty. But this is not possible because $G$ has at least one leg in this case). Since $f$ fixes all legs, we conclude that $\Lambda=G$.
%\end{proof}

\begin{definition} \label{def:atlas}
For every $g,n \in\Z_{\geq 0}$  with $n+2g-2>0$ we define a topological space $\Atlass{g,n}$  by gluing for each cycle-rigidified $n$-marked stable graph $G=(\underline G,(c_i^G)_{1\leq i\leq g})$ the extended cone $\overline \sigma_G\coloneqq\overline \sigma_{\underline G}$ along the face maps $\overline \sigma_{\underline G'}\to \overline \sigma_{\underline G}$ associated to {morphisms} $G\to G'$.  For an open subset $V\subseteq \Atlass{g,n}$ we define $\Gamma(V,\PL_{\Atlass{g,n}})$ as the set of continuous functions from $V$ to $\Rbar$ such that for all $G$, the pull-back to $\overline\sigma_G$ is a piecewise linear function.
\end{definition}

The space $(\Atlass{g,n},\PL_{\Atlass{g,n}})$  is a \TPL{}-space and it follows from Lemma \ref{lem:ridig graphs are rigid} that the morphism $\relint(\sigma_G)\to \Atlass{g,n}$ are injective for all cycle-rigidified $I$-marked stable graphs of genus $g$. {We  now define an affine structure on $\Atlass{g,n}$. The idea is to pretend we had an affine structure on the universal family over $\Atlass{g,n}$ and then check which piecewise linear functions on $\Atlass{g,n}$ are forced to be affine by the exactness of the sequence \eqref{equ:exact sequence}. Informally, we integrate relative one-forms along paths in the fibers of the total space to obtain these functions. To formalize this, we introduce the notions of thickened paths, one-forms along thickened paths, and cross ratios. %\renzo{Think about adding motivating sentence for "thickened path". {\color{purple!50!black} Andreas: okay like that?} Awesome.}
}

\begin{definition}
\label{def:germ of trails}
Let $G$ be a stable graph, and denote by $G^{f}$ the subset of genus zero vertices  or  points of bounded edges. A \textbf{thickened path} $\gamma^{\pm} = (\gamma, f_s, f_e)$ is a path $\gamma: [0,1] \to G^{f}$   together with the choice of an initial  flag $f_s$ incident to $\gamma(0)$ and a terminal flag $f_e$ incident to $\gamma(1)$.

While not strictly necessary, we include the following simplifying assumptions in the definition:
\begin{itemize}
    \item  $\gamma(0) = v_s, \gamma(1) = v_e $ are vertices;
    \item if $\gamma$ is not constant,  $f_s$ is distinct from the flag determined by $\gamma'(0)$ and $f_e$ is distinct from $\gamma'(1)$;
    \item if $\gamma$ is constant, $f_s\not= f_e$.
\end{itemize}

%on $G$  consists of a sequence $v_0,\ldots, v_n$ of genus-zero vertices on $G$, a sequence of flags $f_1\ldots, f_n$ such that $e_G(f_i)\in E_b(G)$ 
%\andreas{while writing this, graph complexes fell out of favor. My suggestion would be to modify the definition of stable graphs simply by adding the datum of a subset $E_b(G)$ of $E(G)\setminus L(G)$ of bounded edges.},
 %$e_G(f_i)$ has vertices $v_{i-1}$ and $v_{i}$, and $v_G(f_i)=v_{i-1}$, and two flags $f_s, f_e\in F(G)$ such that $v_G(f_s)=v_0$ and $v_G(f_e)=v_n$.
\end{definition}

\begin{remark}
Note that in the definition of thickened paths, $\gamma$ may be constant: in this case a thickened path  consists only of a genus-zero vertex $v$ and two flags $f_s$ and $f_e$ with $v_G(f_s)=v_G(f_e)=v$.
\end{remark}

%Given a tropical stable curve $\Gamma$ with dual graph $G$, any thickened path 
%\begin{equation}
%((v_i)_{0\leq i\leq n}, (f_i)_{1\leq i \leq n},f_s,f_e)
%\end{equation}
% on $G$ gives a recipe to define a path $\gamma\colon [0,1]\to \Gamma$. Namely, one starts at a point on $e_G(f_s)$ close to $v_0$, then moves to $v_0$, follows the edges $e_G(f_1)$,..., $e_G(f_n)$ in the orientation prescribed by the $f_i$, and ends at a point on $e_G(f_e)$ close to $v_n$. Of course, the $\gamma$ is not well-defined, but one only obtains a ``germ of a path''. This can be made precise, but we refrain from doing so.\renzo{How is this different from specifying a path plus a tangent direction for the initial and final vertices?}

\begin{definition}
\label{def:1-form along germ}
Let $G$ be a stable graph, $\Gamma$ a tropical curve of combinatorial type $G$ and  $\gamma^{\pm}$ be a thickened path on $G$. A \textbf{one-form along $\gamma^{\pm}$} is  an element $w \in H^0([0,1],\gamma^{-1}\Omega^1_\Gamma)$ such that any local one-form $\omega$ representing $w$  near $\gamma(0)$ (resp. $\gamma(1))$ has slope zero along $f_s$ (resp. $f_e$).
\end{definition}

 Denote by $v_0 = v_s, v_1, \ldots , v_n = v_e$ the ordered multiset of vertices traversed by $\gamma$. Explicitly, a  one-form along $\gamma^{\pm}$ consists of the combinatorial datum of a collection of maps
\begin{equation}
w_i\colon T_{v_i}G \to \Z
\end{equation}
such for $0\leq i\leq n$ such that
\begin{enumerate}
\item 
\begin{equation}
w_0(f_s)=w_n(f_e)=0 \ ,
\end{equation}
\item
for every $0\leq i\leq n$ we have
\begin{equation}
\sum_{f\in T_{v_i}G} w_i(f)=0 \ ,
\end{equation}
\item 
if $1\leq i\leq n$ and $f_{i-1}^+, f_i^-$ denote the two flags belonging to the edge $e_i$ traversed by $\gamma$ between $v_{i-1}$ and $v_i$, then

% $f'_i$ the unique flag distinct from $f_i$ satisfying $e_G(f'_i)=e_G(f_i)$, then we have
\begin{equation}
w_{i-1}(f_{i-1}^+)=-w_i(f_{i}^-) \ .
\end{equation}
We may interpret condition (3) as saying that $w$ assigns an integer $w(e_i)$ to any directed edge when traversed by $\gamma$. 
%\renzo{I think I get this but I find it written in a very complicated way.}
\end{enumerate}

%Given a tropical stable curve $\Gamma$ with dual graph $G$ and a thickened path $t$, let $\gamma\colon [0,1]\to \Gamma$ be a path constructed according to $t$. Let $w$ be a one-form along $t$. Then $w$ specifies a one-form along $\gamma$, that is an element $\omega\in \Gamma([0,1],\gamma^{-1}\Omega^1_\Gamma)$.
% such that the slope of $\omega_x$ along an edge adjacent to $\gamma(x)$ in the direction specified by a flag $f\in F(G)$ is precisely $w(f)$. In particular, $\omega_{\gamma(0)}$ and $\omega_{\gamma(1)}=0$.

\begin{definition}
A \textbf{cross ratio datum} on a stable graph $G$ is a pair $(\gamma^{\pm},w)$ consisting of a thickened path $\gamma^{\pm}$ on $G$ and a one-form $w$ along $\gamma^{\pm}$.
\end{definition}

%\begin{definition}
%Let $G$ be a stable graph.
%A \emph{genus-zero} trail in $G$ consists of two not necessarily distinct flags $s,t\in F(G)$ and a possibly empty oriented trail $f_1,\ldots,f_k$, where $f_i\in F(G)$, from $v_G(s)$ to $v_G(t)$ such that all edges $e_G(f_i)$ are distinct and only adjacent to genus-zero vertices. 
%A cross ratio datum for $G$ consists of two genus-zero trails $(s,t,f_1,\ldots,f_k)$ and $(s',t',f_1',\ldots, f'_{k'})$ such that $s\neq t$ and $\{s,t\}\cap \{s',t'\}=\emptyset$.
%\end{definition}

%Note that if a  graph complex $G'$ is a specialization of a graph complex $G$, then every genus-zero trail on $G'$ lifts uniquely to a genus-zero trail on $G$. Therefore, cross ratio data on $G'$ lift uniquely to cross ratio data on $G'$ as well.

If a stable graph $G$ is obtained as a specialization of a stable graph $\widetilde{G}$, then any thickened path $\gamma^{\pm}$ on $G$, which by definition is contained in the set of genus-zero points of $G$, can be lifted uniquely to a thickened path $\widetilde{\gamma}^{\pm}$ of a trail on $\widetilde{G}$: insert into $\gamma$ the edges of $\widetilde{G}$ that are contracted to vertices that $\gamma$ passes through. Any one-form $w$ along $\gamma^{\pm}$ can be uniquely lifted to a one-form $\widetilde{w}$ along $\widetilde{\gamma}^{\pm}$, the values of $\widetilde{w}$ on contracted flags being determined by part (2) in Definition \ref{def:1-form along germ}. Therefore, any cross ratio datum on $G$ lifts uniquely to a cross ratio datum on $\tilde{G}$.

\begin{definition} \label{def:cycleriggraph}
Let $G$ be an $n$-marked cycle-rigidified stable graph of genus $g$, and let $c=(\gamma^{\pm},w)$ be a cross ratio datum on $G$. We define a piecewise-linear function $\xi_c$, the \textbf{cross ratio associated to $c$}, on
\begin{equation}
V(G)=\bigcup_{\widetilde{G}} \relint(\widetilde{\sigma}_G) \ ,
\end{equation}
where the union is taken over all cycle-rigidified stable graphs $\widetilde{G}$ specializing to $G$. Namely, given a stable graph $\widetilde{G}$ specializing to $G$, let $\widetilde{c}=(\widetilde{\gamma}^{\pm},\widetilde{w})$ denote the unique lift of $c$ to $\widetilde{G}$. We then define the restriction of $\xi_c$ to $\relint(\sigma_{\widetilde{G}})$ by
\begin{equation}
\xi_c\vert_{\relint(\sigma_{\widetilde{G}})}= \sum_{1\leq i\leq m} \widetilde{w}(e_i) l_{e_i}  \ ,
\end{equation}
where $e_1,\ldots, e_m$ denote the ordered multiset of directed edges traversed by $\widetilde{\gamma}$.
\end{definition}

\begin{definition}
We define $\Aff_{\Atlass{g,n}}$ as the subsheaf of $\PL_{\Atlass{g,n}}$ generated by the cross ratios $\xi_c$, where $c$ is a cross ratio datum on some $n$-marked cycle-rigidified stable graph of genus $g$.
\end{definition}

To ease the use of cross ratio data, we introduce a generating set for cross ratio data.

\begin{definition}
A cross ratio datum $c=(\gamma^{\pm},w)$ on a stable graph $G$ is said to be \textbf{primitive} if it is of one of the following two types:
\begin{itemize}
\item[(\textbf{Edge})] the path $\gamma$ parameterizes a bounded edge $e$ connecting two genus-zero vertices $v_0$ and $v_1$.
 The one-form $w$ has $w(e) = 1$, and there exist precisely one flag %$f_{-1} \in T_{v_0}G \smallsetminus \{f_s, \gamma'(0)\}$ such that $w(f_{-1})\not= 0$, and one flag $f_{1} \in T_{v_1}G \smallsetminus \{f_e, \gamma'(1)\}$ such that $w(f_{1})\not= 0$. 
  in $T_{v_0}G \smallsetminus \{f_s, \gamma'(0)\}$ and one flag in $T_{v_1}G \smallsetminus \{f_e, \gamma'(1)\}$ where $w$ has non-zero slope. By the balancing condition the slopes along these two special flags are $\pm1$;  we label the flags $f_{\pm1}$ so that  $w(f_{-1}) = -1$ and $w(f_{1}) = 1$. We denote this cross ratio datum by $c_{(e, (f_s,f_e),(f_{-1},f_{1}))}$ and the associated cross ratio by $\xi_{(e, (f_s,f_e),(f_{-1},f_{1}))}$.

%The trail $t$ consists of a single flag $f$ with $e_G(f)\in E_b(G)$ connecting two (not necessarily distinct) genus-zero vertices $v_0$ and $v_1$. The one-form $w$ satisfies $w_0(f)=1$ and there is precisely one flag $f'_s\in T_{v_0}G\setminus \{f\}$ with  $w_0(f'_s)\neq 0$ and precisely one flag $f'_e\in T_{v_1}G\setminus \{f'\}$ with and $w_1(f'_e)\neq 0$, where $f'$ is the flag that has the opposite orientation on $e_G(f)$. Note that in this case we must have $w_0(f'_s)=-1$ and $w_1(f'_e)=1$. We denote this cross ratio datum by $c_{(f, f_s,f_s',f_e,f_e')}$ and the associated cross ratio by $\xi_{(f, f_s,f_s',f_e,f_e')}$.\renzo{It seems to me that such datum identifies the germ of a four marked rational curve, and I imagine the cross ratio is somehow related to that, but given the way we have it defined now, what is the purpose of the additional data of the f' 's. OK maybe I get it: the data is needed to lift the cross ratio uniquely to generizations of the graph...}
%\item[(loop)]
%$t$ consists of a simple flag $f$ such that $e_G(f)$ is a loop at a genus-zero vertex $v$. The one-form satisfies,  $w_0(f)=1$ and $w_i(f')=0$ for every $f'\in T_v G$ with $e_G(f')\neq e_G(f)$. We denote this cross ratio datum by $c_{(f,f_s,f_t)}$ associated cross ratio by $\xi_{(f,f_s,f_t)}$
\item[(\textbf{Vertex})] 
The image of $\gamma$ consists of a single genus-zero vertex $v$. There exist  two flags $f_{-1},f_{1}\in T_v G \smallsetminus \{f_s, f_e\}$ on which $w$ is nonzero, and we have $w(f_{-1})=-1$ and $w(f_1)=1$. We denote this cross ratio datum by $c_{((f_s,f_e),(f_{-1},f_1))}$ associated cross ratio by $\xi_{((f_s,f_e),(f_{-1},f_1))}$.%\andreas{while working with it, I came to resent this notation. It's not quite as intuitive as it should have been. Any thoughts?}.
\end{itemize}
\end{definition}
\begin{figure}[tb]
    \centering
    \begin{tikzpicture}[thick]
    \draw[thick] (-1,0) -- (-3,0);
    \draw[thick, ->] (2,0.01) -- (2,-0.01);
    %\draw (0,0) arc (-175:175:1.7);
  \draw[thick] (0.5,0) circle (1.5);
     \node at (2.5,0) {$\Gamma$};
      \node at (-3.5,0) {$\dots$};
       \node at (-2,0.3) {$f_s = f_e$};
        \node at (-0,0.5) {$f_1 = \gamma'(0)$};
         \node at (0.1,-0.5) {$f_{-1} = \gamma'(1)$};
   \end{tikzpicture}
   
    \caption{A primitive cross ratio datum on a tropical elliptic curve. }
    \label{fig:crossratioloopexample}
\end{figure}
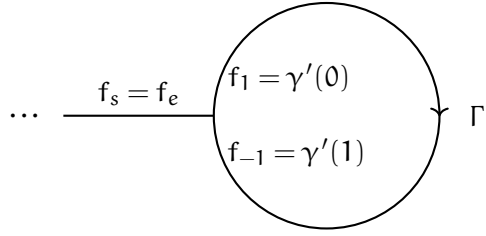

{
\begin{example}
Let $\Gamma$ denote a cycle-rigidified tropical elliptic curve with no points of genus one, as depicted in Figure \ref{fig:crossratioloopexample}. 
As an example of a non-trivial primitive cross ratio of edge type, one may take $\gamma^{\pm}$ to be the path parameterizing the loop in the direction of the cycle-rigidification, with $f_s = f_e$ being the tail of $\Gamma$. The one-form $w$ may be chosen to be the (pull-back via $\gamma^{\pm}$ of the) global one-form $\omega$ with slope zero on the tail of $\Gamma$ and one along the loop. This cross ratio datum produces the function on $(0, \infty)\subset \Atlass{1,1}$ that assigns to each tropical elliptic curve the length of  its loop. There are several other choices for primitive cross ratio data, but up to sign they all produce the length of the loop as a function.  
\end{example}
}

\begin{remark}\label{rem:cross-r}
Given a cycle-rigidified stable graph $\widetilde{G}$ specializing to $G$, a primitive cross ratio datum $c$ on $G$ determines two oriented paths on $\widetilde{G}$: $\widetilde{\gamma}$ connecting $v_{\widetilde{G}}(f_s)$ to $v_{\widetilde{G}}(f_e)$, and $\gamma_w$ connecting $v_{\widetilde{G}}(f_{-1})$ to $v_{\widetilde{G}}(f_1).$
For a point $x$ in $\relint(\sigma_{\widetilde{G}})$, the value $\xi_c(x)$ equals the signed length of the intersection $\widetilde{\gamma} \cap \gamma_w$, where the sign is positive if the two paths are oriented the same way and negative otherwise, see Figure \ref{fig:primitive corssration}. Thus, a primitive cross ratio datum is intimately related to the notions of tropical cross ratios in the literature, which can be viewed as tropicalizations of cross ratios in algebraic geometry \cite{MikhalkinModuli,GathmannMarkwigKontsevich,TyomkinCrossratios,Goldner}. 
\end{remark}

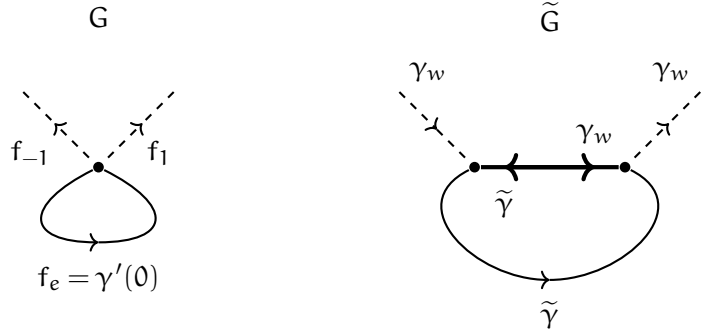
\begin{figure}
    \centering
    \begin{tikzpicture}[thick]
    \begin{scope}
    \node[dot] (v) at (0,0) {};
    \coordinate (vloop) at ($(v)-(0,1)$);
    \node at (0,2) {$G$};
    \draw [->] (v) .. controls ($(v)-(1,.5)$) and ($(vloop)-(1,0)$) .. 
        node [at end, label={below:$f_e=\gamma'(0)$}] {}
    (vloop);

    \draw (vloop) .. controls ($(vloop)+(1,0)$) and ($(v)+(1,-.5)$)  .. (v);
        
    \coordinate (1) at ($(v)+(-1,1)$);
    \coordinate (2) at ($(v)+(1,1)$);
    
    \draw [dashed,-<] (1)--
        node [at end, auto,swap] {$f_{-1}$}
    ($(v)!.5!(1)$);
    \draw [dashed] ($(v)!.5!(1)$) -- (v);
    
    \draw [dashed,-<] (2)--
        node [at end, auto] {$f_{1}$}
    ($(v)!.5!(2)$);
    \draw [dashed] ($(v)!.5!(2)$) -- (v);
    \end{scope}

    \begin{scope} [xshift =5cm]
    \node[dot] (v) at (0,0) {};
    \coordinate (vloop) at ($(v)+(1,-1.5)$);
    \node at (1,2){$\widetilde G$};
    \node [dot] (v2) at (2,0) {};
    \draw [->] (v) .. controls ($(v)-(1,.5)$) and ($(vloop)-(1,0)$) .. 
        node [at end, label={below:$\widetilde \gamma$}] {}
    (vloop);

    \draw (vloop) .. controls ($(vloop)+(1,0)$) and ($(v2)+(1,-.5)$)  .. (v2);
    \draw [ultra thick,->] (v) --
        node [at end, label={above:$\gamma_w$}] {}
    ($(v)!.8!(v2)$);
     \draw [ultra thick,->](v2)-- 
        node [at end, label={below:$\widetilde \gamma$}] {}
     ($(v2)!.8!(v)$);
        
    \coordinate (1) at ($(v)+(-1,1)$);
    \coordinate (2) at ($(v2)+(1,1)$);
    
    \draw [dashed,->] (1)--
        node [at start,auto] {$\gamma_w$}
    ($(v)!.5!(1)$);
    \draw [dashed] ($(v)!.5!(1)$) -- (v);
    
    \draw [dashed,-<] (2)--
        node [at start, auto, swap] {$\gamma_w$}
    ($(v2)!.5!(2)$);
    \draw [dashed] ($(v2)!.5!(2)$) -- (v2);
    \end{scope}
    
    \end{tikzpicture}
    \caption{The thickened edge is traversed by both $\gamma_w$ and $\widetilde \gamma$, but in opposite directions. If $c$ is the primitive cross ration datum on $G$ depicted here, then for every point in $\relint(\sigma_{\widetilde G})$ the value of $\xi_c$ equals the negative of the length of the thickened edge.}
    \label{fig:primitive corssration}
\end{figure}

\begin{proposition}
\label{prop:primitive cross ratios generate all cross ratios}
The sheaf $\Aff_{\Atlass{g,n}}$ is generated by the cross ratios $\xi_c$ associated to primitive cross ratio data $c$.
\end{proposition}

\begin{proof}
Let $G$ be a cycle-rigidified stable graph  and let $c=(\gamma^{\pm},w)$ be a cross ratio datum on $G$. We  show that $\xi_c$ is a linear combination of primitive cross ratios. Let $v_0, \ldots, v_m$ be the ordered multiset of vertices traversed by $\gamma$. We proceed by induction on $m$ First suppose $m=0$. 
%If $f_s=f_e$, then $\xi_c=0$ and we are done. \renzo{Think if it is worth imposing the condition $f_s\not=f_t$ in the definition. } If $v_0$ is three-valent, then $w=0$ and hence $\xi_c=0$. So suppose $f_s\neq f_e$ and $v_0$ is at least four-valent. 
The one-form $w$ along $\gamma^\pm$ is uniquely determined by the map $w_0:T_{v_0}G\setminus\{f_s,f_e\}\to \Z$. We may think of $w_0$ as  an element of the sublattice $H\subset \Z^{\val{v_0}-2}$ of integral vectors whose coordinates add to zero. It is well known that $w_0$ may be written as an integral linear combination of vectors with exactly two non-zero entries, equal to $1$ and $-1$. Each such vector corresponds to a primitive  cross ratio of vertex type, concluding the proof of the base case.

%\renzo{I thought it was zero}. So if we pick any $f_s'\in T_{v_0}G\setminus\{f_s,f_t\}$, then $w$ is a linear combination of one-forms whose value is $-1$ on $f_s'$ and $1$ on a flag $f\in T_{v_0}G\setminus\{f_s,f_t,f_s'\}$, and $0$ on all other flags. Since 
%\begin{equation}
%\xi_{(t,w_1+w_2)}=\xi_{(t,w_1)}+\xi_{(t,w_2)}
%\end{equation}
%for any two one-forms $w_1,w_2$ along $t$, it follows that $\xi_c$ is a linear combination of the primitive cross ratios $\xi_{(f_s,f_t,f_s',f)}$, where $f \in T_{v_0}G\setminus\{f_s,f_t,f_s'\}$.\renzo{Isn't this a complicated way of saying that the integral sublattice of integers adding to zero is generated over the integers by vectors which have exactly two non-zero coordinates, which are $1$ and $-1$?}

Now suppose $m>0$. Denote by $e_{m-1}, e_m$ the last two oriented edges traversed by $\gamma$, by $f_{m-1}^-,f_{m-1}^+$ the flags tangent to $\gamma$   and incident to $v_{m-1}$, and by $h$ a flag incident to $v_{m-1}$ which is not traversed by $\gamma$. % (see Figure \ref{}). 
Consider the integral affine function
\begin{equation}
\xi\coloneqq \xi_c -\sum_{f\in T_{v_m}G\setminus \{f_e\}} w_m(f)\cdot \xi_{(e_m,(f_{m-1}^-,f_e),(h,f))} \ .
\end{equation}

One may verify that the affine function $\xi$ can be written as a cross ratio
$
\xi=\xi_{(\widetilde{\gamma}^\pm,\widetilde{w})}  ,
$
where $\widetilde{\gamma}^\pm$ is the thickened path in $G$ obtained
by truncating $\gamma$ at $v_{m-1}$
and choosing $\widetilde{f}_e = f_{m-1}^+$, and 
 $\widetilde{w}$ is the one-form  given by $\widetilde{w}_i=w_i$ for $i<m-1$ and 
\begin{equation}
\widetilde{w}_{m-1}(f)=\begin{cases}
0 & , \ f=f_{m-1}^+ \ , \\
w_m(h)+w_m(f_{m-1}^+) & , \ f=h \ , \\
w_m(f) & ,\ \text{ else.}
\end{cases}
\end{equation}
The statement is immediate for any point in $\relint(\sigma_G)$; it is a combinatorial exercise to show that the functions $\xi=\xi_{(\widetilde{\gamma}^\pm,\widetilde{w})}$ agree on the relative interior of cones $\sigma_{\widetilde{G}}$, where $\widetilde{G}$ is a graph specializing to $G$.
Using the induction hypothesis finishes the proof.

\end{proof}

\begin{lemma}
\label{lem:properties of vertex cross ratios}
Let $G$ be a stable graph and let $v\in V(G)$ be a genus-zero vertex, and let $f_1, f_2,f_3,f_4\in T_v G$ be distinct flags.
\begin{enumerate}
\item We have 
\begin{equation*}
\xi_{((f_1,f_2),(f_3,f_4))}= \xi_{((f_3,f_4),(f_1,f_2))} \ .
\end{equation*}
\item
We have
\begin{equation*}
\xi_{((f_1,f_2),(f_3,f_4))}=-\xi_{((f_2,f_1),(f_3,f_4))} \ .
\end{equation*}
\item
We have 
\begin{equation*}
\xi_{((f_1,f_3),(f_2,f_4))}= \xi_{((f_1,f_2),(f_3,f_4))} + \xi_{((f_1,f_4),(f_2,f_3))} \ .
\end{equation*}
\item
If $f_5\in T_vG\setminus\{f_1,f_2,f_3,f_4\}$, then
\begin{equation*}
\xi_{((f_1,f_2),(f_3,f_4))}= \xi_{((f_1,f_5),(f_3,f_4))}+\xi_{((f_5,f_2),(f_3,f_4))}.
\end{equation*}
\end{enumerate}
\end{lemma}

\begin{proof}
The interpretation of primitive cross ratios in Remark \ref{rem:cross-r} makes these statements transparent:
$(1)$ states invariance when switching the roles  of $\tilde{\gamma}$ and $\gamma_w$; the sign is changed when reversing the orientation of one of the paths, implying  $(2)$; $(4)$ is stating that the cross ratio is additive when decomposing the path $\widetilde{\gamma}$ from $f_1$ to $f_2$ as the sum of two paths by introducing an intermediate point $f_5$. Finally, $(3)$ is  obtained from the Pl\"ucker relations in $\Mgn{0,4}$ (see Figure \ref{fig:cross ratios on M04}).

\begin{figure}
    \centering
    \begin{tikzpicture}[auto,thick]
        \coordinate (t) at (0,1);
        \coordinate (b) at (0,0);
        \coordinate (ol) at ($(t)+(-1,1)$);
        \coordinate (or) at ($(t)+(1,1)$);
        \coordinate (ul) at ($(b)+ (-1,-1)$);
        \coordinate (ur) at ($(b)+(1,-1)$);
        
        \draw (t) -- 
            node {$e$}
        (b);
        
        \draw [dashed,->] (t)--
            node[at end] {$f_1$}
        ($(t)!.5!(ol)$);
        
        \draw [dashed,->] (t)--
            node[at end,swap] {$f_2$}
        ($(t)!.5!(or)$);
        
        \draw [dashed,->] (b)--
            node[at end,swap] {$f_3$}
        ($(b)!.5!(ul)$);
        
        \draw [dashed, ->] (b)--
            node[at end] {$f_4$}
        ($(b)!.5!(ur)$);
        
        \draw [dashed] ($(t)!.5!(ol)$) -- (ol)
                        ($(t)!.5!(or)$)-- (or)
                        ($(b)!.5!(ul)$) --(ul)
                        ($(b)!.5!(ur)$)--(ur);
                        
        \node at (0,-2.5) {$
        \begin{aligned}
        \xi_{((f_1,f_3),(f_2,f_4))}&= l_e  \\ 
        \xi_{((f_1,f_2),(f_3,f_4))}&= 0 \\
        \xi_{((f_1,f_4),(f_2,f_3))}&= l_e
        \end{aligned}$};
        
        \begin{scope}[xshift=5cm]
            \coordinate (t) at (0,1);
        \coordinate (b) at (0,0);
        \coordinate (ol) at ($(t)+(-1,1)$);
        \coordinate (or) at ($(t)+(1,1)$);
        \coordinate (ul) at ($(b)+ (-1,-1)$);
        \coordinate (ur) at ($(b)+(1,-1)$);
        
         \draw (t) -- 
            node {$e$}
        (b);
        
        \draw [dashed,->] (t)--
            node[at end] {$f_1$}
        ($(t)!.5!(ol)$);
        
        \draw [dashed,->] (t)--
            node[at end,swap] {$f_3$}
        ($(t)!.5!(or)$);
        
        \draw [dashed,->] (b)--
            node[at end,swap] {$f_2$}
        ($(b)!.5!(ul)$);
        
        \draw [dashed, ->] (b)--
            node[at end] {$f_4$}
        ($(b)!.5!(ur)$);
        
        \draw [dashed] ($(t)!.5!(ol)$) -- (ol)
                        ($(t)!.5!(or)$)-- (or)
                        ($(b)!.5!(ul)$) --(ul)
                        ($(b)!.5!(ur)$)--(ur);
                        
        \node at (0,-2.5) {$
        \begin{aligned}
        \xi_{((f_1,f_3),(f_2,f_4))}&= 0  \\ 
        \xi_{((f_1,f_2),(f_3,f_4))}&= l_e \\
        \xi_{((f_1,f_4),(f_2,f_3))}&= -l_e
        \end{aligned}$};
        \end{scope}
        
        \begin{scope}[xshift=10cm]
            \coordinate (t) at (0,1);
        \coordinate (b) at (0,0);
        \coordinate (ol) at ($(t)+(-1,1)$);
        \coordinate (or) at ($(t)+(1,1)$);
        \coordinate (ul) at ($(b)+ (-1,-1)$);
        \coordinate (ur) at ($(b)+(1,-1)$);
         \draw (t) -- 
            node {$e$}
        (b);
        
        \draw [dashed,->] (t)--
            node[at end] {$f_1$}
        ($(t)!.5!(ol)$);
        
        \draw [dashed,->] (t)--
            node[at end,swap] {$f_4$}
        ($(t)!.5!(or)$);
        
        \draw [dashed,->] (b)--
            node[at end,swap] {$f_2$}
        ($(b)!.5!(ul)$);
        
        \draw [dashed, ->] (b)--
            node[at end] {$f_3$}
        ($(b)!.5!(ur)$);
        
        \draw [dashed] ($(t)!.5!(ol)$) -- (ol)
                        ($(t)!.5!(or)$)-- (or)
                        ($(b)!.5!(ul)$) --(ul)
                        ($(b)!.5!(ur)$)--(ur);
                        
        \node at (0,-2.5) {$
        \begin{aligned}
        \xi_{((f_1,f_3),(f_2,f_4))}&= -l_e  \\ 
        \xi_{((f_1,f_2),(f_3,f_4))}&= -l_e \\
        \xi_{((f_1,f_4),(f_2,f_3))}&= 0
        \end{aligned}$};
        \end{scope}
                                
    \end{tikzpicture}
    \caption[]{
    One has $\xi_{((f_1,f_3),(f_2,f_4))}= \xi_{((f_1,f_2),(f_3,f_4))} + \xi_{((f_1,f_4),(f_2,f_3))}$ 
 on $\Mgn{0,4}$.
    }
    \label{fig:cross ratios on M04}
\end{figure}
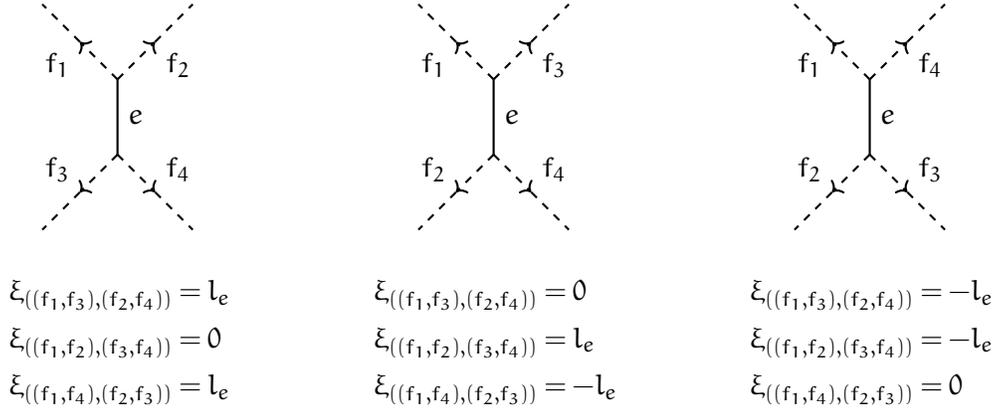

\end{proof}

%\begin{lemma}
%Let $G$ be a stable graph, and let $v\in V(G)$ be a genus-zero vertex.
%\begin{itemize}
%\item
%If $g,h,f_1,f_2,f_3\in T_v(G)$ are distinct flags, then we have
%\begin{equation}
%\xi_{(f_1,g,f_3 ,h)}= \xi_{(f_1,g,f_2 ,h)}+\xi_{(f_2,g,f_3 ,h)}
%\end{equation}
%\item
%If $f\in T_v(G)$ 
%\end{itemize}
%\end{lemma}

\begin{lemma}
\label{lem:properties of edge cross ratios}
Let $G$ be a stable graph, let $e = \{f,f'\}\in E_b(G)$ be a bounded  oriented edge connecting two (not necessarily distinct) genus-zero vertices $v_0$ and $v_1$;%, and let $f,f'$ be the flags that orient $e$, where $v_G(f)=v_0$ and $v_G(f')=v_1$. Furthermore, l
Let $f_s,f_{-1}\in T_{v_0}G\setminus \{f\}$ be two distinct flags, and let $f_e,f_{1}\in T_{v_1}G\setminus \{f'\}$ be two distinct flags. 
\begin{enumerate}
\item We have
\begin{equation*}
\xi_{(e,(f_s,f_e),(f_{-1},f_1))}=\xi_{(-e,(f_e,f_s),(f_1,f_{-1}))} \ .
\end{equation*}
\item
We have
\begin{equation*}
\xi_{(e,(f_s,f_e),(f_{-1},f_1))}=\xi_{(e,(f_{-1},f_1),(f_s,f_e))} \ .
\end{equation*}
\item If $h\in T_{v_1}G\setminus\{f_e,f_{-1},f'\}$, then
\begin{equation*}
\xi_{(e,(f_s,f_e),(f_{-1},f_1))}=\xi_{(e,(f_s,h),(f_{-1},f_1))}+\xi_{((h,f_e),(f',f_1))} \ .
\end{equation*}
\end{enumerate}
\end{lemma}
\begin{proof} These statements and their proofs are analogous to those in Lemma \ref{lem:properties of vertex cross ratios}.
\end{proof}
{The next technical lemma describes cross ratios on cones of cycle-rigidified graphs with some edges of infinite length.}

\begin{lemma}
\label{lem:cross ratios at infinity}
Let $G$ be a cycle-rigidified $I$-marked genus-$g$ stable graph, let $G_b$ the stable graph obtained by making all edges that are not legs bounded, and let $G_\infty$ denote the stable graph obtained by contracting the edges in $G_b$ that correspond to the  bounded edges of $G$. Using the notation $V(G)$ as in Definition \ref{def:cycleriggraph},  the restriction 
\begin{equation}
\Gamma(V(G), \Aff_{\Atlass{g,n}})\to \Gamma(V(G_b), \Aff_{\Atlass{g,n}}) 
\end{equation}

identifies  $\Gamma(V(G), \Aff_{\Atlass{g,n}})$ with the group of all integral affine functions on $V(G_b)$ that are constant on $x+\relint(\sigma_{G_\infty})$ for some (and hence any) $x\in  \relint(\sigma_{G_b})$.
\end{lemma}

\begin{proof}
Since $V(G)$ is contained in the closure of $V(G_b)$, the restriction of functions from $V(G)$ to $V(G_b)$ is injective. By definition, the thickened paths appearing in cross ratio data on $G$ do not include any infinite edges of $G$, which implies that restrictions of integral affine functions on $V(G)$ to $V(G_b)$ are constant on $x+\relint(\sigma_{G_\infty})$ for any $x\in \relint(\sigma_{G,b})$. Conversely, let $\xi$ be an integral affine function on $V(G_b)$ that is constant on $x+\relint(\sigma_{G_\infty})$ for any $x\in \relint(\sigma_{G_b})$. 
{We may write $\xi=\xi_b +\xi_{\infty}$,
where $\xi_b$ is a linear combination of  primitive cross ratios  on $G_b$ that do not involve any of the infinite edges of $G$, and $\xi_{\infty}$ is a linear combination of primitive cross ratios of edge-type on $G_b$ that involve only edges in $E_\infty(G)$.
By part (3) of Lemma \ref{lem:properties of edge cross ratios}, two edge-type primitive cross ratios that use the same edge are proportional modulo primitive cross ratios of vertex-type. Hence, we may assume that
\begin{equation}
    \xi_\infty = \sum_{e \in E_\infty(G)} \xi_e \ ,
\end{equation}
where $\xi_e$ is a primitive edge-type cross ratio that uses the edge $e$.
Noting that the $\xi_e$'s are linearly independent on $x+\relint(\sigma_{G_\infty})$ for any $x\in \relint(\sigma_{G_b})$ implies that $\xi_\infty = 0$.
If $c$ is a primitive cross ratio datum on $G_b$ that does not involve any infinite edge of $G$, then $c$ induces a primitive cross ratio datum $c'$ on $G$ and $\xi_c=\xi_{c'}\vert_{V(G_b)}$. 
Since  $\xi = \xi_b$ is a linear combination of such primitive cross ratios, it follows that $\xi$ is the restriction of an affine function on $V(G)$.
}
\end{proof}

%{Andreas: \color{teal}Do we need to prove this or can we say that the proof is quite similar to that of Lemma \ref{lem:properties of vertex cross ratios}?\renzo{I think we can say the Lemma is proven in analogous way to the previous one. We should be consistent though with the labelings in the two lemmas: if we have $f_1, f_2$ etc in the first, we should not have $f_s, f_{s'}$ in the second.}}

\subsection{Tautological maps as morphisms of tropical spaces.}
\label{subsec:family over atlas}

In \cite[Section 4]{CCUW}, tautological  morphisms of moduli spaces of tropical curves are defined as morphisms of cone stacks. We show they are in fact morphisms of stacks on tropical spaces.

Set theoretically, each point $u\in \Atlass{g,n\sqcup\{\star\}}$ corresponds to a cycle-rigidified $(n\sqcup\{\star\})$-marked tropical stable curve $(C,(x_i)_{i\in n\sqcup\{\star\}})$,  of genus $g$. Omitting the marking $s_\star$ and stabilizing yields a cycle-rigidified $n$-marked tropical stable curve of genus $g$, which corresponds to a point in $\Atlass{g,n}$. This map is called the {\it forgetful morphism}, and we denote it by $\pi_\star$. For each $i\in I$ the map $\pi_\star$ has a section: if $(C,(x_i)_{i\in I})$ is an $I$-marked tropical stable curve of genus $g$, then we can attach a tripod (with three infinite edges) to the point $x_i$, move the marking $x_i$ to one of the two newly created ends, and mark the other end by $\star$. The new compact edge formed has infinite length. We denote this section by $s_i$.

\begin{proposition}
\label{prop:atlas gives family of TPL-curves}
Let $G$ be a cycle-rigidified stable graph. If $\gamma\colon \Atlass{g,n\sqcup \{\star\}} \to \Z$ is the function that assigns to every point in $\relint(\sigma_G)$ the genus of the vertex of $G$ that is adjacent to the leg marked by $\star$, then $(\pi_\star\colon \Atlass{g,n\sqcup\{\star\}}\to \Atlass{g,n}, (s_i)_{i\in [n]},\gamma)$ is a family of $n$-marked stable \TPL{}-curves of genus $g$. 

Morover, if $G$ is a cycle-rigidified $n$-marked genus-$g$ stable graph and $b\in \relint(\sigma_G)$, then the combinatorial type of the fiber $(\Atlass{g,n\sqcup\{\star\}})_b$ can be naturally identified with the underlying graph $\underline{G}$, and  the edge of $(\Atlass{g,n\sqcup\{\star\}})_b$ corresponding to $e\in E_b(G)$ has length $l_e(b)$.
\end{proposition}
\begin{proof} 
This proposition follows from \cite[Proposition 4.5]{CCUW} by observing that by construction $\pi_\star$ pulls-back piecewise linear functions to piecewise-linear functions. 
\end{proof}

\begin{proposition}
\label{prop:left-exactness of sequence on atlas}
The maps $\pi_\star$ and $s_i $ are morphism of tropical spaces. Furthermore, for every $x\in \Atlass{g,n\sqcup \{\star\}}$, the sequence
\begin{equation}
0\to 
\Omega^1_{\Atlass{g,n},\pi_\star(x)}\to 
\Omega^1_{\Atlass{g,n\sqcup\{\star\}},x} \to
\Omega^1_{(\Atlass{g,n\sqcup \{\star\}})_{\pi_\star(x)},x}  \to
0
\end{equation}
is exact.%, thus making  $(\pi_\star\colon \Atlass{g,n\sqcup\{\star\}}\to \Atlass{g,n}, (s_i)_{i\in I},\gamma)$  a family of $I$-marked stable  tropical curves.
\end{proposition}

\begin{proof}
Let $\underline{G}$ be a cycle-rigidified $n$-marked genus-$g$ stable graph, and let $c$ be a cross ratio datum. For any cycle-rigidified $(n\sqcup \{\star\})$-marked genus-$g$ stable graph $G$ whose stabilization after forgetting the $\star$-marked leg is $\underline{G}$, $c$ defines a cross ratio datum $c'$ on $G$ that does not involve the $\star$-marked leg. It follows immediately from the definitions that $(\pi_\star^*\xi_c)_x=(\xi_{c'})_x$ at any point of $\relint(\sigma_G)$. It follows that $\pi_\star$ is linear. It also follows that for any cross ratio datum $c'$ on $H$ which does not involve the $\star$-marked leg, the cross ratio $\xi_{c'}$ agrees with the pull-back of a cross ratio defined by a cross ratio datum $c$ on $\underline{G}$. 

To see that $s_i$ is linear, let $y\in \Atlass{g,n}$, and let $G$ be the cycle-rigidified $(n\sqcup\{\star\})$-marked genus-$g$ stable graph such that $s_i(y)\in \relint(\sigma_G)$. By construction of $s_i$, the $\star$-marked leg in $G$ is adjacent to a three-valent genus-zero vertex, all of whose adjacent edges are unbounded. It follows that there is no cross ratio datum on $G$ that involves the $\star$-marked leg. Therefore, every integral affine linear function $\xi$ at $s_i(y)$ is locally the pull-back via $\pi_\star$ of an integral-linear $\xi'$ function at $y$. Since $s_i$ is a section of $\pi_\star$, we then have $s_i^*\xi=\xi'$. We conclude that $s_i$ is linear. 

To prove the exactness of the sequence, let $x\in \Atlass{g,n\sqcup\{\star\}}$. 
Since the integral affine functions on the fiber over $\pi_\star(x)$ are precisely the restrictions of integral affine functions on $\Atlass{g,n\sqcup\{\star\}}$, the sequence is exact on the right.
The exactness on the left follows immediately from the fact that images of neighborhoods of $x$ under $\pi_\star$ are neighborhoods of $\pi_\star(x)$.  It remains to show exactness in the middle. Let $G$ be the cycle-rigidified $(n\sqcup\{\star\})$-marked genus-$g$ stable graph such that $x\in\relint(\sigma_G)$ and let $\xi\in (Aff_{\Atlass{g,n\sqcup \{\star\}}})_x$ be a function whose restriction to the fiber $(\Atlass{g,n\sqcup\{\star\}})_{\pi_\star(x)}$ is constant in a neighborhood of $x$. Denote by  $f_\star\in L(G)$ the $\star$-marked leg in $G$ and let $v\coloneqq v_G(f_\star)$. If $\gamma_G(v)>0$ or $v$ is a three-valent vertex whose adjacent edges are all unbounded, there is no cross ratio datum on $G$ involving $f_\star$. In that case, all integral affine functions at $x$ are pull-backs under $\pi_\star$ of integral affine functions at $\pi_\star(x)$ and we are done. We may thus assume $v$  is a genus-zero vertex $v$ that is either at least four-valent, or three-valent and adjacent to at least one bounded edge whose vertices have genus zero.

First assume the latter case. %, that is the $v$ is a three-valent genus-zero vertex. 
Let $\{f_1,f_2\}= T_v G \setminus {f_\star}$. Without loss of generality we may assume that $e_G(f_1)$ is a bounded edge whose vertices have genus zero. Let $v_1$ be the second vertex of $e_G(f_1)$  and let $f_{1,e},f_{1,e}'\in T_{v_1}G$ be distinct flags such that $e_G(f_{1,e})\neq e_G(f_1)\neq e_G(f_{1,e}')$. If $e_G(f_2)$ is either unbounded or adjacent to a vertex of positive genus, by Lemma \ref{lem:properties of edge cross ratios}, we have 
\begin{equation}
\xi \equiv \lambda\cdot\xi_{(e_G(f_1),(f_\star,f_{1,e}), (f_2,f_{1,e}'))} 
\end{equation}
for some $\lambda\in \Z$, where the congruence holds modulo cross ratios pulled back from $\Atlass{g,n}$.
The slope of the restriction of the cross ratio $\xi_{(e_G(f_1),(f_\star,f_{1,e}), (f_2,f_{1,e}'))}$ to the fiber $(\Atlass{g,n\sqcup\{\star\}})_{\pi_\star(x)}$ in the direction corresponding to $f_1$ is $-1$, so  $\xi$ being constant on the fiber implies $\lambda=0$. It follows that $\xi$ equals the pull-back of an integral affine function from  $\Atlass{g,n}$. If $e_G(f_2)$ is bounded and its vertices have genus zero, define $v_2$, $f_{2,e}$, and $f_{2,e}'$ similarly as for $f_1$. Then again by Lemma \ref{lem:properties of edge cross ratios}, we have
\begin{equation}
\xi \equiv \lambda_1\cdot\xi_{(e_G(f_1),(f_\star,f_{1,e}), (f_2,f_{1,e}'))}   + \lambda_2\cdot\xi_{(e_G(f_2),(f_\star,f_{2,e}), (f_1,f_{2,e}'))}  
\end{equation}
for some $\lambda_1,\lambda_2\in \Z$, where again the congruence holds modulo cross ratios pulled back from $\Atlass{g,n}$. %The slope of the restriction of the cross ratio $\xi_{(f_i,f_\star,f_{3-i},f_{i,e},f_{i,e}')}$ to the fiber $(\Atlass{g,n\sqcup\{\star\}})_{\pi_\star(x)}$ in the direction corresponding to $f_i$ is $-1$, so 
One may observe that $\xi$ being constant on the fiber implies $\lambda_1=\lambda_2$. Let $\underline{G}$ be the cycle-rigidified $I$-marked genus-$g$ stable graph obtained by removing $f_\star$ and $v$ from $G$ and replacing the edges $e_G(f_1)$ and $e_G(f_2)$ by a single edge $e$ connecting $v_1$ and $v_2$. Then $\pi_{\star}(x)\in \relint(\sigma_{\underline{G}})$.% Let $h\in F(\underline{G})$ be the flag orienting $e$ from $v_1$ to $v_2$. 
Then one sees that
\begin{equation}
\pi_\star^*(\xi_{(e,(f_{1,e},f_{2,e}),(f_{1,e}',f_{2,e}'))}) = \xi_{(e_G(f_1),(f_\star,f_{1,e}), (f_2,f_{1,e}'))}   + \xi_{(e_G(f_2),(f_\star,f_{2,e})(f_1,f_{2,e}'))} \ .
\end{equation}
It follows that $\xi$ equals the pull-back of an integral affine function from the $\Atlass{g,n}$. This concludes the analysis of $v$ being a three-valent vertex.

Now assume that $v$ is a genus-zero vertex of valence at least four. 
Observe that if $\xi$ is a primitive cross ratio of edge type containing $f_\star$, by Lemma \ref{lem:properties of edge cross ratios} (3), it may be replaced by a cross ratio of edge type not containing $f_\star$ plus a cross ratio of vertex type. We may therefore assume that we start with $\xi$  a linear combination of cross ratios of vertex type containing $f_\star$.

We  manipulate $\xi$ using the relations provided by Lemma \ref{lem:properties of vertex cross ratios}. Fix three distinct flags $f_1,f_2, f_3\in T_v G\smallsetminus \{f_\star\}$. Through repeated applications of part $(4)$ we may guarantee that all cross ratios in $\xi$ involve the flags $f_\star, f_1, f_2$. We may use parts $(1),(2),(3)$ to express $\xi$ as a linear combinations of the cross ratios 
$\xi_{((f_\star,f_1),(f_2,h))}$ and $\xi_{((f_\star,h),(f_1,f_2))}$, for any $h\in T_vG \smallsetminus \{f_\star, f_1, f_2\}$. With one more application of $(4)$ we observe:
\begin{equation}
\xi_{((f_\star,h),(f_1,f_2))} = \xi_{((f_\star,f_3),(f_1,f_2))}+\xi_{((f_3,h),(f_1,f_2))},
\end{equation}
hence modulo cross ratios not involving $f_\star$ we may assume that $h=f_3$ in the second family of cross ratios.

In conclusion, we may write 
\begin{equation}
\xi\equiv \mu\cdot \xi_{((f_\star,f_3),(f_1,f_2))}+ \sum_{h\in T_v G\setminus\{f_\star,f_1,f_2\}} \lambda_h\cdot \xi_{((f_\star,f_1),(f_2,h))}
\end{equation}
for appropriate choices of $\mu,\lambda_h\in \Z$, i.e. $\xi$ is a linear combination of $\val(v)-2$ cross ratios. Notice  that $\xi$ being constant on the fiber of $\pi_\star(x)$ imposes $\val(v) -2$ independent linear homogeneous conditions on the coefficients  $\mu$, $ \lambda_h$, hence $\xi\equiv 0$ modulo cross ratios not involving $f_\star$, which is equivalent to $\xi$ being-pulled back from an integral affine function from  $\Atlass{g,n}$.
\end{proof}
%\renzo{Made all graphs upstairs $G$ and downstaris $\underline{G}$. All points upstairs $x$ and downstairs $y$. Adjusted to val -2 because it seems we are referring to the graph upstairs.}

%Given distinct $g_1,\ldots,g_3\in T_v H\setminus\{f_\star\}$, the slope of the restriction of the cross ratio $\xi_{(f_\star,g_1,g_2,g_3)}$ to the fiber $(\Atlass{g,n\sqcup\{\star\}})_{\pi_\star(x)}$  is $-1$ in the direction corresponding to $g_3$ and $0$ in the direction of $g_2$ . Moreover, if $g_4\in T_v H\setminus \{f_\star,g_1,g_2,g_3\}$, then the slope of the cross ratio $\xi_{(f_\star,g_1,g_2,g_3)}$ to the fiber $(\Atlass{g,n\sqcup\{\star\}})_{\pi_\star(x)}$ in the direction corresponding to $g_4$ is $0$. It follows that $\lambda_h=0$ for all $h\neq f_3$ and we have
%\begin{equation}
%\xi\equiv \mu\cdot  \xi_{(f_\star,f_2,f_1,f_3)} + %\lambda_{f_3} \cdot \xi_{(f_\star,f_1,f_2,f_3)} \ .
%\end{equation}
%Because the slope of $\xi$ in the direction of $f_2$ is $0$, we have $\mu=0$, and because the slope of $\xi$ in the direction of $f_1$ is $0$, we have $\lambda_{f_3}=0$. It follows that $\xi$ equals the pull-back of an integral affine function from the $\Atlass{g,n}$.

\begin{proposition}
\label{prop:maps to atlas are linear}
Let $\pi\colon \mc C \to B$ be a family of $n$-marked genus-$g$ stable tropical curves, and let $f\colon B\to \Atlass{g,n}$ and $f_\star\colon \mc C\to \Atlass{g,n\sqcup\{\star\}}$ be morphisms of \TPL{}-spaces such that the diagram
\begin{center}
\begin{tikzpicture}[auto]
\matrix[matrix of math nodes, row sep= 5ex, column sep= 4em, text height=1.5ex, text depth= .25ex]{
|(C)| \mc C	&
|(n1)|	\Atlass{g,n\sqcup\{\star\}}	\\
|(B)| B	&	
|(n)| \Atlass{g,n}	\\
};
\begin{scope}[->,font=\footnotesize]
%%% horizontal
\draw (C) --node{$f_\star$} (n1);
\draw (B)--node{$f$} (n);
%%% vertical
\draw (C) --node{$\pi$} (B);
\draw (n1) --node{$\pi_\star$} (n);
\end{scope}
\end{tikzpicture}
\end{center}
is a Cartesian diagram of \TPL{}-spaces exhibiting the family of $n$-marked stable \TPL{}-curve $\pi\colon \mc C\to B$ as the pull-back of  $\pi_\star\colon \Atlass{g,n\sqcup\{\star\}}\to \Atlass{g,n}$. Then $f$ and $f_\star$ are linear.
\end{proposition}

\begin{proof}
{Since the affine structure on $\Atlass{g,n}$ is generated by  cross ratios, to show that $f$ is linear we need to show that cross ratios induce linear functions on $B$. }Let $b\in B$, and let $G$ be the cycle-rigidified, $n$-marked, genus-$g$ stable graph such that $f(b)\in \relint(\sigma_G)$. Since $f_\star$ induces an isomorphism $\mc C_b \cong (\Atlass{g, n\sqcup\{\star\}})_{f(b)}$, the combinatorial type of $\mc C_b$ is naturally identified with $G$ and the edge-length of the edge in $\mc C_b$ corresponding to  $e\in E_b(G)$ is given by $l_e(f(b))$. Let $c = (\gamma^\pm, w)$ be a cross ratio datum on $G$, and let us first assume that the image of $\gamma$ is a simple curve. Refer to Figure \ref{fig:techp} for the notation in this proof. Let $V_\gamma$ denote a contractible neighborhood of the image of $\gamma$ and $\omega\in \Omega^1_{\mc C_b}$ such that  $w = \gamma^\ast \omega$. By the exactness of \eqref{equ:exact sequence}, $\omega$ may be lifted to a one-form $\widetilde{\omega}\in \Omega^1_{\mc  C} (\widetilde{V}_\gamma)$, with $\widetilde{V}_\gamma$ an open set in $\mc C$ restricting to $V_\gamma$ on $\mc C_b$. Further, we may choose $\widetilde{\phi} \in Aff_{\mc C}(\widetilde{V}_\gamma)$ lifting $\widetilde{\omega}$.

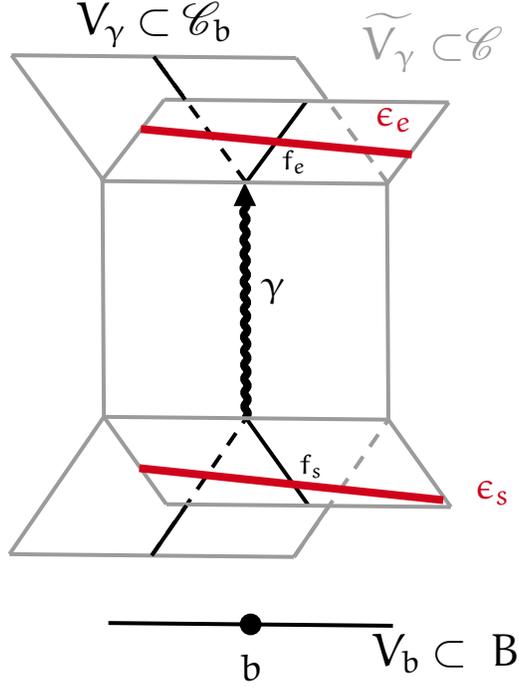
\begin{figure}[tb]
\begin{center}

\begin{tikzpicture}[x=0.75pt,y=0.75pt,yscale=-.7,xscale=.7]

%: \path (0,504); %set diagram left start at 0, and has height of 504

%Straight Lines [id:da29882146475339466] 
\draw [line width=1.5]    (169.13,455.38) -- (374.13,457.38) ;
%Straight Lines [id:da6235048401037409] 
\draw [line width=1.5]    (311.88,79) -- (268.38,137.75) ;
%Straight Lines [id:da8391032933868958] 
\draw [line width=1.5]  [dash pattern={on 5.63pt off 4.5pt}]  (268.38,137.75) -- (201.88,47) ;
%Straight Lines [id:da5245153164021734] 
\draw [line width=3]    (267.41,143.75) -- (267.46,151.75) .. controls (269.14,153.41) and (269.15,155.08) .. (267.49,156.75) .. controls (265.84,158.42) and (265.85,160.09) .. (267.52,161.75) .. controls (269.19,163.41) and (269.2,165.08) .. (267.55,166.75) .. controls (265.9,168.42) and (265.91,170.09) .. (267.58,171.75) .. controls (269.25,173.41) and (269.26,175.08) .. (267.61,176.75) .. controls (265.96,178.42) and (265.97,180.09) .. (267.64,181.75) .. controls (269.31,183.41) and (269.32,185.08) .. (267.67,186.75) .. controls (266.02,188.42) and (266.03,190.09) .. (267.7,191.75) .. controls (269.37,193.41) and (269.38,195.08) .. (267.73,196.75) .. controls (266.08,198.42) and (266.09,200.09) .. (267.76,201.75) .. controls (269.43,203.41) and (269.44,205.08) .. (267.78,206.75) .. controls (266.13,208.42) and (266.14,210.09) .. (267.81,211.75) .. controls (269.48,213.41) and (269.49,215.08) .. (267.84,216.75) .. controls (266.19,218.42) and (266.2,220.09) .. (267.87,221.75) .. controls (269.54,223.41) and (269.55,225.08) .. (267.9,226.75) .. controls (266.25,228.42) and (266.26,230.09) .. (267.93,231.75) .. controls (269.6,233.41) and (269.61,235.08) .. (267.96,236.75) .. controls (266.31,238.42) and (266.32,240.09) .. (267.99,241.75) .. controls (269.66,243.41) and (269.67,245.08) .. (268.02,246.75) .. controls (266.37,248.42) and (266.38,250.09) .. (268.05,251.75) .. controls (269.72,253.41) and (269.73,255.08) .. (268.08,256.75) .. controls (266.43,258.42) and (266.44,260.09) .. (268.11,261.75) .. controls (269.78,263.41) and (269.79,265.08) .. (268.14,266.75) .. controls (266.49,268.42) and (266.5,270.09) .. (268.17,271.75) .. controls (269.84,273.41) and (269.85,275.08) .. (268.2,276.75) .. controls (266.55,278.42) and (266.56,280.09) .. (268.23,281.75) .. controls (269.9,283.41) and (269.91,285.08) .. (268.26,286.75) .. controls (266.6,288.42) and (266.61,290.09) .. (268.28,291.75) .. controls (269.95,293.41) and (269.96,295.08) .. (268.31,296.75) .. controls (266.66,298.42) and (266.67,300.09) .. (268.34,301.75) .. controls (270.01,303.41) and (270.02,305.08) .. (268.37,306.75) -- (268.38,307.75) -- (268.38,307.75) ;
\draw [shift={(267.38,137.75)}, rotate = 89.66] [fill={rgb, 255:red, 0; green, 0; blue, 0 }  ][line width=0.08]  [draw opacity=0] (16.97,-8.15) -- (0,0) -- (16.97,8.15) -- cycle    ;
%Straight Lines [id:da934569850811219] 
\draw [line width=1.5]  [dash pattern={on 5.63pt off 4.5pt}]  (200.38,406.75) -- (268.38,307.75) ;
%Straight Lines [id:da9335899298769961] 
\draw [line width=1.5]    (268.38,307.75) -- (313.38,370.75) ;
%Shape: Circle [id:dp41042139427449786] 
\draw  [fill={rgb, 255:red, 0; green, 0; blue, 0 }  ,fill opacity=1 ][line width=1.5]  (265.25,456.38) .. controls (265.25,452.85) and (268.1,450) .. (271.63,450) .. controls (275.15,450) and (278,452.85) .. (278,456.38) .. controls (278,459.9) and (275.15,462.75) .. (271.63,462.75) .. controls (268.1,462.75) and (265.25,459.9) .. (265.25,456.38) -- cycle ;
%Straight Lines [id:da01215051713915849] 
\draw [color={rgb, 255:red, 155; green, 155; blue, 155 }  ,draw opacity=1 ][line width=1.5]    (164.88,136.75) -- (369.88,138.75) ;
%Straight Lines [id:da6005097638743242] 
\draw [color={rgb, 255:red, 155; green, 155; blue, 155 }  ,draw opacity=1 ][line width=1.5]    (209.38,78) -- (414.38,80) ;
%Straight Lines [id:da9583901752625315] 
\draw [color={rgb, 255:red, 155; green, 155; blue, 155 }  ,draw opacity=1 ][line width=1.5]    (99.38,46) -- (304.38,48) ;
%Straight Lines [id:da21520692153300125] 
\draw [color={rgb, 255:red, 155; green, 155; blue, 155 }  ,draw opacity=1 ][line width=1.5]    (165.88,306.75) -- (370.88,308.75) ;
%Straight Lines [id:da31955199520810207] 
\draw [color={rgb, 255:red, 155; green, 155; blue, 155 }  ,draw opacity=1 ][line width=1.5]    (210.88,369.75) -- (415.88,371.75) ;
%Straight Lines [id:da22253250307077677] 
\draw [color={rgb, 255:red, 155; green, 155; blue, 155 }  ,draw opacity=1 ][line width=1.5]    (97.88,405.75) -- (302.88,407.75) ;
%Straight Lines [id:da4955461575506921] 
\draw [color={rgb, 255:red, 155; green, 155; blue, 155 }  ,draw opacity=1 ][line width=1.5]  [dash pattern={on 5.63pt off 4.5pt}]  (370.88,138.75) -- (304.38,48) ;
%Straight Lines [id:da8810874165162497] 
\draw [color={rgb, 255:red, 155; green, 155; blue, 155 }  ,draw opacity=1 ][line width=1.5]    (414.38,80) -- (369.88,138.75) ;
%Straight Lines [id:da00458839940879896] 
\draw [color={rgb, 255:red, 155; green, 155; blue, 155 }  ,draw opacity=1 ][line width=1.5]    (369.88,138.75) -- (370.88,308.75) ;
%Straight Lines [id:da410595861397379] 
\draw [color={rgb, 255:red, 155; green, 155; blue, 155 }  ,draw opacity=1 ][line width=1.5]    (164.88,136.75) -- (165.88,306.75) ;
%Straight Lines [id:da8457362352232196] 
\draw [color={rgb, 255:red, 155; green, 155; blue, 155 }  ,draw opacity=1 ][line width=1.5]    (209.38,78) -- (163.88,137) ;
%Straight Lines [id:da7384324863361842] 
\draw [color={rgb, 255:red, 155; green, 155; blue, 155 }  ,draw opacity=1 ][line width=1.5]    (165.88,136.75) -- (99.38,46) ;
%Straight Lines [id:da9012360296475774] 
\draw [color={rgb, 255:red, 155; green, 155; blue, 155 }  ,draw opacity=1 ][line width=1.5]    (326.38,78) -- (304.38,48) ;
%Straight Lines [id:da09933871303499164] 
\draw [color={rgb, 255:red, 155; green, 155; blue, 155 }  ,draw opacity=1 ][line width=1.5]    (370.88,308.75) -- (415.88,371.75) ;
%Straight Lines [id:da7369138959971462] 
\draw [color={rgb, 255:red, 155; green, 155; blue, 155 }  ,draw opacity=1 ][line width=1.5]    (165.88,306.75) -- (210.88,369.75) ;
%Straight Lines [id:da9219601902380159] 
\draw [color={rgb, 255:red, 155; green, 155; blue, 155 }  ,draw opacity=1 ][line width=1.5]    (97.88,405.75) -- (165.88,306.75) ;
%Straight Lines [id:da08854975066661064] 
\draw [color={rgb, 255:red, 155; green, 155; blue, 155 }  ,draw opacity=1 ][line width=1.5]  [dash pattern={on 5.63pt off 4.5pt}]  (302.88,407.75) -- (370.88,308.75) ;
%Straight Lines [id:da7116881876093855] 
\draw [color={rgb, 255:red, 155; green, 155; blue, 155 }  ,draw opacity=1 ][line width=1.5]    (302.88,407.75) -- (327.38,371.75) ;
%Straight Lines [id:da8200300776769844] 
\draw [line width=1.5]    (200.38,406.75) -- (225.38,370.75) ;
%Straight Lines [id:da6937469490890911] 
\draw [line width=1.5]    (201.88,47) -- (224.38,78) ;
%Straight Lines [id:da4557899973036539] 
\draw [color={rgb, 255:red, 208; green, 2; blue, 27 }  ,draw opacity=1 ][line width=3]    (410.38,366.75) -- (191.38,343.75) ;
%Straight Lines [id:da8038055230657666] 
\draw [color={rgb, 255:red, 208; green, 2; blue, 27 }  ,draw opacity=1 ][line width=3]    (387.75,117.56) -- (192.5,99.19) ;

% Text Node
\draw (278,205.4) node [anchor=north west][inner sep=0.75pt]  [font=\Large]  {$\gamma $};
% Text Node
\draw (305,331.4) node [anchor=north west][inner sep=0.75pt]    {$f_{s}$};
% Text Node
\draw (292.13,111.78) node [anchor=north west][inner sep=0.75pt]    {$f_{e}$};
% Text Node
\draw (144,6.4) node [anchor=north west][inner sep=0.75pt]  [font=\LARGE]  {$V_{\gamma } \subset \mathcal{C}_{b}$};
% Text Node
\draw (350,13.4) node [anchor=north west][inner sep=0.75pt]  [font=\LARGE]  {$\textcolor[rgb]{0.61,0.61,0.61}{\widetilde{\textcolor[rgb]{0.61,0.61,0.61}{V}\textcolor[rgb]{0.61,0.61,0.61}{_{\gamma }}}}\textcolor[rgb]{0.61,0.61,0.61}{\ \subset }\textcolor[rgb]{0.61,0.61,0.61}{\mathcal{C}}$};
% Text Node
\draw (262,476.4) node [anchor=north west][inner sep=0.75pt]  [font=\Large]  {$b$};
% Text Node
\draw (431,351.4) node [anchor=north west][inner sep=0.75pt]  [font=\Large]  {$\textcolor[rgb]{0.82,0.01,0.11}{\epsilon }\textcolor[rgb]{0.82,0.01,0.11}{_{s}}$};
% Text Node
\draw (359,82.4) node [anchor=north west][inner sep=0.75pt]  [font=\Large]  {$\textcolor[rgb]{0.82,0.01,0.11}{\epsilon }\textcolor[rgb]{0.82,0.01,0.11}{_{e}}$};
% Text Node
\draw (357,461.4) node [anchor=north west][inner sep=0.75pt]  [font=\LARGE]  {$V_{b} \subset \ B$};

\end{tikzpicture}
\end{center}
\caption{We illustrate the notation relevant to the proof of Prop \ref{prop:maps to atlas are linear}.}
\label{fig:techp}
\end{figure}

By Lemma \ref{lem:local structure at interior of edge}, there exists a neighborhood $V$ of $b$ and linear sections $\epsilon_s, \epsilon_e \colon V\to \mc C$ of $\pi$, whose images are contained in $f_s$ and $f_e$.

%After potentially shrinking $U$, we may assume that $\epsilon_0(V)\subseteq U_1$, that $\epsilon_n(V)\subseteq U_n$, that $\epsilon_k(V)\subseteq U_k\cap U_{k+1}$ for all $1\leq k\leq n-1$, and that $\phi_k=\phi_{k+1}$ on a neighborhood of $\epsilon_k(U)$ for all $1\leq  k\leq n-1$. Then 

By construction, we have
\begin{equation}
(f^*\xi_c)_b = \left( \epsilon_e^*\widetilde{\phi}-\epsilon_s^*\widetilde{\phi}\right)_b \ .
\end{equation}
The right side of this equation is contained in $\Aff_{B,b}$. We conclude that $f$ is linear at $b$, and since $b$ was arbitrary that $f$ is linear.
If $\gamma$ is not a simple path, one may apply a standard local lifting argument to a refinement of $\gamma$ and conduct essentially the same proof.

To show that $f_\star$ is linear let $x\in \mc C$ be a point in $\mc C_b$. The position of $x$ in $\mc C_b$, together with the identification of the combinatorial type of $\mc C_b$ with $G$ induced by $g$ uniquely determines the $(n\sqcup \{\star\})$-marked graph $G_\star$ with $f_\star(x)\in \relint(\sigma_{G_\star})$. Let $c=(\gamma^\pm,w)$ be a cross ratio datum on $G_\star$. By Proposition \ref{prop:primitive cross ratios generate all cross ratios} we may assume that $c$ is primitive. If $c$ does not involve the $\star$-marked leg then $c$ defines a cross ratio datum $c'$ on $G$ and $\xi_c=\pi_\star^*\xi_{c'}$. Then $f_\star^*\xi_c=(f\circ \pi)$ is integral affine at $x$ by the linearity of $\pi$ and $f$. We may thus assume that $c$ involves the $\star$-marked leg. By Lemma \ref{lem:properties of vertex cross ratios} and Lemma \ref{lem:properties of edge cross ratios} we may further assume that the thickened path $\gamma^\pm$ does not involve the $\star$-marked leg, but only the one-form does. %In this case, one can construct a path $\gamma\colon[0,1]\to \mc C_b$ according to $t$, but since defines $\mc C_b$ is not $(n\sqcup\{\star\})$-marked, the one-from $w$ does not define a one-form along $\gamma$.
Let $\Gamma$ be the stable tropical curve obtained by attaching to $\mc C_b$  an additional leg at $x$  and marking it by $\star$. Then %$\gamma$ can be viewed a path in $\Gamma$ and since $\Gamma$ is $(n\sqcup\{\star\})$-marked and of combinatorial type $H$, 
$w$ is an element  of $ \Gamma([0,1],\gamma^{-1}\Omega^1_\Gamma)$. Because the inclusion $\mc C_b\to \Gamma$ is not linear, $w$ cannot be pulled back to an element in $\Gamma([0,1],\gamma^{-1}\Omega^1_{\mc C_b})$. However, it can be pulled back to an element $\omega\in \Gamma([0,1],\gamma^{-1}\Omega^1_{\mc C_b}(kx))$ for an appropriate choice of $k\in \Z$ (since $c$ is primitive, $k=\pm 1$). Let $\mc D=\mc C\times_B\mc C$ and let $\Delta\colon \mc C\to \mc D$ be the diagonal map. Then the projection $\mc D\to \mc C$, together with the pull-backs of the sections and the genus function, is an $n$-marked genus-$g$ stable tropical curve. The fiber $\mc D_x$ can be canonically identified with $\mc C_b$, and $x$ maps to $\Delta(x)$ in this identification. Therefore, $\gamma$ can be viewed as a path in $\mc D_x$, and $\omega$ can be viewed as an element in $\Gamma([0,1],\gamma^{-1}\Omega^1_{\mc D_x}(k\Delta(x))$. By Proposition \ref{prop:aff(ks) is a torsor} and Proposition \ref{prop:aff(ks) is everywhere inhabited}, $\omega$ lifts to an element in $\Gamma([0,1],\gamma^{-1}\Omega^1_{\mc D}(k\Delta))$. Using Proposition \ref{prop:aff(ks) is a torsor}, one can now proceed similarly as in the proof of the linearity of $f$ above.

\end{proof}

\begin{definition} \label{def:mumfcusp}
We say that a stable tropical curve is a \textbf{Mumford-curve} if the genus-function of its combinatorial type is identically zero. We denote by $\MgnMf{g,n}$ the full subcategory of $\Mgnbar{g,n}$ consisting of all families of Mumford curves. %category fibered in groupoids as well.
\end{definition}

\begin{proposition}\label{prop:mumfopen}
For every pair $g,n\in \Z^{\geq 0}$  with $2g-2+n>0$, the stack $\MgnMf{g,n}$ is an open substack of $\Mgnbar{g,n}$, that is the inclusion $\MgnMf{g,n}\to \Mgnbar{g,n}$ is represented by an open immersion.
\end{proposition}

\begin{proof}
We need to show that for every family $\mc C\to B$ of stable tropical curves the locus 
\begin{equation}
U\coloneqq \{b\in B \mid \mc C_b \text{ is a Mumford curve}\}
\end{equation}
is open in $B$. Let $b\in U$, and let $(\Sigma,(G_\sigma,f_\sigma,\chi_\sigma)_{\sigma\in \Sigma})$ be a \TPL{}-trivialization of $\mc C\to B$ at $b$. Let $\tau\in \Sigma$ be the inclusion-minimal polyhedron in $\Sigma$ containing $b$. By assumption, the function $\gamma_{G_\tau}$ is identically zero. For every $\sigma$ in $\Sigma$, the polyhedron $\tau$ is a face of $\sigma$, and therefore the combinatorial type $G_\tau$ is a contraction of $G_\sigma$. It follows immediately that $\gamma_{G_\sigma}$ is identically zero as well. It follows that $U$ contains the interior of  $\vert \Sigma\vert$.
\end{proof}

\begin{definition}\label{def:good}
We denote by $\GoodAtlas{g,n}\subseteq \Atlass{g,n}$ the union of the relative interiors of the cones $\sigma_G$ corresponding to cycle-rigidified stable graphs $G$ where every bounded edge is adjacent to at least one genus-zero vertex, and every vertex of an unbounded edge has genus zero.
Moreover, we denote by $\Atlas{g,n}\subset \GoodAtlas{g,n}$ the union of the relative interiors of the cones $\sigma_G$ corresponding to cycle-rigidified stable graphs $G$ with $\gamma_G$  identically zero. 

\end{definition}

\begin{lemma}
\label{lem:forgetful morphism between Mumford loci is proper}
The subset $\Atlas{g,n}$ is open in $\Atlass{g,n}$.  Furthermore, we have $\pi_\star^{-1}\Atlas{g,n}=\Atlas{g,n\sqcup\{\star\}}$.
\end{lemma}

\begin{proof}
The openness follows directly from Proposition \ref{prop:mumfopen} and Proposition \ref{prop:atlas gives family of TPL-curves}. The second statement follows from the fact that forgetting a marking and stabilizing never increases the genus function.
%which means that if $\Gamma'$ is obtained from $\Gamma$ by forgetting a mark and stabilizing, then $\Gamma$ is a Mumford curve if and only if $\Gamma'$ is a Mumford curve. 
\end{proof}

\begin{proposition}
\label{prop:family of tropical curve over atlas}
All integral linear functions on $\Atlass{g,n\sqcup \{\star\}}$ are harmonic on the fibers of $\pi_\star\colon \Atlass{g,n\sqcup\{\star\}}\to \Atlass{g,n}$. Moreover, the fibers of $\pi_\star$ over $\GoodAtlas{g,n}$ are stable tropical curves and $\pi_\star^{-1}\GoodAtlas{g,n}\to \GoodAtlas{g,n}$, together with the induced sections and the induced genus function, is a family of $n$-marked genus-$g$ stable tropical curves.
In particular, $\Atlas{g, n\sqcup\{\star\}}\to \Atlas{g,n}$ is a family of $n$-marked genus-$g$ stable tropical curves.
\end{proposition}

\begin{proof}
Let $G$ be a cycle-rigidified $(n\sqcup\{\star\})$-marked genus-$g$ stable graph and let $c$ be a cross ratio datum on $G$. We need to show that $\xi_c$ is harmonic on all fibers. By Proposition \ref{prop:primitive cross ratios generate all cross ratios} we may assume that $c$ is primitive. If $c$ does not involve the $\star$-marked flag $f_\star$, then $\xi_c$ coincides with the pull-back of a cross ratio from $\Atlass{g,n}$ and thus is constant on all fibers. We may thus assume that $c$ involves $f_\star$. By Lemma \ref{lem:properties of vertex cross ratios} and Lemma \ref{lem:properties of edge cross ratios}, we may assume that either
$
c=c_{((f_\star,f_e),(f_{-1},f_1))}
$
%for appropriate flags $f_1$, $f_2$, and $f_3$, 
or $
c=c_{(e, (f_\star,f_e),(f_{-1},f_1))}.
$
%for appropriate flags $h_1,\ldots, h_4$.
In the first case, the restriction of $\xi_c$ to a fiber near a point in $\relint(\sigma_G)$ has slope $1$ in the direction corresponding to $f_{-1}$, % slope $0$ in the direction corresponding to $f_e$, 
slope $\pm 1$ in the direction of corresponding to $f_{\mp 1}$, 
and slope $0$ in all other directions. It follows that the restrictions of $\xi_c$ to the fibers of $\pi_\star$ are harmonic. Moreover, by 
%fixing $f_{-1}$ and varying $f_e$ and $f_1$, we see 
taking appropriate linear combinations of primitive vertex-type cross ratios, we see
that if $v$ is a genus-zero vertex of valence at least three in a fiber of $\pi_\star$, then all harmonic functions in a neighborhood of $v$ in $\pi_\star^{-1}(\pi_\star(v))$ are restrictions of integral affine functions on $\Atlass{g,n\sqcup\{\star\}}$.

In the second case, the restriction of $\xi_c$ to a fiber near a point in $\relint(\sigma_G)$ has slope $-1$ in the direction corresponding to $e$, slope $1$ in the direction corresponding to $f_{-1}$, and slope $0$ in all other directions. It follows that the restriction of $\xi_c$ to the fibers of $\pi_\star$ is harmonic as well. Moreover, we see that if $v$ is a point in the interior of an edge in a fiber of $\pi_\star$ that is adjacent to at least one finite genus-zero vertex, then all harmonic functions in a neighborhood of $v$ in $\pi_\star^{-1}(\pi_\star(v))$ are restrictions of integral affine functions on $\Atlass{g,n\sqcup\{\star\}}$. 

If $b\in \GoodAtlas{g,n}$, then all points of the fiber $(\Atlas{g,n\sqcup\{\star\}})_b$ are either nodes, markings, genus-zero vertices of valence at least three, or points in the interior of an edge that is adjacent to at least one finite genus-zero vertex. By what we have seen, the restrictions of integral affine functions on $\Atlass{g,n\sqcup\{\star\}}$ to $(\Atlas{g,n\sqcup\{\star\}})_b$ are precisely the harmonic functions. It follows that $(\Atlas{g,n\sqcup\{\star\}})_b$ is a {stable} tropical curve. Together with Proposition \ref{prop:left-exactness of sequence on atlas}, it follows that the $\pi_\star^{-1}\GoodAtlas{g,n}\to \GoodAtlas{g,n}$ is a family of $n$-marked genus-$g$ stable tropical curves.

The ``in particular''-statement is a consequence of Lemma \ref{lem:forgetful morphism between Mumford loci is proper} and the fact that $\Atlas{g,n}$ is contained in $\GoodAtlas{g,n}$.
\end{proof}

\begin{lemma}
\label{lem:bijection on fibers defines isomorphism}
Let $\pi\colon\mc C\to B$ and $\pi'\colon\mc C'\to B$ be two families of $n$-marked genus-$g$ pre-stable \TPL{}-curves, and let $f\colon \mc C\to \mc C'$ be a bijective morphism of families of $n$-marked pre-stable \TPL{}-curves such that the induced morphisms $\mc C_b\to \mc C'_b$ are isomorphisms for every $b\in B$. Then $f$ is an isomorphism of \TPL{}-curves. Moreover, if both $\pi$ and $\pi'$ are families of tropical curves and $f$ is linear, then $f$ is an isomorphism of tropical curves.
\end{lemma}

\begin{proof}
Let $b\in B$. Let $(\Sigma,(G_\sigma,f_\sigma,\chi_\sigma)_{\sigma\in \Sigma})$ and $(\Sigma',(G_\sigma',f_\sigma',\chi_\sigma')_\sigma\in\Sigma)$ be \TPL{}-trivializations of $\pi$ and $\pi'$, respectively. After taking a common refinement of $\Sigma$ and $\Sigma'$, we may assume that $\Sigma=\Sigma'$. Since $f$ is an isomorphism on fibers it induces for every $\sigma\in\Sigma$ an isomorphism $G_\sigma\xrightarrow\cong G_\sigma'$ such that the resulting morphism
\begin{equation}
\mc C\times_B^{\mTPL}\sigma\xrightarrow{\chi_\sigma}
\mc C_{G_\sigma} \times^{\mTPL}_{\overline\sigma_{G_\sigma}} \sigma\xrightarrow\cong
\mc C_{G_\sigma'} \times^{\mTPL}_{\overline\sigma_{G_\sigma'}} \sigma \xrightarrow{\chi'_\sigma{}^{-1}}
\mc C'\times_B^{\mTPL}\sigma
\end{equation}
coincides with the restriction of $f$ to $\mc C\times_B^{\mTPL}\sigma$. Therefore, the restriction of $f$ to $\pi^{-1}\sigma$ is an isomorphism onto its image for all $\sigma\in\Sigma$, which implies that the restriction of $f$ to $\pi^{-1}\vert\Sigma\vert$ is an isomorphism of \TPL{}-spaces. Since being an isomorphism is local, we conclude that $f$ is an isomorphism of \TPL{}-spaces.

Now assume that both $\pi$ and $\pi'$ are families of tropical curves and that $f$ is linear.
%Finally, we show that $f$ is an isomorphism of tropical spaces. 
It suffices to show that the pull-back $\Omega^1_{\mc C'}\to \Omega^1_{\mc C}$ is an isomorphism. So let $p\in \mc C$, let $p'=f(p)$, and let $b=\pi(b)$. Then by the definition of families of semi-stable curves, there exists a commutative diagram
\begin{equation}
\begin{tikzpicture}[auto]
\matrix[matrix of math nodes, row sep= 5ex, column sep= 4em, text height=1.5ex, text depth= .25ex]{
|(0lo)| 0 & [-1.5em]
|(Bo)| \Omega^1_{B,b} &
|(Cp)|  \Omega^1_{\mc C',p'} &
|(Cpfib)| \Omega^1_{\mc C'_b,p'} &
|(0ro)| 0 \\
|(0lu)| 0 &
|(Bu)| \Omega^1_{B,b} &
|(C)|  \Omega^1_{\mc C,p} &
|(Cfib)|  \Omega^1_{\mc C_b,p} &
|(0ru)|  0 \\
};
\begin{scope}[->,font=\footnotesize]
%%%%%% horizontal
%%% row 1
\draw (0lo) -- (Bo);
\draw (Bo)--(Cp);
\draw (Cp)--(Cpfib);
\draw (Cpfib)--(0ro);
%%% row 2
\draw (0lu) -- (Bu);
\draw (Bu)--(C);
\draw (C)--(Cfib);
\draw (Cfib)--(0ru);
%%%%%% vertical
\draw (Bo)--node{$\id$}(Bu);
\draw (Cp)--(C);
\draw (Cpfib)--(Cfib);
\end{scope}
\end{tikzpicture} 
\end{equation}
with exact rows. The leftmost vertical arrow is the identity, and hence an isomorphism, and the rightmost vertical arrow is an isomorphism because the morphism $\mc C_b\to \mc C'_b$ of the fibers induced by $f$ is an isomorphism by assumption. Therefore, the vertical arrow in the middle is an isomorphism as well, finishing the proof.
\end{proof}

\begin{proposition}

\label{prop:local lifts to atlas}
Let $\pi\colon \mc C\to B$ be a family of $n$-marked genus-$g$ Mumford $\TPL{}$-curves, let $b\in B$, and assume we are given a Cartesian diagram of \TPL{}-spaces
\begin{center}
\begin{tikzpicture}[auto]
\matrix[matrix of math nodes, row sep= 5ex, column sep= 4em, text height=1.5ex, text depth= .25ex]{
|(Cb)| \mc C_b	&
|(Istar)|	\Atlas{g,n\sqcup\{\star\}}	\\
|(b)| \{b\}	&	
|(I)| \Atlas{g,n}	\\
};
\begin{scope}[->,font=\footnotesize]
%%% horizontal
\draw (Cb) -- (Istar);
\draw (b)-- (I);
%%% vertical
\draw (Cb) -- (b);
\draw (Istar) --node{$\pi_\star$} (I);
\end{scope}
\end{tikzpicture}
\end{center}
that is compatible with the sections.

Then there exists a neighborhood $U$ of $b$ in $B$ and morphisms
\begin{align*}
f& \colon U \to \Atlas{g,n} \\
f_\star&\colon\mc C_U \to  \Atlas{g,n\sqcup\{\star\}}
\end{align*}
that are compatible with the sections, such that the diagram 

\begin{center}
\begin{tikzpicture}[auto]
\matrix[matrix of math nodes, row sep= 5ex, column sep= 4em, text height=1.5ex, text depth= .25ex]{
|(Cb)| \mc C_b	&
|(CU)|	\mc C_U		&
|(Istar)|	\Atlas{g,n\sqcup\{\star\}}	\\
|(b)| \{b\}	&	
|(U)| U &
|(I)| \Atlas{g,n}	\\
};
\begin{scope}[->,font=\footnotesize]
%%% horizontal
\draw (Cb)-- (CU);
\draw (CU) -- node{$f_\star$}  (Istar);
\draw (b)-- (U);
\draw (U)--node{$f$}(I);
%%% vertical
\draw (Cb) -- (b);
\draw (CU)--(U);
\draw (Istar) --node{$\pi_\star$} (I);
\end{scope}
\end{tikzpicture}
\end{center}
is commutative and both squares are Cartesian in the category of \TPL{}-spaces. Moreover, if $\mc C\to B$ is a family of tropical curves, then both $f$ and $f_\star$ are linear and the squares are Cartesian in the category of tropical spaces.
\end{proposition}

\begin{proof}
let $(\Sigma,(G_\sigma,f_\sigma,\chi_\sigma)_{\sigma\in \Sigma})$ be a \TPL{}-trivialization of $\pi$ at $b$, and let $\tau\in \Sigma$ be the unique stratum with $b\in\relint(\tau)$. Let $x$ be the image of $b$ in $\Atlas{g,n}$, and let $H$ be the unique cycle-rigidified $n$-marked genus-$g$ stable graph with $x\in \relint(\sigma_H)$. The given isomorphism $\mc C_b\cong (\Atlas{g,n\sqcup\{\star\}})_b$ induces an isomorphism $H\cong G_\tau$ of stable graphs, and thus induces a cycle rigidification on $G_\tau$. For every face $\sigma\in\Sigma$, the stable graph $G_\sigma$ specializes to $G_\tau$. Since $\mc C\to B$ is a family of Mumford curves, the cycle rigidification on $G_\tau$ lifts uniquely to a cycle-rigidification on $G_\sigma$. With these cycle rigidifications, each pair of morphisms $(f_\sigma,\chi_\sigma)$ defines a Cartesian square
\begin{center}
\begin{tikzpicture}[auto]
\matrix[matrix of math nodes, row sep= 5ex, column sep= 4em, text height=1.5ex, text depth= .25ex]{
|(CU)|	\mc C_\sigma		&
|(Istar)|	\Atlas{g,n\sqcup\{\star\}}	\\
|(U)| \sigma &
|(I)| \Atlas{g,n}	\\
};
\begin{scope}[->,font=\footnotesize]
%%% horizontal
\draw (CU) -- node{$f_{\star,\sigma}$}  (Istar);
\draw (U)--node{$f_\sigma$}(I);
%%% vertical
\draw (CU)--(U);
\draw (Istar) --node{$\pi_\star$} (I);
\end{scope}
\end{tikzpicture}
\end{center}
of \TPL{}-spaces that is compatible with the sections. By construction, the morphisms $f_\sigma$ and $f_{\star,\sigma}$ glue to morphisms $f\colon \vert \Sigma\vert\to \Atlas{g,n}$ and $f_\star\colon \mc C_{\vert \Sigma\vert} \to \Atlas{g,n \sqcup \{\star\}}$, so we can take $U=\vert\Sigma\vert$.

Now assume that $\mc C \to B$ is a family of tropical curves. By Proposition \ref{prop:maps to atlas are linear}, both $f$ and $f_\star$ are linear. By Lemma \ref{lem:bijection on fibers defines isomorphism}, it follows that the given squares are Cartesian in the category of tropical spaces.
\end{proof}

\begin{theorem}
\label{thm:existance of atlas}
The morphism $\Atlas{g,n}\to \MgnMf{g,n}$  defined by the family $\pi_\star\colon\Atlas{g,n\sqcup\{\star\}}\to \Atlas{g,n}$ is a covering of $\MgnMf{g,n}$, that is for any morphism $T\to \MgnMf{g,n}$ whose domain is a tropical space, the projection $T\times_{\MgnMf{g,n}} \Atlas{g,n}\to T$ is a local isomorphism whose underlying morphism of topological spaces is a covering.
\end{theorem}

\begin{proof}
Let $T\to \MgnMf{g,n}$ be a morphism from a tropical space $T$, represented by a family $\mc C\to T$ of $n$-marked genus-$g$ stable tropical curves. Then the fiber product $F\coloneqq T\times_{\MgnMf{g,n}} \Atlas{g,n}$ represents the functor that assigns to any tropical $T$-space $S\xrightarrow f T$ the set of pairs $(\varphi,\chi)$ consisting of a morphism $\varphi \colon S\to \Atlas{g,n}$ and an isomorphism $\chi\colon f^*\mc C\to \varphi^*\Atlas{g,n\sqcup\{\star\}}$. First we show that $F\to T$ has finite fibers and that the cardinality of the fibers is locally constant on $T$. Let $t\in T$, and let $G$ be the combinatorial type of $\mc C_t$. For every isomorphism type $G'$ of cycle-rigidified graph whose underlying stable graph is isomorphic to $G$  there is a point $x$ in the relative interior of $\overline\sigma_{G'}$ such that $\mc C_t$ is isomorphic to $(\Atlas{g,n\sqcup\{\star\}})_x$. If $C$ denotes the set of cycle-rigidifications of $G$, then there are precisely $\vert C/\Aut(G)\vert$ choices for $G'$. By Lemma \ref{lem:ridig graphs are rigid}, that is $\vert C\vert/\vert \Aut(G)\vert$ many choices. Now for each such choice of $G'$, the group $\Aut(G)$ acts naturally on $\overline\sigma_{G'}$, and it acts transitively on the set of all $x\in \relint(\sigma_{G'})$ such that $\mc C_t$ is isomorphic to $(\Atlas{g,n\sqcup\{\star\}})_x$. Finally, for every such $x$ there are $\Aut(G)_x$ many isomorphisms $\mc C_t\to (\Atlas{g,n\sqcup\{\star\}})_x$. Applying the Orbit-Stabilizer Theorem yields that $t$ has precisely 
\begin{equation}
\vert C\vert/\vert\Aut(G)\vert \cdot \vert \Aut(G)\vert= \vert C\vert 
\end{equation}
inverse images in $F$. We conclude that the number of inverse images of a point $t\in T$ is the number of cycle-rigidifications of the combinatorial type of $\mc C_t$. To see the that this number is locally constant on $T$, let $(\Sigma,(G_\sigma,f_\sigma,\chi_\sigma)_{\sigma\in\Sigma})$ be a \TPL{}-trivialization of $\mc C\to T$ at a point $t\in T$. Let $\tau\in \Sigma$ be the unique minimal polyhedron in $\Sigma$. Then $t\in \relint(\tau)$, and in particular $G_\tau$ is the combinatorial type of $\mc C_t$. For all $\sigma\in \Sigma$, the combinatorial type $G_\tau$ is a specialization of $G_\sigma$, so cycle-rigidifications on $G_\tau$ lift uniquely to cycle-rigidification of $G_\sigma$. Since for every $t'\in \relint(\sigma)$, the curve $\mc C_{t'}$ is of combinatorial type $G_\sigma$, this shows that the fibers of $F\to T$ have the same cardinality on $\vert \Sigma\vert$. Together with Proposition \ref{prop:local lifts to atlas}, this finishes the proof.
\end{proof}

Since in genus zero every curve is a Mumford curve, we immediately obtain the following corollary.

\begin{corollary}
\label{cor:representative for M0n}
For  $n\geq 3$, the stack $\Mgnbar{0,n}$ is represented by $\Atlass{0,n}=\Atlas{0,n}$.
\end{corollary}

{
\begin{example}
\label{ex:fiber product over M11}
We have already noted earlier that there does not exist a local isomorphism  onto the full moduli space $\Mgnbar{g,n}$ for any $g>0$: in fact, the issue is topological; because the isotropy group of a tropical curve with vertices of genus greater one is too small compared to the isotropy groups of Mumford curves that are close to it in the moduli space, there does not  exist a local homeomorphism onto $\Mgnbar{g,n}$. This beckons the question whether there is a well-behaved class of morphisms of tropical spaces besides local isomorphisms in which one could look for an atlas for $\Mgnbar{g,n}$. We do not treat this question in detail, but instead provide a computation that makes this seem like a fruitless endeavor. Namely, we discuss the surjections $f_a\colon \T\PP^1\to \Mgnbar{1,1}$ associated to the families $\mc C^a\to \T\PP^1$ from Example \ref{ex:tropfam}. To answer the question if these are ``well-behaved'', we compute the fiber products $F_{a,b}= \T\PP^1 {\times}_{f_a,\Mgnbar{1,1},f_b}\T\PP^1$ for all $a,b\in \Z$. If $p_1,p_2\colon \T\PP^1\times \T\PP^1\to \T\PP^1$ denote the projections to the first and second factor, the fiber product $F_{a,b}$ is given by $\Isom_{\T\PP^1\times \T\PP^1}(p_1^*\mc C^a,p_2^*\mc C^b)$. This is computed following the proof of Theorem \ref{thm:representable diagonal}. Analogously to Example \ref{ex:diagonal is representable}, one sees that $F_{a,b}$ is the tropical space obtained by doubling the rays of the tropical subspace
\begin{equation}
    S= \Rbar_{\geq 0} \begin{pmatrix}
    1\\1\\0
    \end{pmatrix}
    \cup
    \Rbar_{\geq 0} \begin{pmatrix}
    1\\-1\\b
    \end{pmatrix}
    \cup\Rbar_{\geq 0} \begin{pmatrix}
    -1\\-1\\b-a
    \end{pmatrix}
    \cup\Rbar_{\geq 0} \begin{pmatrix}
    -1\\1\\-a
    \end{pmatrix}
\end{equation}
of $(\T\PP^1)^3$. More precisely, the fiber product $F_{a,b}$ is obtained by gluing two copies of $S$ at their respective origins $(0,0,0)$; the affine functions are generated by pull-backs of affine functions on $S$ on the glued space. The two projections from $F_{a,b}$ to $\T\PP^1$ are given by the composite of the natural map $F_{a,b}\to S$ and the projections from $S$ onto the first and second coordinate of $(\T\PP^1)^3$. One first notes that, as expected, neither of the projections is a local isomorphism, or even a local homeomorphism, over the origin $0\in \T\PP^1$ due to the lack of nontrivial automorphisms of $\mc C^a_0$ and $\mc C^b_0$. Moreover, neither $F_{a,b}$ nor $S$ is smooth over the origin even though $\T\PP^1$ is; if $a\neq b$ there even is a non-harmonic affine function on $F_{a,b}$. In particular, tropical cycles on either of the $\T\PP^1$ factors can in general not be pulled back to $F_{a,b}$ (see Section \ref{sec:intersection theory} for the definition of tropical cycles on tropical spaces).
\end{example}
}

\subsection{The moduli space $\Mgnbar{0,n}$ is a tropical toric variety}

It is known \cite{SpeyerSturmfels,GathmannKerberMarkwig} that the underlying set of $\Mgnbar{0,n}$ can be mapped onto a fan in affine space using the distances between the markings. Let us quickly recall the details. One embeds $\R^n$ into $\R^{n\choose 2}$ (where $n\choose 2$ counts the  two-element subsets of $[n]$) by mapping a tuple $(x_i)_{i}$ to the tuple $(x_i+x_j)_{\{i,j\}}$. We can then map a moduli point $[\Gamma]\in\Mgn{0,n}$ to the point $d(\Gamma)\coloneqq(d_\Gamma(i,j)/2)_{\{i,j\}}$ in the quotient $\R^{n\choose 2}/\R^n$, where $d_\Gamma(i,j)$ denotes the distance between %any pair of points $p_i$ and $p_j$ of $C$ that lie on the interiors of the edges adjacent
the vertices opposite to the markings $i$ and $j$, respectively. The image of all curves of a given combinatorial type is a an open cone, and the closures of these cones define a rational, polyhedral fan in $\R^{n\choose 2}/\R^n$, which we  denote by $\Sigma_n$. 

\begin{remark}
The factor $\frac 12$ in the definition of $d(\Gamma)$ does not appear in \cite{GathmannKerberMarkwig} or \cite{SpeyerSturmfels}, but has been shown in \cite{GrossCorrespondence} to appear when tropicalizing the Pl{\"u}cker-coordinates on the algebraic moduli space  $\Mgnbar{0,n}^\alg$ of $n$-marked genus-zero stable curves.
\end{remark}

\begin{lemma}
\label{lem:cross ratio vs distance coordinates}
%Let $\sigma\in\Sigma_n$ be the cone corresponding to the combinatorial type of a tropical curve $C\in \Mgn{0,n}$. For two distinct numbers $1\leq i,j\leq n$, let $x_{ij}$ denote the coordinate corresponding to the two-element subset $\{i,j\}\subseteq \{1,\ldots, n\}$. 
Let $C$ be the subgroup of the integral linear functions on $\R^{n\choose 2}/\R^n$ generated by the functions
\begin{equation}
x_{ij}+x_{kl}-x_{ik}-x_{jl} 
\end{equation}
corresponding to four distinct elements $i,j,k,l\in [n]$.
% whose corresponding legs in $C$ are adjacent to the same vertex. 
Then the functions in $C$ are precisely the integral linear functions on $\R^{n\choose 2}/\R^n$. 
%that vanish on $\sigma$. 
\end{lemma}

\begin{proof}
First of all, note that the functions $x_{ij}+x_{kl}-x_{ik}-x_{jl}$ are integral linear functions on $\R^{n\choose 2}$ that vanish on $\R^n$. Therefore, they define integral linear functions on the quotient space. 
%Furthermore, if $p$ is a point in the relative interior of $\sigma$ corresponding to a curve $C'$ of the same combinatorial type of $C$, and if the legs  corresponding to $i,j,k,l$ are all adjacent to the same vertex in $C$ (and hence in $C'$), then $p$ can be represented by a point in $\R^{n\choose 2}$ whose $x_{ij}$-, $x_{kl}$-, $x_{ik}$-, and $x_{jl}$-coordinates are all $0$. Therefore, $x_{ij}+x_{kl}-x_{ik}-x_{jl}$ vanishes on $p$, and, since $p$ was chosen arbitrary, on all of $\sigma$.

Assume there exists a nonzero integral linear function $f$ on $\R^{n\choose 2}/\R^n$ that is not contained in $C$. It can be represented by an integral linear function $F$ on $\R^{n\choose 2}$ that vanishes on $\R^n$. We can uniquely express $F$ as
\begin{equation}
F=\sum_{e\in \R^{n\choose 2}} a_e x_e \ ,
\end{equation}
where the coefficients $a_e\in \Z$. Order $\left[{n\choose 2}\right]$ lexicographically. Let
$
S_F=\left\{e \in \left[{n\choose 2}\right]\mid a_e\neq 0\right\} \ $ and
\begin{equation}
s = \max\{\min \left(e \in S_F) \mid F\not\in C\right \};
\end{equation}

in simple terms, any function supported on elements strictly greater than $s$ must belong to $C$. It is an elementary linear algebra exercise to verify that if a function's support is a subset of the four largest coordinate functions, then it is identically zero and hence it belongs to $C$. We may therefore assume that $s = \{m<k\}<\{n-3<n\}$. Choose $F\notin C$  such that $\min(S_F) = s$; since $F$ must vanish on $\R^n$, $k<n$. Choose one element $i \in \{n-3, n-2, n-1\}\smallsetminus \{m,k\}$.

%Assume $F\notin C$ is such that $\{m<k\}\coloneqq\min(S)$ is minimized.  Since $g$ vanishes on the $m$-th standard unit vector $e_m\in \R^I$, there exits $\{m<l\}\in S$ with $k\neq l$. Because $g$ vanishes on $e_m$, $e_k$, and $e_l$, there exists $i\in I\setminus\{m,k,l\}$ such that there exists an $e\in S$ with $i\in e$\renzo{why do we care about this? It seems to me we could pick $i$ basically randomly (so long as it is larger than $m$)? Isn't that right?}. Note that $m<i$ by the choice of $m$. \

Define $F'$ as
\begin{equation}
F'=F-a_{mk}\cdot (x_{mk}+x_{in}-x_{mn}-x_{ik}) \ .
\end{equation}
Then the coefficient of $x_{mk}$ is zero in $F'$ and for all coordinates $x_e$ that appear in $F'$ with nonzero coefficient, we have $e>\{m<k\}$ . As the integral linear function $f'$ on $\R^{n\choose 2}/\R^n$ induced by $F'$ is still not contained in $C$, this is a contradiction. 
\end{proof}

\begin{theorem}
\label{thm:Mg0n is what it's supposed to be}
The map $d\colon \Mgn{0,n}\to \R^{n\choose 2}/\R^n$ is a closed embedding, which extends to an isomorphism $\Mgnbar{0,n}\to \overline \Sigma_n$.
\end{theorem}

\begin{proof}
Since we defined the \TPL{}-space $\overline{\mathcal M}_{0,n}$ as an extended cone complex in the same way as in \cite{SpeyerSturmfels,GathmannKerberMarkwig}, it follows %from \emph{loc.\ cit.} 
that $d$ extends to an isomorphism $\Mgnbar{0,n}\to \overline\Sigma_n$ of \TPL{}-spaces. To show that $d$ is linear, it suffices to show that $d^*(x_{ij}+x_{kl}-x_{ik}-x_{jl})$ is integral affine on $\overline{\mathcal M}_{0,n}$ for all distinct $i,j,k,l\in [n]$ by Lemma \ref{lem:cross ratio vs distance coordinates}. For any moduli point $[\Gamma]\in \Mgn{0,n}$, the value of $x_{ij}+x_{kl}-x_{ik}-x_{jl}$ on $d(\Gamma)$ only depends on the subtree of $\Gamma$ spanned by the leaves marked by $i$, $j$, $k$ , and $l$. By checking this on the three combinatorial types for this subtree, one sees that
\begin{equation}\label{eq:crtocr}
d^*(x_{ij}+x_{kl}-x_{ik}-x_{jl})=\xi_{((f_i,f_l),(f_k,f_j))} \ ,
\end{equation}
where for $m\in \{i,j,k,l\}$ we denote by $f_m$ the leg marked by $m$ in the $n$-marked tree $T_0$ consisting of one vertex with $n$ legs attached. To show that $d$ is a closed embedding we need to show that at every moduli point $p\in \Mgn{0,n}$, the stalk $\Aff_{\Mgn{0,n},p}$ is generated by pull-backs of integral affine functions via $d$. It follows from \eqref{eq:crtocr} that this is true if $p=[T_0]$;  it then suffices to show that the stalk $\Aff_{\Mgn{0,n},p}$ is generated by the cross ratios on $T_0$. Let $G$ be the unique $n$-marked genus-zero graph with $p\in \relint(\sigma_G)$, and let $\xi$ be a primitive cross ratio on $G$. Then either $\xi=\xi_{(e,(f_1,f_2),(f_3,f_4))}$ or $\xi=\xi_{((f_1,f_2),(f_3,f_4))}$. %for appropriate choices of flags $f_i$ on $G$. 
For $1\leq i\leq 4$ choose a leg $g_i$ of $G$ in the connected component of $G\smallsetminus f_i$ that does not contain any other of the $f_j's$ (informally, make a simple path in the graph starting in the direction pointed by $f_i$ until you reach a leg). 
%Following the flag $f_i$ for $0\leq i\leq 4$ one arrives at a leg $g_i$ of $G$. 
The marked legs $g_i$ can also be viewed as legs on $T_0$ and one sees that $\xi=\xi_{((g_1,g_2),(g_3,g_4))}$, where the latter is defined by a cross ratio datum on $T_0$.

To show that $d$ induces an isomorphism $\Mgnbar{0,n}\to \overline\Sigma_n$ we need to compare the affine structures at the infinity points. Let $x$ be an infinite point of $\overline\Sigma_n$, and let $\sigma\in \Sigma_n$ such that $x\in \overline\sigma$, and let $\tau$ be the face of $\sigma$ corresponding to the infinite face of $\overline\sigma$ that contains $x$ in its relative interior. By definition, the integral affine functions at $x$ are precisely the functions in a neighborhood of $\relint(\sigma)$ that are constant on $y+\relint(\tau)$ for any $y\in \relint(\sigma)$. By Lemma \ref{lem:cross ratios at infinity}, the analogous statement holds on $\Mgnbar{0,n}$. As we have already shown that $d$ defines an isomorphism between $\Mgn{0,n}$ and $\Sigma_n$, this completes the proof.
\end{proof}

\section{Tautological morphisms} \label{sec:morphisms}

{
The universal family over $\MgnMf{g,n}$ is constructed using the forgetful morphism (of tropical spaces) described in Section \ref{subsec:family over atlas}. In this section we show the forgetful morphism is a morphism of stacks; to do so,  one needs to define the stabilization of a \emph{family} of stable curves after forgetting a section. 
}

\begin{definition}
\label{def:stabilization}
Let $\pi\colon \mc D\to B$ be a  family of $(n\sqcup\{\star\})$-marked genus-$g$ stable tropical curves. When one forgets the $\star$-section, the fibers of the family are not $n$-marked  stable curves because they have an unmarked smooth infinite point.  A \textbf{stabilization} of $\pi$ consists of a family $\epsilon \colon \mc C\to B$ of $n$-marked genus-$g$ stable tropical curves, together with a morphism $\phi\colon \mc D\to \mc C$ over $B$ that respect the genus functions and the $i$-marked section for all $i\in I$, and such that 
\begin{enumerate}
\item For every $b\in B$, the induced morphism $\mc D_b\to \mc C_b$ is the stabilization of the smooth tropical curve $\mc D_b$ after forgetting the $\star$-marking.
\item $\mc C$ has the quotient topology with respect to $\phi$. 
\item for every open subset $U\subseteq \mc C$, we have 
\begin{equation} 
\Gamma(U,\PL_{\mc C})= \{\varphi\colon U\to \Rbar \mid \varphi\circ \phi\in \Gamma(\phi^{-1} U,\PL_{\mc D})\} \ .
\end{equation}
\item The  pull-back morphism $\Aff_{\mc C}\to \phi_*\Aff_{\mc D}$ is an isomorphism.
\end{enumerate}
\end{definition}
See Figure \ref{fig:stabi} for an illustration of Definition \ref{def:stabilization}.
The following Lemma shows that if a stabilization exists, then it is unique up to isomorphism.

\begin{figure}
    \centering

    \vspace{1cm}
    
    \begin{tikzpicture}[thick]
        \begin{scope}
            \node at (1.5,2) {$\mc D$};
        
            \coordinate (lv) at (0,0);
            \coordinate (rvt) at (2.5,0);
            \coordinate (rvb) at (2.5,-2);
            \coordinate (lstar) at ($(lv)-(1,1.2)$);
            \coordinate (rstar) at ($(rvb)-(1,1.2)$);
            \coordinate (l1) at ($(lv)+(-1,1.2)$);
            \coordinate (r1) at ($(rvt)+(-1,1.2)$);
            \coordinate (l2) at ($(lv)+(1,1.4)$);
            \coordinate (r2) at ($(rvt)+(1,1.4)$);
            \coordinate (l3) at ($(lv)+(.2,-1.4)$);
            \coordinate (r3) at ($(rvb)+(.2,-1.4)$);
            
            %%%%% neighborhood
            %%%%%% star 
            \coordinate (x) at ($(lv)!.6!(rvb)$);
            \begin{scope}
                %\clip  (x) -- ++(-38.66:.4)-- ++(-1,-1.2)-- ++(141.34:.8)--++(1,1.2)--cycle;
                 \clip[rotate=141.34,xslant=-.8,yscale=.6,shift=(x)]   (.4,0) arc (0:180:.4) [yscale=1/0.6,xslant=.8,rotate=-141.34]-- ($(-1,-1.2)+(-38.66:.4)$)--++(141.34:.8)--cycle;
                \foreach \a in {-5,...,5}
                {
                     %\draw [gray,dashed] ($(lv)!.6+\a/27!(rvb)$)--($($(lv)!.6+\a/27!(rvb)$)!.6!($(lv)!.6+\a/27!(rvb)-(1,1.2)$)$);
                     %\draw [gray] ($($(lv)!.6+\a/27!(rvb)$)!.63!($(lv)!.6+\a/27!(rvb)-(1,1.2)$)$) -- ($(lv)!.6+\a/27!(rvb)-(1,1.2)$);
                     \draw [gray] ($(lv)!.6+\a/30!(rvb)$)--($(lv)!.6+\a/30!(rvb)-(1,1.2)$);
                }
            \end{scope}
            \draw[thin]  (x) -- ++(141.34:.4)-- ++(-1,-1.2)--++(-38.66:.8)--++($0.9*(1,1.2)$);
            \node [pin={-170:$\phi^{-1}U$}] at ($(x)-0.8*(1,1.2)$) {};

            %++(-1,-1.2)-- ++(141.34:.8)--++(1,1.2) ;
            
            %%%% middle 
            \begin{scope}            
                \clip (x) --+(-38.66:.4) arc (-38.66:141.34:.4)--cycle;    
                \foreach \a in {-5,...,5}
                    \draw [gray] ($(lv)!.6+\a/27!(rvb)+(0,1)$)--($(lv)!.6+\a/27!(rvb)+(0,-1)$);
            \end{scope}
            \draw[thin] (x) --+(-38.66:.4) arc (-38.66:141.34:.4)--cycle;
            
            %%%% 3
            \begin{scope}
                \clip [rotate=141.34,xslant=-.8,yscale=.6] (x) --+(.4,0) arc (0:180:.4)--cycle;
                \foreach \a in {-5,...,5}
                    \draw [gray] ($(lv)!.6+\a/27!(rvb)$)--($(lv)!.6+\a/27!(rvb)+(.2,-1.4)$);
            \end{scope}            
            \draw [thin,rotate=141.34,xslant=-.8,yscale=.6] (x) --+(.4,0) arc (0:180:.4)--cycle;

            \draw (lv)-- (rvt) -- (rvb)--cycle;
            %%% l1
            \draw (lv)--(l1);
            \draw[dashed]($(l1)!.7!(r1)$)--(r1)--(rvt);
            \draw (l1)--($(l1)!.7!(r1)$);

            \draw (lv)--(l2)--(r2)--(rvt);
            
            %%% lstar
            \draw (lv)--(lstar);
            \draw [ultra thick](lstar) --
                node [at start,label={[label distance=-.15cm]left:$\star$}]{}
            (rstar);
            \draw (rstar)--($(rstar)!.4!(rvb)$);
            \draw[dashed] ($(rstar)!.4!(rvb)$)--(rvb);

            %%% l3
            \draw (lv)--(l3)--(r3)--(rvb);

            % %%% middle
            % \draw ($(lv)!.6!(rvt)$) -- ($(lv)!.6!(rvb)$);
            % \draw [dashed] ($(lv)!.6!(rvt)$) -- ($($(lv)!.6!(rvt)$)!.8!($(l1)!.6!(r1)$)$);
            % \draw ($($(lv)!.6!(rvt)$)!.8!($(l1)!.6!(r1)$)$)--($(l1)!.6!(r1)$);

            % \draw ($(lv)!.6!(rvt)$) -- ($(l2)!.6!(r2)$);
            % \draw ($(lv)!.6!(rvb)$) -- ($(l3)!.6!(r3)$);
            
            % \draw [dashed]($(lv)!.6!(rvb)$) -- ($($(lv)!.6!(rvb)$)!.64!($(lstar)!.6!(rstar)$)$);
            % \draw ($($(lv)!.6!(rvb)$)!.64!($(lstar)!.6!(rstar)$)$) --($(lstar)!.6!(rstar)$);
   
        \end{scope}

        \draw [->] (4,-1) -- 
            node[above]{$\phi$}
        (5.5,-1);

        \begin{scope}[xshift=7cm]
            \node at (1.5,2) {$\mc C$};
            
            \coordinate (lv) at (0,0);
            \coordinate (rvt) at (2.5,0);
            \coordinate (rvb) at (2.5,-2);
            \coordinate (l1) at ($(lv)+(-1,1.2)$);
            \coordinate (r1) at ($(rvt)+(-1,1.2)$);
            \coordinate (l2) at ($(lv)+(1,1.4)$);
            \coordinate (r2) at ($(rvt)+(1,1.4)$);
            \coordinate (l3) at ($(lv)+(0,-3.1)$);
            \coordinate (r3) at ($(rvb)+(0,-1.6)$);

            %%%%% Neighborhood
            \begin{scope}
                %\node[fill=black, circle, inner sep=.05] (center) at ($(lv)!.5!(rvb)$){}; 
                \coordinate (center) at ($(lv)!.6!(rvb)$);
                \node at ($(center)-(0,.65)$) {$U$};
            
            \clip (center) circle (.4);    
            \foreach \a in {-5,...,5}
                \draw [gray] ($(lv)!.6+\a/27!(rvb)+(0,2)$)--($(lv)!.6+\a/27!(rvb)+(0,-2)$);
            \end{scope}
            \draw[thin] (center) circle (.4);

            \draw (lv)-- (rvt) -- (rvb)--cycle;
            %%% l1
            \draw (lv)--(l1);
            \draw[dashed]($(l1)!.7!(r1)$)--(r1)--(rvt);
            \draw (l1)--($(l1)!.7!(r1)$);
            
            %%% l2
            \draw (lv)--(l2)--(r2)--(rvt);
            
            %%% l3
            \draw (lv)--(l3)--(r3)--(rvb);

            % %%% middle
            % \draw ($(lv)!.6!(rvt)$) -- ($(lv)!.6!(rvb)$);
            % \draw [dashed] ($(lv)!.6!(rvt)$) -- ($($(lv)!.6!(rvt)$)!.8!($(l1)!.6!(r1)$)$);
            % \draw ($($(lv)!.6!(rvt)$)!.8!($(l1)!.6!(r1)$)$)--($(l1)!.6!(r1)$);

            % \draw ($(lv)!.6!(rvt)$) -- ($(l2)!.6!(r2)$);
            % \draw ($(lv)!.6!(rvb)$) -- ($(l3)!.6!(r3)$);
  
        \end{scope}
        
        \begin{scope}[xshift=3.5cm, yshift=-6cm]
        \draw [<-] (.2,.7)--
            node[auto] {$\pi$}
        +(-1.5,1.3);
        \draw [<-] (2.3,.7)--
            node[auto,swap]{$\epsilon$}
        +(1.5,1.3);
        
        \draw (0,0)--
            node[below]{$B$}
        (2.5,0);        
        \end{scope}
    \end{tikzpicture}
    \caption{The stabilization of a family of curves. The picture shows a neighborhood  $U$ of a point in the image under $\phi$ of the $\star$-marked leg: its inverse image contains infinite $\star$-marked legs, which implies that any function in  $\phi_*\Aff_{\mc D}(U)$ must have zero slope along the $\star$-leg. }
    \label{fig:stabi}
\end{figure}
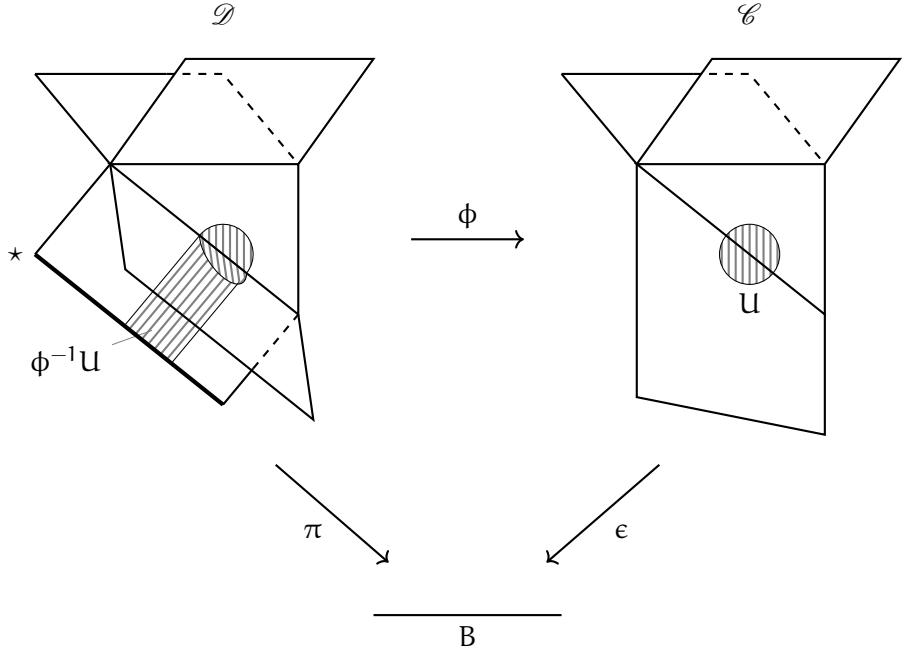

\begin{lemma}
\label{lem:uniqueness of stabilizations}
Let $\phi\colon\mc D\to \mc C$ be the stabilization of an $(n\sqcup\{\star\})$-marked family of stable tropical curves $\mc D\to B$ after forgetting the $\star$-section, let $\mc C'\to B$ be an $n$-marked family of stable tropical curves, and let $\phi'\colon\mc D\to \mc C'$ be a morphism over $B$ that exhibits $\mc C'_b$ as the stabilization of $\mc D_b$ for all $b\in B$. Then there exists a unique morphism $\varphi\colon \mc C\to \mc C'$ such that $\varphi\circ \phi=\phi'$, and this morphism $\varphi$ is an isomorphism of families of tropical curves. 
\end{lemma}

\begin{proof}
The uniqueness of $\varphi$ immediately follows from the fact that $\phi$ is surjective. To show the existence of $\varphi$, we note that for all $b\in B$, the morphisms 
\begin{align*}
\phi\vert_{\mc D_b}&\colon \mc D_b\to \mc C_b \quad\text{ and} \\
\phi'\vert_{\mc D_b}&\colon \mc D_b\to \mc C_b'
\end{align*}
are both stabilizations of $\mc D_b$. This implies the existence of a bijection $\varphi\colon \mc C\to \mc C'$ such that $\varphi\circ \phi=\phi'$. This map is continuous because $\mc C$ has the quotient topology with respect to $\phi$, and it is a morphism of tropical spaces because functions on $\mc C$ are piecewise integral linear or integral affine if and only if their pull-backs under $\phi$ are. By Lemma \ref{lem:bijection on fibers defines isomorphism}, it follows that $\varphi$ is an isomorphism of families of tropical curves.
\end{proof}

\begin{lemma}
\label{lem:stabilizations exists}
Let  $\pi\colon \mc D\to B$ be a family of $(n\sqcup\{\star\})$-marked stable tropical curves.
Then there exists a stabilization of $\pi$ when forgetting the $\star$-section.
\end{lemma}

\begin{proof}
Because stabilizations are unique if they exists by Lemma \ref{lem:uniqueness of stabilizations}, we can prove the existence locally around a point $b\in B$, that is we may assume that $B=\vert\Sigma\vert$ for some \TPL{}-trivialization $(\Sigma,(G_\sigma,f_\sigma,\chi_\sigma)_{\sigma\in \Sigma})$ at $b$. For each $\sigma\in\Sigma$, the stable graph $G_\sigma$ is $(n\sqcup\{\star\})$-marked. Forgetting the $\star$-marked leg and stabilizing one obtains an $n$-marked stable graph $G'_\sigma$ and $f_\sigma$ induces a map $f_\sigma'\colon \sigma\to \overline\sigma_{G'_\sigma}$. Gluing the families of \TPL{}-curves $\sigma\times_{\overline\sigma_{G'_\sigma}}\mc C_{G_{\sigma}'}$ yields an $n$-marked \TPL{}-curve $\mc C\to B$ and a morphism $\phi\colon \mc D\to \mc C$ of \TPL{}-spaces such that the induced morphisms $\mc D_b\to \mc C_b$ are stabilizations for all $b\in B$. Since $\mc C_{G'_\sigma}$ has the quotient topology with respect to the natural morphism $\mc C_{G_\sigma}\to \mc C_{G'_\sigma}$  and the piecewise linear functions on $\mc C_{G'_\sigma}$ are precisely those functions which pull-back to piecewise linear functions on $\mc C_{G_\sigma}$, %\andreas{This is not hard to see. We could either make it a Lemma, or keep it under its proverbial rug}
 the same is true for $\mc C$, that is $\mc C$ has the quotient topology with respect to $\phi$ and the piecewise linear functions on $\mc C$ are precisely those functions whose pull-back under $\phi$ is piecewise linear on $\mc D$. We define an affine structure $\Aff_{\mc C}$ on $\mc C$ by declaring its sections on an open subset $U\subseteq \mc C$ to be
\begin{equation}
\Gamma(U,\Aff_{\mc C})=\{m\in \Gamma(U,\PL_{\mc C})\mid \phi^*m\in \Gamma(\phi^{-1}U,\Aff_{\mc D})\}
\end{equation}
Let $m'\in \Gamma(\phi^{-1}U,\Aff_{\mc D})$. Since $\phi$ stabilizes $\mc D$ fiberwise, the function $m'$ is constant on the fibers of $\phi$. Therefore, $m'=\phi^*m$ for a unique function $m$ on $U$, and by definition of $\Aff_{\mc C}$ we have $m\in \Gamma(U,\Aff_{\mc C})$. It follows that the pull-back defines an isomorphism
\begin{equation}
\Aff_{\mc C}\to \phi_*\Aff_{\mc D} \ .
\end{equation}
It is left to show that the fibers of $\mc C$ are {stable} tropical curves and that the sequences of cotangent sheaves on the fibers of $\mc C$ are exact. Let $b\in B$. Since
\begin{equation}
0\to \Aff_{B,b}\to \Aff_{\mc D}\vert_{\mc D_b} \to \Omega^1_{\mc D_b} \to 0
\end{equation}
is an exact sequence of sheaves on $\mc D_b$, it follows that the lower row in the commutative diagram
\begin{equation}
\begin{tikzpicture}[auto]
\matrix[matrix of math nodes, row sep= 5ex, column sep= 4em, text height=1.5ex, text depth= .25ex]{
|(0lo)| 0 & [-1.5em]
|(Bo)| \Aff_{B,b} &
|(Cp)|  \Aff_{\mc C}\vert_{\mc C_b} &
|(Cpfib)| \Omega^1_{\mc C_b} &
|(0ro)| 0 \\
|(0lu)| 0 &
|(Bu)| \phi_{b*}\Aff_{B,b} &
|(C)|  \phi_{b*}\left(\Aff_{\mc D}\vert_{\mc D_b}\right) &
|(Cfib)| \phi_{b*}\Omega^1_{\mc D_b} &
|(0ru)|  0 \\
};
\begin{scope}[->,font=\footnotesize]
%%%%%% horizontal
%%% row 1
\draw (0lo) -- (Bo);
\draw (Bo)--(Cp);
\draw (Cp)--(Cpfib);
\draw (Cpfib)--(0ro);
%%% row 2
\draw (0lu) -- (Bu);
\draw (Bu)--(C);
\draw (C)--(Cfib);
\draw (Cfib)--(0ru);
%%%%%% vertical
\draw (Bo)--(Bu);
\draw (Cp)--(C);
\draw (Cpfib)--(Cfib);
\end{scope}
\end{tikzpicture} 
\end{equation}
is left-exact. Note that since $\phi$ is proper, one has 
\begin{equation}
\phi_{b*}\left(\Aff_{\mc D}\vert_{\mc D_b}\right)\cong \left(\phi_*\Aff_{\mc D}\right)\vert_{\mc C_b}
\end{equation}
so that the diagram makes sense. The fibers of $\phi_b=\phi\vert_{\mc D_b}$ are contractible and $\phi_b$ is proper, so we have $R^1\phi_{b*}\Aff_{B,b}=0$ and the lower row in the commutative diagram is exact. For the same reason, the vertical morphism to the left is an isomorphism. It follows from the fact that $\Aff_{\mc C}\to \phi_*\Aff_{\mc D}$ is an isomorphism, that the vertical morphism in the middle is an isomorphism as well. Together with the fact that the upper row of the diagram is exact to the right, this implies that the upper row is exact and the vertical morphism to the right is an isomorphism as well. It remains to show that $\Aff_{\mc C_b}$ is the sheaf of harmonic functions on $\mc C_b$. Since $R^1\phi_{b*}\R=0$ and $\phi_{b*}\R=\R$, the fact that $\Omega^1_{\mc C_b}\to \phi_{b*}\Omega^1_{\mc D_b}$ is an isomorphism implies that $\Aff_{\mc C_b}\to \phi_{b*}\Aff_{\mc D_b}$ is an isomorphism. But because $\phi_b$ is the stabilization of $\mc D_b$, one has 
\begin{equation}
\phi_{b*}\Aff_{\mc D_b}= \Harm_{\mc C_b} \ .
\end{equation}
\end{proof}

\begin{proposition}
\label{prop:forgetful morphism is a morphism}
Choosing for every family of $(n\sqcup\{\star\})$-marked genus-$g$ stable tropical curves $\mc C\to B$ a stabilization of $\mc C\to B$ after forgetting the $\star$-section defines a morphism $\Mgnbar{g,n\sqcup\{\star\}}\to \Mgnbar{g,n}$.
\end{proposition}

\begin{proof}
We need to show that if $\mc D'\to B'$ is a family of $(n\sqcup\{\star\})$-marked tropical curves that is the pull-back of a family $(n\sqcup\{\star\})$-marked tropical curves $\mc D\to B$ under a morphism $f\colon B'\to B$, then the pull-back under $f$ of stabilization $\mc C\to B$ of $\mc D\to B$ after forgetting the $\star$-section is the stabilization of $\mc D'\to B'$ after forgetting the $\star$-section. Let $\phi\colon \mc D\to \mc C$ be the morphism exhibiting $\mc C$ as the stabilization of $\mc D$. Then the composite 
\begin{equation}
\mc D'\to \mc D\xrightarrow{\phi} \mc C
\end{equation}
induces a morphism $g\colon\mc D'\to f^*\mc C$. Since $\mc D'$ and $\mc D$ have the same fibers and $\mc C$ is obtained by stabilizing $\mc D$, the morphism $g$ is a fiber-wise stabilization of $\mc D'$. By Lemma \ref{lem:uniqueness of stabilizations}, it follows that $f^*\mc C$ is the stabilization of $\mc D'$.
\end{proof}

\begin{definition}
\label{def:forgetful morphism}
We call the morphism from Proposition \ref{prop:forgetful morphism is a morphism} the \textbf{forgetful morphism} associated to the $\star$-marking, and  we also denote it by
\begin{equation}
\pi_\star\colon \Mgnbar{g,n\sqcup\{\star\}}\to \Mgnbar{g,n} \ .
\end{equation}
\end{definition}

The following result shows that the forgetful morphisms $\pi_\star\colon \Atlas{g,n\sqcup\{\star\}}\to \Atlas{g,n}$ used in \S\ref{subsec:family over atlas} are compatible with the forgetful morphisms $\pi_\star\colon \Mgnbar{g,n\sqcup\{\star\}}\to \Mgnbar{g,n}$.

\begin{proposition}
\label{prop:forgetful morphisms between atlases are forgetful morphisms} 
Let $g,n$  be two non-negative integers such that $n+2g-2>0$. Then the morphism
\begin{equation}
\xymatrix{
\Atlas{g,n\sqcup\{\star,\diamond\}}  \ar[rr]^{\hspace{-1cm}\pi_\diamond\times\pi_\star} 
%\ar[dr]_{\pi_\diamond}
& & \Atlas{g,n\sqcup\{\star\}}\times_{\Atlas{g,n}}\Atlas{g,n\sqcup\{\diamond\}}
%\ar[dl]^{\pi_\star^\ast(\pi_\star)}\\
% &\Atlas{g,n\sqcup\{\star\}}  &
}
\end{equation}
is the stabilization of the family of $(n\sqcup\{\star\})$-marked genus-$g$ stable tropical curves 
\begin{equation}
\Atlas{g,n\sqcup\{\star,\diamond\}}\xrightarrow{\pi_\diamond} \Atlas{g,n\sqcup\{\star\}}
\end{equation}
when forgetting the $\star$-section.
\end{proposition}

\begin{proof}
Consider the diagram
\begin{equation}
\xymatrix{
\Atlas{g,n\sqcup\{\star,\diamond\}}  \ar[r]^{\hspace{-1cm}\pi_\diamond\times\pi_\star} 
\ar[dr]^{\pi_\diamond}
&  \Atlas{g,n\sqcup\{\star\}}\times_{\Atlas{g,n}}\Atlas{g,n\sqcup\{\diamond\}}\ar[d]^{\pi_\diamond} 
\ar[r] & \Atlas{g,n\sqcup\{\diamond\}}\ar[d]^{\pi_\diamond}
\\
 &\Atlas{g,n\sqcup\{\star\}} \ar[r]^{\pi_\star} & \Atlas{g,n}
}
\end{equation}
According to the definition of $\pi_\diamond$ in Section \ref{subsec:family over atlas}, the morphism $\pi_\diamond\times\pi_\star$ stabilizes the family $\pi$ fiber-wise. By \ref{lem:uniqueness of stabilizations} and \ref{lem:stabilizations exists}, it follows that it is the stabilization.
\end{proof}

\begin{lemma}
\label{lem:local structure at genus 0 point}
Let $\pi\colon \mc C\to B$ be a family of $n$-marked genus-$g$ tropical stable curves and let $x\in \mc C_b$ be a vertex with $\gamma_{\mc C_b}(x)=0$ and $\val(x) \geq 3$. Then, denoting $J = T_x(\mc C_b)$, there exist open neighborhoods $U$ and $V$ of $b=\pi(x)$ and $x$, respectively, and linear morphisms $f\colon U\to \Mgnbar{0, J}$ and $f_\star\colon V\to \Mgnbar{0, J\sqcup\{\star\}}$ such that the diagram
\begin{center}
\begin{tikzpicture}[auto]
\matrix[matrix of math nodes, row sep= 5ex, column sep= 4em, text height=1.5ex, text depth= .25ex]{
|(C)| V 	&
|(n1)|	\Mgnbar{0,J\sqcup\{\star\}}	\\
|(B)| U	&	
|(n)| \Mgnbar{0,J}	\\
};
\begin{scope}[->,font=\footnotesize]
%%% horizontal
\draw (C) --node{$f_\star$} (n1);
\draw (B)--node{$f$} (n);
%%% vertical
\draw (C) --node{$\pi$} (B);
\draw (n1) --node{$\pi_\star$} (n);
\end{scope}
\end{tikzpicture}
\end{center}
is commutative and the induced morphism $V\to U\times_{\Mgnbar{0,J}}\Mgnbar{0, J\sqcup\{\star\}}$ is an open immersion. 
\end{lemma}

\begin{proof}
 For each $v\in T_x(\mc C_b)$, let $x_v$ be a point in the interior of the edge in $\mc C_b$ that is adjacent to $x$ in the direction of $v$. By Lemma \ref{lem:local structure at interior of edge}, there exists an open neighborhood $U$ of $B$, neighborhoods $W_v$ of $x_v$ for all $v\in T_x(\mc C_b)$, an $\epsilon>0$, and isomorphisms $W_v\cong (0,\epsilon)\times U$ over $B$. Let $V'$ the connected component of 
\begin{equation}
\mc C_U\setminus\bigcup_{v\in T_x(\mc C_b)} W_v
\end{equation}
containing $x$, and let
\begin{equation}
V= V'\cup \bigcup_{v\in T_x(\mc C_b)} W_v \ .
\end{equation}

Gluing for each $v\in T_x(\mc C_b)$ a copy of $(0, \infty] \times U$ to $V$ via the identification
\begin{equation}
W_v\cong(0,\epsilon) \times U,
\end{equation}
we obtain a family $\widetilde{\mc C}\to U$ of $J$-marked genus-zero stable tropical curves over $U$. By Corollary \ref{cor:representative for M0n}, there exist a Cartesian diagram
\begin{equation}
\begin{tikzpicture}[auto]
\matrix[matrix of math nodes, row sep= 5ex, column sep= 4em, text height=1.5ex, text depth= .25ex]{
|(C)| \widetilde{\mc C} 	&
|(n1)|	\Mgnbar{0,J\sqcup\{\star\}}	\\
|(B)| U	&	
|(n)| \Mgnbar{0,J} \ .	\\
};
\begin{scope}[->,font=\footnotesize]
%%% horizontal
\draw (C) -- (n1);
\draw (B)--node{$f$} (n);
%%% vertical
\draw (C) -- (B);
\draw (n1) --node{$\pi_\star$} (n);
\end{scope}
\end{tikzpicture}
\end{equation}
Noting that by construction, $V$ is an open subset of $\widetilde {\mc C}$ finishes the proof.
\end{proof}

\begin{remark}
Lemma \ref{lem:local structure at genus 0 point} may be generalized to apply to points $x$ in the interior of an edge of a fiber (where $\val(x) = 2$); one must attach  products of $U$  with a tripod, rather than just copies of $(0, \infty]$, to the sets $W_v$.
\end{remark}

\begin{lemma}
\label{lem:remembering a section}
Let $\pi: \mc C\to B$ be a family of $n$-marked genus-$g$ tropical stable curves and let $t\colon B\to \mc C$ be a linear section such that $\gamma_{\mc C}(t(b))=0$ for all $b\in B$. Then there exists a family of $(n\sqcup\{\star\})$-marked genus-$g$ tropical stable curves $\mc D\to B$ and a morphism $\phi\colon\mc D \to \mc C$ over $B$ that exhibits $\mc C\to B$ as the stabilization of $\mc D\to B$ after forgetting the $\star$-section $s_\star\colon B\to \mc D$ and such that $\phi\circ s_\star=t$. Moreover,  $\phi\colon \mc D\to \mc C$ is unique up to unique isomorphism, that is if $\phi'\colon\mc D'\to \mc C$ is a second such family, then there exists a unique isomorphism $\varphi\colon \mc D\to \mc D'$ with $\phi'\circ\varphi=\phi$.
\end{lemma}

\begin{proof}
Since any morphism $\phi\colon \mc D\to \mc C$ in question is an isomorphism over $\mc C\setminus t(B)$, the construction of $\phi$ is local around $t(B)$. Using Lemma \ref{lem:local structure at genus 0 point}, we may thus assume that $\mc C\to B$ is a family tropical curves of genus zero. 
For the remainder of the proof, refer to the following diagram:
\begin{equation}
    \xymatrix{
    \mc D \ar[rr] \ar[dr] \ar[ddr] & & \Mgnbar{0,n\sqcup\{\star,\diamond\}}\ar[d]^{\pi_\diamond}  \ar@/^1.5pc/[dr]^{\pi_\diamond\times \pi_\star} &\\
    & \mc C  \ar@/^3pc/[rr]\hole
    \ar[r]^{f_\star} \ar[d]_{\pi} & \Mgnbar{0,n\sqcup\{\star\}} \ar[d]^{\pi_\star} & \Mgnbar{0,n\sqcup\{\star\}}\times_{\Mgnbar{0,n}}\Mgnbar{0,n\sqcup\{\diamond\}}
    \ar[d]\ar[l]\\
     & B \ar[r]^f \ar[ur]  \ar@/_/[u]_t& \Mgnbar{0,n} & \Mgnbar{0,n\sqcup\{\diamond\}} \ar[l]^{\pi_\diamond}
    }
\end{equation}
By Corollary \ref{cor:representative for M0n}, there exist morphisms $f\colon B\to \Mgnbar{0,n}$ and $f_\star\colon \mc C\to \Mgnbar{0,n\sqcup\{\star\}}$ identifying $\mc C$ with $f^*(\Mgnbar{0,n\sqcup\{\star\}})$. Since $t$ is a section of $\pi$, we obtain an induced identification of $\mc C$ with $(f_\star\circ t)^*(\Mgnbar{0,n\sqcup\{\star\}}\times_{\Mgnbar{0,n}}\Mgnbar{0,n\sqcup\{\diamond\}})$. It follows from Proposition \ref{prop:forgetful morphisms between atlases are forgetful morphisms} that the morphism
\begin{equation}
\mc D\coloneqq (f_\star\circ t)^*(\Mgnbar{0,n\sqcup\{\star,\diamond\}})\xrightarrow{(f_\star\circ t)^*(\pi_\diamond\times\pi_\star)} (f_\star\circ t)^*(\Mgnbar{0,n\sqcup\{\star\}}\times_{\Mgnbar{0,n}}\Mgnbar{0,n\sqcup\{\diamond\}})\cong \mc C 
\end{equation} 
is a stabilization such that the $\star$-section of $\mc D$ is mapped to $t$ in $\mc C$.  This shows existence. Uniqueness follows from the fact that a family of genus-zero stable tropical curves is uniquely determined by its fibers.
\end{proof}

We may now use the construction from Lemma \ref{lem:remembering a section} to define  sections $s_i$ as morphisms of moduli spaces. 

\begin{proposition}
\label{prop:section is a morphism}{
The morphism  $s_i: \Mgnbar{g,n}\to \Mgnbar{g,n\sqcup\{\star\}}$ is defined as follows.
To a family of $n$-marked genus-$g$ stable tropical curves $(\mc C\to B,(s_i)_{i\in I}, \gamma_{\mc C})$, assign a family of $(n\sqcup\{\star\})$-marked genus-$g$ curves $\mc D\to B$
such that:
\begin{itemize}
    \item there is a morphism $\phi\colon \mc D\to \mc C$ over $B$ exhibiting $\mc C$ as the stabilization of $\mc D$ after forgetting the $\star$-section; %$s_\star\colon B\to \mc D$ ;
    \item $\phi\circ s_\star= s_i$ after forgetting the $\star$-section.
\end{itemize}}

\end{proposition}

\begin{proof}
By Lemma \ref{lem:remembering a section}, the morphism $\phi:\mc D\to \mc C$ exists and is unique up to unique isomorphism. To show that this defines a morphism of stacks one needs to show that the choice of $\mc D$ is compatible with fiber products. For this, one proceeds similarly as in the proof of Proposition \ref{prop:forgetful morphism is a morphism}.
\end{proof}

\begin{definition}
\label{def:section}
We call the morphism $s_i\colon \Mgnbar{g,n}\to \Mgnbar{g,n\sqcup\{\star\}}$ from Proposition \ref{prop:section is a morphism} the \textbf{$i$-th section}.%, and  we  denote it by
%\begin{equation}
%s_i\colon \Mgnbar{g,n}\to \Mgnbar{g,n\sqcup\{\star\}} \ .
%\end{equation}
\end{definition}

\begin{theorem}
\label{thm:universal family}
The forgetful morphism $\pi_\star \colon \MgnMf{g,n\sqcup\{\star\}}\to \MgnMf{g,n}$ represents the universal family. More precisely, if $B\to \Mgnbar{g,n}$ is defined via a family of $n$-marked genus-$g$ tropical Mumford curves $\pi\colon\mathcal C\to B$, then there is a morphism $\mc C\to \Mgnbar{g,n\sqcup\{\star\}}$ such that the resulting square
\begin{center}
\begin{tikzpicture}[auto]
\matrix[matrix of math nodes, row sep= 5ex, column sep= 4em, text height=1.5ex, text depth= .25ex]{
|(C)| \mathcal C 	&
|(n1)|	\MgnMf{g,n\sqcup\{\star\}}	\\
|(B)| B	&	
|(n)| \MgnMf{g,n}	\\
};
\begin{scope}[->,font=\footnotesize]
%%% horizontal
\draw (C) -- (n1);
\draw (B)-- (n);
%%% vertical
\draw (C) --node{$\pi$} (B);
\draw (n1) --node{$\pi_\star$} (n);
\end{scope}
\end{tikzpicture}
\end{center}
is Cartesian. Moreover, the pull-back of the $i$-th section from $\MgnMf{g,n}$ to $B$ coincides with the $i$-th section on $B$.
\end{theorem}

\begin{proof}
For any tropical space $T\stackrel{f}{\to} B$ over $B$, the groupoid of morphisms from $T$ to $B\times_{\MgnMf{g,n}}\MgnMf{g,n\sqcup\{\star\}}$ over $B$ has as objects the pairs $(\mc D,\phi)$ consisting of a family $(n\sqcup\{\star\})$-marked genus-$g$ tropical Mumford curves and a morphism $\phi:\mc D\to \mc C_T $ of $n$-marked families of tropical curves that exhibits $\mc C_T = f^\ast \mc C$ as the stabilization of $\mc D$ after forgetting the $\star$-section. By Lemma \ref{lem:remembering a section}, between any two objects in this groupoid there is at most one isomorphism. Therefore, we can take isomorphism classes of theses pairs $(\mc D,\phi)$ and represent $B\times_{\MgnMf{g,n}}\MgnMf{g,n\sqcup\{\star\}}$ by a functor. Again by Lemma \ref{lem:remembering a section}, the equivalence class of one of the pairs $(\mc D,\phi)$ only depends on the composite $\phi\circ s_\star$ of the $\star$-section and $\phi$. Lemma \ref{lem:remembering a section} also states that every linear section $T\to \mc C_T$ appears as $\phi\circ s_\star$ for some pair $(\mc D,\phi)$. Therefore, the fiber product $B\times_{\MgnMf{g,n}}\MgnMf{g,n\sqcup\{\star\}}$ is the functor assigning to $T/B$ the sections of $\mc C_T$. This is naturally equivalent to the functor represented by the family $\pi\colon \mc C\to B$. The statement about the sections follows immediately.
\end{proof}

\section{Tropical intersection theory on $\Mgnbar{g,n}$}
\label{sec:intersection theory}

The main goal of this section is to define  $\psi$-classes on $\Mgnbar{g,n}$. %and to prove the pull-back formula for $\psi$-classes. 
To achieve this goal, we need to develop {some elements of}  intersection theory for tropical spaces and stacks over the category of tropical spaces. For a computational approach to tropical intersection theory, in particular of moduli spaces for curves of genus zero, see the \textsc{polymake} \cite{polymake} extension {\it a-tint} \cite{Atint}.

\subsection{Cycles and divisors on tropical spaces and stacks}
Much of the tropical intersection theory developed for tropical varieties in $\R^n$ \cite{MikhalkinICM,AllermannRau,GathmannKerberMarkwig,RauModuli,Flossmann,FS,Katz}, and more generally for weakly embedded cone complexes \cite{GrossToroidal}, tropical manifolds \cite{Shaw,FrancoisRau}, and rational polyhedral spaces \cite{Lefschetz,GrossShokriehSheaf}, carries over to tropical spaces. We explain the adaptions that have to be made.

\subsubsection{Tropical cycles}

Let $X$ be a tropical space. We say that a function $f\colon X\to \Z$ is \emph{constructible}, if at every $x\in X$ there exists a local face structure $\Sigma$ such that $f\vert_\sigma$ is constant for every $\sigma\in \Sigma$.  For a constructible function $f\colon X\to \Z$, we denote by
\begin{equation}
\vert f\vert=\overline{\{ x\in X\mid f(x)\neq 0\}}
\end{equation}
its support. Let $\mathrm{W}_k(X)$ denote the quotient of the Abelian group of constructible functions $X\to \Z$ whose support is at most $k$-dimensional by the subgroup of all constructible functions $X\to \Z$ whose support is at most $(k-1)$-dimensional. Given an element $\alpha \in \mathrm W_k(X)$ represented by a function $f\colon X\to \Z$, the closure of the set of all points $x\in \vert f\vert$ where the local dimension of $\vert f \vert$ is $k$ only depends on $\alpha$. We call it the \emph{support} of $\alpha$ and denote it by $\vert \alpha \vert$.

Let $\phi\colon X\to Y$ be a morphism of \TPL{}-spaces, where $X$ is purely $n$-dimensional and $\dim(Y)=m$. Let $x\in X$. Then there exist local face structures $\Sigma$ and $\Delta$ of $X$  and $Y$ at $x$ and $\phi(x)$, respectively, such that $\phi(\sigma)\in\Delta$ and the restriction $\phi\vert_\sigma\colon \sigma\to \phi(\sigma)$ is linear for all $\sigma\in \Sigma$. For every point $z$ in the relative interior of some {$n$}-dimensional $\sigma\in\Sigma$ we define the \emph{multiplicity} $\mathrm{mult}_z(\phi)$ of $\phi$ at $z$ as follows: if $\dim(\phi(\sigma))<{n}$, then $\mathrm{mult}_z(\phi)=0$, otherwise $\phi\vert_\sigma$ can be expressed in suitable coordinates as $\R^{n}\to \R^{m},\;\; x\mapsto Ax+b$ for appropriate choices of a matrix $A\in \Z^{m\times n}$ and a vector $b\in \R^{m}$, and we set $\mathrm{mult}_z(\phi)$ equal to the index of $A(\Z^{n})$ in $\Z^{m}$. Note that this can be expressed as the $\gcd$ of all basic minors of $A$. The domain of 
$z\mapsto\mathrm{mult}_z(\phi)$  depends on the choices of $\Sigma$ and $\Delta$, but its values do not. Therefore, we obtain a well-defined element $\mathrm{mult}(\phi)\in W_{n}(X)$.

Still assuming that $X$ is purely $n$-dimensional, we say that $[\alpha]\in \mathrm{W}_n(X)$ is a \emph{tropical $n$-cycle} if {it satisfies the following \textit{balancing condition}:} for every open set $U\subseteq X$, affine function $\phi\in \Gamma(U,\Aff_X)$, local face structure $\Sigma$ on $U$ such that {$\alpha$ and $\mathrm{mult}(\phi)$ %\andreas{wrong sign!\rcolor{Make sure to fold in the word "balancing" in this definition.}}
have} constant value  on the relative interiors of any $\sigma\in\Sigma$, and for every $(n-1)$-dimensional polyhedron $\tau\in \Sigma$ with $\phi\vert_\tau=0$ we have
\begin{equation}\label{def:balancing}
\sum_{\sigma:\tau\preceq\sigma} {{\alpha\cdot \mathrm{sign}(\phi\vert_{\sigma}) \cdot \mathrm{mult}(\phi)}}=0 \ ,
\end{equation}
where the sum is taken over all $n$-dimensional $\sigma\in \Sigma$ containing $\tau$. {Since $\phi\vert_\tau=0$, the sign $\mathrm{sign}(\phi\vert_\sigma)$ of $\phi$ on $\sigma$ is well-defined for all $\sigma\in \Sigma$ containing $\tau$.} We denote the subgroup of $\mathrm{W}_n(X)$ consisting of all tropical $n$-cycles by $Z_n(X)$.

Now assume that $X$ is an arbitrary tropical space, and let $k$ be a nonnegative integer. The group $Z_k(X)$ of \emph{tropical $k$-cycles} is the subgroup consisting of all elements of $\mathrm{W}_k(X)$ represented by a function $\alpha\colon X\to \Z$ such that there exists a purely $k$-dimensional locally polyhedral subset $Y\subseteq X$ such that $\alpha\vert_{X\setminus Y}=0$ and $\alpha\vert_{Y}$ represents an element in $Z_k(Y)$.

\begin{example}
If $X$ is a smooth connected tropical curve, then the tropical one-cycles on $X$ are  represented by  constant functions $X\to \Z$.
\end{example}

\begin{example}
If $X$ can be embedded in $\R^d$ for some $d$, then our definition of balancing is equivalent to the definition in \cite{AllermannRau}. In particular, the constant function with value one defines a tropical cycle $[\Mgnbar{0,n}]$ on $\Mgnbar{0,n}$  by Theorem \ref{thm:Mg0n is what it's supposed to be} and \cite[Theorem 3.7]{GathmannKerberMarkwig}. The proof in \cite{GathmannKerberMarkwig} generalizes easily to higher genus in that the constant function with value one defines a tropical cycle $[\Atlas{g,n}]$ on $\Atlas{g,n}$. 
\end{example}

Let $f\colon X\to Y$ be a proper morphism of tropical spaces, and let $\alpha\in W_k(X)$. Let $f_\alpha\colon \vert \alpha\vert\to f(\vert\alpha\vert)$ denote the morphism induced by $f$. Then we define the \emph{push-forward} $f_*\alpha$ by $f_*\alpha=0$ if $\dim(f(\vert\alpha\vert))<k$, and by
\begin{equation}
f_*\alpha\colon f(\vert\alpha\vert) \to \Z,\;\; y\mapsto \sum_{x\in f_\alpha^{-1}\{y\}} \mathrm{mult}_x(f_\alpha)\cdot \alpha(x)
\end{equation}
otherwise. Note that $f_*\alpha$ is a well-defined element of $W_k(Y)$ when extended by zero to all of $Y$. If $\alpha\in Z_k(X)$, then a straightforward adaption of the proof of \cite[Proposition 2.25]{GathmannKerberMarkwig} to the setting of tropical spaces shows that $f_*\alpha\in Z_k(Y)$.

\subsubsection{Tropical Cartier divisors}

\begin{definition}
Let $X$ be a tropical space. A piecewise linear function $\phi\in \Gamma(X,\PL_X)$ is a \textbf{rational function} if at every $x\in X$ there exists a local face structure $\Sigma$ such that for every  $\sigma\in \Sigma$ there are an open neighborhood $U$ of $\sigma$  and $m\in \Gamma(U,\Aff_X)$ with $\phi\vert_\sigma=m\vert_\sigma$. We denote by $\APL_X$ the subsheaf of $\PL_X$ whose sections over an open subset $U\subseteq X$ are rational functions on $U$.
\end{definition}

By definition, every  integral affine function is rational.

\begin{definition}
Let $X$ be a tropical space. The \textbf{sheaf of tropical Cartier divisors} on $X$, denoted by $\Divs_X$, is defined by
\begin{equation}
\Divs_X\coloneqq ~\raisebox{.7ex}{$\APL_X$}\Big/\raisebox{-.7ex}{$\Aff_X$} \ .
\end{equation}
The group $\Div(X)$ of tropical Cartier divisors on $X$ is defined by 
\begin{equation}
\Div(X)\coloneqq \Gamma(X, \Divs_X) \ .
\end{equation}
For any $\phi\in \Gamma(X,\APL_X)$ we denote its image in $\Div(X)$ by $\mathrm{div}(\phi)$.
\end{definition}

Let $X$ be a purely $n$-dimensional tropical space, let $\alpha\in Z_n(X)$ be a tropical $n$-cycle on $X$, and let $\phi\in \Gamma(X,\APL_X)$ be a rational function on $X$. By definition,  there exists at any $x\in X$ a local face structure $\Sigma$ such that for every $\sigma\in \Sigma$ there exists a neighborhood $U_\sigma$ of $\sigma$ and an integral affine function $m_\sigma \in \Gamma(U,\Aff_X)$ such that $\phi\vert_\sigma =m_\sigma\vert_\sigma$, and such that $\alpha$ is constant on the relative interior of every $\sigma\in\Sigma$. Let $\tau\in\Sigma$ be an $(n-1)$-dimensional polyhedron. 
%After modifying $\Sigma$ accordingly, we may assume that $\tau$ is inclusion-minimal in $\Sigma$ and such that there exists an integral affine function $m\in \Gamma(U,\Aff_X)$ such that $\phi\vert_\tau=m\vert_\tau$. 
  Since $\alpha$ is a tropical cycle,  the sum
\begin{equation}
d_\tau=\sum_{\sigma :\tau\preceq\sigma} \mathrm{sign}(m_\sigma-m_\tau)\cdot \mathrm{mult}(m_\sigma-m_\tau)\cdot \alpha \ 
\end{equation}
 is independent of the choice of $m_\tau$. Define a function $\vert\Sigma\vert\to\Z$ whose value is $d_\tau$ on the relative interior of an $(n-1)$-dimensional polyhedron $\tau\in \Sigma$, and zero everywhere else. Since this construction is local, it defines an element in $\mathrm{W}_{n-1}(X)$ which we denote by $\mathrm{div}(\phi)\cdot \alpha$. If $\phi\in \Gamma(X,\Aff_X)$, then by construction, one has $\mathrm{div}(\phi)\cdot \alpha=0$. So because the construction of $\mathrm{div}(\phi)\cdot \alpha$ is local and bilinear, it induces a bilinear pairing 
\begin{equation}
\Div(X)\times Z_n(X)\to \mathrm{W}_{n-1}(X) \ .
\end{equation}
We denote the image of a pair $(D,\alpha)$ under this pairing by $D\cdot \alpha$. Exactly as in \cite[Proposition 3.7 (a)]{AllermannRau}, one proves that $D\cdot \alpha\in Z_{n-1}(X)$. 

The above construction generalizes to any tropical space $X$, $D\in \Div(X)$ and $\alpha\in Z_k(X)$ for any non-negative integer $k$, by setting
\begin{equation}
D\cdot \alpha\coloneqq D\vert_{\vert\alpha\vert}\cdot \alpha ~.
\end{equation}
For any $k\geq 0$ one thus obtains a bilinear pairing
\begin{equation}
\Div(X)\times Z_k(X)\to Z_{k-1}(X) ,\;\; (D,\alpha)\mapsto D\cdot \alpha ~.
\end{equation}

\begin{remark}
\label{rem:divisors at infinity}
The constant function with value one on $X=\T$ defines a cycle $[X]\in Z_1(X)$. One would expect that for every zero-cycle $\alpha\in Z_0(X)$ there exists a divisor $D\in \Div(X)$ with $D\cdot[X]=\alpha$. However, with our definitions this is true only if $\vert\alpha\vert$ does not contain the point $\infty$. This can be fixed by modifying the sheaf $\APL_X$ in a way that allows functions with value $\pm\infty$ so that there would be a rational function on $X$ approaching $\infty$ with nonzero slope. As doing so would introduce several technical difficulties without increasing the generality of our statements, we decided to work with the simpler version of $\APL_X$ (and hence $\Div(X)$) introduced above.
\end{remark}

\begin{definition}
Given  a tropical space $X$, we define the \textbf{weak Chow group $A_*(X)$ of $X$} as the quotient of $Z_*(X)$ by its subgroup generated by all elements of the form $\mathrm{div}(\phi)\cdot \alpha$, where $\alpha\in Z_k(X)$ for some non-negative integer $k$, and $\phi\in \Gamma(\vert\alpha\vert,\APL_{\vert \alpha\vert})$. 
We say that two tropical cycles $\alpha$ and $\beta$ on $X$ are \textbf{weakly rational equivalent} if they have the same image in $A_*(X)$.
\end{definition}

%\andreas{remark about why 'weak' here?}

Exactly as in \cite[Proposition 3.7 (b)]{AllermannRau}, one proves that for any two divisors $D,E\in \Div(X)$ and $\alpha\in Z_*(X)$ one has
\begin{equation}
D\cdot (E\cdot \alpha)=E\cdot(D\cdot\alpha)  \ .
\end{equation}
It follows from this that every divisor $D\in \Div(X)$ defines a morphism
\begin{equation}
A_*(X)\to A_*(X),\;\;[\alpha]\mapsto [D\cdot\alpha] \ .
\end{equation}

If $X$ is a compact tropical space, then every tropical zero-cycle $\alpha$ on $X$ can only have finite support. Therefore, one can integrate zero-cycles and obtain an integer
\begin{equation}
\int_X\alpha\coloneqq \sum_{x\in X}\alpha(x) \ ,
\end{equation}
the \emph{degree} of $\alpha$. Since the sum over all points of the sum of all outgoing slopes at a point of a rational function on a compact tropical curve is zero, the degree induces a map
\begin{equation}
\int_X\colon A_0(X)\to \Z
\end{equation}
on the Chow group of a compact tropical curve $X$.

\subsubsection{Tropical line bundles}

If $X$ is a tropical space, then  a tropical line bundle may be described in the following equivalent ways:
\begin{enumerate}
    \item a tropical space $\pi\colon \mc L\to X$ over $X$ such that $X$ can be covered by open subsets $U$ such that $\mc L_U=\pi^{-1}U\cong U \times \T$ over $U$.
    \item an $\Aff_X$-torsor, that is a sheaf $\mc A$ of sets with an $\Aff_X$-action such that $\mc A_x$ is an $\Aff_{X,x}$-torsor for all $x\in X$.
    \item  a cocycle $\{\xi_{\alpha, \alpha'}\}$ determining an element  of $H^1(X,\Aff_X)$.
\end{enumerate}
If $\mc L$ is a tropical space as in (1), the linear sections $U\to \mc L_U$ whose values are never infinite form an $\Aff_X$-torsor, showing that (1) implies (2).
Given an  $\Aff_X$-torsor $\mc A$, differences of local sections for the torsor give the elements of a cocycle, yielding  (2) implies (3).
Finally, a cocycle provides patching data to construct a tropical space $\mc L$, showing that (3) implies (1).

%The equivalence of $(1)$ and $(2)$ follows from the fact that the only isomorphisms of $\T$ are the translations. 

% and $\mc A$ is the corresponding $\Aff_X$-torsor, then one can naturally identify the sections in $\Gamma(U,\mc A)$ over some open subset $U$ of $X$ with the linear sections $U\to \mc L_U$ whose values are never infinite. 

\begin{example}
Given a family of tropical curves $\pi: \mc C \to B$ together with an affine linear section $s:B\to \mc C$ supported on genus zero-points, the sheaf $Aff_{\mc C}(ks)$ from Definition \ref{def:funwipo} defines a tropical line bundle on $\mc C$.
\end{example}

\begin{example}[Line bundles on $\T\PP^1$]
\label{ex:troplinebundles}
For $a\in \ZZ$, we describe a line bundle that we denote $\calO_{\T\PP^1}(a)$ in the three different ways introduced in this section:
\begin{enumerate}
    \item Let $\mc L = \mc L_{-\infty} \cup \mc L_{\infty}$, where $\mc L_{-\infty}= [-\infty, \infty)\times \TT$, with affine integral coordinates $(x_-, y_-)$ and $L_{\infty}= (-\infty, \infty]\times \TT$, with affine integral coordinates $(x_+, y_+)$. The two charts are glued together via the transition functions:
    \begin{align}
        x_+ &= x_- & y_+ &= y_-+ax_-.
    \end{align}
    Projection to the first factor provides a map to $\T\PP^1$, and by construction $\T\PP^1$ is covered by two open sets $U$ with $\mc L_U\cong U\times \TT$. We wish to consider the constant section $s_0:[-\infty, \infty)\to \mc L_{-\infty}$ given by $y_- = 0$, and observe that in the $+$-coordinates the local equation for $s_0$ is $y_+ = a x_+$.
      \item Consider the sheaf $Aff_{\T\PP^1}(a \infty)$: its sections in a (small) neighborhood of $-\infty$ are constant, whereas its sections in a  neighborhood of $+\infty$ have slope $-a$. Given a section of $Aff_{\T\PP^1}(a \infty)$, one may obtain a section of $\mc L$ by adding to it the distinguished section $s_0$.
      \item Given two open sets $[-\infty, \infty)$ and $(-\infty, \infty]$ covering $\T\PP^1$, we get the cocycle
    \begin{equation}
        \xi_{\infty, -\infty} =  ax.
    \end{equation}
  
\end{enumerate}
\end{example}

\begin{definition}
Let $X$ be a tropical space. The \textbf{Picard group $\Pic(X)$ of $X$} is given by 
\begin{equation}
\Pic(X)=H^1(X,\Aff_X) \ .
\end{equation}
\end{definition}

%Given the cocycle in $H^1(X,\Aff_X)$ defined by a tropical line bundle $\mc L\to X$, one can glue trivial $\Rbar$-fiber bundles according to the cocycle and obtain an $\Rbar$-fiber bundle $\mc L_\infty\to X$. \renzo{Why isn't this just $\mc L$?}
A rational section $s$ of a tropical line bundle $\mc L\to X$ is a continuous section $s\colon X\to \mc L$ such that whenever $U\subseteq X$ is an open subset and $\mc L_U\xrightarrow{\cong} \T\times U$ is an isomorphism, the composite
\begin{equation}
U\xrightarrow{s\vert_U} \mc L_U\to  \T\times  U\to \T
\end{equation}
is contained in $\Gamma(U, \APL_X)$.

Given a tropical Cartier divisor $D\in \Div(X)$, one defines a tropical line bundle $\mc L(D)$ by taking the image of $D$ under the connecting morphism
%\begin{equation}
$\Div(X)\to \Pic(X)$
%\end{equation}
associated to the short exact sequence
\begin{equation}
0\to \Aff_X\to \APL_X\to \Divs_X \to 0 \ .
\end{equation}
The bundle $\mc L(D)$ comes with a canonical rational section, and conversely every rational section of a tropical line bundle $\mc L$ defines a Cartier divisor. Just as in algebraic geometry, one obtains a bijection between $\Div(X)$ and isomorphism classes of Pairs $(\mc L, s)$ consisting of a tropical line bundle $\mc L$ on $X$ and a rational section $s$ of $\mc L$.
Differently from algebraic geometry, tropical line bundles need not have any rational section, as shown in Example \ref{ex:ratsec}.
%\renzo{This part seems stuff which should be written somewhere. Reference?}\andreas{I found a version for curves in Mikhalkin and Zharkov's tropical curves paper, but haven't seen a proof for this. Torchiani's diploma thesis probably contains this, but one would have to travel to Kaiserslautern to check this}
\begin{example}\label{ex:ratsec}
Let $X =\R$ be the tropical space with the exotic affine structure {generated by the restriction of the identity to $(-1,1)$: it satisfies} $Aff_{X,x} = \Z \oplus \R$ for $-1< x< 1$, and $Aff_{X,x} = \R$ otherwise. 
Let $U_{1} = (-\infty, 1)$ and $U_{-1} = (-1, \infty)$ be an open cover of $X$, and let $m \in Aff_X(-1,1)\setminus \R$ be a cocycle defining a line bundle $\mc L$ on $X$. Since the only sections of $Aff_X(U_{\pm1})$ are constants, we deduce that $\mc L$ does not admit any rational section. 

\end{example}
We denote by $\SPic(X)$ the subgroup of $\Pic(X)$ consisting of all line bundles that admit a rational section. In other words, $\SPic(X)$ is the image of $\Div(X)$ in $\Pic(X)$. 
For every $\mc L\in \SPic(S)$ there exists a divisor $D\in \Div(X)$ such that $\mc L=\mc L(D)$. This divisor induces a map
\begin{equation}
A_*(X)\to A_*(X),\;\; \alpha\mapsto D\cdot \alpha \ ,
\end{equation}
which does not depend on the choice of $D$. This map is called the \textbf{first Chern class of $\mc L$}, denoted by $c_1(\mc L)$. The image of $\alpha\in A_*(X)$ under $c_1(\mc L)$ is denoted by $c_1(\mc L)\frown \alpha$.

\begin{remark}
If every point of a tropical space $X$ has a neighborhood that allows a closed embedding into $\T^r$ for some positive integer $r$, then the sheaf $\APL_X$ agrees with the sheaf $\mc M$ from \cite{Lefschetz}. It is shown in Lemma 4.5 in \textit{loc.\ cit.\ }that $H^1(X,\mc M)=0$, which implies that $\SPic(X)=\Pic(X)$ in this case.
\end{remark}

\subsubsection{Line bundles and their Chern classes on stacks over tropical spaces}

We  define divisors and line bundles on a stack $\mc M$ over the site of tropical spaces with representable diagonal. Because we want our definitions to apply to $\Mgnbar{g,n}$, we do not assume that $\mc M$ has an atlas. The lack of an atlas makes it impossible to define a group of tropical cycles on $\mc M$ and we settle with considering tropical cycles on tropical spaces $T$ equipped with a morphism $T\to \mc M$.

{
\begin{remark} \label{rem:fundcl}
One might be tempted to define tropical cycles on the stack $\Mgnbar{g,n}$ as balanced functions on its underlying set  $\vert \Mgnbar{g,n}\vert=\Mgnbar{g,n}(pt)$ of isomorphism classes of $n$-marked genus-$g$ stable tropical curves. The underlying set $\vert\Mgnbar{g,n}\vert$ does have a natural \TPL{}-structure, obtained for example by viewing it is a quotient of the underlying set of $\Atlass{g,n}$, but since there is no natural surjection from a tropical space onto $\Mgnbar{g,n}$, this \TPL{}-space $\vert\Mgnbar{g,n}\vert$ does not have a natural affine structure. In particular, there is no notion of balancing on $\vert\Mgnbar{g,n}\vert$, and thus no notion of tropical cycles. In the case $g=n=1$ this is illustrated by Example \ref{ex:fiber product over M11}:  there is a surjection $f_a\colon \T\PP^1\to \Mgnbar{1,1}$ for every $a\in \Z$ via which one could induce an affine structure on $\vert\Mgnbar{1,1}\vert$; however,  the fiber products $\T\PP^1\times_{f_a,\Mgnbar{1,1},f_b}\T\PP^1$ are pathological: for example, the constant weights do not define a one-cycle on them if $a\neq b$. This shows that these affine structure are incompatible. 
\end{remark}
}

It is possible to define  tropical line bundles and their Chern classes on $\mc M$.
%A \emph{tropical Cartier divisor on $\mc M$} is a collection of divisors $f^*D\in \Div(T)$ for every open morphism $T\xrightarrow f \mc M$ such that whenever $S\xrightarrow g \mc M$ is the pull-back of $f$ via an open morphism $h\colon S\to T$, one has $g^*D=h^*(f^*D)$. Since the pull-backs via open morphisms are morphisms of Abelian groups, we obtain an Abelian group $\Div(\mc M)$ of tropical Cartier divisors on $\mc M$. By construction, one can pull-back divisors along open morphisms from a tropical space  to $\mc M$.
A \emph{tropical line bundle on $\mc M$} is 
%a morphism $\mc M\to \mc X$, where $\mc X$ denote the stack of tropical line bundles \renzo{Is there a need to introduce the stack of tropical line bundles? Why not just saying tlb are line bundles on every pull-back?}. This means that a line bundle $\mc L$ on $\mc M$ is 
given by specifying for every morphism $T\xrightarrow f\mc M$ a line bundle $f^*\mc L$ in a way that is compatible with pull-backs.
Since line bundles can be tensored in a way compatible with pull-backs, all tropical line bundles on $\mc M$ form a group, the tropical Picard group $\Pic(\mc M)$ of $\mc M$. By construction, one can pull-back tropical line bundles along any morphism from a tropical space to $\mc M$. {The first Chern class of $
\mc L$ on the stack may be defined using the first Chern class of $f^*\mc L$ for all suitable pull-backs.}

%\begin{remark}
%As we have not assumed that $\mc M$ has an atlas, a morphism $T\to \mc M$ does not necessarily factor locally through an open morphism $U\to \mc M$. Therefore,  a morphism $\Div(\mc M)\to \Pic(\mc M)$ does not necessarily exist.
%\end{remark}
%As explained in the preceding remark, the absence of an atlas for $\mc M$ makes it impossible to define tropical cycles on $\mc M$. In particular, we cannot define first Chern classes of tropical line bundles on $\mc M$ the way we did for tropical spaces. However, 
% to have a definition on the stack.

\begin{definition}
\label{def:fcc} Given a tropical line bundle $\mc L$ on $\mc M$, a morphism $T\xrightarrow f \mc M$ such that $f^*\mc L\in \SPic(T)$, and a cycle $\alpha\in Z_*(T)$, we define the cycle
\begin{equation}
c_1(\mc L)\frown \alpha\coloneqq c_1(f^*\mc L)\frown \alpha \ .
\end{equation}

We call the collections of the morphisms
\begin{equation}
A_*(T)\to A_*(T) ,\;\; \alpha\mapsto c_1(\mc L)\frown \alpha
\end{equation}
for all maps $T\xrightarrow f\mc M$ such that $f^*\mc L\in \SPic(T)$ the \textbf{first Chern class of {$\mc L$ on} $\mc M$}.
\end{definition}

\subsection{Boundary divisors and $\psi$-classes}

We begin by defining tropical line bundles that are closely related to the cone subcomplexes at infinity that are usually thought of as tropical boundary divisors (\cite{ACP}).

\begin{definition}
\label{def:boundary line bundle}
Let   $g,n \in \Z^{\geq 0}$ with $n+2g-2>1$, and let $i,j\in I$ with $i\neq j$. We define the \textbf{boundary line bundle} $\mc L(D_{i,j})\in \Pic(\Mgnbar{g,n})$ by defining 
%\begin{equation}
%f^*D_{i,j} =s_i^*\phi^*(\phi\circ s_j)\renzo{Why is this a rational function on $B$?}
%\end{equation}
% for an open morphism $f\colon B\to \Mgnbar{g,n}$  corresponding to a family of tropical stable curves $\mc D\to B$ with section $(s_k)_{k\in I}$, where $\phi\colon \mc D\to \mc C$ is the stabilization of $\mc D$ after forgetting the $i$-th mark. Moreover, we define $\mc L(D_{i,j})\in \Pic(\Mgnbar{g,n})$ by setting 
\begin{equation}
f^*\mc L(D_{i,j})= s_i^*\phi^*(\Aff_{\mc C}(\phi\circ s_j)) \ ,
\end{equation}
for a morphism $f\colon B\to \Mgnbar{g,n}$  corresponding to a family of tropical stable curves $\mc D\to B$ with sections $(s_k)_{k\in I}$, where $\phi\colon \mc D\to \mc C$ is the stabilization of $\mc D$ after forgetting the $i$-th mark.
\end{definition}

\begin{remark}
As observed in Remark \ref{rem:divisors at infinity}, our definition of tropical Cartier divisors does not allow them to be supported on strata at infinity. Therefore, given a family of stable tropical curves $\mc C\to B$ and a section corresponding to a marked end $s\colon B\to \mc C$, there is no tropical Cartier divisor on $\mc C$ that is supported on $s(B)$. As already explained in Remark \ref{rem:divisors at infinity}, one can enlarge the sheaf $\APL_X$ to remedy this, which would allow us to define boundary divisors $D_{i,j}$ whose associated line bundles $\mc L(D_{i,j})$ agree with the boundary line bundles from Definition \ref{def:boundary line bundle}. But because this involves several technicalities while not increasing the generality of our results, we refrain from doing so. 
\end{remark}

While the definition of  $\mc L(D_{i,j})$ is asymmetrical in $i$ and $j$, one has $\mc L(D_{i,j})=\mc L(D_{j,i})$. We omit the proof of this fact as it is a technical exercise relying on the intuitive fact that for families of tropical curves where the $i$-th and $j$-th leg of every fiber are attached at the same vertex, there exists an isomorphism of tropical spaces exchanging the two legs. We have all ingredients to define the tropical $\psi$-classes.

\begin{definition}\label{def:psi}
Let  $g,n \in \Z^{\geq 0}$ with $n+2g-2>0$, and let $i\in [n]$. We define the \textbf{$i$-th cotangent line bundle} $\bL_i\in \Pic(\Mgnbar{g,n})$ by defining $f^*\bL_i =s_i^*(\mc \Aff_{\mc C}(-s_i))$ for a morphism $f\colon B\to \Mgnbar{g,n}$  corresponding to a family of $n$-marked genus-$g$ stable tropical curves $\mc C\to B$ with $i$-th section $s_i$. The \textbf{$i$-th psi-class} $\psi_i$ is defined by
\begin{equation*}
\psi_i=c_1(\bL_i) \ .
\end{equation*}
\end{definition}

{
\begin{example}
\label{ex:psimanyvalues}
We compute the class $f_a^\ast\psi_1$ for the family of elliptic curves described in Example \ref{ex:tropfam}, corresponding to a map $f_a: \T\PP^1 \to \Mgnbar{1,1}$. 

Consider the section $s_1: \T\PP^1 \to \mc C^a$, defined by $s_1 (x) = (x, \infty)$. 
%Its image is the locus  denoted $H_\infty$ in Example \ref{ex:tropfam}.

Denote by $T = T_-\cup T_+ \subset \mc C^a$, and construct two open sets covering $s_1$:
\begin{align}
    T_{-\infty} &= T \smallsetminus \mc C^a_\infty &  T_{\infty} &= T \smallsetminus \mc C^a_{-\infty}.
\end{align}

For these two open sets, we consider  sections of the sheaf $Aff_{\mc C^a}(-s_1)$:
\begin{align}
 y_- &= y_+ - ax_+  \in Aff_{\mc C^a}(-s_1)(T_{-\infty}) &
 y_- -ax_- &= y_+ \in Aff_{\mc C^a}(-s_1)(T_{\infty}).
\end{align}
Observe that neither $y_-$ nor $y_+$ restrict to the germ of an affine function for any point $x\in s_1$, making it impossible to pull them back to $\T\PP^1$ via $s_1$. However, we may consider the cocycle $\xi \in Aff_{\mc C^a}( T_{-\infty}\cap T_{\infty}) $:
\begin{equation}
    \xi = y_+-y_- = ax_+ = -ax_-.
\end{equation}
Since $\xi$ is constant along vertical fibers, its pull-back via $s_1$ determines a cocycle for $f_a^\ast\bL_1$. 
\begin{equation} \label{eq:cocforpsi}
    f_a^\ast\bL_1  = \left\{ s_1^\ast(T_{-\infty})\cap s_1^\ast(T_{-\infty}), s_1^\ast (\xi) \right\} = \left\{(-\infty, \infty), ax \right\}.
\end{equation}
As seen in Example \ref{ex:troplinebundles}, \eqref{eq:cocforpsi} shows that $f_a^\ast\bL_1 \cong \calO_{\T\PP^1}(a)$, which implies that the degree of 
$f_a^\ast \psi_1= c_1(f_a^\ast\bL_1) = a$.

This example illustrates that the degree of the class $f_a^\ast\psi_1$ does not depend only on the set theoretical map from the base to the moduli space, but it detects in an essential way the affine structure on the family of curves. 
\end{example}
}
The following lemma will be needed in the proof of the pull-back formula for the tropical $\psi$-classes.

\begin{lemma}
\label{lem:section at vertex}
Let $\mc C\to B$ be a family of $n$-marked stable tropical curves and let $i\in [n]$. Assume that for every $b\in B$ the vertex $v(b)$ of $\mc C_b$  that is adjacent to the $i$-marked leg is three-valent. Then the map $v\colon B\to \mc C$ is linear. 
\end{lemma}

\begin{proof}
Let $b\in B$, and let $J=T_{v(b)}\mc C_b$. By Lemma \ref{lem:local structure at genus 0 point}, there is an open neighborhood $V$ of $v(b)$ in $\mc C$, an open neighborhood $U$ of $b$ in $B$, and linear morphisms $f\colon U\to \Mgnbar{0, J}$ and $f_\star\colon V\to \Mgnbar{0, J\sqcup\{\star\}}$ such that the diagram
\begin{equation}
\begin{tikzpicture}[auto]
\matrix[matrix of math nodes, row sep= 5ex, column sep= 4em, text height=1.5ex, text depth= .25ex]{
|(C)| V 	&
|(n1)|	\Mgnbar{0,J\sqcup\{\star\}}	\\
|(B)| U	&	
|(n)| \Mgnbar{0,J}	\\
};
\begin{scope}[->,font=\footnotesize]
%%% horizontal
\draw (C) --node{$f_\star$} (n1);
\draw (B)--node{$f$} (n);
%%% vertical
\draw (C) -- (B);
\draw (n1) --node{$\pi_\star$} (n);
\end{scope}
\end{tikzpicture}
\end{equation}
is commutative and the induced morphism $V\to U\times_{\Mgnbar{0,J}}\Mgnbar{0, J\sqcup\{\star\}}$ is an open immersion. Since $\#J=3$, this means that $\mc C\to B$ is isomorphic to an open subset of a constant family in a neighborhood of $v(b)$. It follows immediately that $v$ is linear. 
\end{proof}

\begin{theorem}
\label{thm:pull-back of psi class}
Let $g,n\in \Z^{\geq 0}$ with $n+2g-2>0$, and let $i\in [n]$. Then
\begin{equation*}
\bL_i=\pi_\star^*\bL_i\otimes \mc L(D_{i,\star})  \ ,
\end{equation*}
where $\pi_\star$ denotes the forgetful morphism from Definition \ref{def:forgetful morphism}.
\end{theorem}

{
\begin{proof}
Let $\mc D\to B$ be a family of $(n\sqcup\{\star\})$-marked genus-$g$ stable tropical curves and  $\phi\colon \mc D\to \mc C$ its stabilization after forgetting the $\star$-section. Let $Y$ denote the open subset of $B$ consisting of all $b\in B$ such that the $i$-marked and $\star$-marked legs of $\mc D_b$ are adjacent to the same three-valent vertex $v(b)$, and let $D$ denote the closed subset of $Y$ consisting of all $b\in Y$ such that all three edges adjacent to $v(b)$ are infinite. Alternatively, one may characterize the points $b\in D$ as those whose fiber's $i$-th marked leg  is contracted by $\phi$.
%Let $\{U_\lambda\}_{\lambda\in \Lambda}$ such that for every $\lambda\in \Lambda$ there exists a section $t_\lambda\colon U_\lambda\to \mc D_{U_\lambda}$ such that $t_\lambda(b)$ is contained in the $i$-marked leg of $\mc D_b$ for every $b\in U_\lambda$. By Lemma \ref{lem:section at vertex}, we may assume that if $U_\lambda$ is contained in $V$, then $t_\lambda(b)$ is the vertex adjacent to the $i$-marked leg for all $b\in U_\lambda$. Let  
%After potentially refining the cover, we may assume that if $U_\lambda\cap D\neq \emptyset$, then there exists an open neighborhood 
Let $\{U_\delta\}_{\delta\in \Delta}$ be an open cover of $D$, and $\{U_\lambda\}_{\lambda\in \Lambda}$ an open cover of $B \smallsetminus Y$ such that:
\begin{itemize}
    \item  for any $\delta \in \Delta,\   \lambda\in \Lambda$, one has $U_\delta \cap U_\lambda = \varnothing;$ 
    \item for any $\lambda\in \Lambda$, there is a linear section $t_\lambda\colon U_\lambda \to \mc D_{U_\lambda}$  whose image is contained in the $i$-th marked leg of $\mc D_b$ for all $b\in U_\lambda$;
     \item for any $\delta\in \Delta$, there is a linear section $r_\delta \colon U_\delta \to \mc C_{U_\delta}$  whose image is contained in the $i$-th marked leg of $\mc C_b$ for all $b\in U_\delta$.
\end{itemize}

All this data is organized in the following diagram:
\begin{equation}
    \xymatrix{
 \bigcup_{\lambda\in \lambda} \mc D_{U_\lambda} \ar[dd] & \ar@{_{(}->}[l] \mc D \ar[rr]^\phi \ar[dr]&  & \mc C \ar[dl] & {\mc C}_D  \ar[dd] \ar@{_{(}->}[l] \ar@{^{(}->}[r]& \bigcup_{\delta\in \Delta} \mc C_{U_\delta} \ar[dd] \\
   &  & B \ar@/^/[ul]^{s_i}  \ar@/^2pc/[ul]^{s_\star} & & &  \\
\bigcup_{\lambda\in \lambda} U_\lambda \ar@/^2pc/[uu]^{\bigcup_{\lambda\in \Lambda} t_\lambda} & \ar@{_{(}->}[l] B\smallsetminus Y  \ar@{^{(}->}[ur] &  & Y \ar@{_{(}->}[ul]  & D  \ar@{_{(}->}[l] \ar@{^{(}->}[r]& \bigcup_{\delta\in \Delta} U_\delta \ar@/_2pc/[uu]_{\bigcup_{\delta\in \Delta} r_\delta},
 }
\end{equation}

{Let $A=\Delta\sqcup \{v\}\sqcup \Lambda$; denote $U_v=Y\setminus D$ and $t_v:U_v \to {\mc D}_{U_v}$ to be the linear section $b\mapsto v(b)$. Then  $\{U_\alpha\}_{\alpha\in A}$
% \begin{equation} \label{eq:opcov}
%  \{U_\alpha\}_{\alpha\in A}:=  \{U_\delta\}_{\delta\in \Delta} \cup Y \cup  \{U_\lambda\}_{\lambda\in \Lambda}, 
% \end{equation}
is an open cover of $B$}, which we will use to describe the cocycles for the line bundles we are interested in comparing.

To describe the cocycle for $\bL_i$ we need to 
construct local sections of $s_i^\ast (Aff_{\mc D}(-s_i))$ for all open sets in the open cover.
%\eqref{eq:opcov}
For $\alpha\in \Lambda \sqcup \{v\}$, let $\varphi_\alpha$ be the fiberwise harmonic function defined in a fiberwise connected neighorhood of $t_\alpha(U_\alpha)\cup s_i(U_\alpha)$  in $\mc D_{U_\alpha}$  that has constant value zero on $t_\alpha(U_\alpha)$ and slope one in the direction of $s_i(U_\alpha)$. For $\alpha\in \Delta$, define and denote $\varphi_\alpha = \varphi_v$, the fiberwise harmonic function equal to zero at $v(b)$ and with slope one towards $s_i(U_\alpha)$.
%Furthermore, let $\varphi_v$ be a fiberwise harmonic function that has constant value zero on \acolor{$v(Y\setminus D)$} to which the $i$-th leg attaches, and slope $1$ in the direction of \acolor{$s_i(Y\setminus D)$}.
Note that the local section $t_v$ is linear by Lemma \ref{lem:section at vertex}, hence  by Lemma \ref{lem:affineness is a clopen condition on fibers} and Lemma \ref{lem:characterization of affineness in terms of section and harmonicity} we see that $\varphi_\alpha$ defines a section of $\Aff_{\mc D}(-s_i)$ {for all $\alpha\in A$}.
The line bundle $\bL_i$ is represented by the cocycle $(\xi_{\alpha,\alpha'}) \in \check{H}^1(\{U_\alpha\}_{\alpha\in A},\Aff_B)$, where
% \begin{equation}\label{eq:coc0}
% \xi_{\alpha,\alpha'}=
% \begin{cases}
% s_i^*(\varphi_v-\varphi_v) = 0, &   \delta,\delta'\in \Delta\\
% s_i^*(\varphi_v-\varphi_v) = 0, &  U_\delta \cap Y \\
% s_i^*(\varphi_\lambda-\varphi_v), & U_\lambda \cap Y \\
% s_i^*(\varphi_\lambda-\varphi_{\lambda'}),  & \lambda,\lambda'\in \Lambda \ .
% \end{cases}
% \end{equation} 
{
\begin{equation}\label{eq:coc0}
\xi_{\alpha,\alpha'}=
\begin{cases}
s_i^*(\varphi_v-\varphi_v) = 0, &   \alpha,\alpha'\in \Delta\\
s_i^*(\varphi_v-\varphi_v) = 0, &  \alpha\in \Delta,~ \alpha'=v \\
s_i^*(\varphi_\alpha-\varphi_v), & \alpha\in \Lambda,~ \alpha'=v \\
s_i^*(\varphi_\alpha-\varphi_{\alpha'}),  & \alpha,\alpha'\in \Lambda \ .
\end{cases}
\end{equation} 
}

To describe the cocycles for $\pi_\star^*\bL_i$ and  $\mc L(D_{i,\star})$, we need to work with linear (local) sections of the stabilization $\mc C$. For $\delta\in \Delta$, we are given sections $r_\delta$ by construction. For $\alpha\in \Lambda \sqcup \{v\}$ , define $r_\alpha=\phi\circ t_\alpha$ ; for $\alpha \in A$,  let $\chi_\alpha$ be the fiberwise harmonic function defined in a fiberwise connected neighborhood of $r_\alpha(U_\alpha)\cup \phi(s_i(U_\alpha))$ in $\mc C_{U_\alpha}$ that has constant value zero on $r_\alpha(U_\alpha)$ and slope one in the direction of $\phi(s_i(U_\alpha))$ %\andreas{this is the part that would not work if $U_v=Y$}. 
By Lemma \ref{lem:affineness is a clopen condition on fibers} and Lemma \ref{lem:characterization of affineness in terms of section and harmonicity},  each $\chi_\alpha$ defines a section of $\Aff_{\mc C}(-\phi\circ s_i)$. For $\alpha\in \Lambda \sqcup \{v\}$, note that $\phi^\ast(\chi_\alpha)$ is constantly zero along the end marked by $\star$.

 The line bundle $\pi_\star^*\bL_i$ is represented by the cocycle 
 \begin{equation}\label{eq:coc1}
    ((\phi\circ s_i)^*(\chi_{\alpha}-\chi_{\alpha'})), 
 \end{equation} and $\mc L(D_{i,\star})$ is represented by the cocycle $(\zeta_{\alpha,\alpha'})$, where
% \begin{equation}\label{eq:coc2}
% \zeta_{\alpha,\alpha'}=
% \begin{cases}
% (\phi\circ s_\star)^*(\chi_{\delta'}-\chi_{\delta}), &   \delta,\delta'\in \Delta\\
% (\phi\circ s_\star)^*(\chi_{v}-\chi_{\delta}) , &  U_\delta \cap Y \\
% (\phi\circ s_\star)^*(\chi_{v}) = 0, & U_\lambda \cap Y \\
% 0,  & \lambda,\lambda'\in \Lambda \ .
% \end{cases}
% \end{equation} 
{
\begin{equation}\label{eq:coc2}
\zeta_{\alpha,\alpha'}=
\begin{cases}
(\phi\circ s_\star)^*(\chi_{\alpha'}-\chi_{\alpha}), &   \alpha,\alpha'\in \Delta\\
(\phi\circ s_\star)^*(-\chi_{\alpha}) , &  \alpha\in \Delta,~ \alpha'=v \\
 0, & \alpha\in \Lambda, \alpha'=v \\
0,  & \alpha,\alpha'\in \Lambda \ .
\end{cases}
\end{equation} 
We implicitly assume here that $\chi_\delta$ is defined on $(\phi\circ s_\star)(U_\delta)$ for all $\delta\in \Delta$. This can always be achieved after potentially shrinking the open sets $U_\delta$.%\renzo{I am not sure this sentence is needed: the $U_\delta$ are all inside $Y$, so $\chi_\delta$ should  already be defined. }
}

Since $\chi_{\alpha}$ and $\chi_{\alpha'}$ always have the same slope on the $i$-marked leg, we have 
\begin{equation}\label{eq:coc3}
(\phi\circ s_i)^*(\chi_{\alpha'}-\chi_{\alpha})=(\phi\circ s_\star)^*(\chi_{\alpha'}-\chi_{\alpha})
\end{equation}
for all {$\alpha,\alpha'\in \Delta\cup \{v\}$}. Moreover, for $\alpha \in \Lambda \cup\{v\}$, in a neighborhood of $s_i(U_\alpha)$, we have
\begin{equation}\label{eq:coc4}
\phi^*\chi_\alpha = \varphi_\alpha. 
\end{equation}

Combining \eqref{eq:coc0}, \eqref{eq:coc1}, \eqref{eq:coc2}, \eqref{eq:coc3}, \eqref{eq:coc4} {with the fact that $(\phi\circ s_\star)^*\chi_v=0$}, one sees  that for every $(\alpha, \alpha')\in A\times A$,
\begin{equation}
 (\phi\circ s_i)^*(\chi_{\alpha}-\chi_{\alpha'}) + \zeta_{\alpha,\alpha'}=\xi_{\alpha,\alpha'},
\end{equation}
which implies
\begin{equation}
\pi_\star^*\bL_i\otimes \mc L(D_{i,\star}) \cong \bL_i .
\end{equation}
\end{proof}
}

{In case the vertex $v$ adjacent to the $i$-marked leg has always genus zero, one can  find a Cartier divisor $D$ with $\mc L(D)\cong\bL_i$ such that  $D$ is supported at the points $b$ where $v(b)$ has valence greater than three.}

\begin{proposition}
\label{prop:representing psi-class by divisor}
Let $\pi\colon\mc C\to B$ be a family of $n$-marked genus-$g$ stable tropical curves corresponding to a map $f\colon B\to \Mgnbar{g,n}$, let $i\in [n]$, and assume that for every $b\in B$ the vertex $v(b)$ of $\mc C_b$ that is adjacent to the $i$-marked leg has genus zero. For every $b\in B$, let $J_b= T_{v(b)}\mc C_b$, where we identify $i$ with the element of $J_b$ corresponding to the $i$-marked leg, let $j_b,k_b\in J_b\setminus\{i\}$ be distinct, and let $V_b$ be an open neighborhood of $v(b)$ that is isomorphic to an open subset of a family of  $J_b$-marked genus-zero stable tropical curves $\mc C_b\to U_b$ over $U_b=\pi(V_b)$, which exists by Lemma \ref{lem:local structure at genus 0 point}. For every $x\in U_b$, the intersection in $\mc C_b$ of the paths from $v(x)$ to the $j_b$-th and $k_b$-th marked point, respectively, is contained in $V_b$ and we denote by $\chi_b(x)$ the length of that path. Then the datum $\{(U_b,\chi_b)\}_{b\in B}$ defines a Cartier divisor $D\in \Div(B)$ such that $\mc L(D)\cong f^*\bL_i$. 
\end{proposition}

\begin{proof}
For every $b\in B$, the identification of $V_b$ with an open subset of the family of genus-zero curves $\mc C_b\to U_b$ induces a morphism $f_b\colon V_b \to \Mgn{0,J_b\sqcup\{\star\}}$. By definition of $\chi_b$, we have
\begin{equation}
\chi_b=(f_b\circ v)^*\xi_{(\star, j_b), (i,k_b)} \ .
\end{equation}
It follows from this that $\chi_b\in \Gamma(U_b,\APL_X)$. Moreover, ${f}_b^*\xi_{(\star,j_b),(i,k_b)}$ has outgoing slope one from $v(b)$ in the direction determined by {$i$}. By Lemma \ref{lem:affineness is a clopen condition on fibers} and Lemma \ref{lem:characterization of affineness in terms of section and harmonicity}, ${f}_b^*\xi_{(\star,j_b),(i,k_b)}$ can be extended to a section $\xi_b$ of $\Aff_{\mc C}(-s_i)$ over an open set including $V_b$ and the whole $i$-marked leg of $\mc C_x$ for every $x\in U_b$. Given two points $b,b'\in B$, the functions $\xi_b$ and $\xi_{b'}$ have the same slope on the $i$-marked leg, and hence {their difference is} constant on that leg of $\mc C_x$ for all $x\in U_b\cap U_{b'}$. It follows that 
\begin{equation}
s_i^*(\xi_b-\xi_{b'})=v^*(\xi_b-\xi_{b'})=\chi_b-\chi_b' \ ,
\end{equation}
from which we conclude that $\{(U_b,\chi_b)\}_{b\in B}$ defines a divisor $D$ on $B$ for which $\mc L(D)\cong f^*\bL_i$.
\end{proof}

\begin{corollary}
\label{cor:support of psi-class}
Let $\mc C\to B$ be a family of $n$-marked tropical stable curves corresponding to a morphism $f\colon B\to \Mgnbar{g,n}$, and let $i\in [n]$ such that the vertex $v(b)$ adjacent to the $i$-marked leg of $\mc C_b$ has genus zero for all $b\in B$. Moreover, assume that $v(b)$ is three-valent for all $b$ in a dense open subset of $B$.  Then there exists a divisor $D\in \Div(B)$ with $\mc L(D)\cong f^*\bL_i$ and whose support is contained in the set of all $b\in B$ such that $v(b)$ is at least four-valent.
\end{corollary}

\begin{proof}
Let $\{(U_b,\chi_b)\}_{b\in B}$ be as in the statement of Proposition \ref{prop:representing psi-class by divisor}, and let $D$ be the Cartier divisor represented by that data. For every $b\in B$ for which $v(b)$ is trivalent we may shrink $U_b$ such that $v(x)$ is three-valent for all $x\in U_b$. In that case $\chi_b=0$, and hence $b$ is not contained in the support of $D$. 
\end{proof}

{One  recovers the description of the cycle representing $\psi_i$ on $M_{0,n}^{trop}$ from \cite{MikhalkinModuli,KerberMarkwig}.}

%\begin{proof}
%Let $\{(U_b,\chi_b)\}_{b\in B}$ be as in the statement of Proposition \ref{prop:representing psi-class by divisor}, and let $D$ be the Cartier divisor represented by that data. For every $b\in B$ for which $v(b)$ is trivalent we may shrink $U_b$ such that $v(x)$ is three-valent for all $x\in U_b$. In that case $\chi_b=0$, and hence $b$ is not contained in the support of $D$. 
%\end{proof}

\begin{corollary}
\label{cor:psi class in genus 0}
Let $[\Mgnbar{0,n}]\in Z_*(\Mgnbar{0,n})$ denote the top-dimensional cycle on $\Mgnbar{0,n}$ represented by the constant function with value one. Then $\psi_i\cdot [\Mgnbar{0,n}]$ is represented by the function on $\Mgnbar{0,n}$ that is one on all points corresponding to curves where the $i$-marked leg is adjacent to a four-valent vertex, and zero everywhere else.
\end{corollary}

\begin{proof}
Let $\{(U_b,\chi_b)\}_{b\in B}$ be as in the statement of Proposition \ref{prop:representing psi-class by divisor}, and let $D$ be the Cartier divisor represented by that data. As we have seen in the proof of Corollary  \ref{cor:support of psi-class}, the support of $D$ are precisely those points of $\Mgnbar{0,n}$ corresponding to curves where the $i$-marked leg is adjacent to a vertex of valence at least four. Let $G$ be a $n$-marked genus-zero graph such that all edges except legs of $G$ are bounded and all vertices of $G$ are three-valent, except the vertex $v$ adjacent to the $i$-marked leg, which is four-valent. Let $b\in \relint(\sigma_G)$ be a point corresponding to a curve $\Gamma$. Then $\chi_b$ vanishes on $U_b\cap \sigma_G$. Let $j_b,k_b\in T_v G\setminus\{i\}$ be distinct, as in the statement of Proposition \ref{prop:representing psi-class by divisor}, and let $l_b$ be the fourth element of $T_v G$. The cone $\sigma_G$ is adjacent to precisely three top-dimensional cones of $\Mgnbar{0,n}$, corresponding to three-valent stable graphs $G_{j_b}$, $G_{k_b}$ and $G_{l_b}$, where the subscript denotes the tangent direction which remains adjacent to the $i$-marked leg. By construction, the rational function $\chi_b$ has multiplicity one on $U_b\cap\sigma_{G_{l_b}}$ and it is identically zero on $U_b\cap \sigma_{G_{j_b}}$ and on $U_b\cap \sigma_{G_{k_b}}$. It follows that $\psi_i\cap [\Mgnbar{0,n}]$ has weight one at $b$.
\end{proof}

\begin{theorem}
\label{thm:germ of dilaton}
Let $\pi\colon \mc C\to B$ be a family of $n$-marked genus-$g$ tropical Mumford curves corresponding to a map $f\colon B\to \MgnMf{g,n}$ and let $f_\star\colon \mc C\to \MgnMf{g,n\sqcup\{\star\}}$ be the pull-back of $f$ under the forgetful morphism. Furthermore, assume that $B$ is purely $n$-dimensional and let $\alpha\in Z_n(B)$. Then $\alpha\circ \pi$ is a tropical $(n+1)$-cycle on $\mc C$, and $\psi_\star\cdot (\alpha\circ\pi)$ is represented by the tropical $n$-cycle $\mc C\to \Z$ that has value zero on all % that are marks in their fibers
genus zero, one-valent vertices of fibers, and $(\val(x)-2)\cdot \alpha(\pi(x))$ at all other points $x\in \mc C$.
\end{theorem}

\begin{proof}
By Proposition \ref{lem:local structure at genus 0 point}, the total space $\mc C$ is locally isomorphic to the product of $B$ and a stable tropical curve away from a codimension-two locus on $\mc C$. The composition $\alpha\circ \pi$ is a tropical cycle because of this, the facts that $\alpha$ is a tropical cycle on $B$ and that constant integer weights define tropical cycles on smooth tropical curves. 

To compute $\psi_\star\cdot (\alpha\circ \pi)$, let $\{(U_x,\chi_x)\}_{x\in \mc C}$ be as in the statement of Proposition \ref{prop:representing psi-class by divisor} (applied to the $(n\sqcup \{\star\})$-marked family over $\mc C$), and let $D$ be the Cartier divisor represented by that data. As in the proof of Corollary  \ref{cor:support of psi-class}, the support of $D$ is  the set of all $x\in \mc C$ having valence at least three in its fiber. Let $x$ be a generic point in the support of $D$, and let $v$ be the vertex adjacent to the $\star$-marked leg in the fiber over $x$ in the $(n\sqcup\{\star\})$-marked family $\mc C_\star\to \mc C$. The rational function $\chi_x$ depends on a choice of two distinct tangent vectors in $T_v (\mc C_\star)_x$ that do not point in the direction of the $\star$-marked leg. This corresponds to the choice of two distinct tangent vectors $j,k\in T_x(\mc C_{\pi(x)})$. With this choice, the function $\chi_x$ has slope one in the direction corresponding to any $t\in T_x(\mc C_{\pi(x)})$, and slope zero in the directions corresponding to $j$ and $k$. It follows that the multiplicity of $D\cdot (\alpha \circ \pi)$ at $x$ is given by $(\val(x)-2)\alpha(\pi(x))$.
\end{proof}

\begin{remark}
Theorem \ref{thm:germ of dilaton} can be interpreted as the germ of a tropical dilaton equation. Although we have not mentioned this in \S\ref{sec:intersection theory}, the push-forward of tropical cycles respects rational equivalence. Together with the fact the sum $\sum_v (\val(v)-2)$ over all finite vertices $v$ of an $n$-marked genus-$g$ stable graph equals $2g-2+n$, we conclude that
\begin{equation*}
\pi_*(\psi_\star\cdot(\alpha\circ\pi))=(2g-2+n)\cdot\alpha 
\end{equation*}
in the situation of the theorem. In the case where the base is $\MgnMf{g,n}$, we obtain 
\begin{equation}
\pi_*(\psi_\star\cdot[\MgnMf{g,n\sqcup\{\star\}}])=(2g-2+n)\cdot [\MgnMf{g,n}] \ .
\end{equation}
%This is the main ingredient in the proof of the classical dilaton equation and the only reason we cannot state the dilaton equation tropically is the lack of a fundamental class for $\Mgnbar{g,n}$.
\end{remark}

\section{Correspondence theorems in genus $1$}\label{sec:classes}

In this section, we evaluate $\psi$ classes on some one-dimensional cycles of genus one tropical curves.
The first example is a Gromov-Witten cycle, i.e.\ a family of tropical stable maps to the plane subject to point conditions \cite{Tyomkin,GathmannKerberMarkwig,R17,RSW19}. The second family of examples consists of one-dimensional spaces of admissible covers of the projective line \cite{Caporaso, CMRadmissible}. In both cases, the theory of tropical $\psi$-classes provides a correspondence with its algebro-geometric counterpart.

%Since $\Mgnbar{g,n}$ does not have a fundamental cycle, there is no natural way to produce families of $I$-marked genus-$g$ tropical stable curves over a compact base $B$. 
%In this section, we show how to obtain such families using input from algebraic geometry.

%In this section we study $\psi$ classes on construct some examples of 
%families of tropical cycles; we focus our attention on two types of family, constructed via   

%$\psi$-classes to these families  then results in a  correspondence theorem.

\subsection{Families of tropical curves from well-spaced stable maps.}
\label{sec:71}

%Let $\T\PP^2$ be  tropical projective two-space, that is the tropical toric variety associated to the complete fan in $\R^2$ whose rays are $\begin{pmatrix}1\\0 \end{pmatrix}$, $\begin{pmatrix} 0\\1\end{pmatrix}$, and $\begin{pmatrix} -1\\-1\end{pmatrix}$.
Given a degree $d$ and a non-negative integer $n$, we  consider the space  of all tropical stable maps into $\T\PP^2$ of degree $d$ with $n$ marked points as described in \cite{R17}. We slightly modify the construction of \emph{loc. cit}.\ by adding cycle-rigidifications to the source curves and taking the extended cone complex to obtain a compact \TPL{}-space $\TSM_{1,n}(\T\PP^2,d)$.  As some combinatorial types of these tropical stable maps are super-abundant, the dimension of $\TSM_{1,n}(\T\PP^2)$ is larger than the expected dimension. %One way of dealing with this is by replacing $\TSM_{1,n}(\T\PP^2)$ by 
We consider the  closed subspace $\WS_{n}(\T\PP^2,d) \subset \TSM_{1,n}(\T\PP^2)$ consisting of all tropical stable maps that are {\it well-spaced} in the sense of \cite[Definition 4.4.4]{RSW19} (following well-spacedness as a realizability condition in \cite{SpeyerDiss,SpeyerUniformizing}). This space is pure-dimensional of dimension
\begin{equation}
\dim(\WS_{n}(\T\PP^2,d))= 3d+n \ .
\end{equation}
Evaluating stable maps at the marked points yields the evaluation maps 
\begin{equation}
\ev_i\colon \WS_{n}(\T\PP^2,d) \to \T\PP^2
\end{equation}
for each $1\leq i\leq n$;  forgetting the map and stabilizing the source curve defines a map
\begin{equation}
\src\colon \WS_{n}(\T\PP^2,d)\to \Atlass{1,n\acolor{+3d}} \ .
\end{equation}
We endow $\WS_{n}(\T\PP^2,d)$ with the smallest affine structure such that $\src$ and $\ev_i$ for all $1\leq i\leq n$ become morphisms of tropical spaces.

Choose eight general points $p_1,\ldots,p_8$ in $\R^2$. We may assume that the points are horizontally stretched: they lie in a thin neighborhood of a line of  very small negative slope, their numbering and  horizontal spacing increases from left to right.  Because the points are in general position and $\WS_{8}(\T\PP^2,3)$ is $17$-dimensional, the set
\begin{equation}
B=\bigcap_{i=1}^8\ev^{-1}_i\{p_i\}
\end{equation}
is a one-dimensional subset of $\WS_{8}(\T\PP^2,3)$. The corresponding family of tropical stable maps is  described explicitly in Appendix \ref{app:main}.
Since $B$ is defined as a generic intersection of the sets $\ev_i^{-1}\{p_i\}$,  we may define a tropical one-cycle on $B$ is via tropical intersection theory. %For this we will need the following lemma.

\begin{lemma}
\label{lem:cycle on neighborhood of pencil}
There exists an open neighborhood $U$ of $B$ in $\WS_{8}(\T\PP^2,3)$ such that the function  $\alpha_U$ assigning to a point $b\in B$ corresponding to a tropical stable map $\Gamma_b\to \T\PP^2$ the value zero, if $\Gamma$ has a loop attached to a trivalent vertex or the midpoint of an edge of multiplicity one, and the number of automorphisms of $\Gamma_b$, else, defines a tropical $17$-cycle on $U$. 
\end{lemma}

\begin{proof}
{We define $U$ to be the union of relative interiors of cones in $\WS_{8}(\T\PP^2,3)$ (considered as an extended cone complex) that have nonempty intersection with $B$. Because the eight points used to define $B$ are chosen generically, $U$ is open and only contains the relative interiors of $16$- and $17$-dimensional cones.
We have to check balancing around every $16$-dimensional facet. If the facet is contained in the locus of Mumford curves,  this follows from \cite[Theorem 3.3.8]{Torchiani}. If, on the other hand, a facet $\tau$ is contained in the locus of curves that have a genus one vertex, then it parametrizes maps whose source curve has a two-valent genus one vertex that is mapped to the midpoint of an edge of multiplicity two. The integer affine functions defined in a neighborhood of $\tau$ are  the linear combinations of the edge lengths $l_e$ for all edges $e$ contained in the set $E_0$ of edges not adjacent to the genus one vertex and pull-backs of integer affine functions on $\T\PP^2$ via the evaluation maps. There are two $17$-dimensional faces adjacent to $\tau$. In one of them, call it $\sigma_1$, the genus one vertex is resolved by inserting a loop that is contracted by the map. In the other one, call it $\sigma_2$, the genus one vertex is resolved by inserting two parallel edges on which the map has multiplicity one. In either case, every combination of values of evaluation maps and edge lengths $l_e$ for $e\in E_0$ on $\sigma_i$ is also attained on $\tau$. More precisely, we have 
\begin{equation}
\left(\prod_{e\in E_0} l_e\times \prod_{i=1}^8 \ev_i\right)(\sigma_k) = \left(\prod_{e\in E_0} l_e\times \prod_{i=1}^8 \ev_i\right)(\tau)  
\end{equation}
for all $k\in \{1,2\}$. It follows that every integral affine function that vanishes on $\tau$ also vanishes on a small neighborhood of $\tau$. Therefore, the balancing condition is vacuously true around $\tau$.
}
\end{proof}

\noindent
For the following definition, note that for every $p\in \R^2$, the tropical line $L_p$ with vertex at $p$
% \[
% L_p\coloneqq \overline{p+ \left(\R_{\geq 0} \begin{pmatrix}
% 1\\0
% \end{pmatrix} \cup \R_{\geq 0} \begin{pmatrix}
% 0\\1
% \end{pmatrix} 
% \cup \R_{\geq 0} \begin{pmatrix}
% -1\\-1
% \end{pmatrix}
% \right)}
% \]
defines a tropical Cartier divisor on $\T\PP^2$ and that $L_p^2=[p]$. 

\begin{definition}
We define the cycle $\alpha_B$ on $B$ as
\begin{equation}
\alpha_B=\left(\prod_{i=1}^n (\ev_i^*L_{p_i})^2\right) \cdot \alpha_U \ ,
\end{equation}
where $U$ is a sufficiently small neighborhood of $B$ in $\WS_{8}(\T\PP^2,3)$ and $\alpha_U$ is defined as in Lemma \ref{lem:cycle on neighborhood of pencil}.
\end{definition}

No leg of the source curve of any tropical stable map corresponding to some $b\in B$ is adjacent to a genus one vertex, hence the set $\src(B)$ is contained in $\GoodAtlas{1,17}$. Therefore, there exists a natural family of $17$-marked genus one stable tropical curves $\mc C\coloneqq B\times_{\GoodAtlas{1,17}}\GoodAtlas{1,17\sqcup \{\star\}} \to B$ by Proposition \ref{prop:family of tropical curve over atlas}. For $1\leq i\leq 8$, let $\mc C_i\to B$ be the family of one-marked genus one stable tropical curves obtained from $\mc C\to B$ by forgetting all marks but the $i$-th one.

\begin{proposition}
\label{prop:degree of psi1 on pencil}
For $i = 1, \ldots , 8$, let $f_i\colon B\to \Mgnbar{1,\{i\}}$ be the morphism corresponding to the family $\mc C_i\to B$. Then we have 
\begin{equation}
\int_Bf_i^*\psi_i\cdot \alpha_B =432 \ .
\end{equation}
\end{proposition}

\begin{proof}
First we observe that the $i$-th marked point never sits on a component that is contracted by the stabilization $\mc C\to \mc C_i$. So by Theorem \ref{thm:pull-back of psi class}, it suffices to prove that
\begin{equation}
\int_B g^*\psi_i\cdot \alpha_B=432 \ ,
\end{equation}
where $g\colon B\to \Mgnbar{1,17}$ is the morphism induced by the $17$-marked family $\mc C \to B$.

The $i$-marked leg is attached to a genus-zero vertex in every fiber of $\mc C$, allowing us to use Proposition \ref{prop:representing psi-class by divisor} to compute the degree of $\psi_i$. Using the representative for $\psi_i$ produced by the proposition, we see that there is a contribution for every curve in the family contained in the subset $\Psi$ of $B$ consisting of curves where the $i$-marked leg is adjacent to a four-valent vertex. As one sees combinatorially (see Figure \ref{fig:vertmar}), all points in $B$ for which this happens have the same image curve $\Gamma'$, that is they correspond to the same well-spaced map to $\T\PP^2$ up to orientation of the cycle and marking of the legs on which the map is not constant. 

There are two cases to consider. In the first case, for $i\in\{3,4,6,7,8\}$,  the image curve $\Gamma'$ does not have any edges of multiplicity two. There are $2\cdot (3!)^3=432$ different curves in $B$ with given image curve $\Gamma'$, so $\Psi$ has $432$ elements. For each $b=[\Gamma\to \T\PP^2]\in \Psi$, the $i$-th marked point maps to a three-valent vertex in $\Gamma'$. Moreover, the star of  $\Gamma'$ at that vertex is a tropical line in suitable coordinates. Therefore, locally around $b$, the space $B$ is isomorphic to the space of $\{i\}$-marked plane tropical lines passing through a given point in the plane (see Figure \ref{fig:local picture at 3-valent vertex}). By \cite[Proposition 4.7]{GathmannKerberMarkwig}, this space is isomorphic to $\Mgnbar{0,4}$. {The weights of $\alpha_B$ are computed as explained in Appendix \ref{app:main}; computations analogous to the one illustrated in Figure \ref{fig:matrixmult} show the weights are equal to one locally around $b$.} By Corollary \ref{cor:psi class in genus 0} it follows that the contribution of $b$ to the degree of $\psi_i$ is one. The total degree of $\psi_i$ on $B$ is thus equal to $432$.

In the second case, occurring when $i\in\{1,2,5\},$  the image curve $\Gamma'$ has an edge of multiplicity two.  If $i=5$, then the $i$-th leg is adjacent to the loop of $\Gamma$, whereas when $i\in\{1,2\}$ it is adjacent to two legs on which the map is non-constant. In either case, $\Psi$ consists of $216$ points. In any map $b=[\Gamma\to \T\PP^2]\in \Psi$, the $i$-th leg is adjacent to a four-valent vertex $v$ of $\Gamma$ that maps to the interior of the multiplicity-two edge in $\Gamma'$. Let $f_1,f_2,f_3$, the three remaining directions of $\Gamma$ at $v$ such that the map has dilation factor two in direction $f_1$. Each of these flags corresponds to the three directions $e_1,e_2,e_3$ on the base at $b$, as shown in Figure \ref{fig:local picture at multiplicity 2 edge}. { Computing the weights as in Appendix \ref{app:main}, Figure \ref{fig:matrixmult}, one sees the weight $\alpha_B$ is two in the direction of $e_1$ and one in the other two directions.} Using Lemma \ref{lem:local structure at genus 0 point}, the contribution of $\psi_i$ at $b$ is uniquely determined by the map $\phi\colon V\to \Mgnbar{0,4}$ defined in a neighborhood $V$ of $b$  by restricting one's attention to a neighborhood of $v$ in the total space of the family. In each direction $e_i$ one obtains a different combinatorial type of four-marked curves, so $\phi$ maps each of the directions $e_i$ to different directions at the cone-point of $\Mgnbar{0,4}$. Furthermore, $\phi$ has slope one in the direction of $e_1$. Therefore, $\phi_*(\alpha_B\vert_U)$ has weight two everywhere. By the tropical projection formula and Corollary \ref{cor:psi class in genus 0}, we conclude that the contribution at $b$ to $\psi_i$ is two. The total degree of $\psi_i$ on $B$ is thus equal to $2\cdot 216=432$.
\end{proof}

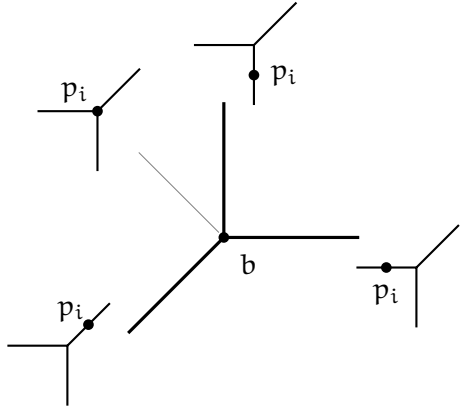
\begin{figure}
    \centering
    \begin{tikzpicture}[thick]
        \begin{scope}[very thick]
            \node[dot,label={-45:$b$},pin={[pin distance=1.5cm]135:{}}]{};
            \clip (0,0) circle (1.8cm);
            \draw (0,0)-- (2,0)
                  (0,0)--(0,2)
                  (0,0)--(-2,-2);   
        \end{scope}
    
        %%% middle
        \begin{scope}[scale=.8,xshift=-2.1cm, yshift=2.1cm]
            \node [dot, label={[label distance=-.1cm]135:$p_i$}] at (0,0) {};
            \clip (0,0) circle (1cm);
            \draw (0,0)-- (-1,0)
                  (0,0)--(0,-1)
                  (0,0)--(1,1);            
        \end{scope}
        
        %%% right
        \begin{scope}[scale=.8,xshift=3.2cm,yshift=-.5cm]
            \node [dot, label=below:$p_i$] at (-.5,0) {};
            \clip (0,0) circle (1cm);
            \draw (0,0)-- (-1,0)
                  (0,0)--(0,-1)
                  (0,0)--(1,1);            
        \end{scope}
        
        %%% top
        \begin{scope}[scale=.8,xshift=.5cm,yshift=3.2cm]
            \node [dot, label=right:$p_i$] at (0,-.5) {};
            \clip (0,0) circle (1cm);
            \draw (0,0)-- (-1,0)
                  (0,0)--(0,-1)
                  (0,0)--(1,1);            
        \end{scope}
        
        %%% below left
        \begin{scope}[scale=.8,yshift=-1.8cm, xshift=-2.6cm]
            \node [dot, label={[label distance=-.2cm]135:$p_i$}] at (.35,.35) {};
            \clip (0,0) circle (1cm);
            \draw (0,0)-- (-1,0)
                  (0,0)--(0,-1)
                  (0,0)--(1,1);            
        \end{scope}        
    \end{tikzpicture}
    \caption{If $b$ corresponds to a point where $p_i$ coincides with a three-valent vertex of the image curve, the family $B$ is isomorphic to a tropical line locally around $b$.}
    \label{fig:local picture at 3-valent vertex}
\end{figure}

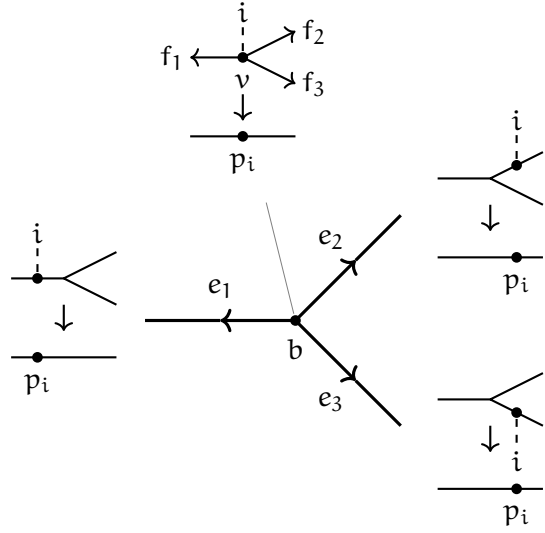
\begin{figure}
    \centering

    \begin{tikzpicture}[thick]
        
        \begin{scope}[very thick]
            \node at (0,0) [dot,label=below:$b$,pin={[pin distance=1.5cm]100:{}}] {} ;
            \draw[->] (0,0)--
                node [at end, label=above:$e_1$]{}
            (-1,0);
            \draw (-1,0)--(-2,0);
            
            \draw[->] (0,0)--
                node [at end, label={[label distance=-.2cm]135:$e_2$}]{}
            (.8,.8);
            \draw (.8,.8)--(1.4,1.4);
            
            \draw[->] (0,0)--
                node [at end, label={[label distance=-.2cm]-135:$e_3$}]{}
            (.8,-.8);
            \draw (.8,-.8)--(1.4,-1.4);

            % \coordinate (mid) at (1,.7);
            % \draw (0,0)[->] .. controls (.3,.3) and ($(mid)+(-.3,-.1)$).. (mid);
            % \draw (mid) .. controls ($(mid)+(.3,.1)$) and (1.7,.95).. (2,1);
            
            % \coordinate (mid2) at (1,-.7);
            % \draw (0,0)[->] .. controls (.3,-.3) and ($(mid2)+(-.3,.1)$).. (mid2);
            % \draw (mid2) .. controls ($(mid2)+(.3,-.1)$) and (1.7,-.95).. (2,-1);
        \end{scope}

        \begin{scope}[scale=.7,xshift=-2cm,yshift=5cm]
            \node[dot, label=below:$v$] (vert) at (1,0){};
            \draw [<-] (0,0) --
                node [at start, label={[label distance=-.2cm]left:$f_1$}] {} 
            (1,0);
            
            \draw [->]( 1,0)--
                node [at end, label={[label distance=-.2cm]right:$f_2$}]{}
            (2,.5);
            \draw [->] (1,0)--
                node [at end, label={[label distance=-.2cm]right:$f_3$}]{}
            (2,-.5);
            \draw [dashed] (vert)--
                node[at end,label={[label distance=-.3cm]above:$i$}] {}
            +(0,.7);
            \draw[->] (1,-.7)--(1,-1.2);
            \draw (0,-1.5)--(2,-1.5);
            \node [dot,label=below:$p_i$] at (1,-1.5){};
        \end{scope}

        %%% left
        \begin{scope}[scale=.7,xshift=-5.4cm, yshift=.8cm]
            \node[dot] (vert) at (.5,0){};
            \draw (0,0) -- (1,0)
                  (1,0)-- (2,.5)
                  (1,0)--(2,-.5);
            \draw [dashed] (vert)--
                node[at end,label={[label distance=-.3cm]above:$i$}] {}
            +(0,.7);
            \draw[->] (1,-.5)--(1,-1);
            \draw (0,-1.5)--(2,-1.5);
            \node [dot,label=below:$p_i$] at (.5,-1.5){};
        \end{scope}

        %%%% top right
        \begin{scope}[scale=.7,xshift=2.7cm, yshift=2.7cm]
            \node[dot] (vert) at (1.5,.25){};
            \draw (0,0) -- (1,0)
                  (1,0)-- (2,.5)
                  (1,0)--(2,-.5);
            \draw [dashed] (vert)--
                node[at end,label={[label distance=-.3cm]above:$i$}] {}
            +(0,.7);
            \draw[->] (1,-.5)--(1,-1);
            \draw (0,-1.5)--(2,-1.5);
            \node [dot,label=below:$p_i$] at (1.5,-1.5){};
        \end{scope}

        %%% bottom right
        \begin{scope}[scale=.7,xshift=2.7cm, yshift=-1.5cm]
            \node[dot] (vert) at (1.5,-.25){};
            \draw (0,0) -- (1,0)
                  (1,0)-- (2,.5)
                  (1,0)--(2,-.5);
            \draw [dashed] (vert)--
                node[at end,label={[label distance=-.2cm]below:$i$}] {}
            +(0,-.6);
            \draw[->] (1,-.5)--(1,-1);
            \draw (0,-1.7)--(2,-1.7);
            \node [dot,label=below:$p_i$] at (1.5,-1.7){};
        \end{scope}        
    \end{tikzpicture}
    
    \caption{The flag $f_i$ corresponds to the direction $e_i\in T_b B$ where the $i$-marked leg moves away from the vertex in the direction given by $f_i$.}
    \label{fig:local picture at multiplicity 2 edge}
\end{figure}

\begin{proposition}
\label{prop:degree of map from pencil to moduli space}
Let $g\colon B\to \Atlass{1,1}$ be the composition $B\hookrightarrow \WS_{8}(\T\PP^2,3)\xrightarrow{\src} \Atlass{1,17}\to \Atlass{1,1}$. Then we have 
\begin{equation}
g_*\alpha_B=12\cdot 432\cdot [\Atlass{1,1}]
\end{equation}
\end{proposition}

\begin{proof}
Since every point in $\Atlass{1,1}$ except the initial point has a neighborhood isomorphic to an open subset of $\T$, every tropical one-cycle on $\Atlass{1,1}$ is a multiple of $[\Atlass{1,1}]$. Therefore, we have to prove that $g$ has degree $12\cdot 432$, where the degree is defined with respect to $\alpha_B$. By definition of $\alpha_B$, the degree of $g$ is the same as the degree with respect to $\alpha_U$ of the map
\begin{equation}
U\xrightarrow{g\times \prod_{i=1}^8\ev_i} \Atlass{1,1}\times \prod_{i=1}^8 \T\PP^2 \ ,
\end{equation}
where $U$ is the neighborhood of $B$ from Proposition \ref{prop:cycle on admissible cover space}. The weights of $\alpha_U$ are defined exactly as  in \cite[Definition 3.5]{KerberMarkwig}, except that pairs of strata appearing in part (c) of that definition are replaced by one stratum, parameterizing curves with a loop contracting at the mid-point of an edge of weight two, with twice the multiplicity. The degree of $g$ is $12\cdot 432$ from \cite[Theorem 6.3]{KerberMarkwig} and the fact that there are $12$ tropical rational curves passing through eight general points in $\T\PP^2$ when counted with multiplicity \cite{FloorDiagrams}. The factor of $432$ does not appear in \cite{KerberMarkwig}: it arises from  orienting the cycle and marking all unbounded edges.
\end{proof}

%\begin{remark}
%Let us unravel the preceding proof to expose the underlying geometry. To determine the degree of $g$, one needs to determine the preimage of a given curve in $\Atlass{1,1}$. One convenient choices is to pick a curve with a very large loop. For all maps in $B$ where the loop of the base curve is immersed into $\T\PP^2$, the length of the loop can be bounded from above in terms of the distances between the eight fixed points. Therefore, if a map in $B$ is mapped by $g$ to a curve with very large loop, some edge of the loop has to be contracted. This can only happen if the image curve is rational. Counted with multiplicity, there are $12$ rational curves in the linear system determined by the eight points \cite{FloorDiagrams}. Each of these curves has $432$ preimages in $B$. What is left is to show that the multiplicity of the rational image curves coincides with local degrees of $g$ at the preimages. In each of the $12$ cases, this amounts to computing the determinant of a block-triangular matrix with $2\times 2$ blocks, which has been done in much greater generality in \cite{KerberMarkwig}. \andreas{we might also want to replace the proof above with a slightly modified version of this remark}
%\end{remark}

We can interpret Proposition \ref{prop:degree of psi1 on pencil} and Proposition \ref{prop:degree of map from pencil to moduli space} combined as the tropical version of the equality
\begin{equation}
\psi_1=\frac 1{24}[\mathrm{pt}]
\end{equation}
that holds on the algebraic moduli space of  genus one stable curves $\Mgnbar{1,1}$.
The lack of a fundamental class for $\Mgnbar{1,1}$ does not allow to  integrate  the class of a point. However, we  may view $\Atlass{1,1}$ as a two-to-one cover of $\Mgnbar{1,1}$. Then the map $f\colon B\to \Mgnbar{1,1}$ defined by the family $\mc C_1\to B$ has degree $24\cdot 432$. We thus have obtained that the pull-back of $\psi_1$ via a map of degree $24\cdot 432$ is a zero-dimensional cycle of degree $432$.

%Using the projection formula (which also holds tropically \andreas{reference}), one would obtain
%\begin{equation}
%\int_{\Mgnbar{1,1}}\psi= \frac{1}{24\cdot 432}\int_B\psi\cdot \alpha_B= \frac 1{24} \ .
%\]
%The obstruction to making this rigorous are the lack of a canonical family over $\Atlass{1,1}$ \andreas{see example..?} and the fact that the one-marked genus-one curve with a genus-one vertex does not have a non-trivial automorphism. 

\subsection{Families of tropical curves from admissible covers}
\label{sec:72}

We turn our attention to families coming from tropicalizing admissible covers. Adapting the notions from \cite{Caporaso,CMRadmissible}  to allow nodes at infinity, we say that a \emph{tropical admissible cover} is a harmonic map $f\colon\Gamma_1\to \Gamma_2$ of stable curves such that $f$  does not contract any edges, preimages of marked points are marked points, preimages of nodes at infinity are nodes at infinity, and the local Riemann-Hurwitz condition 
\begin{equation}
    2g_x-2+\val(x)= d_x(2g_{f(x)}-2+\val(f(x)))
\end{equation}is satisfied for any point   $x\in \Gamma_1$. Note that $f$ might not be a linear morphism of tropical curves because the affine structure of $\Gamma_1$ at a point $x\in \Gamma_1$ of genus greater zero is trivial, whereas the affine structure of $\Gamma_2$ at $f(x)$ need not be trivial. However, $f$ is linear away from the locus of points of higher genus since there harmonicity is equivalent to linearity.  Modifying the construction from \cite[\S 3.2.2]{CMRadmissible} of a cone complex of tropical admissible covers by adding cycle-rigidifications to the source curves and taking the extended cone complex, one obtains compact \TPL{}-spaces of cycle-oriented tropical admissible covers.

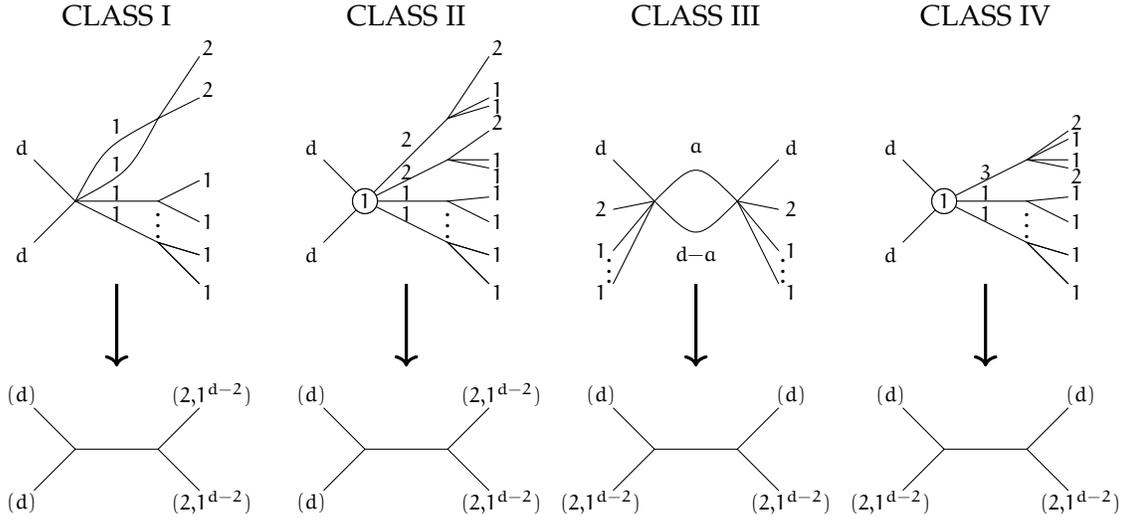
\begin{figure}

\begin{tikzpicture}
\begin{scope}[scale=0.55, shift={(0,0)}]
\draw (-1,-1)--(0,0)--(-1,1);
\draw (0,0) --(2,0);
\draw (3,-1)--(2,0)--(3,1);
\node at (-1.3, -1.3) {$\scriptstyle (d)$};
\node at (-1.3, +1.3) {$\scriptstyle (d)$};
\node at (3.3, -1.3) {$\scriptstyle (2,1^{d-2})$};
\node at (3.3, 1.3) {$\scriptstyle (2,1^{d-2})$};
\end{scope}

\begin{scope}[scale=0.55, shift={(0,6)}]

\node at (1,4.5) {CLASS I};

\draw (-1,-1)--(0,0)--(-1,1);
\draw (0,0) --(2,0);
\draw (0,0) --(2,-1);

\draw (0,0) .. controls (.7,1.3) ..(2,2);
\draw (0,0) .. controls (1.3,.7) ..(2,2);

\draw (3,3.5) -- (2,2) -- (3 , 2.5);
\draw (3,.5) -- (2,0) -- (3 , -.5);
\draw (3,-1.3) -- (2,-1) -- (3 , -2);
\draw (3,-1.3) -- (2,-1) -- (3 , -2);

\node at (-1.3, -1.3) {$\scriptstyle d$};
\node at (-1.3, +1.3) {$\scriptstyle d$};

\node at (1, 1.8) {$\scriptstyle 1$};
\node at (1, .9) {$\scriptstyle 1$};
\node at (1, .2) {$\scriptstyle 1$};
\node at (1, -.3) {$\scriptstyle 1$};
\node at (2, -.4) {$\vdots$};
\node at (3.2, 2.7) {$\scriptstyle 2$};
\node at (3.2, 3.7) {$\scriptstyle 2$};
\node at (3.2, -1.3) {$\scriptstyle 1$};
\node at (3.2, -2.2) {$\scriptstyle 1$};
\node at (3.2, -.5) {$\scriptstyle 1$};
\node at (3.2, .5) {$\scriptstyle 1$};

\draw[very thick, ->] (1,-2)--(1,-4);
\end{scope}

\begin{scope}[scale=0.55, shift={(7,0)}]
\draw (-1,-1)--(0,0)--(-1,1);
\draw (0,0) --(2,0);
\draw (3,-1)--(2,0)--(3,1);
\node at (-1.3, -1.3) {$\scriptstyle (d)$};
\node at (-1.3, +1.3) {$\scriptstyle (d)$};
\node at (3.3, -1.3) {$\scriptstyle (2,1^{d-2})$};
\node at (3.3, 1.3) {$\scriptstyle (2,1^{d-2})$};
\end{scope}

\begin{scope}[scale=0.55, shift={(7,6)}]
\node at (1,4.5) {CLASS II};

\filldraw[ fill= white](0,0) circle (0.3);
\draw (-1,-1)--(0,0)--(-1,1);
\draw (0,0) --(2,0);
\draw (0,0) --(2,-1);

\draw (0,0) -- (2,2);
\draw (0,0) -- (2,1);

\draw (3,3.5) -- (2,2) -- (3 , 2.5);
\draw (2,2) -- (3 , 2.3);

\draw (3,.8) -- (2,1) -- (3 , 1.7);
\draw (2,1) -- (3 , 1);

\draw (3,.1) -- (2,0) -- (3 , -.5);
\draw (3,-1.3) -- (2,-1) -- (3 , -2);
\draw (3,-1.3) -- (2,-1) -- (3 , -2);

\filldraw[ fill= white](0,0) circle (0.3);
\node at (0, 0) {$\scriptstyle 1$};
\node at (-1.3, -1.3) {$\scriptstyle d$};
\node at (-1.3, +1.3) {$\scriptstyle d$};
\node at (1, 1.5) {$\scriptstyle 2$};
\node at (1, .7) {$\scriptstyle 2$};
\node at (1, .2) {$\scriptstyle 1$};
\node at (1, -.3) {$\scriptstyle 1$};
\node at (2, -.4) {$\vdots$};
\node at (3.2, 2.7) {$\scriptstyle 1$};
\node at (3.2, 2.3) {$\scriptstyle 1$};
\node at (3.2, 3.7) {$\scriptstyle 2$};
\node at (3.2, -1.3) {$\scriptstyle 1$};
\node at (3.2, -2.2) {$\scriptstyle 1$};
\node at (3.2, -.5) {$\scriptstyle 1$};
\node at (3.2, .2) {$\scriptstyle 1$};

\node at (3.2, 1) {$\scriptstyle 1$};
\node at (3.2, .6) {$\scriptstyle 1$};
\node at (3.2, 1.9) {$\scriptstyle 2$};

\draw[very thick, ->] (1,-2)--(1,-4);

\end{scope}

\begin{scope}[scale=0.55, shift={(14,0)}]
\draw (-1,-1)--(0,0)--(-1,1);
\draw (0,0) --(2,0);
\draw (3,-1)--(2,0)--(3,1);
\node at (-1.3, -1.3) {$\scriptstyle (2,1^{d-2})$};
\node at (-1.3, +1.3) {$\scriptstyle (d)$};
\node at (3.3, -1.3) {$\scriptstyle (2,1^{d-2})$};
\node at (3.3, 1.3) {$\scriptstyle (d)$};
\end{scope}

\begin{scope}[scale=0.55, shift={(14,6)}]

\node at (1,4.5) {CLASS III};
\draw (-1,-.2)--(0,0)--(-1,1);
\draw (3,-.2)--(2,0)--(3,1);

\draw (0,0) .. controls (1,1) .. (2,0);
\draw (0,0) .. controls (1,-1) .. (2,0);
\draw (0,0)-- (-1,-2);
\draw (0,0)-- (-1,-1.2);
\draw (2,0)-- (3,-2);
\draw (2,0)-- (3,-1.2);

\node at (-1.3, +1.3) {$\scriptstyle d$};
\node at (3.3, +1.3) {$\scriptstyle d$};
\node at (-1.3, -.2) {$\scriptstyle 2$};
\node at (3.3, -.2) {$\scriptstyle 2$};
\node at (-1.3, -1.2) {$\scriptstyle 1$};
\node at (3.3, -1.2) {$\scriptstyle 1$};
\node at (1, 1.3) {$\scriptstyle a$};
\node at (1, -1.3) {$\scriptstyle d-a$};

\node at (-1.3, -2.2) {$\scriptstyle 1$};
\node at (3.3, -2.2) {$\scriptstyle 1$};

\node at (-1.1,-1.5) {$\vdots$};
\node at (3.1,-1.5) {$\vdots$};

\draw[very thick, ->] (1,-2)--(1,-4);

\end{scope}

\begin{scope}[scale=0.55, shift={(21,6)}]
\node at (1,4.5) {CLASS IV};

\filldraw[ fill= white](0,0) circle (0.3);
\draw (-1,-1)--(0,0)--(-1,1);
\draw (0,0) --(2,0);
\draw (0,0) --(2,-1);

%\draw (0,0) -- (2,2);
\draw (0,0) -- (2,1);

%\draw (3,3.5) -- (2,2) -- (3 , 2.5);
%\draw (2,2) -- (3 , 2.3);

\draw (3,.8) -- (2,1) -- (3 , 1.7);
\draw (3,1.5) -- (2,1) -- (3 , 1);

\draw (3,.1) -- (2,0) -- (3 , -.5);
\draw (3,-1.3) -- (2,-1) -- (3 , -2);
\draw (3,-1.3) -- (2,-1) -- (3 , -2);

\filldraw[ fill= white](0,0) circle (0.3);
\node at (0, 0) {$\scriptstyle 1$};
\node at (-1.3, -1.3) {$\scriptstyle d$};
\node at (-1.3, +1.3) {$\scriptstyle d$};
%\node at (1, 1.5) {$\scriptstyle $};
\node at (1, .7) {$\scriptstyle 3$};
\node at (1, .2) {$\scriptstyle 1$};
\node at (1, -.3) {$\scriptstyle 1$};
\node at (2, -.4) {$\vdots$};
%\node at (3.2, 2.7) {$\scriptstyle 1$};
%\node at (3.2, 2.3) {$\scriptstyle 1$};
%\node at (3.2, 3.7) {$\scriptstyle 2$};
\node at (3.2, -1.3) {$\scriptstyle 1$};
\node at (3.2, -2.2) {$\scriptstyle 1$};
\node at (3.2, -.5) {$\scriptstyle 1$};
\node at (3.2, .2) {$\scriptstyle 1$};

\node at (3.2, 1.5) {$\scriptstyle 1$};
\node at (3.2, 1) {$\scriptstyle 1$};
\node at (3.2, .6) {$\scriptstyle 2$};
\node at (3.2, 1.9) {$\scriptstyle 2$};

\draw[very thick, ->] (1,-2)--(1,-4);

\end{scope}

\begin{scope}[scale=0.55, shift={(21,0)}]
\draw (-1,-1)--(0,0)--(-1,1);
\draw (0,0) --(2,0);
\draw (3,-1)--(2,0)--(3,1);
\node at (-1.3, -1.3) {$\scriptstyle (2,1^{d-2})$};
\node at (-1.3, +1.3) {$\scriptstyle (d)$};
\node at (3.3, -1.3) {$\scriptstyle (2,1^{d-2})$};
\node at (3.3, 1.3) {$\scriptstyle (d)$};
\end{scope}
\end{tikzpicture}

\caption{Classes of combinatorial types of degree-$d$ tropical admissible covers of a tropical genus-zero curve with ramification profile $((d),(d),(2,1^{d-2}),(2,1^{d-2}))$}
\label{fig:admissible combinatorial types}
\end{figure}

We  focus on a specific class of one dimensional examples.
For $d\geq 2$, denote by $B_d$ the compact \TPL{}-space of tropical admissible covers of degree-$d$  of a genus-zero tropical curve with ramification profile  $((d),(d),(2,1^{d-2}),(2,1^{d-2}))$. By the Riemann-Hurwitz formula the genus of the source curve is one.

 The combinatorial types of covers in $B_d$ can be divided into four  classes, depicted in Figure \ref{fig:admissible combinatorial types}, where one has to add markings and  cycle-rigidifications  to obtain a combinatorial type. There are $(d-2)!$ combinatorial types of class I, ${d-2 \choose 2}^2(d-4)!$ types of class II, $2(d-1)$ combinatorial types of class III, and $(d-2)^2\cdot (d-3)!$ types of class IV. So $B_d$ is isomorphic, as a \TPL{}-space, to the space obtained by gluing together $(d-2)!+{d-2 \choose 2}^2(d-4)!+2(d-1)+(d-2)^2\cdot (d-3)!$ copies of $[0,\infty]$ at their respective element $0$.

By construction, there are natural \emph{source} and \emph{branch} maps
\begin{align*}
\src &\colon B_d \to \Atlass{1,2d},\;\; [\Gamma_1\to \Gamma_2]\mapsto [\Gamma_1] \ ,\text{ and} \\
\br &\colon B_d\to \Mgnbar{0,4},\;\; [\Gamma_1\to \Gamma_2]\mapsto [\Gamma_2] \ .
\end{align*}
Define $\Aff_{B_d}$ to be the smallest affine structure on $B_d$
that makes $\src$ and $\br$ linear. %In other words, a piecewise integral linear function $\phi$ on $B_d$ is a section in $\Aff_{B_d}$ if and only if there exists a positive integer $n$, such that $n\phi$ is an integral linear combination of affine functions pulled-back via $\src$ and $\br$.

%\acolor{
%Let $\Aff'_B$ as the smallest affine structure on $B$ such that $\src$ and $\br$ become linear, and let $\Aff_B$ be the saturation of $\Aff'_B$ in $\PL_B$. In other words, a piecewise integral affine function $\phi$ on $B$ is a section in $\Aff_B$ if and only if $n\phi$ is a section of $\Aff'_B$ for some positive integer $n$.}

Let $\mc C_1= B_d\times_{\Atlass{1,2d}}\Atlass{1,2d\sqcup\{\star\}}$, and let $\mc C_2= B_d\times_{\Mgnbar{0,4}}\Mgnbar{0,4\sqcup\{\star\}}$. Then the projection $\pi_1\colon\mc C_1\to B_d$ is a family of $2d$-marked genus-one stable \TPL{}-curves and the projection $\pi_2\colon\mc C_2 \to B_d$ is a family of four-marked genus-zero stable tropical curves. By construction of $B_d$, there is a morphism $F\colon\mc C_1\to \mc C_2$ of \TPL{}-spaces over $B_d$ such that the induced morphism $F_b\colon(\mc C_1)_b\to (\mc C_2)_b$ identifies $b$ with the modular point $[(\mc C_1)_b\xrightarrow{F_b}(\mc C_2)_b]$.

{Unfortunately, $\mc C_1\to B_d$ is not a family of tropical curves because $\Atlass{1,2d\sqcup\{\star\}}\to \Atlass{1,2d}$ is not a family of tropical curves. The problem is not the exactness of the sequence of cotangent sheaves, but the lack of enough affine functions on $\mc C_1$ to make the fibers stable tropical curves. The idea to fix this is to define the pull-back of affine functions via the map $F$ to be affine functions on  $\mc C_1$. On a leg of a fiber of $\mc C_1\to B_d$ where the multiplicity of $F$ is $d$, one then obtains a function with slope $d$. One would like to divide this function by $d$ to obtain a fiberwise harmonic function on $\mc C_1$ with vertical slope one, but if the horizontal slopes of that function are not all divisible by $d$, which they usually are not, this does not define a piecewise integral linear function. To fix this, one takes a degree-$\lcm(2,d)$ cover $\widetilde B_d$ of $B_d$ whose underlying topological space is $B_d$, whose sheaf of piecewise linear functions is given by $\PL_{\widetilde B_d}=\frac 1{\lcm(2,d)}\PL_{B_d}$, and whose sheaf of affine functions $\Aff_{\widetilde B_d}$ is given by the saturation of $\Aff_{B_d}$ in $\PL_{\widetilde B_d}$. In other words, a piecewise integral linear function $\phi$ on $\widetilde B_d$ is a section in $\Aff_{\widetilde B_d}$ if and only if there exists a positive integer $n$, such that $n\phi$ is a section of $\Aff_{B_d}$. Let $\widetilde{\mc C_1}=\widetilde B_d\times_{B_d}^{\mathrm{TPL}} \mc C_1$, let $\widetilde{\mc C_2}=\widetilde B_d\times_{B_d}\mc C_2$, and let $\widetilde F\colon \widetilde{\mc C_1}\to \widetilde{\mc C_2}$ be the morphism of \TPL{}-spaces induced by $F$. Finally,  $\widetilde{\mc C_1}$ is made into a tropical space by equipping it with the smallest affine structure that is saturated in $\PL_{\widetilde{\mc C_1}}$, makes the projection $\widetilde{\mc C_1}\to \mc C_1$ linear everywhere, and makes $\widetilde F$ linear at all points of genus zero.

\begin{lemma}
$\widetilde{ \mc C_1}\to \widetilde B_d$ is a family of tropical stable curves.
\end{lemma}

\begin{proof} 
It easily follows from the fact that the cotangent sequence \eqref{equ:exact sequence} is exact on $\mc C_1$ and $\widetilde {\mc C_2}$ by Proposition \ref{prop:left-exactness of sequence on atlas}, together with the fact that $\Aff_{\widetilde B_d}$ is saturated, that the cotangent sequence is also exact on $\widetilde{\mc C_1}$ as well. Since all functions in $\Aff_{\widetilde{\mc C_1}}$ are fiberwise harmonic by construction, it remains to be shown that the fibers are tropical stable curves, that is that all harmonic functions on the fibers are affine. By the proof of Proposition \ref{prop:family of tropical curve over atlas}, this is clear everywhere except at points $x\in \mc C_2$ that are contained in a leg in its fiber that is adjacent to a genus-one vertex. So assume that $x$ is contained in such a leg. Let $m\in \{1,2,d\}$ denote the multiplicity of $\widetilde F$ at $x$, and let $\phi$ be an integral affine function in a neighborhood of $\widetilde F(x)$ with vertical slope one. Then by construction $\frac 1m \widetilde F^*(\phi)$ is an integral affine function in a neighborhood of $x$ of slope one (for this we needed to replace $B_d$ by $\widetilde B_d$). This finishes the proof.
\end{proof}
}

{
We will not refer to $B_d$ again and work exclusively with $\widetilde B_d$ in the rest of this section. To ease the notation, we replace $B_d$ with $\widetilde B_d$, $\mc C_i$ with $\widetilde{\mc C}_i$, and $\br$ ($\src$) with the composite of $\br$ ($\src$) and the natural map $\widetilde B_d\to B_d$. 
}

We define a weight on $B_d$ following \cite[Definition 2.6]{BBM}: the value at a point $f: \Gamma_1 \to \Gamma_2$ is the automorphism-weighted product of multiplicities on the edges and the local Hurwitz numbers at the vertices of the corresponding  source curve:

\begin{definition}
The weight $\alpha_{B_d}$ on $B_d$ depends on the class of the combinatorial type of the covers parameterized, as labeled in Figure \ref{fig:admissible combinatorial types}. Specifically: 
\begin{equation}
 \alpha_{B_d}(b) = 
 \left\{
 \begin{array}{cl}
  (d-1)!    &  \mbox{$F_b$ is of class I,} \\
 \frac 1 {3} {(d-4)!\cdot (d-1)(d-2)(d-3)}    &  \mbox{$F_b$ is of class II,}\\
 \gcd(a,d-a) [(d-2)!]^2 &
 \mbox{$F_b$ is of class III,}\\
 \frac 1 {3} {(d-3)!\cdot (d-1)(d-2)}    &  \mbox{$F_b$ is of class IV.}
 \end{array}
 \right.
\end{equation}
\end{definition}

%\acolor{\begin{definition}
%We define a weight $\alpha_B$ on $B$ by assigning to a point in $B$ on a ray corresponding to a combinatorial type weight $(d-1)!$ if the type is of class a), weight $\frac{1}{3}(d-4)!(d-1)(d-2)(d+3)$ if the type is of class b), and weight $\gcd(a,d-a) [(d-2)!]^2$, if the type is of class c).
%\end{definition}}
We assume that the first marked point of the family $\mc C_1\to B_d$ is one of the two points of ramification order $d$. Let $\tilde\pi: B_d \to \Atlass{1,1}$ the the composition of  the source map $\src$ with the forgetful morphism forgetting all but the first marked point.
\begin{proposition}
\label{prop:cycle on admissible cover space}
\begin{enumerate}[label=\alph*)]
\item[]
\item The weight $\alpha_{B_d}$ defines a tropical one-cycle on ${B_d}$.
\item We have 
\begin{equation}
\br_* \alpha_{B_d}= \frac{\lcm(2,d)[(d-2)!]^2(d-1)d(d+1)}{6} [\Mgnbar{0,4}] \ .
\end{equation}
\item We have 
\begin{equation}
\tilde\pi_*\alpha_{B_d}= 2\lcm(2,d)[(d-2)!]^2 (d-1)(d+1)[\Atlass{1,1}] \ .
\end{equation}
\end{enumerate}
\end{proposition}

\begin{proof}
For  a) we need to check that the weights are balanced at the vertex of $B_d$. For germs of affine functions obtained by pull-back via $\br$, balancing  is equivalent to  the fact that $\br$ has a well-defined degree, hence this follows from b). There are no-nontrivial germs of affine functions on the vertex of $\Atlass{1,2d}$, hence balancing holds vacuously for functions pulled-back via $\src$.

Part b) is  \cite[Theorem 2.11]{BBM}, which in particular shows that the degree of the branch morphism is well-defined.

For c), we compute the contribution to $\tilde\pi_\ast \alpha_{B_d}$ by rays of $B_d$ of combinatorial types belonging to each of the three classes identified in Figure \ref{fig:admissible combinatorial types}.
For a ray parameterizing curves of class $I$, $\tilde\pi$ is an integral map of slope $2\lcm(2,d)$. There are $(d-2)!$ such rays, and to each of them $\alpha_{B_d}$ assigns weight $(d-1)!$.

Rays of class $II$ and $IV$ are contracted to the vertex of $\Atlass{1,1}$, and hence do not contribute.

The function $\tilde\pi$ restricted to a ray parameterizing curves of class $III$ has slope $\lcm(2,d)lcm(a, d-a)\left(\frac{1}{a}+ \frac{1}{d-a}\right)$. The weight on the ray is  $\gcd(a,d-a) [(d-2)!]^2$, hence the contribution of the ray is independent of $a$, equal to $\lcm(2,d)d[(d-2)!]^2$. There are $2(d-1)$ rays of class III. Putting all contributions together we see that the degree of $\tilde \pi$ equals
\begin{multline}
% \deg(\tilde\pi_*\alpha_{B_d}) = 
\overbrace{(d-2)!\cdot 2\lcm(2,d) (d-1)!}^{\text{contribution from class I}}+ \overbrace{2(d-1)\cdot \lcm(2,d)d[(d-2)!]^2}^{\text{contribution from class III}} =\\
=2\lcm(2,d)[(d-2)!]^2 (d-1)(d+1)
\end{multline}
\end{proof}

\begin{proposition}
\label{prop:degree of psi on admissible covers space}
Let $\pi: B_d \to \Mgnbar{1,1}$ the the composition of  the source map $\src$ with the forgetful morphism forgetting all but the first marked point.
We have 
\begin{equation} 
\int_{B_d}\pi^\ast\psi_1 \cdot \alpha_{B_d}= \frac{\lcm(2,d)[(d-2)!]^2(d-1)(d+1)}{6} \ .
\end{equation}
\end{proposition}

\begin{proof}
The map $F\colon \mc C_1\to \mc C_2$ has degree $d$ on the one-marked leg in every fiber, so if $F$ maps the one-marked leg in $\mc C_1$ to the $\diamond$-marked leg in $\mc C_2$, then we have 
\begin{equation}
d\cdot \pi^\ast\psi_1= \br^*\psi_\diamond \ .
\end{equation}
Using the tropical projection formula \cite[Proposition 7.7]{AllermannRau} and Proposition \ref{prop:cycle on admissible cover space} b), we thus obtain
\begin{multline*}
\int_{B_d}\pi^\ast\psi_1\cdot\alpha_{B_d}=
\frac 1d \int_{B_d}\br^*\psi_\diamond\cdot \alpha_{B_d}=  
\frac 1d \int_{\Mgnbar{0,4}}\br_*(\br^*\psi_\diamond\cdot \alpha_{B_d})=
\frac 1d\int_{\Mgnbar{0,4}}\psi_\diamond \cdot \br_*\alpha_{B_d} =\\
=\frac{[\lcm(2,d)(d-2)]^2(d-1)(d+1)}6\int_{\Mgnbar{0,4}} \psi_\diamond \cdot [\Mgnbar{0,4}]=
\frac{\lcm(2,d)[(d-2)]^2(d-1)(d+1)}6 \ ,
\end{multline*}
where the last equality follows from Corollary \ref{cor:psi class in genus 0}.
\end{proof}

Proposition \ref{prop:degree of psi on admissible covers space} together with Proposition \ref{prop:cycle on admissible cover space} c) show that the ratio of the evaluation of $\pi^\ast\psi_1\cdot \alpha_{B_d}$ with the degree of $src$ is $1/12$. Together with the fact that stack $\Mgnbar{1,1}$ has a generic $B\mu_2$ stabilizer, we have that for the admissible covers families studied $\psi$ behaves as the pull-back of a class of degree $1/24$.
%; in our theory it lacks a fundamental cycle, which prevents us from making the statement that
%$
%\psi_1=\frac 1{24}[\mathrm{pt}];
%$
%however such behavior holds numerically .
%on families that are obtained by tropicalizing families of algebraic genus-one stable curves.

%\acolor{
%The combination of Proposition \ref{prop:degree of psi on admissible covers space} and Proposition \ref{prop:cycle on admissible cover space} c) is another indicator for the fact that while 
%\[
%\psi_1=\frac 1{24}[\mathrm{pt}]
%\]
%does not hold tropically, it does hold at least numerically on families that are obtained by tropicalizing families of algebraic genus-one stable curves.} \renzo{Think of a more momentuous sentence to conclude the paper with ;)}

%\andreas{I performed an appendectomy. To view the severed appendix, uncomment the next line}
\appendix
\addtocontents{toc}{\protect\setcounter{tocdepth}{1}}
\renewcommand*{\thetheorem}{\Alph{section}.\arabic{theorem}}

\newcommand{\inputTikZ}[2]{%  
     \resizebox{#1}{!}{\input{#2}}  
}
%\pagebreak
\section{Cycle of degree $3$ tropical stable maps} \label{app:main}
In this Appendix we explicitly describe the one dimensional cycle 
$
B=\bigcap_{i=1}^8 \ev^{-1}_i\{p_i\} \subset \WS_{8}(\mathbb{TP}^2,3)
$ studied in Section \ref{sec:71}.
\begin{figure}[h] 
     \resizebox{1 \textwidth}{!}{\input{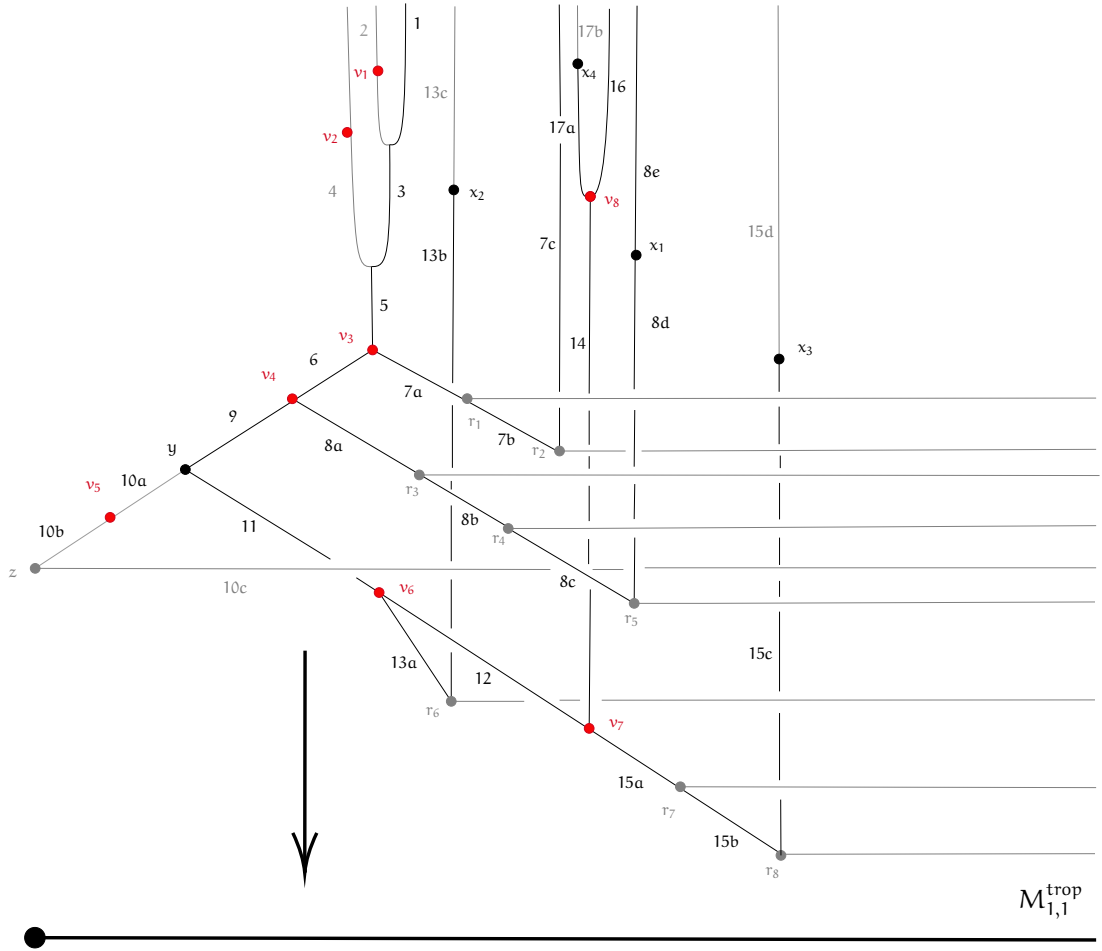}}  

\caption{ Genus one, degree three maps through $8$ general points in the plane. The part of $B$ which is drawn in black coincides with a tropical pencil of cubics through $8$ points. In gray are the tropical stable maps with un underlying constant image curve.}
\label{fig:fam}
\end{figure}
To be precise, Figure \ref{fig:fam} depicts a quotient of the cycle $B$  by the action of the group $S_3^3$ permuting the labels of the infinite ends. Equivalently, one may think of parameterizing tropical stable maps where the infinite ends are unlabeled. The vertical map to $M_{1,1}^{trop}$ is given by the tropical $j$-invariant.

%\pagebreak
The different combinatorial types of curves appearing, together with the corresponding lattice paths and floor diagrams \cite{MikhalkinEnumerative,FominMikhalkin}, are depicted in Figure \ref{fig:combty} (spanning multiple pages).

\pagebreak
The vertices of $B$ contain  eigth rational nodal curves, as depicted in Figure \ref{fig:ratnod}.
Among the vertices of $B$ are when each of the marked points sits at a four-valent vertex  (or equivalently the mark is incident to a three-valent vertex of the image cubic).
These are the points of $B$ where the various $\psi$ classes are supported. These are depicted in Figure \ref{fig:vertmar}. The remaining vertices in $B$ are collected in Figure \ref{fig:remver}.
\pagebreak
%\thispagestyle{empty}

%\begin{figure}[h]
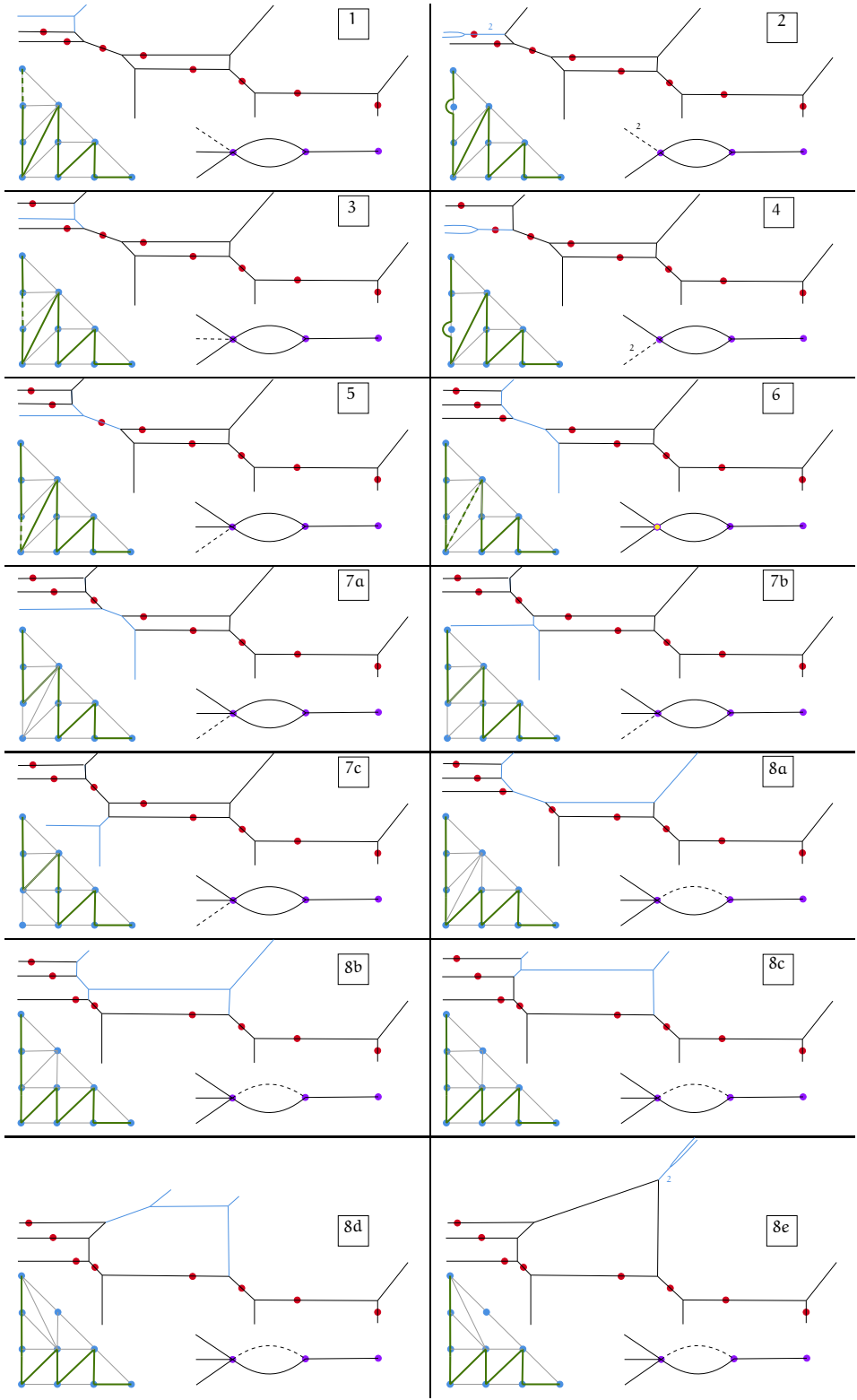
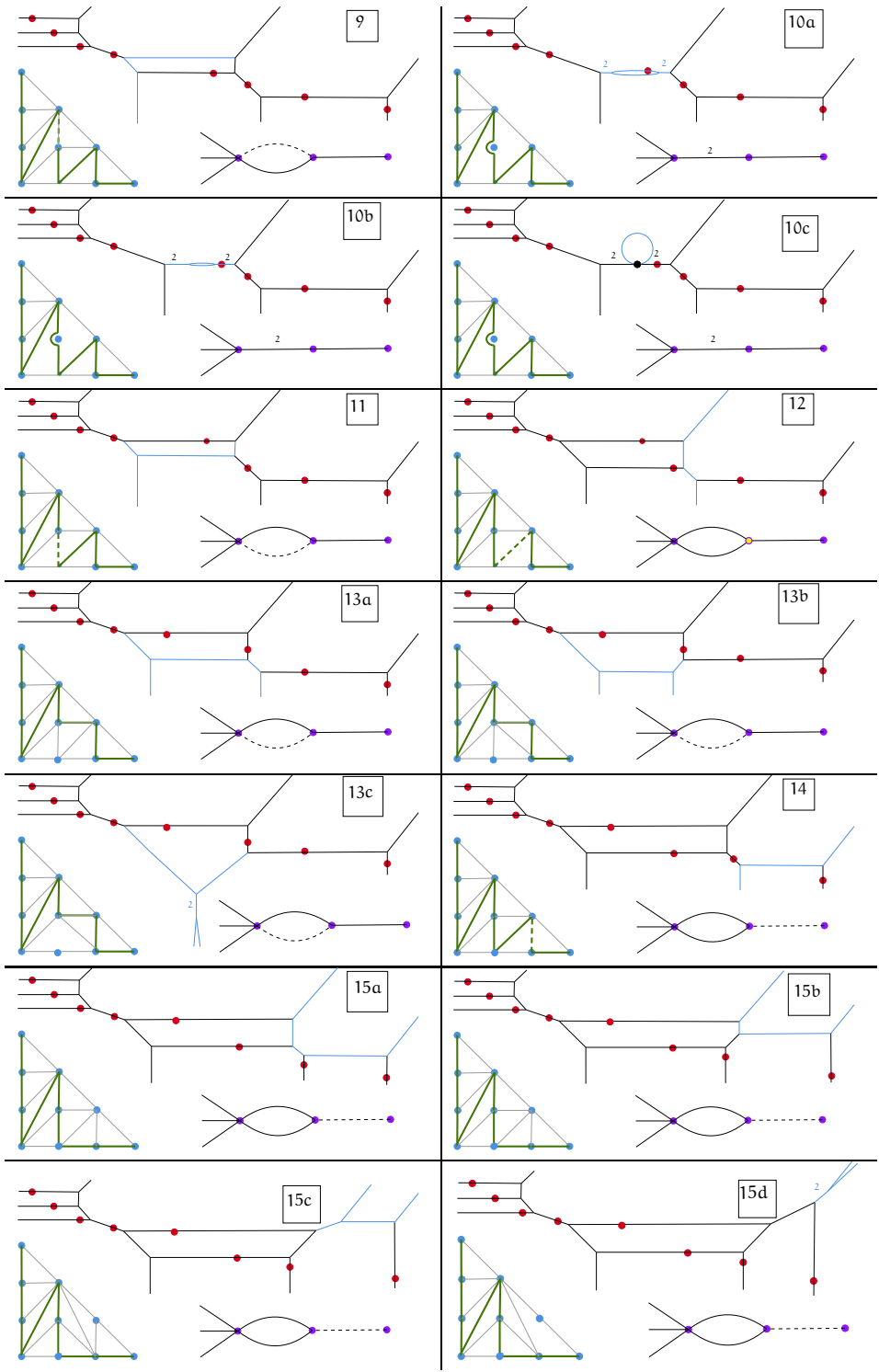
\begin{tabular}{c|c}
     \resizebox{.4 \textwidth}{!}{\input{1.tikz}}  
 & %  
     \resizebox{.4 \textwidth}{!}{\input{2.tikz}}  
 \\ \hline
     \resizebox{.4 \textwidth}{!}{\input{3.tikz}}  
 & %  
     \resizebox{.4 \textwidth}{!}{\input{4.tikz}}  
 \\ \hline
     \resizebox{.4 \textwidth}{!}{\input{5.tikz}}  
 & %  
     \resizebox{.4 \textwidth}{!}{\input{6.tikz}}  
 \\ \hline
     \resizebox{.4 \textwidth}{!}{\input{7a.tikz}}  
 & %  
     \resizebox{.4 \textwidth}{!}{\input{7b.tikz}}  
 \\ \hline
     \resizebox{.4 \textwidth}{!}{\input{7c.tikz}}  
 & %  
     \resizebox{.4 \textwidth}{!}{\input{8a.tikz}}  
 \\  \hline
     \resizebox{.4 \textwidth}{!}{\input{8b.tikz}}  
 & %  
     \resizebox{.4 \textwidth}{!}{\input{8c.tikz}}  
 \\ \hline
     \resizebox{.4 \textwidth}{!}{\input{8d.tikz}}  
 & %  
     \resizebox{.4 \textwidth}{!}{\input{8e.tikz}}  
\\ 
\end{tabular}
%\end{figure}

\pagebreak
%\thispagestyle{empty}

%\begin{figure}[h]
\begin{tabular}{c|c}
     \resizebox{.4 \textwidth}{!}{\input{9.tikz}}  
 & %  
     \resizebox{.4 \textwidth}{!}{\input{10a.tikz}}  
 \\ \hline
     \resizebox{.4 \textwidth}{!}{\input{10b.tikz}}  
 & %  
     \resizebox{.4 \textwidth}{!}{\input{10c.tikz}}  
 \\ \hline
     \resizebox{.4 \textwidth}{!}{\input{11.tikz}}  
 & %  
     \resizebox{.4 \textwidth}{!}{\input{12.tikz}}  
 \\ \hline
     \resizebox{.4 \textwidth}{!}{\input{13a.tikz}}  
 & %  
     \resizebox{.4 \textwidth}{!}{\input{13b.tikz}}  
 \\ \hline
     \resizebox{.4 \textwidth}{!}{\input{13c.tikz}}  
 & %  
     \resizebox{.4 \textwidth}{!}{\input{14.tikz}}  
 \\  \hline
     \resizebox{.4 \textwidth}{!}{\input{15a.tikz}}  
 & %  
     \resizebox{.4 \textwidth}{!}{\input{15b.tikz}}  
 \\ \hline
     \resizebox{.4 \textwidth}{!}{\input{15c.tikz}}  
 & %  
     \resizebox{.4 \textwidth}{!}{\input{15d.tikz}}  
\\ 
\end{tabular}
%\end{figure}

\pagebreak

\begin{figure}[h]
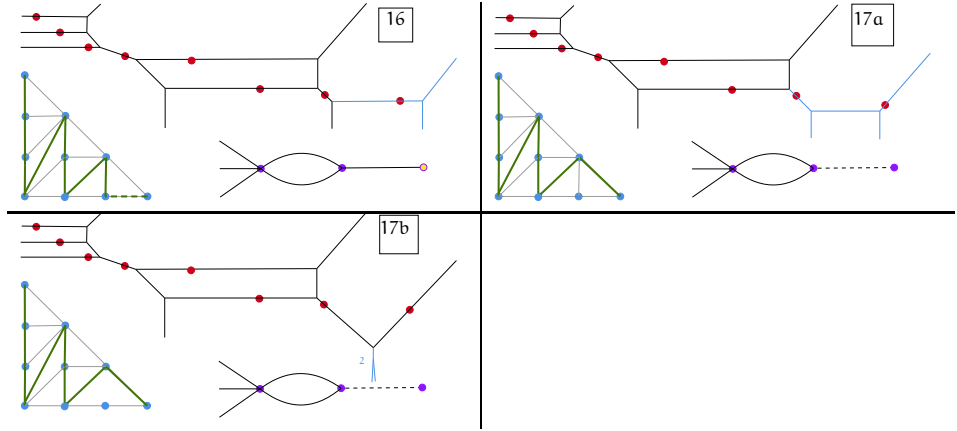

\begin{tabular}{c|c}
     \resizebox{.4 \textwidth}{!}{\input{16.tikz}}  
 & %  
     \resizebox{.4 \textwidth}{!}{\input{17a.tikz}}  
 \\ \hline
     \resizebox{.4 \textwidth}{!}{\input{17b.tikz}}  
 &  \\ 
\end{tabular}
\caption{The combinatorial types of tropical stable maps associated to each segment of the cycle $B$.}
\label{fig:combty}
\end{figure}
\vspace{1cm}

\begin{figure}[h]
\begin{tabular}{c|c}
     \resizebox{.4 \textwidth}{!}{%\tikzset{every picture/.style={line width=0.75pt}} %set default line width to 0.75pt        

\begin{tikzpicture}[x=0.75pt,y=0.75pt,yscale=-1,xscale=1]

%Shape: Circle [id:dp049478145904912285] 
\draw  [color={rgb, 255:red, 208; green, 2; blue, 27 }  ,draw opacity=1 ][fill={rgb, 255:red, 208; green, 2; blue, 27 }  ,fill opacity=1 ] (61,32) .. controls (61,29.24) and (63.24,27) .. (66,27) .. controls (68.76,27) and (71,29.24) .. (71,32) .. controls (71,34.76) and (68.76,37) .. (66,37) .. controls (63.24,37) and (61,34.76) .. (61,32) -- cycle ;
%Shape: Circle [id:dp597865605180764] 
\draw  [color={rgb, 255:red, 208; green, 2; blue, 27 }  ,draw opacity=1 ][fill={rgb, 255:red, 208; green, 2; blue, 27 }  ,fill opacity=1 ] (96,55) .. controls (96,52.24) and (98.24,50) .. (101,50) .. controls (103.76,50) and (106,52.24) .. (106,55) .. controls (106,57.76) and (103.76,60) .. (101,60) .. controls (98.24,60) and (96,57.76) .. (96,55) -- cycle ;
%Shape: Circle [id:dp03321279691471113] 
\draw  [color={rgb, 255:red, 208; green, 2; blue, 27 }  ,draw opacity=1 ][fill={rgb, 255:red, 208; green, 2; blue, 27 }  ,fill opacity=1 ] (162,69) .. controls (162,66.24) and (164.24,64) .. (167,64) .. controls (169.76,64) and (172,66.24) .. (172,69) .. controls (172,71.76) and (169.76,74) .. (167,74) .. controls (164.24,74) and (162,71.76) .. (162,69) -- cycle ;
%Straight Lines [id:da15855389453970536] 
\draw [color={rgb, 255:red, 0; green, 0; blue, 0 }  ,draw opacity=1 ]   (44,84) -- (181,83) ;
%Shape: Circle [id:dp7534753985751717] 
\draw  [color={rgb, 255:red, 208; green, 2; blue, 27 }  ,draw opacity=1 ][fill={rgb, 255:red, 208; green, 2; blue, 27 }  ,fill opacity=1 ] (212,84) .. controls (212,81.24) and (214.24,79) .. (217,79) .. controls (219.76,79) and (222,81.24) .. (222,84) .. controls (222,86.76) and (219.76,89) .. (217,89) .. controls (214.24,89) and (212,86.76) .. (212,84) -- cycle ;
%Shape: Circle [id:dp5270532448174468] 
\draw  [color={rgb, 255:red, 208; green, 2; blue, 27 }  ,draw opacity=1 ][fill={rgb, 255:red, 208; green, 2; blue, 27 }  ,fill opacity=1 ] (294,107) .. controls (294,104.24) and (296.24,102) .. (299,102) .. controls (301.76,102) and (304,104.24) .. (304,107) .. controls (304,109.76) and (301.76,112) .. (299,112) .. controls (296.24,112) and (294,109.76) .. (294,107) -- cycle ;
%Shape: Circle [id:dp7779686155492531] 
\draw  [color={rgb, 255:red, 208; green, 2; blue, 27 }  ,draw opacity=1 ][fill={rgb, 255:red, 208; green, 2; blue, 27 }  ,fill opacity=1 ] (375,126.5) .. controls (375,123.74) and (377.24,121.5) .. (380,121.5) .. controls (382.76,121.5) and (385,123.74) .. (385,126.5) .. controls (385,129.26) and (382.76,131.5) .. (380,131.5) .. controls (377.24,131.5) and (375,129.26) .. (375,126.5) -- cycle ;
%Shape: Circle [id:dp23212820006386536] 
\draw  [color={rgb, 255:red, 208; green, 2; blue, 27 }  ,draw opacity=1 ][fill={rgb, 255:red, 208; green, 2; blue, 27 }  ,fill opacity=1 ] (466.5,146) .. controls (466.5,143.24) and (468.74,141) .. (471.5,141) .. controls (474.26,141) and (476.5,143.24) .. (476.5,146) .. controls (476.5,148.76) and (474.26,151) .. (471.5,151) .. controls (468.74,151) and (466.5,148.76) .. (466.5,146) -- cycle ;
%Shape: Circle [id:dp5439300296757643] 
\draw  [color={rgb, 255:red, 208; green, 2; blue, 27 }  ,draw opacity=1 ][fill={rgb, 255:red, 208; green, 2; blue, 27 }  ,fill opacity=1 ] (599,166) .. controls (599,163.24) and (601.24,161) .. (604,161) .. controls (606.76,161) and (609,163.24) .. (609,166) .. controls (609,168.76) and (606.76,171) .. (604,171) .. controls (601.24,171) and (599,168.76) .. (599,166) -- cycle ;
%Straight Lines [id:da9939861155158032] 
\draw [color={rgb, 255:red, 0; green, 0; blue, 0 }  ,draw opacity=1 ]   (152.5,32) -- (44,33) ;
%Straight Lines [id:da6460141075147547] 
\draw [color={rgb, 255:red, 0; green, 0; blue, 0 }  ,draw opacity=1 ]   (152.5,32) -- (172.5,13) ;
%Straight Lines [id:da796824295382208] 
\draw    (181,83) -- (360,83) ;
%Straight Lines [id:da9753553652203049] 
\draw [color={rgb, 255:red, 0; green, 0; blue, 0 }  ,draw opacity=1 ]   (181,83) -- (203,106) ;
%Straight Lines [id:da48627689074126645] 
\draw    (203,106) -- (359,107) ;
%Straight Lines [id:da49133266575001033] 
\draw    (360,83) -- (359,107) ;
%Straight Lines [id:da51467780889709] 
\draw [color={rgb, 255:red, 0; green, 0; blue, 0 }  ,draw opacity=1 ]   (203,106) -- (203,188) ;
%Straight Lines [id:da7007627180528562] 
\draw    (360,83) -- (432,1) ;
%Straight Lines [id:da7832199179696879] 
\draw    (359,107) -- (401,146) ;
%Straight Lines [id:da44753667002377306] 
\draw    (401,146) -- (604,147) ;
%Straight Lines [id:da6859330855524519] 
\draw    (604,147) -- (604,184) ;
%Straight Lines [id:da6255332306196266] 
\draw    (604,147) -- (654,84) ;
%Straight Lines [id:da6529064472478432] 
\draw    (401,146) -- (401,186) ;
%Shape: Square [id:dp7171542865243234] 
\draw   (539,7) -- (589,7) -- (589,57) -- (539,57) -- cycle ;
%Straight Lines [id:da16444256483632191] 
\draw [color={rgb, 255:red, 0; green, 0; blue, 0 }  ,draw opacity=1 ][fill={rgb, 255:red, 74; green, 144; blue, 226 }  ,fill opacity=1 ]   (152.5,32) -- (153.5,55) ;
%Straight Lines [id:da666703974482969] 
\draw [color={rgb, 255:red, 0; green, 0; blue, 0 }  ,draw opacity=1 ]   (153.5,55) -- (180.5,83) ;
%Straight Lines [id:da0840655116527731] 
\draw    (153.5,55) -- (40.5,55) ;

% Text Node
\draw (541.08,10.35) node [anchor=north west][inner sep=0.75pt]  [font=\huge,rotate=-358.58]  {$r_{1}$};

\end{tikzpicture}}  
 & %  
     \resizebox{.4 \textwidth}{!}{%\tikzset{every picture/.style={line width=0.75pt}} %set default line width to 0.75pt        

\begin{tikzpicture}[x=0.75pt,y=0.75pt,yscale=-1,xscale=1]

%Shape: Circle [id:dp049478145904912285] 
\draw  [color={rgb, 255:red, 208; green, 2; blue, 27 }  ,draw opacity=1 ][fill={rgb, 255:red, 208; green, 2; blue, 27 }  ,fill opacity=1 ] (29,20) .. controls (29,17.24) and (31.24,15) .. (34,15) .. controls (36.76,15) and (39,17.24) .. (39,20) .. controls (39,22.76) and (36.76,25) .. (34,25) .. controls (31.24,25) and (29,22.76) .. (29,20) -- cycle ;
%Shape: Circle [id:dp597865605180764] 
\draw  [color={rgb, 255:red, 208; green, 2; blue, 27 }  ,draw opacity=1 ][fill={rgb, 255:red, 208; green, 2; blue, 27 }  ,fill opacity=1 ] (64,43) .. controls (64,40.24) and (66.24,38) .. (69,38) .. controls (71.76,38) and (74,40.24) .. (74,43) .. controls (74,45.76) and (71.76,48) .. (69,48) .. controls (66.24,48) and (64,45.76) .. (64,43) -- cycle ;
%Shape: Circle [id:dp03321279691471113] 
\draw  [color={rgb, 255:red, 208; green, 2; blue, 27 }  ,draw opacity=1 ][fill={rgb, 255:red, 208; green, 2; blue, 27 }  ,fill opacity=1 ] (130,57) .. controls (130,54.24) and (132.24,52) .. (135,52) .. controls (137.76,52) and (140,54.24) .. (140,57) .. controls (140,59.76) and (137.76,62) .. (135,62) .. controls (132.24,62) and (130,59.76) .. (130,57) -- cycle ;
%Straight Lines [id:da15855389453970536] 
\draw [color={rgb, 255:red, 0; green, 0; blue, 0 }  ,draw opacity=1 ]   (22.5,107) -- (92.5,106.49) -- (159.5,106) ;
%Shape: Circle [id:dp7534753985751717] 
\draw  [color={rgb, 255:red, 208; green, 2; blue, 27 }  ,draw opacity=1 ][fill={rgb, 255:red, 208; green, 2; blue, 27 }  ,fill opacity=1 ] (212,84) .. controls (212,81.24) and (214.24,79) .. (217,79) .. controls (219.76,79) and (222,81.24) .. (222,84) .. controls (222,86.76) and (219.76,89) .. (217,89) .. controls (214.24,89) and (212,86.76) .. (212,84) -- cycle ;
%Shape: Circle [id:dp5270532448174468] 
\draw  [color={rgb, 255:red, 208; green, 2; blue, 27 }  ,draw opacity=1 ][fill={rgb, 255:red, 208; green, 2; blue, 27 }  ,fill opacity=1 ] (294,107) .. controls (294,104.24) and (296.24,102) .. (299,102) .. controls (301.76,102) and (304,104.24) .. (304,107) .. controls (304,109.76) and (301.76,112) .. (299,112) .. controls (296.24,112) and (294,109.76) .. (294,107) -- cycle ;
%Shape: Circle [id:dp7779686155492531] 
\draw  [color={rgb, 255:red, 208; green, 2; blue, 27 }  ,draw opacity=1 ][fill={rgb, 255:red, 208; green, 2; blue, 27 }  ,fill opacity=1 ] (375,126.5) .. controls (375,123.74) and (377.24,121.5) .. (380,121.5) .. controls (382.76,121.5) and (385,123.74) .. (385,126.5) .. controls (385,129.26) and (382.76,131.5) .. (380,131.5) .. controls (377.24,131.5) and (375,129.26) .. (375,126.5) -- cycle ;
%Shape: Circle [id:dp23212820006386536] 
\draw  [color={rgb, 255:red, 208; green, 2; blue, 27 }  ,draw opacity=1 ][fill={rgb, 255:red, 208; green, 2; blue, 27 }  ,fill opacity=1 ] (466.5,146) .. controls (466.5,143.24) and (468.74,141) .. (471.5,141) .. controls (474.26,141) and (476.5,143.24) .. (476.5,146) .. controls (476.5,148.76) and (474.26,151) .. (471.5,151) .. controls (468.74,151) and (466.5,148.76) .. (466.5,146) -- cycle ;
%Shape: Circle [id:dp5439300296757643] 
\draw  [color={rgb, 255:red, 208; green, 2; blue, 27 }  ,draw opacity=1 ][fill={rgb, 255:red, 208; green, 2; blue, 27 }  ,fill opacity=1 ] (599,166) .. controls (599,163.24) and (601.24,161) .. (604,161) .. controls (606.76,161) and (609,163.24) .. (609,166) .. controls (609,168.76) and (606.76,171) .. (604,171) .. controls (601.24,171) and (599,168.76) .. (599,166) -- cycle ;
%Straight Lines [id:da9939861155158032] 
\draw [color={rgb, 255:red, 0; green, 0; blue, 0 }  ,draw opacity=1 ]   (118,20) -- (9.5,21) ;
%Straight Lines [id:da6460141075147547] 
\draw [color={rgb, 255:red, 0; green, 0; blue, 0 }  ,draw opacity=1 ]   (120.5,20) -- (140.5,1) ;
%Straight Lines [id:da796824295382208] 
\draw    (159.5,83) -- (360,83) ;
%Straight Lines [id:da48627689074126645] 
\draw [color={rgb, 255:red, 0; green, 0; blue, 0 }  ,draw opacity=1 ]   (159.5,106) -- (359,107) ;
%Straight Lines [id:da49133266575001033] 
\draw    (360,83) -- (359,107) ;
%Straight Lines [id:da51467780889709] 
\draw [color={rgb, 255:red, 0; green, 0; blue, 0 }  ,draw opacity=1 ]   (159.5,106) -- (159.5,146) -- (159.5,188) ;
%Straight Lines [id:da7007627180528562] 
\draw    (360,83) -- (432,1) ;
%Straight Lines [id:da7832199179696879] 
\draw    (359,107) -- (401,146) ;
%Straight Lines [id:da44753667002377306] 
\draw    (401,146) -- (604,147) ;
%Straight Lines [id:da6859330855524519] 
\draw    (604,147) -- (604,184) ;
%Straight Lines [id:da6255332306196266] 
\draw    (604,147) -- (654,84) ;
%Straight Lines [id:da6529064472478432] 
\draw    (401,146) -- (401,186) ;
%Shape: Square [id:dp7171542865243234] 
\draw   (539,7) -- (589,7) -- (589,57) -- (539,57) -- cycle ;
%Straight Lines [id:da16444256483632191] 
\draw [color={rgb, 255:red, 0; green, 0; blue, 0 }  ,draw opacity=1 ][fill={rgb, 255:red, 74; green, 144; blue, 226 }  ,fill opacity=1 ]   (120.5,20) -- (121.5,43) ;
%Straight Lines [id:da666703974482969] 
\draw [color={rgb, 255:red, 0; green, 0; blue, 0 }  ,draw opacity=1 ]   (121.5,43) -- (159.5,83) ;
%Straight Lines [id:da0840655116527731] 
\draw    (121.5,43) -- (8.5,43) ;
%Straight Lines [id:da546700243076577] 
\draw [color={rgb, 255:red, 0; green, 0; blue, 0 }  ,draw opacity=1 ]   (159.5,83) -- (159.5,106) ;

% Text Node
\draw (549.81,8.55) node [anchor=north west][inner sep=0.75pt]  [font=\huge,rotate=-358.58]  {$r_{2}$};

\end{tikzpicture}}  
 \\ \hline
     \resizebox{.4 \textwidth}{!}{%\tikzset{every picture/.style={line width=0.75pt}} %set default line width to 0.75pt        

\begin{tikzpicture}[x=0.75pt,y=0.75pt,yscale=-1,xscale=1]

%Shape: Circle [id:dp049478145904912285] 
\draw  [color={rgb, 255:red, 208; green, 2; blue, 27 }  ,draw opacity=1 ][fill={rgb, 255:red, 208; green, 2; blue, 27 }  ,fill opacity=1 ] (23,37) .. controls (23,34.24) and (25.24,32) .. (28,32) .. controls (30.76,32) and (33,34.24) .. (33,37) .. controls (33,39.76) and (30.76,42) .. (28,42) .. controls (25.24,42) and (23,39.76) .. (23,37) -- cycle ;
%Shape: Circle [id:dp597865605180764] 
\draw  [color={rgb, 255:red, 208; green, 2; blue, 27 }  ,draw opacity=1 ][fill={rgb, 255:red, 208; green, 2; blue, 27 }  ,fill opacity=1 ] (73,62) .. controls (73,59.24) and (75.24,57) .. (78,57) .. controls (80.76,57) and (83,59.24) .. (83,62) .. controls (83,64.76) and (80.76,67) .. (78,67) .. controls (75.24,67) and (73,64.76) .. (73,62) -- cycle ;
%Shape: Circle [id:dp03321279691471113] 
\draw  [color={rgb, 255:red, 208; green, 2; blue, 27 }  ,draw opacity=1 ][fill={rgb, 255:red, 208; green, 2; blue, 27 }  ,fill opacity=1 ] (137,82) .. controls (137,79.24) and (139.24,77) .. (142,77) .. controls (144.76,77) and (147,79.24) .. (147,82) .. controls (147,84.76) and (144.76,87) .. (142,87) .. controls (139.24,87) and (137,84.76) .. (137,82) -- cycle ;
%Straight Lines [id:da15855389453970536] 
\draw [color={rgb, 255:red, 0; green, 0; blue, 0 }  ,draw opacity=1 ]   (14,83) -- (181,83) ;
%Shape: Circle [id:dp7534753985751717] 
\draw  [color={rgb, 255:red, 208; green, 2; blue, 27 }  ,draw opacity=1 ][fill={rgb, 255:red, 208; green, 2; blue, 27 }  ,fill opacity=1 ] (187,95) .. controls (187,92.24) and (189.24,90) .. (192,90) .. controls (194.76,90) and (197,92.24) .. (197,95) .. controls (197,97.76) and (194.76,100) .. (192,100) .. controls (189.24,100) and (187,97.76) .. (187,95) -- cycle ;
%Shape: Circle [id:dp5270532448174468] 
\draw  [color={rgb, 255:red, 208; green, 2; blue, 27 }  ,draw opacity=1 ][fill={rgb, 255:red, 208; green, 2; blue, 27 }  ,fill opacity=1 ] (294,107) .. controls (294,104.24) and (296.24,102) .. (299,102) .. controls (301.76,102) and (304,104.24) .. (304,107) .. controls (304,109.76) and (301.76,112) .. (299,112) .. controls (296.24,112) and (294,109.76) .. (294,107) -- cycle ;
%Shape: Circle [id:dp7779686155492531] 
\draw  [color={rgb, 255:red, 208; green, 2; blue, 27 }  ,draw opacity=1 ][fill={rgb, 255:red, 208; green, 2; blue, 27 }  ,fill opacity=1 ] (375,126.5) .. controls (375,123.74) and (377.24,121.5) .. (380,121.5) .. controls (382.76,121.5) and (385,123.74) .. (385,126.5) .. controls (385,129.26) and (382.76,131.5) .. (380,131.5) .. controls (377.24,131.5) and (375,129.26) .. (375,126.5) -- cycle ;
%Shape: Circle [id:dp23212820006386536] 
\draw  [color={rgb, 255:red, 208; green, 2; blue, 27 }  ,draw opacity=1 ][fill={rgb, 255:red, 208; green, 2; blue, 27 }  ,fill opacity=1 ] (466.5,146) .. controls (466.5,143.24) and (468.74,141) .. (471.5,141) .. controls (474.26,141) and (476.5,143.24) .. (476.5,146) .. controls (476.5,148.76) and (474.26,151) .. (471.5,151) .. controls (468.74,151) and (466.5,148.76) .. (466.5,146) -- cycle ;
%Shape: Circle [id:dp5439300296757643] 
\draw  [color={rgb, 255:red, 208; green, 2; blue, 27 }  ,draw opacity=1 ][fill={rgb, 255:red, 208; green, 2; blue, 27 }  ,fill opacity=1 ] (599,166) .. controls (599,163.24) and (601.24,161) .. (604,161) .. controls (606.76,161) and (609,163.24) .. (609,166) .. controls (609,168.76) and (606.76,171) .. (604,171) .. controls (601.24,171) and (599,168.76) .. (599,166) -- cycle ;
%Straight Lines [id:da9939861155158032] 
\draw [color={rgb, 255:red, 0; green, 0; blue, 0 }  ,draw opacity=1 ]   (162,38) -- (10,38) ;
%Straight Lines [id:da6460141075147547] 
\draw [color={rgb, 255:red, 0; green, 0; blue, 0 }  ,draw opacity=1 ]   (162,38) -- (198,-2) ;
%Straight Lines [id:da644383057111223] 
\draw [color={rgb, 255:red, 0; green, 0; blue, 0 }  ,draw opacity=1 ]   (161.5,61) -- (181,83) ;
%Straight Lines [id:da796824295382208] 
\draw    (181,83) -- (360,83) ;
%Straight Lines [id:da9753553652203049] 
\draw [color={rgb, 255:red, 0; green, 0; blue, 0 }  ,draw opacity=1 ]   (181,83) -- (203,106) ;
%Straight Lines [id:da48627689074126645] 
\draw    (203,106) -- (359,107) ;
%Straight Lines [id:da49133266575001033] 
\draw    (360,83) -- (359,107) ;
%Straight Lines [id:da51467780889709] 
\draw [color={rgb, 255:red, 0; green, 0; blue, 0 }  ,draw opacity=1 ]   (203,106) -- (203,188) ;
%Straight Lines [id:da7007627180528562] 
\draw    (360,83) -- (432,1) ;
%Straight Lines [id:da7832199179696879] 
\draw    (359,107) -- (401,146) ;
%Straight Lines [id:da44753667002377306] 
\draw    (401,146) -- (604,147) ;
%Straight Lines [id:da6859330855524519] 
\draw    (604,147) -- (604,184) ;
%Straight Lines [id:da6255332306196266] 
\draw    (604,147) -- (654,84) ;
%Straight Lines [id:da6529064472478432] 
\draw    (401,146) -- (401,186) ;
%Shape: Square [id:dp7171542865243234] 
\draw   (539,7) -- (589,7) -- (589,57) -- (539,57) -- cycle ;
%Straight Lines [id:da16444256483632191] 
\draw [color={rgb, 255:red, 0; green, 0; blue, 0 }  ,draw opacity=1 ]   (162,38) -- (161.5,61) ;
%Straight Lines [id:da0840655116527731] 
\draw    (161.25,61.5) -- (10.75,62.5) ;

% Text Node
\draw (551.08,8.35) node [anchor=north west][inner sep=0.75pt]  [font=\huge,rotate=-358.58]  {$r_{3}$};

\end{tikzpicture}}  
 & %  
     \resizebox{.4 \textwidth}{!}{%\tikzset{every picture/.style={line width=0.75pt}} %set default line width to 0.75pt        

\begin{tikzpicture}[x=0.75pt,y=0.75pt,yscale=-1,xscale=1]

%Shape: Circle [id:dp049478145904912285] 
\draw  [color={rgb, 255:red, 208; green, 2; blue, 27 }  ,draw opacity=1 ][fill={rgb, 255:red, 208; green, 2; blue, 27 }  ,fill opacity=1 ] (29,20) .. controls (29,17.24) and (31.24,15) .. (34,15) .. controls (36.76,15) and (39,17.24) .. (39,20) .. controls (39,22.76) and (36.76,25) .. (34,25) .. controls (31.24,25) and (29,22.76) .. (29,20) -- cycle ;
%Shape: Circle [id:dp597865605180764] 
\draw  [color={rgb, 255:red, 208; green, 2; blue, 27 }  ,draw opacity=1 ][fill={rgb, 255:red, 208; green, 2; blue, 27 }  ,fill opacity=1 ] (64,43) .. controls (64,40.24) and (66.24,38) .. (69,38) .. controls (71.76,38) and (74,40.24) .. (74,43) .. controls (74,45.76) and (71.76,48) .. (69,48) .. controls (66.24,48) and (64,45.76) .. (64,43) -- cycle ;
%Shape: Circle [id:dp03321279691471113] 
\draw  [color={rgb, 255:red, 208; green, 2; blue, 27 }  ,draw opacity=1 ][fill={rgb, 255:red, 208; green, 2; blue, 27 }  ,fill opacity=1 ] (102,82) .. controls (102,79.24) and (104.24,77) .. (107,77) .. controls (109.76,77) and (112,79.24) .. (112,82) .. controls (112,84.76) and (109.76,87) .. (107,87) .. controls (104.24,87) and (102,84.76) .. (102,82) -- cycle ;
%Straight Lines [id:da15855389453970536] 
\draw [color={rgb, 255:red, 0; green, 0; blue, 0 }  ,draw opacity=1 ]   (11,82) -- (128,82) ;
%Shape: Circle [id:dp7534753985751717] 
\draw  [color={rgb, 255:red, 208; green, 2; blue, 27 }  ,draw opacity=1 ][fill={rgb, 255:red, 208; green, 2; blue, 27 }  ,fill opacity=1 ] (133,92) .. controls (133,89.24) and (135.24,87) .. (138,87) .. controls (140.76,87) and (143,89.24) .. (143,92) .. controls (143,94.76) and (140.76,97) .. (138,97) .. controls (135.24,97) and (133,94.76) .. (133,92) -- cycle ;
%Shape: Circle [id:dp5270532448174468] 
\draw  [color={rgb, 255:red, 208; green, 2; blue, 27 }  ,draw opacity=1 ][fill={rgb, 255:red, 208; green, 2; blue, 27 }  ,fill opacity=1 ] (294,107) .. controls (294,104.24) and (296.24,102) .. (299,102) .. controls (301.76,102) and (304,104.24) .. (304,107) .. controls (304,109.76) and (301.76,112) .. (299,112) .. controls (296.24,112) and (294,109.76) .. (294,107) -- cycle ;
%Shape: Circle [id:dp7779686155492531] 
\draw  [color={rgb, 255:red, 208; green, 2; blue, 27 }  ,draw opacity=1 ][fill={rgb, 255:red, 208; green, 2; blue, 27 }  ,fill opacity=1 ] (375,126.5) .. controls (375,123.74) and (377.24,121.5) .. (380,121.5) .. controls (382.76,121.5) and (385,123.74) .. (385,126.5) .. controls (385,129.26) and (382.76,131.5) .. (380,131.5) .. controls (377.24,131.5) and (375,129.26) .. (375,126.5) -- cycle ;
%Shape: Circle [id:dp23212820006386536] 
\draw  [color={rgb, 255:red, 208; green, 2; blue, 27 }  ,draw opacity=1 ][fill={rgb, 255:red, 208; green, 2; blue, 27 }  ,fill opacity=1 ] (466.5,146) .. controls (466.5,143.24) and (468.74,141) .. (471.5,141) .. controls (474.26,141) and (476.5,143.24) .. (476.5,146) .. controls (476.5,148.76) and (474.26,151) .. (471.5,151) .. controls (468.74,151) and (466.5,148.76) .. (466.5,146) -- cycle ;
%Shape: Circle [id:dp5439300296757643] 
\draw  [color={rgb, 255:red, 208; green, 2; blue, 27 }  ,draw opacity=1 ][fill={rgb, 255:red, 208; green, 2; blue, 27 }  ,fill opacity=1 ] (599,166) .. controls (599,163.24) and (601.24,161) .. (604,161) .. controls (606.76,161) and (609,163.24) .. (609,166) .. controls (609,168.76) and (606.76,171) .. (604,171) .. controls (601.24,171) and (599,168.76) .. (599,166) -- cycle ;
%Straight Lines [id:da9939861155158032] 
\draw [color={rgb, 255:red, 0; green, 0; blue, 0 }  ,draw opacity=1 ]   (128,21) -- (12.5,19) ;
%Straight Lines [id:da6460141075147547] 
\draw [color={rgb, 255:red, 0; green, 0; blue, 0 }  ,draw opacity=1 ]   (128,21) -- (148,2) ;
%Straight Lines [id:da796824295382208] 
\draw [color={rgb, 255:red, 74; green, 144; blue, 226 }  ,draw opacity=1 ]   (157,-2) -- (390,-2) ;
%Straight Lines [id:da9753553652203049] 
\draw [color={rgb, 255:red, 0; green, 0; blue, 0 }  ,draw opacity=1 ]   (128,82) -- (150,105) ;
%Straight Lines [id:da48627689074126645] 
\draw    (150,105) -- (359,107) ;
%Straight Lines [id:da49133266575001033] 
\draw [color={rgb, 255:red, 0; green, 0; blue, 0 }  ,draw opacity=1 ]   (360,42) -- (359,107) ;
%Straight Lines [id:da51467780889709] 
\draw [color={rgb, 255:red, 0; green, 0; blue, 0 }  ,draw opacity=1 ]   (150,105) -- (150,187) ;
%Straight Lines [id:da7007627180528562] 
\draw [color={rgb, 255:red, 0; green, 0; blue, 0 }  ,draw opacity=1 ]   (360,42) -- (396,0) ;
%Straight Lines [id:da7832199179696879] 
\draw    (359,107) -- (401,146) ;
%Straight Lines [id:da44753667002377306] 
\draw    (401,146) -- (604,147) ;
%Straight Lines [id:da6859330855524519] 
\draw    (604,147) -- (604,184) ;
%Straight Lines [id:da6255332306196266] 
\draw    (604,147) -- (654,84) ;
%Straight Lines [id:da6529064472478432] 
\draw    (401,146) -- (401,186) ;
%Shape: Square [id:dp7171542865243234] 
\draw   (539,7) -- (589,7) -- (589,57) -- (539,57) -- cycle ;
%Straight Lines [id:da0840655116527731] 
\draw [color={rgb, 255:red, 0; green, 0; blue, 0 }  ,draw opacity=1 ]   (360,42) -- (15.25,43) ;
%Straight Lines [id:da46996664213146677] 
\draw [color={rgb, 255:red, 0; green, 0; blue, 0 }  ,draw opacity=1 ]   (128,21) -- (128,82) ;

% Text Node
\draw (547.08,8.35) node [anchor=north west][inner sep=0.75pt]  [font=\huge,rotate=-358.58]  {$r_{4}$};

\end{tikzpicture}}  
 \\ \hline
     \resizebox{.4 \textwidth}{!}{%\tikzset{every picture/.style={line width=0.75pt}} %set default line width to 0.75pt        

\begin{tikzpicture}[x=0.75pt,y=0.75pt,yscale=-1,xscale=1]
%uncomment if require: \path (0,192); %set diagram left start at 0, and has height of 192

%Shape: Circle [id:dp049478145904912285] 
\draw  [color={rgb, 255:red, 208; green, 2; blue, 27 }  ,draw opacity=1 ][fill={rgb, 255:red, 208; green, 2; blue, 27 }  ,fill opacity=1 ] (17,22) .. controls (17,19.24) and (19.24,17) .. (22,17) .. controls (24.76,17) and (27,19.24) .. (27,22) .. controls (27,24.76) and (24.76,27) .. (22,27) .. controls (19.24,27) and (17,24.76) .. (17,22) -- cycle ;
%Shape: Circle [id:dp597865605180764] 
\draw  [color={rgb, 255:red, 208; green, 2; blue, 27 }  ,draw opacity=1 ][fill={rgb, 255:red, 208; green, 2; blue, 27 }  ,fill opacity=1 ] (57,44) .. controls (57,41.24) and (59.24,39) .. (62,39) .. controls (64.76,39) and (67,41.24) .. (67,44) .. controls (67,46.76) and (64.76,49) .. (62,49) .. controls (59.24,49) and (57,46.76) .. (57,44) -- cycle ;
%Shape: Circle [id:dp03321279691471113] 
\draw  [color={rgb, 255:red, 208; green, 2; blue, 27 }  ,draw opacity=1 ][fill={rgb, 255:red, 208; green, 2; blue, 27 }  ,fill opacity=1 ] (95,83) .. controls (95,80.24) and (97.24,78) .. (100,78) .. controls (102.76,78) and (105,80.24) .. (105,83) .. controls (105,85.76) and (102.76,88) .. (100,88) .. controls (97.24,88) and (95,85.76) .. (95,83) -- cycle ;
%Straight Lines [id:da15855389453970536] 
\draw [color={rgb, 255:red, 0; green, 0; blue, 0 }  ,draw opacity=1 ]   (4,83) -- (121,83) ;
%Shape: Circle [id:dp7534753985751717] 
\draw  [color={rgb, 255:red, 208; green, 2; blue, 27 }  ,draw opacity=1 ][fill={rgb, 255:red, 208; green, 2; blue, 27 }  ,fill opacity=1 ] (126,93) .. controls (126,90.24) and (128.24,88) .. (131,88) .. controls (133.76,88) and (136,90.24) .. (136,93) .. controls (136,95.76) and (133.76,98) .. (131,98) .. controls (128.24,98) and (126,95.76) .. (126,93) -- cycle ;
%Shape: Circle [id:dp5270532448174468] 
\draw  [color={rgb, 255:red, 208; green, 2; blue, 27 }  ,draw opacity=1 ][fill={rgb, 255:red, 208; green, 2; blue, 27 }  ,fill opacity=1 ] (287,108) .. controls (287,105.24) and (289.24,103) .. (292,103) .. controls (294.76,103) and (297,105.24) .. (297,108) .. controls (297,110.76) and (294.76,113) .. (292,113) .. controls (289.24,113) and (287,110.76) .. (287,108) -- cycle ;
%Shape: Circle [id:dp7779686155492531] 
\draw  [color={rgb, 255:red, 208; green, 2; blue, 27 }  ,draw opacity=1 ][fill={rgb, 255:red, 208; green, 2; blue, 27 }  ,fill opacity=1 ] (368,127.5) .. controls (368,124.74) and (370.24,122.5) .. (373,122.5) .. controls (375.76,122.5) and (378,124.74) .. (378,127.5) .. controls (378,130.26) and (375.76,132.5) .. (373,132.5) .. controls (370.24,132.5) and (368,130.26) .. (368,127.5) -- cycle ;
%Shape: Circle [id:dp23212820006386536] 
\draw  [color={rgb, 255:red, 208; green, 2; blue, 27 }  ,draw opacity=1 ][fill={rgb, 255:red, 208; green, 2; blue, 27 }  ,fill opacity=1 ] (459.5,147) .. controls (459.5,144.24) and (461.74,142) .. (464.5,142) .. controls (467.26,142) and (469.5,144.24) .. (469.5,147) .. controls (469.5,149.76) and (467.26,152) .. (464.5,152) .. controls (461.74,152) and (459.5,149.76) .. (459.5,147) -- cycle ;
%Shape: Circle [id:dp5439300296757643] 
\draw  [color={rgb, 255:red, 208; green, 2; blue, 27 }  ,draw opacity=1 ][fill={rgb, 255:red, 208; green, 2; blue, 27 }  ,fill opacity=1 ] (592,167) .. controls (592,164.24) and (594.24,162) .. (597,162) .. controls (599.76,162) and (602,164.24) .. (602,167) .. controls (602,169.76) and (599.76,172) .. (597,172) .. controls (594.24,172) and (592,169.76) .. (592,167) -- cycle ;
%Straight Lines [id:da9939861155158032] 
\draw [color={rgb, 255:red, 0; green, 0; blue, 0 }  ,draw opacity=1 ]   (352,21) -- (5,21) ;
%Straight Lines [id:da6460141075147547] 
\draw [color={rgb, 255:red, 0; green, 0; blue, 0 }  ,draw opacity=1 ]   (121,45) -- (132.8,33.79) -- (164,4) ;
%Straight Lines [id:da9753553652203049] 
\draw [color={rgb, 255:red, 0; green, 0; blue, 0 }  ,draw opacity=1 ]   (121,83) -- (143,106) ;
%Straight Lines [id:da48627689074126645] 
\draw    (143,106) -- (352,108) ;
%Straight Lines [id:da49133266575001033] 
\draw [color={rgb, 255:red, 0; green, 0; blue, 0 }  ,draw opacity=1 ]   (352,21) -- (352,108) ;
%Straight Lines [id:da51467780889709] 
\draw [color={rgb, 255:red, 0; green, 0; blue, 0 }  ,draw opacity=1 ]   (143,106) -- (143,188) ;
%Straight Lines [id:da7007627180528562] 
\draw [color={rgb, 255:red, 0; green, 0; blue, 0 }  ,draw opacity=1 ]   (352,21) -- (370,5) ;
%Straight Lines [id:da7832199179696879] 
\draw    (352,108) -- (394,147) ;
%Straight Lines [id:da44753667002377306] 
\draw    (394,147) -- (597,148) ;
%Straight Lines [id:da6859330855524519] 
\draw    (597,148) -- (597,185) ;
%Straight Lines [id:da6255332306196266] 
\draw    (597,148) -- (647,85) ;
%Straight Lines [id:da6529064472478432] 
\draw    (394,147) -- (394,187) ;
%Shape: Square [id:dp7171542865243234] 
\draw   (532,8) -- (582,8) -- (582,58) -- (532,58) -- cycle ;
%Straight Lines [id:da0840655116527731] 
\draw [color={rgb, 255:red, 0; green, 0; blue, 0 }  ,draw opacity=1 ]   (121,45) -- (3.13,45) ;
%Straight Lines [id:da46996664213146677] 
\draw [color={rgb, 255:red, 0; green, 0; blue, 0 }  ,draw opacity=1 ]   (121,45) -- (121,83) ;

% Text Node
\draw (540.08,9.35) node [anchor=north west][inner sep=0.75pt]  [font=\huge,rotate=0]  {$r_{5}$};

\end{tikzpicture}}  
 & %  
     \resizebox{.4 \textwidth}{!}{%\tikzset{every picture/.style={line width=0.75pt}} %set default line width to 0.75pt        

\begin{tikzpicture}[x=0.75pt,y=0.75pt,yscale=-1,xscale=1]
%uncomment if require: \path (0,501); %set diagram left start at 0, and has height of 501

%Shape: Circle [id:dp049478145904912285] 
\draw  [color={rgb, 255:red, 208; green, 2; blue, 27 }  ,draw opacity=1 ][fill={rgb, 255:red, 208; green, 2; blue, 27 }  ,fill opacity=1 ] (29,20) .. controls (29,17.24) and (31.24,15) .. (34,15) .. controls (36.76,15) and (39,17.24) .. (39,20) .. controls (39,22.76) and (36.76,25) .. (34,25) .. controls (31.24,25) and (29,22.76) .. (29,20) -- cycle ;
%Shape: Circle [id:dp597865605180764] 
\draw  [color={rgb, 255:red, 208; green, 2; blue, 27 }  ,draw opacity=1 ][fill={rgb, 255:red, 208; green, 2; blue, 27 }  ,fill opacity=1 ] (64,43) .. controls (64,40.24) and (66.24,38) .. (69,38) .. controls (71.76,38) and (74,40.24) .. (74,43) .. controls (74,45.76) and (71.76,48) .. (69,48) .. controls (66.24,48) and (64,45.76) .. (64,43) -- cycle ;
%Shape: Circle [id:dp03321279691471113] 
\draw  [color={rgb, 255:red, 208; green, 2; blue, 27 }  ,draw opacity=1 ][fill={rgb, 255:red, 208; green, 2; blue, 27 }  ,fill opacity=1 ] (106,65) .. controls (106,62.24) and (108.24,60) .. (111,60) .. controls (113.76,60) and (116,62.24) .. (116,65) .. controls (116,67.76) and (113.76,70) .. (111,70) .. controls (108.24,70) and (106,67.76) .. (106,65) -- cycle ;
%Straight Lines [id:da15855389453970536] 
\draw [color={rgb, 255:red, 0; green, 0; blue, 0 }  ,draw opacity=1 ]   (11,65) -- (128,65) ;
%Shape: Circle [id:dp7534753985751717] 
\draw  [color={rgb, 255:red, 208; green, 2; blue, 27 }  ,draw opacity=1 ][fill={rgb, 255:red, 208; green, 2; blue, 27 }  ,fill opacity=1 ] (160,78) .. controls (160,75.24) and (162.24,73) .. (165,73) .. controls (167.76,73) and (170,75.24) .. (170,78) .. controls (170,80.76) and (167.76,83) .. (165,83) .. controls (162.24,83) and (160,80.76) .. (160,78) -- cycle ;
%Shape: Circle [id:dp7779686155492531] 
\draw  [color={rgb, 255:red, 208; green, 2; blue, 27 }  ,draw opacity=1 ][fill={rgb, 255:red, 208; green, 2; blue, 27 }  ,fill opacity=1 ] (375.01,109.76) .. controls (375.01,107) and (377.24,104.76) .. (380.01,104.76) .. controls (382.77,104.76) and (385.01,107) .. (385.01,109.76) .. controls (385.01,112.52) and (382.77,114.76) .. (380.01,114.76) .. controls (377.24,114.76) and (375.01,112.52) .. (375.01,109.76) -- cycle ;
%Shape: Circle [id:dp23212820006386536] 
\draw  [color={rgb, 255:red, 208; green, 2; blue, 27 }  ,draw opacity=1 ][fill={rgb, 255:red, 208; green, 2; blue, 27 }  ,fill opacity=1 ] (466.5,146) .. controls (466.5,143.24) and (468.74,141) .. (471.5,141) .. controls (474.26,141) and (476.5,143.24) .. (476.5,146) .. controls (476.5,148.76) and (474.26,151) .. (471.5,151) .. controls (468.74,151) and (466.5,148.76) .. (466.5,146) -- cycle ;
%Shape: Circle [id:dp5439300296757643] 
\draw  [color={rgb, 255:red, 208; green, 2; blue, 27 }  ,draw opacity=1 ][fill={rgb, 255:red, 208; green, 2; blue, 27 }  ,fill opacity=1 ] (599,166) .. controls (599,163.24) and (601.24,161) .. (604,161) .. controls (606.76,161) and (609,163.24) .. (609,166) .. controls (609,168.76) and (606.76,171) .. (604,171) .. controls (601.24,171) and (599,168.76) .. (599,166) -- cycle ;
%Straight Lines [id:da9939861155158032] 
\draw [color={rgb, 255:red, 0; green, 0; blue, 0 }  ,draw opacity=1 ]   (109,20) -- (12.5,19) ;
%Straight Lines [id:da6460141075147547] 
\draw [color={rgb, 255:red, 0; green, 0; blue, 0 }  ,draw opacity=1 ]   (109,20) -- (129,1) ;
%Straight Lines [id:da644383057111223] 
\draw [color={rgb, 255:red, 0; green, 0; blue, 0 }  ,draw opacity=1 ]   (108.5,43) -- (128,65) ;
%Straight Lines [id:da796824295382208] 
\draw [color={rgb, 255:red, 0; green, 0; blue, 0 }  ,draw opacity=1 ]   (181,83) -- (380.01,83.02) ;
%Straight Lines [id:da9753553652203049] 
\draw [color={rgb, 255:red, 0; green, 0; blue, 0 }  ,draw opacity=1 ]   (181,83) -- (224,125.5) ;
%Straight Lines [id:da48627689074126645] 
\draw [color={rgb, 255:red, 0; green, 0; blue, 0 }  ,draw opacity=1 ]   (245,145) -- (401,146) ;
%Straight Lines [id:da49133266575001033] 
\draw [color={rgb, 255:red, 0; green, 0; blue, 0 }  ,draw opacity=1 ]   (380,126.5) -- (380.01,83.02) ;
%Straight Lines [id:da51467780889709] 
\draw [color={rgb, 255:red, 0; green, 0; blue, 0 }  ,draw opacity=1 ]   (224,125.5) -- (245,145) ;
%Straight Lines [id:da7007627180528562] 
\draw [color={rgb, 255:red, 0; green, 0; blue, 0 }  ,draw opacity=1 ]   (380.01,83.02) -- (452.01,1.02) ;
%Straight Lines [id:da44753667002377306] 
\draw [color={rgb, 255:red, 0; green, 0; blue, 0 }  ,draw opacity=1 ]   (401,146) -- (604,147) ;
%Straight Lines [id:da6859330855524519] 
\draw    (604,147) -- (604,184) ;
%Straight Lines [id:da6255332306196266] 
\draw    (604,147) -- (654,84) ;
%Straight Lines [id:da6529064472478432] 
\draw [color={rgb, 255:red, 0; green, 0; blue, 0 }  ,draw opacity=1 ]   (380,126.5) -- (379.01,182.02) ;
%Shape: Square [id:dp7171542865243234] 
\draw   (539,7) -- (589,7) -- (589,57) -- (539,57) -- cycle ;
%Straight Lines [id:da16444256483632191] 
\draw [color={rgb, 255:red, 0; green, 0; blue, 0 }  ,draw opacity=1 ]   (109,20) -- (108.5,43) ;
%Straight Lines [id:da0840655116527731] 
\draw [color={rgb, 255:red, 0; green, 0; blue, 0 }  ,draw opacity=1 ]   (108.5,43) -- (11,43) ;
%Straight Lines [id:da46996664213146677] 
\draw [color={rgb, 255:red, 0; green, 0; blue, 0 }  ,draw opacity=1 ]   (128,65) -- (172.76,80.2) -- (181,83) ;
%Shape: Circle [id:dp6845779969407073] 
\draw  [color={rgb, 255:red, 208; green, 2; blue, 27 }  ,draw opacity=1 ][fill={rgb, 255:red, 208; green, 2; blue, 27 }  ,fill opacity=1 ] (245.01,85.76) .. controls (245.01,83) and (247.24,80.76) .. (250.01,80.76) .. controls (252.77,80.76) and (255.01,83) .. (255.01,85.76) .. controls (255.01,88.52) and (252.77,90.76) .. (250.01,90.76) .. controls (247.24,90.76) and (245.01,88.52) .. (245.01,85.76) -- cycle ;
%Straight Lines [id:da7969571271225022] 
\draw [color={rgb, 255:red, 0; green, 0; blue, 0 }  ,draw opacity=1 ]   (245,145) -- (245,182) ;

% Text Node
\draw (541.08,10.35) node [anchor=north west][inner sep=0.75pt]  [font=\huge,rotate=-358.58]  {$r_{6}$};

\end{tikzpicture}}  
 \\ \hline
     \resizebox{.4 \textwidth}{!}{%\tikzset{every picture/.style={line width=0.75pt}} %set default line width to 0.75pt        

\begin{tikzpicture}[x=0.75pt,y=0.75pt,yscale=-1,xscale=1]

%Shape: Circle [id:dp049478145904912285] 
\draw  [color={rgb, 255:red, 208; green, 2; blue, 27 }  ,draw opacity=1 ][fill={rgb, 255:red, 208; green, 2; blue, 27 }  ,fill opacity=1 ] (29,20) .. controls (29,17.24) and (31.24,15) .. (34,15) .. controls (36.76,15) and (39,17.24) .. (39,20) .. controls (39,22.76) and (36.76,25) .. (34,25) .. controls (31.24,25) and (29,22.76) .. (29,20) -- cycle ;
%Shape: Circle [id:dp597865605180764] 
\draw  [color={rgb, 255:red, 208; green, 2; blue, 27 }  ,draw opacity=1 ][fill={rgb, 255:red, 208; green, 2; blue, 27 }  ,fill opacity=1 ] (64,43) .. controls (64,40.24) and (66.24,38) .. (69,38) .. controls (71.76,38) and (74,40.24) .. (74,43) .. controls (74,45.76) and (71.76,48) .. (69,48) .. controls (66.24,48) and (64,45.76) .. (64,43) -- cycle ;
%Shape: Circle [id:dp03321279691471113] 
\draw  [color={rgb, 255:red, 208; green, 2; blue, 27 }  ,draw opacity=1 ][fill={rgb, 255:red, 208; green, 2; blue, 27 }  ,fill opacity=1 ] (106,65) .. controls (106,62.24) and (108.24,60) .. (111,60) .. controls (113.76,60) and (116,62.24) .. (116,65) .. controls (116,67.76) and (113.76,70) .. (111,70) .. controls (108.24,70) and (106,67.76) .. (106,65) -- cycle ;
%Straight Lines [id:da15855389453970536] 
\draw [color={rgb, 255:red, 0; green, 0; blue, 0 }  ,draw opacity=1 ]   (11,65) -- (128,65) ;
%Shape: Circle [id:dp7534753985751717] 
\draw  [color={rgb, 255:red, 208; green, 2; blue, 27 }  ,draw opacity=1 ][fill={rgb, 255:red, 208; green, 2; blue, 27 }  ,fill opacity=1 ] (160,78) .. controls (160,75.24) and (162.24,73) .. (165,73) .. controls (167.76,73) and (170,75.24) .. (170,78) .. controls (170,80.76) and (167.76,83) .. (165,83) .. controls (162.24,83) and (160,80.76) .. (160,78) -- cycle ;
%Shape: Circle [id:dp7779686155492531] 
\draw  [color={rgb, 255:red, 208; green, 2; blue, 27 }  ,draw opacity=1 ][fill={rgb, 255:red, 208; green, 2; blue, 27 }  ,fill opacity=1 ] (359,126.5) .. controls (359,123.74) and (361.24,121.5) .. (364,121.5) .. controls (366.76,121.5) and (369,123.74) .. (369,126.5) .. controls (369,129.26) and (366.76,131.5) .. (364,131.5) .. controls (361.24,131.5) and (359,129.26) .. (359,126.5) -- cycle ;
%Shape: Circle [id:dp23212820006386536] 
\draw  [color={rgb, 255:red, 208; green, 2; blue, 27 }  ,draw opacity=1 ][fill={rgb, 255:red, 208; green, 2; blue, 27 }  ,fill opacity=1 ] (444,141) .. controls (444,138.24) and (446.24,136) .. (449,136) .. controls (451.76,136) and (454,138.24) .. (454,141) .. controls (454,143.76) and (451.76,146) .. (449,146) .. controls (446.24,146) and (444,143.76) .. (444,141) -- cycle ;
%Shape: Circle [id:dp5439300296757643] 
\draw  [color={rgb, 255:red, 208; green, 2; blue, 27 }  ,draw opacity=1 ][fill={rgb, 255:red, 208; green, 2; blue, 27 }  ,fill opacity=1 ] (575.5,161.5) .. controls (575.5,158.74) and (577.74,156.5) .. (580.5,156.5) .. controls (583.26,156.5) and (585.5,158.74) .. (585.5,161.5) .. controls (585.5,164.26) and (583.26,166.5) .. (580.5,166.5) .. controls (577.74,166.5) and (575.5,164.26) .. (575.5,161.5) -- cycle ;
%Straight Lines [id:da9939861155158032] 
\draw [color={rgb, 255:red, 0; green, 0; blue, 0 }  ,draw opacity=1 ]   (109,20) -- (12.5,19) ;
%Straight Lines [id:da6460141075147547] 
\draw [color={rgb, 255:red, 0; green, 0; blue, 0 }  ,draw opacity=1 ]   (109,20) -- (129,1) ;
%Straight Lines [id:da644383057111223] 
\draw [color={rgb, 255:red, 0; green, 0; blue, 0 }  ,draw opacity=1 ]   (108.5,43) -- (128,65) ;
%Straight Lines [id:da796824295382208] 
\draw [color={rgb, 255:red, 0; green, 0; blue, 0 }  ,draw opacity=1 ]   (181,83) -- (449.01,82.02) ;
%Straight Lines [id:da9753553652203049] 
\draw [color={rgb, 255:red, 0; green, 0; blue, 0 }  ,draw opacity=1 ]   (181,83) -- (224,125.5) ;
%Straight Lines [id:da48627689074126645] 
\draw [color={rgb, 255:red, 0; green, 0; blue, 0 }  ,draw opacity=1 ]   (224,125.5) -- (449,125.5) ;
%Straight Lines [id:da49133266575001033] 
\draw [color={rgb, 255:red, 0; green, 0; blue, 0 }  ,draw opacity=1 ]   (449,125.5) -- (449.01,82.02) ;
%Straight Lines [id:da51467780889709] 
\draw [color={rgb, 255:red, 0; green, 0; blue, 0 }  ,draw opacity=1 ]   (224,125.5) -- (224.01,184.02) ;
%Straight Lines [id:da7007627180528562] 
\draw [color={rgb, 255:red, 0; green, 0; blue, 0 }  ,draw opacity=1 ]   (449.01,82.02) -- (521.01,0.02) ;
%Straight Lines [id:da44753667002377306] 
\draw [color={rgb, 255:red, 0; green, 0; blue, 0 }  ,draw opacity=1 ]   (449,126) -- (581.5,127) ;
%Straight Lines [id:da6859330855524519] 
\draw [color={rgb, 255:red, 0; green, 0; blue, 0 }  ,draw opacity=1 ]   (581.5,127) -- (581.01,173.02) ;
%Straight Lines [id:da6255332306196266] 
\draw [color={rgb, 255:red, 0; green, 0; blue, 0 }  ,draw opacity=1 ]   (581.5,127) -- (631.5,64) ;
%Straight Lines [id:da6529064472478432] 
\draw [color={rgb, 255:red, 0; green, 0; blue, 0 }  ,draw opacity=1 ]   (449,126) -- (449,166) ;
%Shape: Square [id:dp7171542865243234] 
\draw   (539,7) -- (589,7) -- (589,57) -- (539,57) -- cycle ;
%Straight Lines [id:da16444256483632191] 
\draw [color={rgb, 255:red, 0; green, 0; blue, 0 }  ,draw opacity=1 ]   (109,20) -- (108.5,43) ;
%Straight Lines [id:da0840655116527731] 
\draw [color={rgb, 255:red, 0; green, 0; blue, 0 }  ,draw opacity=1 ]   (108.5,43) -- (11,43) ;
%Straight Lines [id:da46996664213146677] 
\draw [color={rgb, 255:red, 0; green, 0; blue, 0 }  ,draw opacity=1 ]   (128,65) -- (181,83) ;
%Shape: Circle [id:dp09440932328926133] 
\draw  [color={rgb, 255:red, 208; green, 2; blue, 27 }  ,draw opacity=1 ][fill={rgb, 255:red, 208; green, 2; blue, 27 }  ,fill opacity=1 ] (258,84.5) .. controls (258,81.74) and (260.24,79.5) .. (263,79.5) .. controls (265.76,79.5) and (268,81.74) .. (268,84.5) .. controls (268,87.26) and (265.76,89.5) .. (263,89.5) .. controls (260.24,89.5) and (258,87.26) .. (258,84.5) -- cycle ;

% Text Node
\draw (547.08,10.35) node [anchor=north west][inner sep=0.75pt]  [font=\huge,rotate=-358.58]  {$r_{7}$};

\end{tikzpicture}}  
 & %  
     \resizebox{.4 \textwidth}{!}{%\tikzset{every picture/.style={line width=0.75pt}} %set default line width to 0.75pt        

\begin{tikzpicture}[x=0.75pt,y=0.75pt,yscale=-1,xscale=1]

%Shape: Circle [id:dp049478145904912285] 
\draw  [color={rgb, 255:red, 208; green, 2; blue, 27 }  ,draw opacity=1 ][fill={rgb, 255:red, 208; green, 2; blue, 27 }  ,fill opacity=1 ] (29,20) .. controls (29,17.24) and (31.24,15) .. (34,15) .. controls (36.76,15) and (39,17.24) .. (39,20) .. controls (39,22.76) and (36.76,25) .. (34,25) .. controls (31.24,25) and (29,22.76) .. (29,20) -- cycle ;
%Shape: Circle [id:dp597865605180764] 
\draw  [color={rgb, 255:red, 208; green, 2; blue, 27 }  ,draw opacity=1 ][fill={rgb, 255:red, 208; green, 2; blue, 27 }  ,fill opacity=1 ] (64,43) .. controls (64,40.24) and (66.24,38) .. (69,38) .. controls (71.76,38) and (74,40.24) .. (74,43) .. controls (74,45.76) and (71.76,48) .. (69,48) .. controls (66.24,48) and (64,45.76) .. (64,43) -- cycle ;
%Shape: Circle [id:dp03321279691471113] 
\draw  [color={rgb, 255:red, 208; green, 2; blue, 27 }  ,draw opacity=1 ][fill={rgb, 255:red, 208; green, 2; blue, 27 }  ,fill opacity=1 ] (106,65) .. controls (106,62.24) and (108.24,60) .. (111,60) .. controls (113.76,60) and (116,62.24) .. (116,65) .. controls (116,67.76) and (113.76,70) .. (111,70) .. controls (108.24,70) and (106,67.76) .. (106,65) -- cycle ;
%Straight Lines [id:da15855389453970536] 
\draw [color={rgb, 255:red, 0; green, 0; blue, 0 }  ,draw opacity=1 ]   (11,65) -- (128,65) ;
%Shape: Circle [id:dp7534753985751717] 
\draw  [color={rgb, 255:red, 208; green, 2; blue, 27 }  ,draw opacity=1 ][fill={rgb, 255:red, 208; green, 2; blue, 27 }  ,fill opacity=1 ] (160,78) .. controls (160,75.24) and (162.24,73) .. (165,73) .. controls (167.76,73) and (170,75.24) .. (170,78) .. controls (170,80.76) and (167.76,83) .. (165,83) .. controls (162.24,83) and (160,80.76) .. (160,78) -- cycle ;
%Shape: Circle [id:dp7779686155492531] 
\draw  [color={rgb, 255:red, 208; green, 2; blue, 27 }  ,draw opacity=1 ][fill={rgb, 255:red, 208; green, 2; blue, 27 }  ,fill opacity=1 ] (359,126.5) .. controls (359,123.74) and (361.24,121.5) .. (364,121.5) .. controls (366.76,121.5) and (369,123.74) .. (369,126.5) .. controls (369,129.26) and (366.76,131.5) .. (364,131.5) .. controls (361.24,131.5) and (359,129.26) .. (359,126.5) -- cycle ;
%Shape: Circle [id:dp23212820006386536] 
\draw  [color={rgb, 255:red, 208; green, 2; blue, 27 }  ,draw opacity=1 ][fill={rgb, 255:red, 208; green, 2; blue, 27 }  ,fill opacity=1 ] (444,141) .. controls (444,138.24) and (446.24,136) .. (449,136) .. controls (451.76,136) and (454,138.24) .. (454,141) .. controls (454,143.76) and (451.76,146) .. (449,146) .. controls (446.24,146) and (444,143.76) .. (444,141) -- cycle ;
%Shape: Circle [id:dp5439300296757643] 
\draw  [color={rgb, 255:red, 208; green, 2; blue, 27 }  ,draw opacity=1 ][fill={rgb, 255:red, 208; green, 2; blue, 27 }  ,fill opacity=1 ] (614.75,171.51) .. controls (614.75,168.75) and (616.99,166.51) .. (619.75,166.51) .. controls (622.51,166.51) and (624.75,168.75) .. (624.75,171.51) .. controls (624.75,174.27) and (622.51,176.51) .. (619.75,176.51) .. controls (616.99,176.51) and (614.75,174.27) .. (614.75,171.51) -- cycle ;
%Straight Lines [id:da9939861155158032] 
\draw [color={rgb, 255:red, 0; green, 0; blue, 0 }  ,draw opacity=1 ]   (109,20) -- (12.5,19) ;
%Straight Lines [id:da6460141075147547] 
\draw [color={rgb, 255:red, 0; green, 0; blue, 0 }  ,draw opacity=1 ]   (109,20) -- (129,1) ;
%Straight Lines [id:da644383057111223] 
\draw [color={rgb, 255:red, 0; green, 0; blue, 0 }  ,draw opacity=1 ]   (108.5,43) -- (128,65) ;
%Straight Lines [id:da796824295382208] 
\draw [color={rgb, 255:red, 0; green, 0; blue, 0 }  ,draw opacity=1 ]   (181,83) -- (471,84.02) ;
%Straight Lines [id:da9753553652203049] 
\draw [color={rgb, 255:red, 0; green, 0; blue, 0 }  ,draw opacity=1 ]   (181,83) -- (224,125.5) ;
%Straight Lines [id:da48627689074126645] 
\draw [color={rgb, 255:red, 0; green, 0; blue, 0 }  ,draw opacity=1 ]   (224,125.5) -- (449,125.5) ;
%Straight Lines [id:da49133266575001033] 
\draw [color={rgb, 255:red, 0; green, 0; blue, 0 }  ,draw opacity=1 ]   (449,126) -- (471,104.02) ;
%Straight Lines [id:da51467780889709] 
\draw [color={rgb, 255:red, 0; green, 0; blue, 0 }  ,draw opacity=1 ]   (224,125.5) -- (224.01,184.02) ;
%Straight Lines [id:da7007627180528562] 
\draw [color={rgb, 255:red, 0; green, 0; blue, 0 }  ,draw opacity=1 ]   (471,104.02) -- (543,22.02) ;
%Straight Lines [id:da44753667002377306] 
\draw [color={rgb, 255:red, 0; green, 0; blue, 0 }  ,draw opacity=1 ]   (470,84.02) -- (617.5,83) ;
%Straight Lines [id:da6859330855524519] 
\draw [color={rgb, 255:red, 0; green, 0; blue, 0 }  ,draw opacity=1 ]   (617.5,83) -- (618.02,189.02) ;
%Straight Lines [id:da6255332306196266] 
\draw [color={rgb, 255:red, 0; green, 0; blue, 0 }  ,draw opacity=1 ]   (617.5,83) -- (655,37.02) ;
%Straight Lines [id:da6529064472478432] 
\draw [color={rgb, 255:red, 0; green, 0; blue, 0 }  ,draw opacity=1 ]   (449,126) -- (449,183.02) ;
%Shape: Square [id:dp7171542865243234] 
\draw   (557.98,8) -- (608,8) -- (608,58.02) -- (557.98,58.02) -- cycle ;
%Straight Lines [id:da16444256483632191] 
\draw [color={rgb, 255:red, 0; green, 0; blue, 0 }  ,draw opacity=1 ]   (109,20) -- (108.5,43) ;
%Straight Lines [id:da0840655116527731] 
\draw [color={rgb, 255:red, 0; green, 0; blue, 0 }  ,draw opacity=1 ]   (108.5,43) -- (11,43) ;
%Straight Lines [id:da46996664213146677] 
\draw [color={rgb, 255:red, 0; green, 0; blue, 0 }  ,draw opacity=1 ]   (128,65) -- (181,83) ;
%Shape: Circle [id:dp09440932328926133] 
\draw  [color={rgb, 255:red, 208; green, 2; blue, 27 }  ,draw opacity=1 ][fill={rgb, 255:red, 208; green, 2; blue, 27 }  ,fill opacity=1 ] (258,84.5) .. controls (258,81.74) and (260.24,79.5) .. (263,79.5) .. controls (265.76,79.5) and (268,81.74) .. (268,84.5) .. controls (268,87.26) and (265.76,89.5) .. (263,89.5) .. controls (260.24,89.5) and (258,87.26) .. (258,84.5) -- cycle ;

% Text Node
\draw (565.08,9.35) node [anchor=north west][inner sep=0.75pt]  [font=\huge,rotate=-358.58]  {$r_{8}$};

\end{tikzpicture}}  
 
\end{tabular}
\caption{The vertices of $B$ corresponding to rational nodal curves.}
\label{fig:ratnod}
\end{figure}
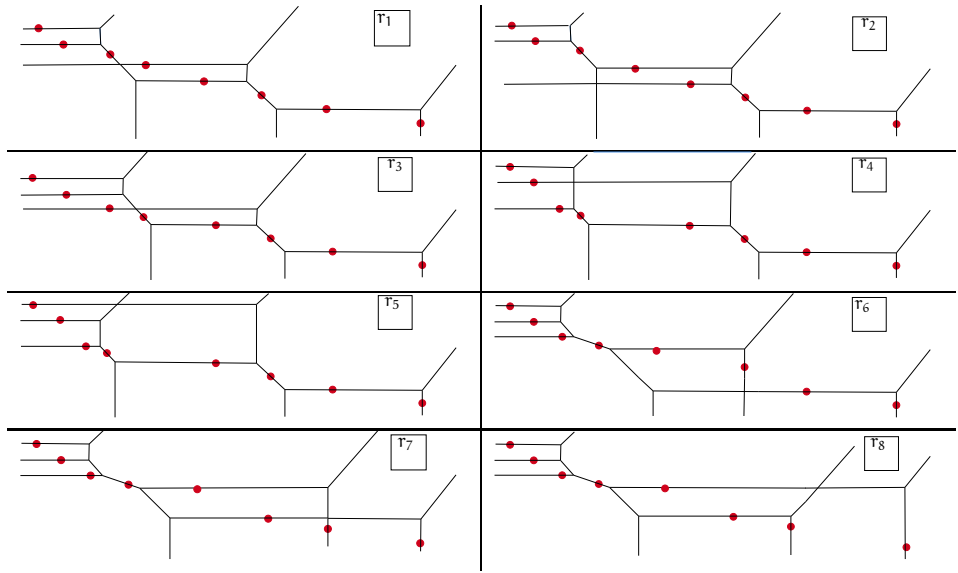

%\pagebreak
% \thispagestyle{empty}

\begin{figure}[h]
\begin{tabular}{c|c}
     \resizebox{.4 \textwidth}{!}{%\tikzset{every picture/.style={line width=0.75pt}} %set default line width to 0.75pt        

\begin{tikzpicture}[x=0.75pt,y=0.75pt,yscale=-1,xscale=1]

%Shape: Circle [id:dp049478145904912285] 
\draw  [color={rgb, 255:red, 208; green, 2; blue, 27 }  ,draw opacity=1 ][fill={rgb, 255:red, 208; green, 2; blue, 27 }  ,fill opacity=1 ] (31.5,45) .. controls (31.5,42.24) and (33.74,40) .. (36.5,40) .. controls (39.26,40) and (41.5,42.24) .. (41.5,45) .. controls (41.5,47.76) and (39.26,50) .. (36.5,50) .. controls (33.74,50) and (31.5,47.76) .. (31.5,45) -- cycle ;
%Shape: Circle [id:dp597865605180764] 
\draw  [color={rgb, 255:red, 208; green, 2; blue, 27 }  ,draw opacity=1 ][fill={rgb, 255:red, 208; green, 2; blue, 27 }  ,fill opacity=1 ] (86,61) .. controls (86,58.24) and (88.24,56) .. (91,56) .. controls (93.76,56) and (96,58.24) .. (96,61) .. controls (96,63.76) and (93.76,66) .. (91,66) .. controls (88.24,66) and (86,63.76) .. (86,61) -- cycle ;
%Shape: Circle [id:dp03321279691471113] 
\draw  [color={rgb, 255:red, 208; green, 2; blue, 27 }  ,draw opacity=1 ][fill={rgb, 255:red, 208; green, 2; blue, 27 }  ,fill opacity=1 ] (145,72) .. controls (145,69.24) and (147.24,67) .. (150,67) .. controls (152.76,67) and (155,69.24) .. (155,72) .. controls (155,74.76) and (152.76,77) .. (150,77) .. controls (147.24,77) and (145,74.76) .. (145,72) -- cycle ;
%Shape: Circle [id:dp7534753985751717] 
\draw  [color={rgb, 255:red, 208; green, 2; blue, 27 }  ,draw opacity=1 ][fill={rgb, 255:red, 208; green, 2; blue, 27 }  ,fill opacity=1 ] (212,84) .. controls (212,81.24) and (214.24,79) .. (217,79) .. controls (219.76,79) and (222,81.24) .. (222,84) .. controls (222,86.76) and (219.76,89) .. (217,89) .. controls (214.24,89) and (212,86.76) .. (212,84) -- cycle ;
%Shape: Circle [id:dp5270532448174468] 
\draw  [color={rgb, 255:red, 208; green, 2; blue, 27 }  ,draw opacity=1 ][fill={rgb, 255:red, 208; green, 2; blue, 27 }  ,fill opacity=1 ] (294,107) .. controls (294,104.24) and (296.24,102) .. (299,102) .. controls (301.76,102) and (304,104.24) .. (304,107) .. controls (304,109.76) and (301.76,112) .. (299,112) .. controls (296.24,112) and (294,109.76) .. (294,107) -- cycle ;
%Shape: Circle [id:dp7779686155492531] 
\draw  [color={rgb, 255:red, 208; green, 2; blue, 27 }  ,draw opacity=1 ][fill={rgb, 255:red, 208; green, 2; blue, 27 }  ,fill opacity=1 ] (375,126.5) .. controls (375,123.74) and (377.24,121.5) .. (380,121.5) .. controls (382.76,121.5) and (385,123.74) .. (385,126.5) .. controls (385,129.26) and (382.76,131.5) .. (380,131.5) .. controls (377.24,131.5) and (375,129.26) .. (375,126.5) -- cycle ;
%Shape: Circle [id:dp23212820006386536] 
\draw  [color={rgb, 255:red, 208; green, 2; blue, 27 }  ,draw opacity=1 ][fill={rgb, 255:red, 208; green, 2; blue, 27 }  ,fill opacity=1 ] (466.5,146) .. controls (466.5,143.24) and (468.74,141) .. (471.5,141) .. controls (474.26,141) and (476.5,143.24) .. (476.5,146) .. controls (476.5,148.76) and (474.26,151) .. (471.5,151) .. controls (468.74,151) and (466.5,148.76) .. (466.5,146) -- cycle ;
%Shape: Circle [id:dp5439300296757643] 
\draw  [color={rgb, 255:red, 208; green, 2; blue, 27 }  ,draw opacity=1 ][fill={rgb, 255:red, 208; green, 2; blue, 27 }  ,fill opacity=1 ] (599,166) .. controls (599,163.24) and (601.24,161) .. (604,161) .. controls (606.76,161) and (609,163.24) .. (609,166) .. controls (609,168.76) and (606.76,171) .. (604,171) .. controls (601.24,171) and (599,168.76) .. (599,166) -- cycle ;
%Straight Lines [id:da9939861155158032] 
\draw [color={rgb, 255:red, 0; green, 0; blue, 0 }  ,draw opacity=1 ]   (104,45) -- (36.5,45) ;
%Straight Lines [id:da6460141075147547] 
\draw [color={rgb, 255:red, 0; green, 0; blue, 0 }  ,draw opacity=1 ]   (104,45) -- (141.5,0) ;
%Straight Lines [id:da5830237590948697] 
\draw    (103.5,45.5) -- (119,61) ;
%Straight Lines [id:da5695455783405858] 
\draw    (119,61) -- (11,61) ;
%Straight Lines [id:da644383057111223] 
\draw    (119,61) -- (181,83) ;
%Straight Lines [id:da796824295382208] 
\draw    (181,83) -- (360,83) ;
%Straight Lines [id:da9753553652203049] 
\draw    (181,83) -- (203,106) ;
%Straight Lines [id:da48627689074126645] 
\draw    (203,106) -- (359,107) ;
%Straight Lines [id:da49133266575001033] 
\draw    (360,83) -- (359,107) ;
%Straight Lines [id:da51467780889709] 
\draw    (203,106) -- (203,188) ;
%Straight Lines [id:da7007627180528562] 
\draw    (360,83) -- (432,1) ;
%Straight Lines [id:da7832199179696879] 
\draw    (359,107) -- (401,146) ;
%Straight Lines [id:da44753667002377306] 
\draw    (401,146) -- (604,147) ;
%Straight Lines [id:da6859330855524519] 
\draw    (604,147) -- (604,184) ;
%Straight Lines [id:da6255332306196266] 
\draw    (604,147) -- (654,84) ;
%Straight Lines [id:da6529064472478432] 
\draw    (401,146) -- (401,186) ;
%Curve Lines [id:da36308332329991666] 
\draw [color={rgb, 255:red, 0; green, 0; blue, 0 }  ,draw opacity=1 ]   (0.5,42) .. controls (32.5,40) and (32.5,44) .. (36.5,45) ;
%Curve Lines [id:da3951564775994856] 
\draw [color={rgb, 255:red, 0; green, 0; blue, 0 }  ,draw opacity=1 ]   (-1.5,50) .. controls (17.5,49) and (32.5,51) .. (36.5,45) ;
%Shape: Square [id:dp007266994236426072] 
\draw   (544,9) -- (594,9) -- (594,59) -- (544,59) -- cycle ;

% Text Node
\draw (554,10.4) node [anchor=north west][inner sep=0.75pt]  [font=\huge]  {$v_{1}$};
% Text Node
\draw (64,25.4) node [anchor=north west][inner sep=0.75pt]  [color={rgb, 255:red, 0; green, 0; blue, 0 }  ,opacity=1 ]  {$\textcolor[rgb]{0,0,0}{2}$};

\end{tikzpicture}}  
 & %  
     \resizebox{.4 \textwidth}{!}{%\tikzset{every picture/.style={line width=0.75pt}} %set default line width to 0.75pt        

\begin{tikzpicture}[x=0.75pt,y=0.75pt,yscale=-1,xscale=1]

%Shape: Circle [id:dp049478145904912285] 
\draw  [color={rgb, 255:red, 208; green, 2; blue, 27 }  ,draw opacity=1 ][fill={rgb, 255:red, 208; green, 2; blue, 27 }  ,fill opacity=1 ] (7.5,21) .. controls (7.5,18.24) and (9.74,16) .. (12.5,16) .. controls (15.26,16) and (17.5,18.24) .. (17.5,21) .. controls (17.5,23.76) and (15.26,26) .. (12.5,26) .. controls (9.74,26) and (7.5,23.76) .. (7.5,21) -- cycle ;
%Shape: Circle [id:dp597865605180764] 
\draw  [color={rgb, 255:red, 208; green, 2; blue, 27 }  ,draw opacity=1 ][fill={rgb, 255:red, 208; green, 2; blue, 27 }  ,fill opacity=1 ] (57,60.5) .. controls (57,57.74) and (59.24,55.5) .. (62,55.5) .. controls (64.76,55.5) and (67,57.74) .. (67,60.5) .. controls (67,63.26) and (64.76,65.5) .. (62,65.5) .. controls (59.24,65.5) and (57,63.26) .. (57,60.5) -- cycle ;
%Shape: Circle [id:dp03321279691471113] 
\draw  [color={rgb, 255:red, 208; green, 2; blue, 27 }  ,draw opacity=1 ][fill={rgb, 255:red, 208; green, 2; blue, 27 }  ,fill opacity=1 ] (145,72) .. controls (145,69.24) and (147.24,67) .. (150,67) .. controls (152.76,67) and (155,69.24) .. (155,72) .. controls (155,74.76) and (152.76,77) .. (150,77) .. controls (147.24,77) and (145,74.76) .. (145,72) -- cycle ;
%Straight Lines [id:da15855389453970536] 
\draw [color={rgb, 255:red, 0; green, 0; blue, 0 }  ,draw opacity=1 ]   (62,60.5) -- (120,61.5) ;
%Shape: Circle [id:dp7534753985751717] 
\draw  [color={rgb, 255:red, 208; green, 2; blue, 27 }  ,draw opacity=1 ][fill={rgb, 255:red, 208; green, 2; blue, 27 }  ,fill opacity=1 ] (212,84) .. controls (212,81.24) and (214.24,79) .. (217,79) .. controls (219.76,79) and (222,81.24) .. (222,84) .. controls (222,86.76) and (219.76,89) .. (217,89) .. controls (214.24,89) and (212,86.76) .. (212,84) -- cycle ;
%Shape: Circle [id:dp5270532448174468] 
\draw  [color={rgb, 255:red, 208; green, 2; blue, 27 }  ,draw opacity=1 ][fill={rgb, 255:red, 208; green, 2; blue, 27 }  ,fill opacity=1 ] (294,107) .. controls (294,104.24) and (296.24,102) .. (299,102) .. controls (301.76,102) and (304,104.24) .. (304,107) .. controls (304,109.76) and (301.76,112) .. (299,112) .. controls (296.24,112) and (294,109.76) .. (294,107) -- cycle ;
%Shape: Circle [id:dp7779686155492531] 
\draw  [color={rgb, 255:red, 208; green, 2; blue, 27 }  ,draw opacity=1 ][fill={rgb, 255:red, 208; green, 2; blue, 27 }  ,fill opacity=1 ] (375,126.5) .. controls (375,123.74) and (377.24,121.5) .. (380,121.5) .. controls (382.76,121.5) and (385,123.74) .. (385,126.5) .. controls (385,129.26) and (382.76,131.5) .. (380,131.5) .. controls (377.24,131.5) and (375,129.26) .. (375,126.5) -- cycle ;
%Shape: Circle [id:dp23212820006386536] 
\draw  [color={rgb, 255:red, 208; green, 2; blue, 27 }  ,draw opacity=1 ][fill={rgb, 255:red, 208; green, 2; blue, 27 }  ,fill opacity=1 ] (466.5,146) .. controls (466.5,143.24) and (468.74,141) .. (471.5,141) .. controls (474.26,141) and (476.5,143.24) .. (476.5,146) .. controls (476.5,148.76) and (474.26,151) .. (471.5,151) .. controls (468.74,151) and (466.5,148.76) .. (466.5,146) -- cycle ;
%Shape: Circle [id:dp5439300296757643] 
\draw  [color={rgb, 255:red, 208; green, 2; blue, 27 }  ,draw opacity=1 ][fill={rgb, 255:red, 208; green, 2; blue, 27 }  ,fill opacity=1 ] (599,166) .. controls (599,163.24) and (601.24,161) .. (604,161) .. controls (606.76,161) and (609,163.24) .. (609,166) .. controls (609,168.76) and (606.76,171) .. (604,171) .. controls (601.24,171) and (599,168.76) .. (599,166) -- cycle ;
%Straight Lines [id:da9939861155158032] 
\draw [color={rgb, 255:red, 0; green, 0; blue, 0 }  ,draw opacity=1 ]   (120,21) -- (7.5,21) ;
%Straight Lines [id:da6460141075147547] 
\draw [color={rgb, 255:red, 0; green, 0; blue, 0 }  ,draw opacity=1 ]   (120,21) -- (140,2) ;
%Straight Lines [id:da644383057111223] 
\draw    (120.5,61) -- (181,83) ;
%Straight Lines [id:da796824295382208] 
\draw    (181,83) -- (360,83) ;
%Straight Lines [id:da9753553652203049] 
\draw    (181,83) -- (203,106) ;
%Straight Lines [id:da48627689074126645] 
\draw    (203,106) -- (359,107) ;
%Straight Lines [id:da49133266575001033] 
\draw    (360,83) -- (359,107) ;
%Straight Lines [id:da51467780889709] 
\draw    (203,106) -- (203,188) ;
%Straight Lines [id:da7007627180528562] 
\draw    (360,83) -- (432,1) ;
%Straight Lines [id:da7832199179696879] 
\draw    (359,107) -- (401,146) ;
%Straight Lines [id:da44753667002377306] 
\draw    (401,146) -- (604,147) ;
%Straight Lines [id:da6859330855524519] 
\draw    (604,147) -- (604,184) ;
%Straight Lines [id:da6255332306196266] 
\draw    (604,147) -- (654,84) ;
%Straight Lines [id:da6529064472478432] 
\draw    (401,146) -- (401,186) ;
%Shape: Square [id:dp7171542865243234] 
\draw   (539,7) -- (589,7) -- (589,57) -- (539,57) -- cycle ;
%Straight Lines [id:da16444256483632191] 
\draw    (120,21) -- (120.5,61) ;
%Curve Lines [id:da7479580476171681] 
\draw [color={rgb, 255:red, 0; green, 0; blue, 0 }  ,draw opacity=1 ]   (6.5,66) .. controls (48,67) and (47.5,69) .. (62.5,60) ;
%Curve Lines [id:da7365465856601703] 
\draw [color={rgb, 255:red, 0; green, 0; blue, 0 }  ,draw opacity=1 ]   (5.5,54) .. controls (49.5,52) and (49.5,55) .. (62.5,60) ;

% Text Node
\draw (547.81,10.55) node [anchor=north west][inner sep=0.75pt]  [font=\huge,rotate=-358.58]  {$v_{2}$};

\end{tikzpicture}}  
 \\ \hline
     \resizebox{.4 \textwidth}{!}{%\tikzset{every picture/.style={line width=0.75pt}} %set default line width to 0.75pt        

\begin{tikzpicture}[x=0.75pt,y=0.75pt,yscale=-1,xscale=1]

%Shape: Circle [id:dp049478145904912285] 
\draw  [color={rgb, 255:red, 208; green, 2; blue, 27 }  ,draw opacity=1 ][fill={rgb, 255:red, 208; green, 2; blue, 27 }  ,fill opacity=1 ] (29,20) .. controls (29,17.24) and (31.24,15) .. (34,15) .. controls (36.76,15) and (39,17.24) .. (39,20) .. controls (39,22.76) and (36.76,25) .. (34,25) .. controls (31.24,25) and (29,22.76) .. (29,20) -- cycle ;
%Shape: Circle [id:dp597865605180764] 
\draw  [color={rgb, 255:red, 208; green, 2; blue, 27 }  ,draw opacity=1 ][fill={rgb, 255:red, 208; green, 2; blue, 27 }  ,fill opacity=1 ] (64,43) .. controls (64,40.24) and (66.24,38) .. (69,38) .. controls (71.76,38) and (74,40.24) .. (74,43) .. controls (74,45.76) and (71.76,48) .. (69,48) .. controls (66.24,48) and (64,45.76) .. (64,43) -- cycle ;
%Shape: Circle [id:dp03321279691471113] 
\draw  [color={rgb, 255:red, 208; green, 2; blue, 27 }  ,draw opacity=1 ][fill={rgb, 255:red, 208; green, 2; blue, 27 }  ,fill opacity=1 ] (145,72) .. controls (145,69.24) and (147.24,67) .. (150,67) .. controls (152.76,67) and (155,69.24) .. (155,72) .. controls (155,74.76) and (152.76,77) .. (150,77) .. controls (147.24,77) and (145,74.76) .. (145,72) -- cycle ;
%Straight Lines [id:da15855389453970536] 
\draw [color={rgb, 255:red, 0; green, 0; blue, 0 }  ,draw opacity=1 ]   (13,71) -- (150,72) ;
%Shape: Circle [id:dp7534753985751717] 
\draw  [color={rgb, 255:red, 208; green, 2; blue, 27 }  ,draw opacity=1 ][fill={rgb, 255:red, 208; green, 2; blue, 27 }  ,fill opacity=1 ] (212,84) .. controls (212,81.24) and (214.24,79) .. (217,79) .. controls (219.76,79) and (222,81.24) .. (222,84) .. controls (222,86.76) and (219.76,89) .. (217,89) .. controls (214.24,89) and (212,86.76) .. (212,84) -- cycle ;
%Shape: Circle [id:dp5270532448174468] 
\draw  [color={rgb, 255:red, 208; green, 2; blue, 27 }  ,draw opacity=1 ][fill={rgb, 255:red, 208; green, 2; blue, 27 }  ,fill opacity=1 ] (294,107) .. controls (294,104.24) and (296.24,102) .. (299,102) .. controls (301.76,102) and (304,104.24) .. (304,107) .. controls (304,109.76) and (301.76,112) .. (299,112) .. controls (296.24,112) and (294,109.76) .. (294,107) -- cycle ;
%Shape: Circle [id:dp7779686155492531] 
\draw  [color={rgb, 255:red, 208; green, 2; blue, 27 }  ,draw opacity=1 ][fill={rgb, 255:red, 208; green, 2; blue, 27 }  ,fill opacity=1 ] (375,126.5) .. controls (375,123.74) and (377.24,121.5) .. (380,121.5) .. controls (382.76,121.5) and (385,123.74) .. (385,126.5) .. controls (385,129.26) and (382.76,131.5) .. (380,131.5) .. controls (377.24,131.5) and (375,129.26) .. (375,126.5) -- cycle ;
%Shape: Circle [id:dp23212820006386536] 
\draw  [color={rgb, 255:red, 208; green, 2; blue, 27 }  ,draw opacity=1 ][fill={rgb, 255:red, 208; green, 2; blue, 27 }  ,fill opacity=1 ] (466.5,146) .. controls (466.5,143.24) and (468.74,141) .. (471.5,141) .. controls (474.26,141) and (476.5,143.24) .. (476.5,146) .. controls (476.5,148.76) and (474.26,151) .. (471.5,151) .. controls (468.74,151) and (466.5,148.76) .. (466.5,146) -- cycle ;
%Shape: Circle [id:dp5439300296757643] 
\draw  [color={rgb, 255:red, 208; green, 2; blue, 27 }  ,draw opacity=1 ][fill={rgb, 255:red, 208; green, 2; blue, 27 }  ,fill opacity=1 ] (599,166) .. controls (599,163.24) and (601.24,161) .. (604,161) .. controls (606.76,161) and (609,163.24) .. (609,166) .. controls (609,168.76) and (606.76,171) .. (604,171) .. controls (601.24,171) and (599,168.76) .. (599,166) -- cycle ;
%Straight Lines [id:da9939861155158032] 
\draw [color={rgb, 255:red, 0; green, 0; blue, 0 }  ,draw opacity=1 ]   (120,20) -- (13.5,19) ;
%Straight Lines [id:da6460141075147547] 
\draw [color={rgb, 255:red, 0; green, 0; blue, 0 }  ,draw opacity=1 ]   (120,20) -- (140,1) ;
%Straight Lines [id:da644383057111223] 
\draw [color={rgb, 255:red, 0; green, 0; blue, 0 }  ,draw opacity=1 ]   (150,72) -- (181,83) ;
%Straight Lines [id:da796824295382208] 
\draw    (181,83) -- (360,83) ;
%Straight Lines [id:da9753553652203049] 
\draw    (181,83) -- (203,106) ;
%Straight Lines [id:da48627689074126645] 
\draw    (203,106) -- (359,107) ;
%Straight Lines [id:da49133266575001033] 
\draw    (360,83) -- (359,107) ;
%Straight Lines [id:da51467780889709] 
\draw    (203,106) -- (203,188) ;
%Straight Lines [id:da7007627180528562] 
\draw    (360,83) -- (432,1) ;
%Straight Lines [id:da7832199179696879] 
\draw    (359,107) -- (401,146) ;
%Straight Lines [id:da44753667002377306] 
\draw    (401,146) -- (604,147) ;
%Straight Lines [id:da6859330855524519] 
\draw    (604,147) -- (604,184) ;
%Straight Lines [id:da6255332306196266] 
\draw    (604,147) -- (654,84) ;
%Straight Lines [id:da6529064472478432] 
\draw    (401,146) -- (401,186) ;
%Shape: Square [id:dp7171542865243234] 
\draw   (539,7) -- (589,7) -- (589,57) -- (539,57) -- cycle ;
%Straight Lines [id:da16444256483632191] 
\draw [color={rgb, 255:red, 0; green, 0; blue, 0 }  ,draw opacity=1 ][fill={rgb, 255:red, 74; green, 144; blue, 226 }  ,fill opacity=1 ]   (120,20) -- (121,43) ;
%Straight Lines [id:da666703974482969] 
\draw [color={rgb, 255:red, 0; green, 0; blue, 0 }  ,draw opacity=1 ]   (121,43) -- (150,72) ;
%Straight Lines [id:da0840655116527731] 
\draw    (122.75,44) -- (13,43) ;

% Text Node
\draw (547.81,9.55) node [anchor=north west][inner sep=0.75pt]  [font=\huge,rotate=-358.58]  {$v_{3}$};

\end{tikzpicture}}  
 & %  
     \resizebox{.4 \textwidth}{!}{%\tikzset{every picture/.style={line width=0.75pt}} %set default line width to 0.75pt        

\begin{tikzpicture}[x=0.75pt,y=0.75pt,yscale=-1,xscale=1]

%Shape: Circle [id:dp049478145904912285] 
\draw  [color={rgb, 255:red, 208; green, 2; blue, 27 }  ,draw opacity=1 ][fill={rgb, 255:red, 208; green, 2; blue, 27 }  ,fill opacity=1 ] (29,20) .. controls (29,17.24) and (31.24,15) .. (34,15) .. controls (36.76,15) and (39,17.24) .. (39,20) .. controls (39,22.76) and (36.76,25) .. (34,25) .. controls (31.24,25) and (29,22.76) .. (29,20) -- cycle ;
%Shape: Circle [id:dp597865605180764] 
\draw  [color={rgb, 255:red, 208; green, 2; blue, 27 }  ,draw opacity=1 ][fill={rgb, 255:red, 208; green, 2; blue, 27 }  ,fill opacity=1 ] (64,43) .. controls (64,40.24) and (66.24,38) .. (69,38) .. controls (71.76,38) and (74,40.24) .. (74,43) .. controls (74,45.76) and (71.76,48) .. (69,48) .. controls (66.24,48) and (64,45.76) .. (64,43) -- cycle ;
%Shape: Circle [id:dp03321279691471113] 
\draw  [color={rgb, 255:red, 208; green, 2; blue, 27 }  ,draw opacity=1 ][fill={rgb, 255:red, 208; green, 2; blue, 27 }  ,fill opacity=1 ] (106,66) .. controls (106,63.24) and (108.24,61) .. (111,61) .. controls (113.76,61) and (116,63.24) .. (116,66) .. controls (116,68.76) and (113.76,71) .. (111,71) .. controls (108.24,71) and (106,68.76) .. (106,66) -- cycle ;
%Straight Lines [id:da15855389453970536] 
\draw [color={rgb, 255:red, 0; green, 0; blue, 0 }  ,draw opacity=1 ]   (11,65) -- (128,65) ;
%Shape: Circle [id:dp7534753985751717] 
\draw  [color={rgb, 255:red, 208; green, 2; blue, 27 }  ,draw opacity=1 ][fill={rgb, 255:red, 208; green, 2; blue, 27 }  ,fill opacity=1 ] (177,84) .. controls (177,81.24) and (179.24,79) .. (182,79) .. controls (184.76,79) and (187,81.24) .. (187,84) .. controls (187,86.76) and (184.76,89) .. (182,89) .. controls (179.24,89) and (177,86.76) .. (177,84) -- cycle ;
%Shape: Circle [id:dp5270532448174468] 
\draw  [color={rgb, 255:red, 208; green, 2; blue, 27 }  ,draw opacity=1 ][fill={rgb, 255:red, 208; green, 2; blue, 27 }  ,fill opacity=1 ] (294,107) .. controls (294,104.24) and (296.24,102) .. (299,102) .. controls (301.76,102) and (304,104.24) .. (304,107) .. controls (304,109.76) and (301.76,112) .. (299,112) .. controls (296.24,112) and (294,109.76) .. (294,107) -- cycle ;
%Shape: Circle [id:dp7779686155492531] 
\draw  [color={rgb, 255:red, 208; green, 2; blue, 27 }  ,draw opacity=1 ][fill={rgb, 255:red, 208; green, 2; blue, 27 }  ,fill opacity=1 ] (375,126.5) .. controls (375,123.74) and (377.24,121.5) .. (380,121.5) .. controls (382.76,121.5) and (385,123.74) .. (385,126.5) .. controls (385,129.26) and (382.76,131.5) .. (380,131.5) .. controls (377.24,131.5) and (375,129.26) .. (375,126.5) -- cycle ;
%Shape: Circle [id:dp23212820006386536] 
\draw  [color={rgb, 255:red, 208; green, 2; blue, 27 }  ,draw opacity=1 ][fill={rgb, 255:red, 208; green, 2; blue, 27 }  ,fill opacity=1 ] (466.5,146) .. controls (466.5,143.24) and (468.74,141) .. (471.5,141) .. controls (474.26,141) and (476.5,143.24) .. (476.5,146) .. controls (476.5,148.76) and (474.26,151) .. (471.5,151) .. controls (468.74,151) and (466.5,148.76) .. (466.5,146) -- cycle ;
%Shape: Circle [id:dp5439300296757643] 
\draw  [color={rgb, 255:red, 208; green, 2; blue, 27 }  ,draw opacity=1 ][fill={rgb, 255:red, 208; green, 2; blue, 27 }  ,fill opacity=1 ] (599,166) .. controls (599,163.24) and (601.24,161) .. (604,161) .. controls (606.76,161) and (609,163.24) .. (609,166) .. controls (609,168.76) and (606.76,171) .. (604,171) .. controls (601.24,171) and (599,168.76) .. (599,166) -- cycle ;
%Straight Lines [id:da9939861155158032] 
\draw [color={rgb, 255:red, 0; green, 0; blue, 0 }  ,draw opacity=1 ]   (109,20) -- (12.5,19) ;
%Straight Lines [id:da6460141075147547] 
\draw [color={rgb, 255:red, 0; green, 0; blue, 0 }  ,draw opacity=1 ]   (109,20) -- (129,1) ;
%Straight Lines [id:da644383057111223] 
\draw [color={rgb, 255:red, 0; green, 0; blue, 0 }  ,draw opacity=1 ]   (108.5,43) -- (128,65) ;
%Straight Lines [id:da796824295382208] 
\draw    (181,83) -- (360,83) ;
%Straight Lines [id:da9753553652203049] 
\draw [color={rgb, 255:red, 0; green, 0; blue, 0 }  ,draw opacity=1 ]   (182,84) -- (204,107) ;
%Straight Lines [id:da48627689074126645] 
\draw    (203,106) -- (359,107) ;
%Straight Lines [id:da49133266575001033] 
\draw    (360,83) -- (359,107) ;
%Straight Lines [id:da51467780889709] 
\draw [color={rgb, 255:red, 0; green, 0; blue, 0 }  ,draw opacity=1 ]   (203,106) -- (203,188) ;
%Straight Lines [id:da7007627180528562] 
\draw    (360,83) -- (432,1) ;
%Straight Lines [id:da7832199179696879] 
\draw    (359,107) -- (401,146) ;
%Straight Lines [id:da44753667002377306] 
\draw    (401,146) -- (604,147) ;
%Straight Lines [id:da6859330855524519] 
\draw    (604,147) -- (604,184) ;
%Straight Lines [id:da6255332306196266] 
\draw    (604,147) -- (654,84) ;
%Straight Lines [id:da6529064472478432] 
\draw    (401,146) -- (401,186) ;
%Shape: Square [id:dp7171542865243234] 
\draw   (538,7) -- (588,7) -- (588,57) -- (538,57) -- cycle ;
%Straight Lines [id:da16444256483632191] 
\draw [color={rgb, 255:red, 0; green, 0; blue, 0 }  ,draw opacity=1 ]   (109,20) -- (108.5,43) ;
%Straight Lines [id:da0840655116527731] 
\draw    (108.5,43) -- (11,43) ;
%Straight Lines [id:da46996664213146677] 
\draw [color={rgb, 255:red, 0; green, 0; blue, 0 }  ,draw opacity=1 ]   (128,65) -- (181,83) ;

% Text Node
\draw (547.08,11.35) node [anchor=north west][inner sep=0.75pt]  [font=\huge,rotate=-358.58]  {$v_{4}$};

\end{tikzpicture}}  
 \\ \hline
     \resizebox{.4 \textwidth}{!}{%\tikzset{every picture/.style={line width=0.75pt}} %set default line width to 0.75pt        

\begin{tikzpicture}[x=0.75pt,y=0.75pt,yscale=-1,xscale=1]

%Shape: Circle [id:dp049478145904912285] 
\draw  [color={rgb, 255:red, 208; green, 2; blue, 27 }  ,draw opacity=1 ][fill={rgb, 255:red, 208; green, 2; blue, 27 }  ,fill opacity=1 ] (29,20) .. controls (29,17.24) and (31.24,15) .. (34,15) .. controls (36.76,15) and (39,17.24) .. (39,20) .. controls (39,22.76) and (36.76,25) .. (34,25) .. controls (31.24,25) and (29,22.76) .. (29,20) -- cycle ;
%Shape: Circle [id:dp597865605180764] 
\draw  [color={rgb, 255:red, 208; green, 2; blue, 27 }  ,draw opacity=1 ][fill={rgb, 255:red, 208; green, 2; blue, 27 }  ,fill opacity=1 ] (64,43) .. controls (64,40.24) and (66.24,38) .. (69,38) .. controls (71.76,38) and (74,40.24) .. (74,43) .. controls (74,45.76) and (71.76,48) .. (69,48) .. controls (66.24,48) and (64,45.76) .. (64,43) -- cycle ;
%Shape: Circle [id:dp03321279691471113] 
\draw  [color={rgb, 255:red, 208; green, 2; blue, 27 }  ,draw opacity=1 ][fill={rgb, 255:red, 208; green, 2; blue, 27 }  ,fill opacity=1 ] (106,65) .. controls (106,62.24) and (108.24,60) .. (111,60) .. controls (113.76,60) and (116,62.24) .. (116,65) .. controls (116,67.76) and (113.76,70) .. (111,70) .. controls (108.24,70) and (106,67.76) .. (106,65) -- cycle ;
%Straight Lines [id:da15855389453970536] 
\draw [color={rgb, 255:red, 0; green, 0; blue, 0 }  ,draw opacity=1 ]   (11,65) -- (128,65) ;
%Shape: Circle [id:dp7534753985751717] 
\draw  [color={rgb, 255:red, 208; green, 2; blue, 27 }  ,draw opacity=1 ][fill={rgb, 255:red, 208; green, 2; blue, 27 }  ,fill opacity=1 ] (160,78) .. controls (160,75.24) and (162.24,73) .. (165,73) .. controls (167.76,73) and (170,75.24) .. (170,78) .. controls (170,80.76) and (167.76,83) .. (165,83) .. controls (162.24,83) and (160,80.76) .. (160,78) -- cycle ;
%Shape: Circle [id:dp5270532448174468] 
\draw  [color={rgb, 255:red, 208; green, 2; blue, 27 }  ,draw opacity=1 ][fill={rgb, 255:red, 208; green, 2; blue, 27 }  ,fill opacity=1 ] (334.5,106.75) .. controls (334.5,103.99) and (336.74,101.75) .. (339.5,101.75) .. controls (342.26,101.75) and (344.5,103.99) .. (344.5,106.75) .. controls (344.5,109.51) and (342.26,111.75) .. (339.5,111.75) .. controls (336.74,111.75) and (334.5,109.51) .. (334.5,106.75) -- cycle ;
%Shape: Circle [id:dp7779686155492531] 
\draw  [color={rgb, 255:red, 208; green, 2; blue, 27 }  ,draw opacity=1 ][fill={rgb, 255:red, 208; green, 2; blue, 27 }  ,fill opacity=1 ] (375,126.5) .. controls (375,123.74) and (377.24,121.5) .. (380,121.5) .. controls (382.76,121.5) and (385,123.74) .. (385,126.5) .. controls (385,129.26) and (382.76,131.5) .. (380,131.5) .. controls (377.24,131.5) and (375,129.26) .. (375,126.5) -- cycle ;
%Shape: Circle [id:dp23212820006386536] 
\draw  [color={rgb, 255:red, 208; green, 2; blue, 27 }  ,draw opacity=1 ][fill={rgb, 255:red, 208; green, 2; blue, 27 }  ,fill opacity=1 ] (466.5,146) .. controls (466.5,143.24) and (468.74,141) .. (471.5,141) .. controls (474.26,141) and (476.5,143.24) .. (476.5,146) .. controls (476.5,148.76) and (474.26,151) .. (471.5,151) .. controls (468.74,151) and (466.5,148.76) .. (466.5,146) -- cycle ;
%Shape: Circle [id:dp5439300296757643] 
\draw  [color={rgb, 255:red, 208; green, 2; blue, 27 }  ,draw opacity=1 ][fill={rgb, 255:red, 208; green, 2; blue, 27 }  ,fill opacity=1 ] (599,166) .. controls (599,163.24) and (601.24,161) .. (604,161) .. controls (606.76,161) and (609,163.24) .. (609,166) .. controls (609,168.76) and (606.76,171) .. (604,171) .. controls (601.24,171) and (599,168.76) .. (599,166) -- cycle ;
%Straight Lines [id:da9939861155158032] 
\draw [color={rgb, 255:red, 0; green, 0; blue, 0 }  ,draw opacity=1 ]   (109,20) -- (12.5,19) ;
%Straight Lines [id:da6460141075147547] 
\draw [color={rgb, 255:red, 0; green, 0; blue, 0 }  ,draw opacity=1 ]   (109,20) -- (129,1) ;
%Straight Lines [id:da644383057111223] 
\draw [color={rgb, 255:red, 0; green, 0; blue, 0 }  ,draw opacity=1 ]   (108.5,43) -- (128,65) ;
%Straight Lines [id:da51467780889709] 
\draw [color={rgb, 255:red, 0; green, 0; blue, 0 }  ,draw opacity=1 ]   (246.5,107) -- (246.5,189) ;
%Straight Lines [id:da7007627180528562] 
\draw [color={rgb, 255:red, 0; green, 0; blue, 0 }  ,draw opacity=1 ]   (359,107) -- (445.5,3) ;
%Straight Lines [id:da7832199179696879] 
\draw [color={rgb, 255:red, 0; green, 0; blue, 0 }  ,draw opacity=1 ]   (359,107) -- (401,146) ;
%Straight Lines [id:da44753667002377306] 
\draw    (401,146) -- (604,147) ;
%Straight Lines [id:da6859330855524519] 
\draw    (604,147) -- (604,184) ;
%Straight Lines [id:da6255332306196266] 
\draw    (604,147) -- (654,84) ;
%Straight Lines [id:da6529064472478432] 
\draw    (401,146) -- (401,186) ;
%Shape: Square [id:dp7171542865243234] 
\draw   (539,7) -- (589,7) -- (589,57) -- (539,57) -- cycle ;
%Straight Lines [id:da16444256483632191] 
\draw [color={rgb, 255:red, 0; green, 0; blue, 0 }  ,draw opacity=1 ]   (109,20) -- (108.5,43) ;
%Straight Lines [id:da0840655116527731] 
\draw [color={rgb, 255:red, 0; green, 0; blue, 0 }  ,draw opacity=1 ]   (108.5,43) -- (11,43) ;
%Straight Lines [id:da46996664213146677] 
\draw [color={rgb, 255:red, 0; green, 0; blue, 0 }  ,draw opacity=1 ]   (128,65) -- (246.5,107) ;
%Shape: Ellipse [id:dp35430366760200704] 
\draw  [color={rgb, 255:red, 0; green, 0; blue, 0 }  ,draw opacity=1 ] (264,106.75) .. controls (264,104.71) and (280.9,103.06) .. (301.75,103.06) .. controls (322.6,103.06) and (339.5,104.71) .. (339.5,106.75) .. controls (339.5,108.79) and (322.6,110.44) .. (301.75,110.44) .. controls (280.9,110.44) and (264,108.79) .. (264,106.75) -- cycle ;
%Straight Lines [id:da9000303830718615] 
\draw [color={rgb, 255:red, 0; green, 0; blue, 0 }  ,draw opacity=1 ]   (246.5,107) -- (266,107.25) ;
%Straight Lines [id:da9862549752568031] 
\draw [color={rgb, 255:red, 0; green, 0; blue, 0 }  ,draw opacity=1 ]   (339.5,106.75) -- (359,107) ;

% Text Node
\draw (541.08,10.35) node [anchor=north west][inner sep=0.75pt]  [font=\huge,rotate=-358.58]  {$v_{5}$};
% Text Node
\draw (252,89.4) node [anchor=north west][inner sep=0.75pt]    {$2$};
% Text Node
\draw (343,88.4) node [anchor=north west][inner sep=0.75pt]    {$2$};

\end{tikzpicture}}  
 & %  
     \resizebox{.4 \textwidth}{!}{%\tikzset{every picture/.style={line width=0.75pt}} %set default line width to 0.75pt        

\begin{tikzpicture}[x=0.75pt,y=0.75pt,yscale=-1,xscale=1]

%Shape: Circle [id:dp049478145904912285] 
\draw  [color={rgb, 255:red, 208; green, 2; blue, 27 }  ,draw opacity=1 ][fill={rgb, 255:red, 208; green, 2; blue, 27 }  ,fill opacity=1 ] (29,20) .. controls (29,17.24) and (31.24,15) .. (34,15) .. controls (36.76,15) and (39,17.24) .. (39,20) .. controls (39,22.76) and (36.76,25) .. (34,25) .. controls (31.24,25) and (29,22.76) .. (29,20) -- cycle ;
%Shape: Circle [id:dp597865605180764] 
\draw  [color={rgb, 255:red, 208; green, 2; blue, 27 }  ,draw opacity=1 ][fill={rgb, 255:red, 208; green, 2; blue, 27 }  ,fill opacity=1 ] (64,43) .. controls (64,40.24) and (66.24,38) .. (69,38) .. controls (71.76,38) and (74,40.24) .. (74,43) .. controls (74,45.76) and (71.76,48) .. (69,48) .. controls (66.24,48) and (64,45.76) .. (64,43) -- cycle ;
%Shape: Circle [id:dp03321279691471113] 
\draw  [color={rgb, 255:red, 208; green, 2; blue, 27 }  ,draw opacity=1 ][fill={rgb, 255:red, 208; green, 2; blue, 27 }  ,fill opacity=1 ] (106,65) .. controls (106,62.24) and (108.24,60) .. (111,60) .. controls (113.76,60) and (116,62.24) .. (116,65) .. controls (116,67.76) and (113.76,70) .. (111,70) .. controls (108.24,70) and (106,67.76) .. (106,65) -- cycle ;
%Straight Lines [id:da15855389453970536] 
\draw [color={rgb, 255:red, 0; green, 0; blue, 0 }  ,draw opacity=1 ]   (11,65) -- (128,65) ;
%Shape: Circle [id:dp7534753985751717] 
\draw  [color={rgb, 255:red, 208; green, 2; blue, 27 }  ,draw opacity=1 ][fill={rgb, 255:red, 208; green, 2; blue, 27 }  ,fill opacity=1 ] (160,78) .. controls (160,75.24) and (162.24,73) .. (165,73) .. controls (167.76,73) and (170,75.24) .. (170,78) .. controls (170,80.76) and (167.76,83) .. (165,83) .. controls (162.24,83) and (160,80.76) .. (160,78) -- cycle ;
%Shape: Circle [id:dp7779686155492531] 
\draw  [color={rgb, 255:red, 208; green, 2; blue, 27 }  ,draw opacity=1 ][fill={rgb, 255:red, 208; green, 2; blue, 27 }  ,fill opacity=1 ] (375,126.5) .. controls (375,123.74) and (377.24,121.5) .. (380,121.5) .. controls (382.76,121.5) and (385,123.74) .. (385,126.5) .. controls (385,129.26) and (382.76,131.5) .. (380,131.5) .. controls (377.24,131.5) and (375,129.26) .. (375,126.5) -- cycle ;
%Shape: Circle [id:dp23212820006386536] 
\draw  [color={rgb, 255:red, 208; green, 2; blue, 27 }  ,draw opacity=1 ][fill={rgb, 255:red, 208; green, 2; blue, 27 }  ,fill opacity=1 ] (466.5,146) .. controls (466.5,143.24) and (468.74,141) .. (471.5,141) .. controls (474.26,141) and (476.5,143.24) .. (476.5,146) .. controls (476.5,148.76) and (474.26,151) .. (471.5,151) .. controls (468.74,151) and (466.5,148.76) .. (466.5,146) -- cycle ;
%Shape: Circle [id:dp5439300296757643] 
\draw  [color={rgb, 255:red, 208; green, 2; blue, 27 }  ,draw opacity=1 ][fill={rgb, 255:red, 208; green, 2; blue, 27 }  ,fill opacity=1 ] (599,166) .. controls (599,163.24) and (601.24,161) .. (604,161) .. controls (606.76,161) and (609,163.24) .. (609,166) .. controls (609,168.76) and (606.76,171) .. (604,171) .. controls (601.24,171) and (599,168.76) .. (599,166) -- cycle ;
%Straight Lines [id:da9939861155158032] 
\draw [color={rgb, 255:red, 0; green, 0; blue, 0 }  ,draw opacity=1 ]   (109,20) -- (12.5,19) ;
%Straight Lines [id:da6460141075147547] 
\draw [color={rgb, 255:red, 0; green, 0; blue, 0 }  ,draw opacity=1 ]   (109,20) -- (129,1) ;
%Straight Lines [id:da644383057111223] 
\draw [color={rgb, 255:red, 0; green, 0; blue, 0 }  ,draw opacity=1 ]   (108.5,43) -- (128,65) ;
%Straight Lines [id:da796824295382208] 
\draw [color={rgb, 255:red, 0; green, 0; blue, 0 }  ,draw opacity=1 ]   (181,83) -- (380.01,83.02) ;
%Straight Lines [id:da9753553652203049] 
\draw [color={rgb, 255:red, 0; green, 0; blue, 0 }  ,draw opacity=1 ]   (181,83) -- (224,125.5) ;
%Straight Lines [id:da48627689074126645] 
\draw [color={rgb, 255:red, 0; green, 0; blue, 0 }  ,draw opacity=1 ]   (224,125.5) -- (380,126.5) ;
%Straight Lines [id:da49133266575001033] 
\draw [color={rgb, 255:red, 0; green, 0; blue, 0 }  ,draw opacity=1 ]   (380,126.5) -- (380.01,83.02) ;
%Straight Lines [id:da51467780889709] 
\draw [color={rgb, 255:red, 0; green, 0; blue, 0 }  ,draw opacity=1 ]   (224,125.5) -- (224.01,184.02) ;
%Straight Lines [id:da7007627180528562] 
\draw    (380.01,83.02) -- (452.01,1.02) ;
%Straight Lines [id:da7832199179696879] 
\draw    (380,126.5) -- (401,146) ;
%Straight Lines [id:da44753667002377306] 
\draw    (401,146) -- (604,147) ;
%Straight Lines [id:da6859330855524519] 
\draw    (604,147) -- (604,184) ;
%Straight Lines [id:da6255332306196266] 
\draw    (604,147) -- (654,84) ;
%Straight Lines [id:da6529064472478432] 
\draw    (401,146) -- (401,186) ;
%Shape: Square [id:dp7171542865243234] 
\draw   (539,7) -- (589,7) -- (589,57) -- (539,57) -- cycle ;
%Straight Lines [id:da16444256483632191] 
\draw [color={rgb, 255:red, 0; green, 0; blue, 0 }  ,draw opacity=1 ]   (109,20) -- (108.5,43) ;
%Straight Lines [id:da0840655116527731] 
\draw [color={rgb, 255:red, 0; green, 0; blue, 0 }  ,draw opacity=1 ]   (108.5,43) -- (11,43) ;
%Straight Lines [id:da46996664213146677] 
\draw [color={rgb, 255:red, 0; green, 0; blue, 0 }  ,draw opacity=1 ]   (128,65) -- (181,83) ;
%Shape: Circle [id:dp7962742852398128] 
\draw  [color={rgb, 255:red, 208; green, 2; blue, 27 }  ,draw opacity=1 ][fill={rgb, 255:red, 208; green, 2; blue, 27 }  ,fill opacity=1 ] (296.01,83.01) .. controls (296.01,80.25) and (298.25,78.01) .. (301.01,78.01) .. controls (303.76,78.01) and (306,80.25) .. (306,83.01) .. controls (306,85.76) and (303.76,88) .. (301.01,88) .. controls (298.25,88) and (296.01,85.76) .. (296.01,83.01) -- cycle ;

% Text Node
\draw (547.08,11.35) node [anchor=north west][inner sep=0.75pt]  [font=\huge,rotate=-358.58]  {$v_{6}$};

\end{tikzpicture}}  
 \\ \hline
     \resizebox{.4 \textwidth}{!}{%\tikzset{every picture/.style={line width=0.75pt}} %set default line width to 0.75pt        

\begin{tikzpicture}[x=0.75pt,y=0.75pt,yscale=-1,xscale=1]

%Shape: Circle [id:dp049478145904912285] 
\draw  [color={rgb, 255:red, 208; green, 2; blue, 27 }  ,draw opacity=1 ][fill={rgb, 255:red, 208; green, 2; blue, 27 }  ,fill opacity=1 ] (29,20) .. controls (29,17.24) and (31.24,15) .. (34,15) .. controls (36.76,15) and (39,17.24) .. (39,20) .. controls (39,22.76) and (36.76,25) .. (34,25) .. controls (31.24,25) and (29,22.76) .. (29,20) -- cycle ;
%Shape: Circle [id:dp597865605180764] 
\draw  [color={rgb, 255:red, 208; green, 2; blue, 27 }  ,draw opacity=1 ][fill={rgb, 255:red, 208; green, 2; blue, 27 }  ,fill opacity=1 ] (64,43) .. controls (64,40.24) and (66.24,38) .. (69,38) .. controls (71.76,38) and (74,40.24) .. (74,43) .. controls (74,45.76) and (71.76,48) .. (69,48) .. controls (66.24,48) and (64,45.76) .. (64,43) -- cycle ;
%Shape: Circle [id:dp03321279691471113] 
\draw  [color={rgb, 255:red, 208; green, 2; blue, 27 }  ,draw opacity=1 ][fill={rgb, 255:red, 208; green, 2; blue, 27 }  ,fill opacity=1 ] (106,65) .. controls (106,62.24) and (108.24,60) .. (111,60) .. controls (113.76,60) and (116,62.24) .. (116,65) .. controls (116,67.76) and (113.76,70) .. (111,70) .. controls (108.24,70) and (106,67.76) .. (106,65) -- cycle ;
%Straight Lines [id:da15855389453970536] 
\draw [color={rgb, 255:red, 0; green, 0; blue, 0 }  ,draw opacity=1 ]   (11,65) -- (128,65) ;
%Shape: Circle [id:dp7534753985751717] 
\draw  [color={rgb, 255:red, 208; green, 2; blue, 27 }  ,draw opacity=1 ][fill={rgb, 255:red, 208; green, 2; blue, 27 }  ,fill opacity=1 ] (160,78) .. controls (160,75.24) and (162.24,73) .. (165,73) .. controls (167.76,73) and (170,75.24) .. (170,78) .. controls (170,80.76) and (167.76,83) .. (165,83) .. controls (162.24,83) and (160,80.76) .. (160,78) -- cycle ;
%Shape: Circle [id:dp7779686155492531] 
\draw  [color={rgb, 255:red, 208; green, 2; blue, 27 }  ,draw opacity=1 ][fill={rgb, 255:red, 208; green, 2; blue, 27 }  ,fill opacity=1 ] (359,126.5) .. controls (359,123.74) and (361.24,121.5) .. (364,121.5) .. controls (366.76,121.5) and (369,123.74) .. (369,126.5) .. controls (369,129.26) and (366.76,131.5) .. (364,131.5) .. controls (361.24,131.5) and (359,129.26) .. (359,126.5) -- cycle ;
%Shape: Circle [id:dp23212820006386536] 
\draw  [color={rgb, 255:red, 208; green, 2; blue, 27 }  ,draw opacity=1 ][fill={rgb, 255:red, 208; green, 2; blue, 27 }  ,fill opacity=1 ] (466.5,146) .. controls (466.5,143.24) and (468.74,141) .. (471.5,141) .. controls (474.26,141) and (476.5,143.24) .. (476.5,146) .. controls (476.5,148.76) and (474.26,151) .. (471.5,151) .. controls (468.74,151) and (466.5,148.76) .. (466.5,146) -- cycle ;
%Shape: Circle [id:dp5439300296757643] 
\draw  [color={rgb, 255:red, 208; green, 2; blue, 27 }  ,draw opacity=1 ][fill={rgb, 255:red, 208; green, 2; blue, 27 }  ,fill opacity=1 ] (599,166) .. controls (599,163.24) and (601.24,161) .. (604,161) .. controls (606.76,161) and (609,163.24) .. (609,166) .. controls (609,168.76) and (606.76,171) .. (604,171) .. controls (601.24,171) and (599,168.76) .. (599,166) -- cycle ;
%Straight Lines [id:da9939861155158032] 
\draw [color={rgb, 255:red, 0; green, 0; blue, 0 }  ,draw opacity=1 ]   (109,20) -- (12.5,19) ;
%Straight Lines [id:da6460141075147547] 
\draw [color={rgb, 255:red, 0; green, 0; blue, 0 }  ,draw opacity=1 ]   (109,20) -- (129,1) ;
%Straight Lines [id:da644383057111223] 
\draw [color={rgb, 255:red, 0; green, 0; blue, 0 }  ,draw opacity=1 ]   (108.5,43) -- (128,65) ;
%Straight Lines [id:da796824295382208] 
\draw [color={rgb, 255:red, 0; green, 0; blue, 0 }  ,draw opacity=1 ]   (181,83) -- (449.01,82.02) ;
%Straight Lines [id:da9753553652203049] 
\draw [color={rgb, 255:red, 0; green, 0; blue, 0 }  ,draw opacity=1 ]   (181,83) -- (224,125.5) ;
%Straight Lines [id:da48627689074126645] 
\draw [color={rgb, 255:red, 0; green, 0; blue, 0 }  ,draw opacity=1 ]   (224,125.5) -- (449,125.5) ;
%Straight Lines [id:da49133266575001033] 
\draw [color={rgb, 255:red, 0; green, 0; blue, 0 }  ,draw opacity=1 ]   (449,125.5) -- (449.01,82.02) ;
%Straight Lines [id:da51467780889709] 
\draw [color={rgb, 255:red, 0; green, 0; blue, 0 }  ,draw opacity=1 ]   (224,125.5) -- (224.01,184.02) ;
%Straight Lines [id:da7007627180528562] 
\draw [color={rgb, 255:red, 0; green, 0; blue, 0 }  ,draw opacity=1 ]   (449.01,82.02) -- (521.01,0.02) ;
%Straight Lines [id:da7832199179696879] 
\draw [color={rgb, 255:red, 0; green, 0; blue, 0 }  ,draw opacity=1 ]   (449,125.5) -- (470,145) ;
%Straight Lines [id:da44753667002377306] 
\draw    (471.5,146) -- (604,147) ;
%Straight Lines [id:da6859330855524519] 
\draw    (604,147) -- (604,184) ;
%Straight Lines [id:da6255332306196266] 
\draw    (604,147) -- (654,84) ;
%Straight Lines [id:da6529064472478432] 
\draw [color={rgb, 255:red, 0; green, 0; blue, 0 }  ,draw opacity=1 ]   (470,145) -- (470,185) ;
%Shape: Square [id:dp7171542865243234] 
\draw   (539,7) -- (589,7) -- (589,57) -- (539,57) -- cycle ;
%Straight Lines [id:da16444256483632191] 
\draw [color={rgb, 255:red, 0; green, 0; blue, 0 }  ,draw opacity=1 ]   (109,20) -- (108.5,43) ;
%Straight Lines [id:da0840655116527731] 
\draw [color={rgb, 255:red, 0; green, 0; blue, 0 }  ,draw opacity=1 ]   (108.5,43) -- (11,43) ;
%Straight Lines [id:da46996664213146677] 
\draw [color={rgb, 255:red, 0; green, 0; blue, 0 }  ,draw opacity=1 ]   (128,65) -- (181,83) ;
%Shape: Circle [id:dp09440932328926133] 
\draw  [color={rgb, 255:red, 208; green, 2; blue, 27 }  ,draw opacity=1 ][fill={rgb, 255:red, 208; green, 2; blue, 27 }  ,fill opacity=1 ] (258,84.5) .. controls (258,81.74) and (260.24,79.5) .. (263,79.5) .. controls (265.76,79.5) and (268,81.74) .. (268,84.5) .. controls (268,87.26) and (265.76,89.5) .. (263,89.5) .. controls (260.24,89.5) and (258,87.26) .. (258,84.5) -- cycle ;

% Text Node
\draw (548.08,11.35) node [anchor=north west][inner sep=0.75pt]  [font=\huge,rotate=-358.58]  {$v_{7}$};

\end{tikzpicture}}  
 & %  
     \resizebox{.4 \textwidth}{!}{%\tikzset{every picture/.style={line width=0.75pt}} %set default line width to 0.75pt        

\begin{tikzpicture}[x=0.75pt,y=0.75pt,yscale=-1,xscale=1]

%Shape: Circle [id:dp049478145904912285] 
\draw  [color={rgb, 255:red, 208; green, 2; blue, 27 }  ,draw opacity=1 ][fill={rgb, 255:red, 208; green, 2; blue, 27 }  ,fill opacity=1 ] (29,20) .. controls (29,17.24) and (31.24,15) .. (34,15) .. controls (36.76,15) and (39,17.24) .. (39,20) .. controls (39,22.76) and (36.76,25) .. (34,25) .. controls (31.24,25) and (29,22.76) .. (29,20) -- cycle ;
%Shape: Circle [id:dp597865605180764] 
\draw  [color={rgb, 255:red, 208; green, 2; blue, 27 }  ,draw opacity=1 ][fill={rgb, 255:red, 208; green, 2; blue, 27 }  ,fill opacity=1 ] (64,43) .. controls (64,40.24) and (66.24,38) .. (69,38) .. controls (71.76,38) and (74,40.24) .. (74,43) .. controls (74,45.76) and (71.76,48) .. (69,48) .. controls (66.24,48) and (64,45.76) .. (64,43) -- cycle ;
%Shape: Circle [id:dp03321279691471113] 
\draw  [color={rgb, 255:red, 208; green, 2; blue, 27 }  ,draw opacity=1 ][fill={rgb, 255:red, 208; green, 2; blue, 27 }  ,fill opacity=1 ] (106,65) .. controls (106,62.24) and (108.24,60) .. (111,60) .. controls (113.76,60) and (116,62.24) .. (116,65) .. controls (116,67.76) and (113.76,70) .. (111,70) .. controls (108.24,70) and (106,67.76) .. (106,65) -- cycle ;
%Straight Lines [id:da15855389453970536] 
\draw [color={rgb, 255:red, 0; green, 0; blue, 0 }  ,draw opacity=1 ]   (11,65) -- (128,65) ;
%Shape: Circle [id:dp7534753985751717] 
\draw  [color={rgb, 255:red, 208; green, 2; blue, 27 }  ,draw opacity=1 ][fill={rgb, 255:red, 208; green, 2; blue, 27 }  ,fill opacity=1 ] (160,78) .. controls (160,75.24) and (162.24,73) .. (165,73) .. controls (167.76,73) and (170,75.24) .. (170,78) .. controls (170,80.76) and (167.76,83) .. (165,83) .. controls (162.24,83) and (160,80.76) .. (160,78) -- cycle ;
%Shape: Circle [id:dp7779686155492531] 
\draw  [color={rgb, 255:red, 208; green, 2; blue, 27 }  ,draw opacity=1 ][fill={rgb, 255:red, 208; green, 2; blue, 27 }  ,fill opacity=1 ] (359,126.5) .. controls (359,123.74) and (361.24,121.5) .. (364,121.5) .. controls (366.76,121.5) and (369,123.74) .. (369,126.5) .. controls (369,129.26) and (366.76,131.5) .. (364,131.5) .. controls (361.24,131.5) and (359,129.26) .. (359,126.5) -- cycle ;
%Shape: Circle [id:dp23212820006386536] 
\draw  [color={rgb, 255:red, 208; green, 2; blue, 27 }  ,draw opacity=1 ][fill={rgb, 255:red, 208; green, 2; blue, 27 }  ,fill opacity=1 ] (455.25,135.75) .. controls (455.25,132.99) and (457.49,130.75) .. (460.25,130.75) .. controls (463.01,130.75) and (465.25,132.99) .. (465.25,135.75) .. controls (465.25,138.51) and (463.01,140.75) .. (460.25,140.75) .. controls (457.49,140.75) and (455.25,138.51) .. (455.25,135.75) -- cycle ;
%Shape: Circle [id:dp5439300296757643] 
\draw  [color={rgb, 255:red, 208; green, 2; blue, 27 }  ,draw opacity=1 ][fill={rgb, 255:red, 208; green, 2; blue, 27 }  ,fill opacity=1 ] (599,147) .. controls (599,144.24) and (601.24,142) .. (604,142) .. controls (606.76,142) and (609,144.24) .. (609,147) .. controls (609,149.76) and (606.76,152) .. (604,152) .. controls (601.24,152) and (599,149.76) .. (599,147) -- cycle ;
%Straight Lines [id:da9939861155158032] 
\draw [color={rgb, 255:red, 0; green, 0; blue, 0 }  ,draw opacity=1 ]   (109,20) -- (12.5,19) ;
%Straight Lines [id:da6460141075147547] 
\draw [color={rgb, 255:red, 0; green, 0; blue, 0 }  ,draw opacity=1 ]   (109,20) -- (129,1) ;
%Straight Lines [id:da644383057111223] 
\draw [color={rgb, 255:red, 0; green, 0; blue, 0 }  ,draw opacity=1 ]   (108.5,43) -- (128,65) ;
%Straight Lines [id:da796824295382208] 
\draw [color={rgb, 255:red, 0; green, 0; blue, 0 }  ,draw opacity=1 ]   (181,83) -- (449.01,82.02) ;
%Straight Lines [id:da9753553652203049] 
\draw [color={rgb, 255:red, 0; green, 0; blue, 0 }  ,draw opacity=1 ]   (181,83) -- (224,125.5) ;
%Straight Lines [id:da48627689074126645] 
\draw [color={rgb, 255:red, 0; green, 0; blue, 0 }  ,draw opacity=1 ]   (224,125.5) -- (449,125.5) ;
%Straight Lines [id:da49133266575001033] 
\draw [color={rgb, 255:red, 0; green, 0; blue, 0 }  ,draw opacity=1 ]   (449,125.5) -- (449.01,82.02) ;
%Straight Lines [id:da51467780889709] 
\draw [color={rgb, 255:red, 0; green, 0; blue, 0 }  ,draw opacity=1 ]   (224,125.5) -- (224.01,184.02) ;
%Straight Lines [id:da7007627180528562] 
\draw [color={rgb, 255:red, 0; green, 0; blue, 0 }  ,draw opacity=1 ]   (449.01,82.02) -- (521.01,0.02) ;
%Straight Lines [id:da7832199179696879] 
\draw [color={rgb, 255:red, 0; green, 0; blue, 0 }  ,draw opacity=1 ]   (449,125.5) -- (471.5,146) ;
%Straight Lines [id:da44753667002377306] 
\draw    (471.5,146) -- (604,147) ;
%Straight Lines [id:da6859330855524519] 
\draw    (604,147) -- (604,184) ;
%Straight Lines [id:da6255332306196266] 
\draw    (604,147) -- (654,84) ;
%Straight Lines [id:da6529064472478432] 
\draw [color={rgb, 255:red, 0; green, 0; blue, 0 }  ,draw opacity=1 ]   (471.5,146) -- (471.5,191) ;
%Shape: Square [id:dp7171542865243234] 
\draw   (539,7) -- (589,7) -- (589,57) -- (539,57) -- cycle ;
%Straight Lines [id:da16444256483632191] 
\draw [color={rgb, 255:red, 0; green, 0; blue, 0 }  ,draw opacity=1 ]   (109,20) -- (108.5,43) ;
%Straight Lines [id:da0840655116527731] 
\draw [color={rgb, 255:red, 0; green, 0; blue, 0 }  ,draw opacity=1 ]   (108.5,43) -- (11,43) ;
%Straight Lines [id:da46996664213146677] 
\draw [color={rgb, 255:red, 0; green, 0; blue, 0 }  ,draw opacity=1 ]   (128,65) -- (181,83) ;
%Shape: Circle [id:dp09440932328926133] 
\draw  [color={rgb, 255:red, 208; green, 2; blue, 27 }  ,draw opacity=1 ][fill={rgb, 255:red, 208; green, 2; blue, 27 }  ,fill opacity=1 ] (258,84.5) .. controls (258,81.74) and (260.24,79.5) .. (263,79.5) .. controls (265.76,79.5) and (268,81.74) .. (268,84.5) .. controls (268,87.26) and (265.76,89.5) .. (263,89.5) .. controls (260.24,89.5) and (258,87.26) .. (258,84.5) -- cycle ;

% Text Node
\draw (547.08,12.35) node [anchor=north west][inner sep=0.75pt]  [font=\huge,rotate=-358.58]  {$v_{8}$};

\end{tikzpicture}}  
 \\ \end{tabular}
\caption{Vertices of $B$ corresponding to where one of the the marks sits at a four-valent vertex.}
\label{fig:vertmar}
\end{figure}
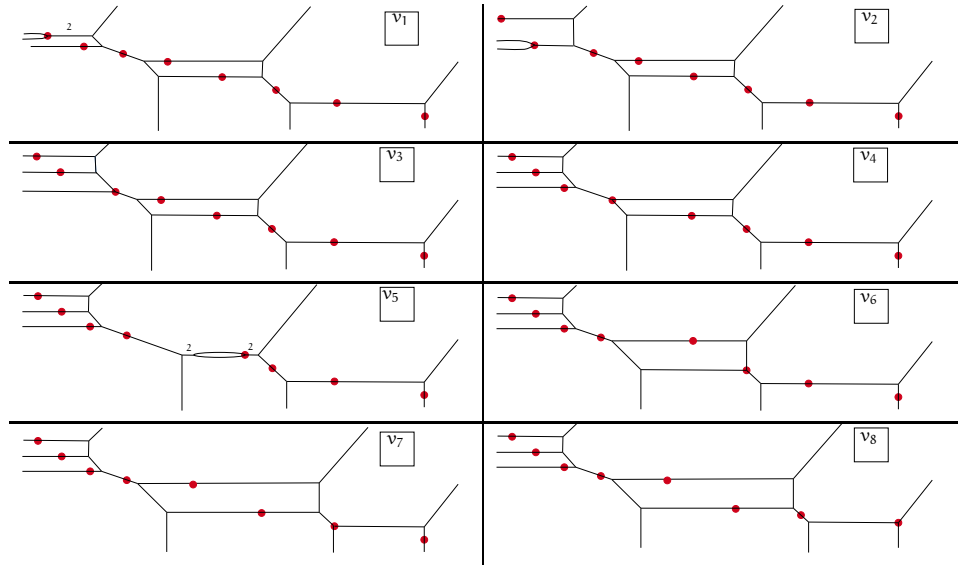

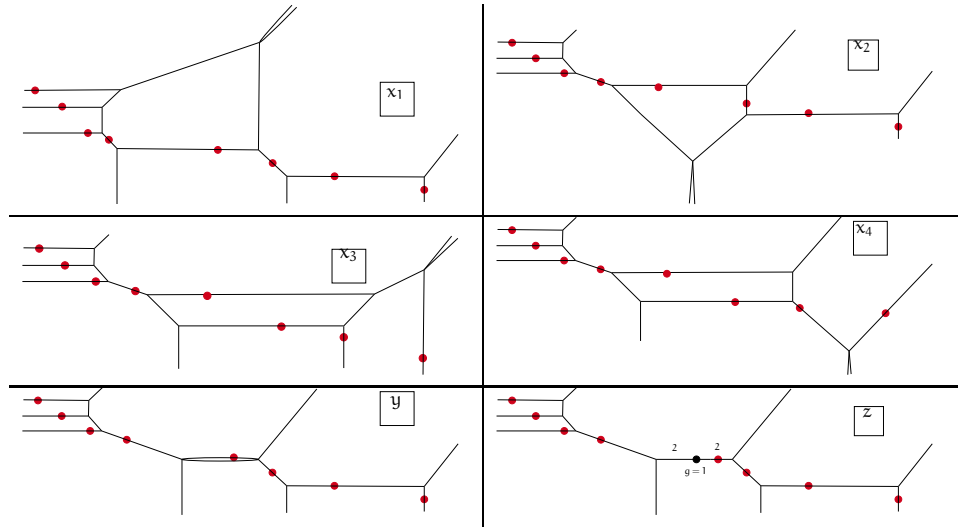
\begin{figure}[h]
\begin{tabular}{c|c}
     \resizebox{.4 \textwidth}{!}{%\tikzset{every picture/.style={line width=0.75pt}} %set default line width to 0.75pt        

\begin{tikzpicture}[x=0.75pt,y=0.75pt,yscale=-1,xscale=1]
%uncomment if require: \path (0,488); %set diagram left start at 0, and has height of 488

%Shape: Circle [id:dp049478145904912285] 
\draw  [color={rgb, 255:red, 208; green, 2; blue, 27 }  ,draw opacity=1 ][fill={rgb, 255:red, 208; green, 2; blue, 27 }  ,fill opacity=1 ] (20,196) .. controls (20,193.24) and (22.24,191) .. (25,191) .. controls (27.76,191) and (30,193.24) .. (30,196) .. controls (30,198.76) and (27.76,201) .. (25,201) .. controls (22.24,201) and (20,198.76) .. (20,196) -- cycle ;
%Shape: Circle [id:dp597865605180764] 
\draw  [color={rgb, 255:red, 208; green, 2; blue, 27 }  ,draw opacity=1 ][fill={rgb, 255:red, 208; green, 2; blue, 27 }  ,fill opacity=1 ] (60,220) .. controls (60,217.24) and (62.24,215) .. (65,215) .. controls (67.76,215) and (70,217.24) .. (70,220) .. controls (70,222.76) and (67.76,225) .. (65,225) .. controls (62.24,225) and (60,222.76) .. (60,220) -- cycle ;
%Shape: Circle [id:dp03321279691471113] 
\draw  [color={rgb, 255:red, 208; green, 2; blue, 27 }  ,draw opacity=1 ][fill={rgb, 255:red, 208; green, 2; blue, 27 }  ,fill opacity=1 ] (98,259) .. controls (98,256.24) and (100.24,254) .. (103,254) .. controls (105.76,254) and (108,256.24) .. (108,259) .. controls (108,261.76) and (105.76,264) .. (103,264) .. controls (100.24,264) and (98,261.76) .. (98,259) -- cycle ;
%Straight Lines [id:da15855389453970536] 
\draw [color={rgb, 255:red, 0; green, 0; blue, 0 }  ,draw opacity=1 ]   (7,259) -- (124,259) ;
%Shape: Circle [id:dp7534753985751717] 
\draw  [color={rgb, 255:red, 208; green, 2; blue, 27 }  ,draw opacity=1 ][fill={rgb, 255:red, 208; green, 2; blue, 27 }  ,fill opacity=1 ] (129,269) .. controls (129,266.24) and (131.24,264) .. (134,264) .. controls (136.76,264) and (139,266.24) .. (139,269) .. controls (139,271.76) and (136.76,274) .. (134,274) .. controls (131.24,274) and (129,271.76) .. (129,269) -- cycle ;
%Shape: Circle [id:dp5270532448174468] 
\draw  [color={rgb, 255:red, 208; green, 2; blue, 27 }  ,draw opacity=1 ][fill={rgb, 255:red, 208; green, 2; blue, 27 }  ,fill opacity=1 ] (290,284) .. controls (290,281.24) and (292.24,279) .. (295,279) .. controls (297.76,279) and (300,281.24) .. (300,284) .. controls (300,286.76) and (297.76,289) .. (295,289) .. controls (292.24,289) and (290,286.76) .. (290,284) -- cycle ;
%Shape: Circle [id:dp7779686155492531] 
\draw  [color={rgb, 255:red, 208; green, 2; blue, 27 }  ,draw opacity=1 ][fill={rgb, 255:red, 208; green, 2; blue, 27 }  ,fill opacity=1 ] (371,303.5) .. controls (371,300.74) and (373.24,298.5) .. (376,298.5) .. controls (378.76,298.5) and (381,300.74) .. (381,303.5) .. controls (381,306.26) and (378.76,308.5) .. (376,308.5) .. controls (373.24,308.5) and (371,306.26) .. (371,303.5) -- cycle ;
%Shape: Circle [id:dp23212820006386536] 
\draw  [color={rgb, 255:red, 208; green, 2; blue, 27 }  ,draw opacity=1 ][fill={rgb, 255:red, 208; green, 2; blue, 27 }  ,fill opacity=1 ] (462.5,323) .. controls (462.5,320.24) and (464.74,318) .. (467.5,318) .. controls (470.26,318) and (472.5,320.24) .. (472.5,323) .. controls (472.5,325.76) and (470.26,328) .. (467.5,328) .. controls (464.74,328) and (462.5,325.76) .. (462.5,323) -- cycle ;
%Shape: Circle [id:dp5439300296757643] 
\draw  [color={rgb, 255:red, 208; green, 2; blue, 27 }  ,draw opacity=1 ][fill={rgb, 255:red, 208; green, 2; blue, 27 }  ,fill opacity=1 ] (595,343) .. controls (595,340.24) and (597.24,338) .. (600,338) .. controls (602.76,338) and (605,340.24) .. (605,343) .. controls (605,345.76) and (602.76,348) .. (600,348) .. controls (597.24,348) and (595,345.76) .. (595,343) -- cycle ;
%Straight Lines [id:da9939861155158032] 
\draw [color={rgb, 255:red, 0; green, 0; blue, 0 }  ,draw opacity=1 ]   (151.5,195) -- (9,196) ;
%Straight Lines [id:da6460141075147547] 
\draw [color={rgb, 255:red, 0; green, 0; blue, 0 }  ,draw opacity=1 ]   (124,221) -- (135.8,209.79) -- (151.5,195) ;
%Straight Lines [id:da9753553652203049] 
\draw [color={rgb, 255:red, 0; green, 0; blue, 0 }  ,draw opacity=1 ]   (124,259) -- (146,282) ;
%Straight Lines [id:da48627689074126645] 
\draw    (146,282) -- (355,284) ;
%Straight Lines [id:da49133266575001033] 
\draw [color={rgb, 255:red, 0; green, 0; blue, 0 }  ,draw opacity=1 ]   (356.5,125) -- (355,284) ;
%Straight Lines [id:da51467780889709] 
\draw [color={rgb, 255:red, 0; green, 0; blue, 0 }  ,draw opacity=1 ]   (146,282) -- (146,364) ;
%Straight Lines [id:da7832199179696879] 
\draw    (355,284) -- (397,323) ;
%Straight Lines [id:da44753667002377306] 
\draw    (397,323) -- (600,324) ;
%Straight Lines [id:da6859330855524519] 
\draw    (600,324) -- (600,361) ;
%Straight Lines [id:da6255332306196266] 
\draw    (600,324) -- (650,261) ;
%Straight Lines [id:da6529064472478432] 
\draw    (397,323) -- (397,363) ;
%Shape: Square [id:dp7171542865243234] 
\draw   (535,184) -- (585,184) -- (585,234) -- (535,234) -- cycle ;
%Straight Lines [id:da0840655116527731] 
\draw [color={rgb, 255:red, 0; green, 0; blue, 0 }  ,draw opacity=1 ]   (124,221) -- (6.13,221) ;
%Straight Lines [id:da46996664213146677] 
\draw [color={rgb, 255:red, 0; green, 0; blue, 0 }  ,draw opacity=1 ]   (124,221) -- (124,259) ;
%Straight Lines [id:da11899879960415971] 
\draw [color={rgb, 255:red, 0; green, 0; blue, 0 }  ,draw opacity=1 ]   (151.5,195) -- (356.5,125) ;
%Curve Lines [id:da8864369919904196] 
\draw    (404.5,70) .. controls (347.5,131) and (332.5,153) .. (411.5,73) ;

% Text Node
\draw (543.08,191.35) node [anchor=north west][inner sep=0.75pt]  [font=\huge,rotate=-358.58]  {$x_{1}$};

\end{tikzpicture}}  
 & %  
     \resizebox{.4 \textwidth}{!}{%\tikzset{every picture/.style={line width=0.75pt}} %set default line width to 0.75pt        

\begin{tikzpicture}[x=0.75pt,y=0.75pt,yscale=-1,xscale=1]
%uncomment if require: \path (0,504); %set diagram left start at 0, and has height of 504

%Shape: Circle [id:dp049478145904912285] 
\draw  [color={rgb, 255:red, 208; green, 2; blue, 27 }  ,draw opacity=1 ][fill={rgb, 255:red, 208; green, 2; blue, 27 }  ,fill opacity=1 ] (29,20) .. controls (29,17.24) and (31.24,15) .. (34,15) .. controls (36.76,15) and (39,17.24) .. (39,20) .. controls (39,22.76) and (36.76,25) .. (34,25) .. controls (31.24,25) and (29,22.76) .. (29,20) -- cycle ;
%Shape: Circle [id:dp597865605180764] 
\draw  [color={rgb, 255:red, 208; green, 2; blue, 27 }  ,draw opacity=1 ][fill={rgb, 255:red, 208; green, 2; blue, 27 }  ,fill opacity=1 ] (64,43) .. controls (64,40.24) and (66.24,38) .. (69,38) .. controls (71.76,38) and (74,40.24) .. (74,43) .. controls (74,45.76) and (71.76,48) .. (69,48) .. controls (66.24,48) and (64,45.76) .. (64,43) -- cycle ;
%Shape: Circle [id:dp03321279691471113] 
\draw  [color={rgb, 255:red, 208; green, 2; blue, 27 }  ,draw opacity=1 ][fill={rgb, 255:red, 208; green, 2; blue, 27 }  ,fill opacity=1 ] (106,65) .. controls (106,62.24) and (108.24,60) .. (111,60) .. controls (113.76,60) and (116,62.24) .. (116,65) .. controls (116,67.76) and (113.76,70) .. (111,70) .. controls (108.24,70) and (106,67.76) .. (106,65) -- cycle ;
%Straight Lines [id:da15855389453970536] 
\draw [color={rgb, 255:red, 0; green, 0; blue, 0 }  ,draw opacity=1 ]   (11,65) -- (128,65) ;
%Shape: Circle [id:dp7534753985751717] 
\draw  [color={rgb, 255:red, 208; green, 2; blue, 27 }  ,draw opacity=1 ][fill={rgb, 255:red, 208; green, 2; blue, 27 }  ,fill opacity=1 ] (160,78) .. controls (160,75.24) and (162.24,73) .. (165,73) .. controls (167.76,73) and (170,75.24) .. (170,78) .. controls (170,80.76) and (167.76,83) .. (165,83) .. controls (162.24,83) and (160,80.76) .. (160,78) -- cycle ;
%Shape: Circle [id:dp7779686155492531] 
\draw  [color={rgb, 255:red, 208; green, 2; blue, 27 }  ,draw opacity=1 ][fill={rgb, 255:red, 208; green, 2; blue, 27 }  ,fill opacity=1 ] (375.01,109.76) .. controls (375.01,107) and (377.24,104.76) .. (380.01,104.76) .. controls (382.77,104.76) and (385.01,107) .. (385.01,109.76) .. controls (385.01,112.52) and (382.77,114.76) .. (380.01,114.76) .. controls (377.24,114.76) and (375.01,112.52) .. (375.01,109.76) -- cycle ;
%Shape: Circle [id:dp23212820006386536] 
\draw  [color={rgb, 255:red, 208; green, 2; blue, 27 }  ,draw opacity=1 ][fill={rgb, 255:red, 208; green, 2; blue, 27 }  ,fill opacity=1 ] (466.5,124) .. controls (466.5,121.24) and (468.74,119) .. (471.5,119) .. controls (474.26,119) and (476.5,121.24) .. (476.5,124) .. controls (476.5,126.76) and (474.26,129) .. (471.5,129) .. controls (468.74,129) and (466.5,126.76) .. (466.5,124) -- cycle ;
%Shape: Circle [id:dp5439300296757643] 
\draw  [color={rgb, 255:red, 208; green, 2; blue, 27 }  ,draw opacity=1 ][fill={rgb, 255:red, 208; green, 2; blue, 27 }  ,fill opacity=1 ] (599,144) .. controls (599,141.24) and (601.24,139) .. (604,139) .. controls (606.76,139) and (609,141.24) .. (609,144) .. controls (609,146.76) and (606.76,149) .. (604,149) .. controls (601.24,149) and (599,146.76) .. (599,144) -- cycle ;
%Straight Lines [id:da9939861155158032] 
\draw [color={rgb, 255:red, 0; green, 0; blue, 0 }  ,draw opacity=1 ]   (109,20) -- (12.5,19) ;
%Straight Lines [id:da6460141075147547] 
\draw [color={rgb, 255:red, 0; green, 0; blue, 0 }  ,draw opacity=1 ]   (109,20) -- (129,1) ;
%Straight Lines [id:da644383057111223] 
\draw [color={rgb, 255:red, 0; green, 0; blue, 0 }  ,draw opacity=1 ]   (108.5,43) -- (128,65) ;
%Straight Lines [id:da796824295382208] 
\draw [color={rgb, 255:red, 0; green, 0; blue, 0 }  ,draw opacity=1 ]   (181,83) -- (380.01,83.02) ;
%Straight Lines [id:da9753553652203049] 
\draw [color={rgb, 255:red, 0; green, 0; blue, 0 }  ,draw opacity=1 ]   (181,83) -- (224,125.5) ;
%Straight Lines [id:da49133266575001033] 
\draw [color={rgb, 255:red, 0; green, 0; blue, 0 }  ,draw opacity=1 ]   (380,126.5) -- (380.01,83.02) ;
%Straight Lines [id:da51467780889709] 
\draw [color={rgb, 255:red, 0; green, 0; blue, 0 }  ,draw opacity=1 ]   (224,125.5) -- (300.51,195.02) ;
%Straight Lines [id:da7007627180528562] 
\draw [color={rgb, 255:red, 0; green, 0; blue, 0 }  ,draw opacity=1 ]   (380.01,83.02) -- (452.01,1.02) ;
%Straight Lines [id:da44753667002377306] 
\draw [color={rgb, 255:red, 0; green, 0; blue, 0 }  ,draw opacity=1 ]   (380,126.5) -- (604,125) ;
%Straight Lines [id:da6859330855524519] 
\draw    (604,125) -- (604,162) ;
%Straight Lines [id:da6255332306196266] 
\draw    (604,125) -- (654,62) ;
%Straight Lines [id:da6529064472478432] 
\draw [color={rgb, 255:red, 0; green, 0; blue, 0 }  ,draw opacity=1 ]   (380,126.5) ;
%Shape: Square [id:dp7171542865243234] 
\draw   (530,16) -- (574.51,16) -- (574.51,60.51) -- (530,60.51) -- cycle ;
%Straight Lines [id:da16444256483632191] 
\draw [color={rgb, 255:red, 0; green, 0; blue, 0 }  ,draw opacity=1 ]   (109,20) -- (108.5,43) ;
%Straight Lines [id:da0840655116527731] 
\draw [color={rgb, 255:red, 0; green, 0; blue, 0 }  ,draw opacity=1 ]   (108.5,43) -- (11,43) ;
%Straight Lines [id:da46996664213146677] 
\draw [color={rgb, 255:red, 0; green, 0; blue, 0 }  ,draw opacity=1 ]   (128,65) -- (181,83) ;
%Shape: Circle [id:dp6845779969407073] 
\draw  [color={rgb, 255:red, 208; green, 2; blue, 27 }  ,draw opacity=1 ][fill={rgb, 255:red, 208; green, 2; blue, 27 }  ,fill opacity=1 ] (245.01,85.76) .. controls (245.01,83) and (247.24,80.76) .. (250.01,80.76) .. controls (252.77,80.76) and (255.01,83) .. (255.01,85.76) .. controls (255.01,88.52) and (252.77,90.76) .. (250.01,90.76) .. controls (247.24,90.76) and (245.01,88.52) .. (245.01,85.76) -- cycle ;
%Straight Lines [id:da7969571271225022] 
\draw [color={rgb, 255:red, 0; green, 0; blue, 0 }  ,draw opacity=1 ]   (300.51,195.02) -- (303.51,258.27) ;
%Straight Lines [id:da2654805906625559] 
\draw [color={rgb, 255:red, 0; green, 0; blue, 0 }  ,draw opacity=1 ]   (300.51,195.02) -- (380,126.5) ;
%Straight Lines [id:da8271800511888983] 
\draw [color={rgb, 255:red, 0; green, 0; blue, 0 }  ,draw opacity=1 ]   (300.51,195.02) -- (295.51,257.27) ;

% Text Node
\draw (536.08,15.35) node [anchor=north west][inner sep=0.75pt]  [font=\huge,rotate=-358.58]  {$x_{2}$};

\end{tikzpicture}}  
 \\ \hline
     \resizebox{.4 \textwidth}{!}{%\tikzset{every picture/.style={line width=0.75pt}} %set default line width to 0.75pt        

\begin{tikzpicture}[x=0.75pt,y=0.75pt,yscale=-1,xscale=1]

%Shape: Circle [id:dp049478145904912285] 
\draw  [color={rgb, 255:red, 208; green, 2; blue, 27 }  ,draw opacity=1 ][fill={rgb, 255:red, 208; green, 2; blue, 27 }  ,fill opacity=1 ] (29,20) .. controls (29,17.24) and (31.24,15) .. (34,15) .. controls (36.76,15) and (39,17.24) .. (39,20) .. controls (39,22.76) and (36.76,25) .. (34,25) .. controls (31.24,25) and (29,22.76) .. (29,20) -- cycle ;
%Shape: Circle [id:dp597865605180764] 
\draw  [color={rgb, 255:red, 208; green, 2; blue, 27 }  ,draw opacity=1 ][fill={rgb, 255:red, 208; green, 2; blue, 27 }  ,fill opacity=1 ] (64,43) .. controls (64,40.24) and (66.24,38) .. (69,38) .. controls (71.76,38) and (74,40.24) .. (74,43) .. controls (74,45.76) and (71.76,48) .. (69,48) .. controls (66.24,48) and (64,45.76) .. (64,43) -- cycle ;
%Shape: Circle [id:dp03321279691471113] 
\draw  [color={rgb, 255:red, 208; green, 2; blue, 27 }  ,draw opacity=1 ][fill={rgb, 255:red, 208; green, 2; blue, 27 }  ,fill opacity=1 ] (106,65) .. controls (106,62.24) and (108.24,60) .. (111,60) .. controls (113.76,60) and (116,62.24) .. (116,65) .. controls (116,67.76) and (113.76,70) .. (111,70) .. controls (108.24,70) and (106,67.76) .. (106,65) -- cycle ;
%Straight Lines [id:da15855389453970536] 
\draw [color={rgb, 255:red, 0; green, 0; blue, 0 }  ,draw opacity=1 ]   (11,65) -- (128,65) ;
%Shape: Circle [id:dp7534753985751717] 
\draw  [color={rgb, 255:red, 208; green, 2; blue, 27 }  ,draw opacity=1 ][fill={rgb, 255:red, 208; green, 2; blue, 27 }  ,fill opacity=1 ] (160,78) .. controls (160,75.24) and (162.24,73) .. (165,73) .. controls (167.76,73) and (170,75.24) .. (170,78) .. controls (170,80.76) and (167.76,83) .. (165,83) .. controls (162.24,83) and (160,80.76) .. (160,78) -- cycle ;
%Shape: Circle [id:dp7779686155492531] 
\draw  [color={rgb, 255:red, 208; green, 2; blue, 27 }  ,draw opacity=1 ][fill={rgb, 255:red, 208; green, 2; blue, 27 }  ,fill opacity=1 ] (359,126.5) .. controls (359,123.74) and (361.24,121.5) .. (364,121.5) .. controls (366.76,121.5) and (369,123.74) .. (369,126.5) .. controls (369,129.26) and (366.76,131.5) .. (364,131.5) .. controls (361.24,131.5) and (359,129.26) .. (359,126.5) -- cycle ;
%Shape: Circle [id:dp23212820006386536] 
\draw  [color={rgb, 255:red, 208; green, 2; blue, 27 }  ,draw opacity=1 ][fill={rgb, 255:red, 208; green, 2; blue, 27 }  ,fill opacity=1 ] (444,141) .. controls (444,138.24) and (446.24,136) .. (449,136) .. controls (451.76,136) and (454,138.24) .. (454,141) .. controls (454,143.76) and (451.76,146) .. (449,146) .. controls (446.24,146) and (444,143.76) .. (444,141) -- cycle ;
%Shape: Circle [id:dp5439300296757643] 
\draw  [color={rgb, 255:red, 208; green, 2; blue, 27 }  ,draw opacity=1 ][fill={rgb, 255:red, 208; green, 2; blue, 27 }  ,fill opacity=1 ] (551.75,169.51) .. controls (551.75,166.75) and (553.99,164.51) .. (556.75,164.51) .. controls (559.51,164.51) and (561.75,166.75) .. (561.75,169.51) .. controls (561.75,172.27) and (559.51,174.51) .. (556.75,174.51) .. controls (553.99,174.51) and (551.75,172.27) .. (551.75,169.51) -- cycle ;
%Straight Lines [id:da9939861155158032] 
\draw [color={rgb, 255:red, 0; green, 0; blue, 0 }  ,draw opacity=1 ]   (109,20) -- (12.5,19) ;
%Straight Lines [id:da6460141075147547] 
\draw [color={rgb, 255:red, 0; green, 0; blue, 0 }  ,draw opacity=1 ]   (109,20) -- (129,1) ;
%Straight Lines [id:da644383057111223] 
\draw [color={rgb, 255:red, 0; green, 0; blue, 0 }  ,draw opacity=1 ]   (108.5,43) -- (128,65) ;
%Straight Lines [id:da796824295382208] 
\draw [color={rgb, 255:red, 0; green, 0; blue, 0 }  ,draw opacity=1 ]   (181,83) -- (491.01,82.02) ;
%Straight Lines [id:da9753553652203049] 
\draw [color={rgb, 255:red, 0; green, 0; blue, 0 }  ,draw opacity=1 ]   (181,83) -- (224,125.5) ;
%Straight Lines [id:da48627689074126645] 
\draw [color={rgb, 255:red, 0; green, 0; blue, 0 }  ,draw opacity=1 ]   (224,125.5) -- (449,125.5) ;
%Straight Lines [id:da49133266575001033] 
\draw [color={rgb, 255:red, 0; green, 0; blue, 0 }  ,draw opacity=1 ]   (449,125.5) -- (491.01,82.02) ;
%Straight Lines [id:da51467780889709] 
\draw [color={rgb, 255:red, 0; green, 0; blue, 0 }  ,draw opacity=1 ]   (224,125.5) -- (224.01,184.02) ;
%Straight Lines [id:da7007627180528562] 
\draw [color={rgb, 255:red, 0; green, 0; blue, 0 }  ,draw opacity=1 ]   (491.01,82.02) -- (558.52,49.02) ;
%Straight Lines [id:da6859330855524519] 
\draw [color={rgb, 255:red, 0; green, 0; blue, 0 }  ,draw opacity=1 ]   (558.52,49.02) -- (557.02,192.02) ;
%Straight Lines [id:da6255332306196266] 
\draw [color={rgb, 255:red, 0; green, 0; blue, 0 }  ,draw opacity=1 ]   (558.52,49.02) -- (596.02,3.03) ;
%Straight Lines [id:da6529064472478432] 
\draw [color={rgb, 255:red, 0; green, 0; blue, 0 }  ,draw opacity=1 ]   (449,126) -- (449,183.02) ;
%Shape: Square [id:dp7171542865243234] 
\draw   (432.99,21) -- (479.01,21) -- (479.01,67.02) -- (432.99,67.02) -- cycle ;
%Straight Lines [id:da16444256483632191] 
\draw [color={rgb, 255:red, 0; green, 0; blue, 0 }  ,draw opacity=1 ]   (109,20) -- (108.5,43) ;
%Straight Lines [id:da0840655116527731] 
\draw [color={rgb, 255:red, 0; green, 0; blue, 0 }  ,draw opacity=1 ]   (108.5,43) -- (11,43) ;
%Straight Lines [id:da46996664213146677] 
\draw [color={rgb, 255:red, 0; green, 0; blue, 0 }  ,draw opacity=1 ]   (128,65) -- (181,83) ;
%Shape: Circle [id:dp09440932328926133] 
\draw  [color={rgb, 255:red, 208; green, 2; blue, 27 }  ,draw opacity=1 ][fill={rgb, 255:red, 208; green, 2; blue, 27 }  ,fill opacity=1 ] (258,84.5) .. controls (258,81.74) and (260.24,79.5) .. (263,79.5) .. controls (265.76,79.5) and (268,81.74) .. (268,84.5) .. controls (268,87.26) and (265.76,89.5) .. (263,89.5) .. controls (260.24,89.5) and (258,87.26) .. (258,84.5) -- cycle ;
%Straight Lines [id:da46779899225592536] 
\draw [color={rgb, 255:red, 0; green, 0; blue, 0 }  ,draw opacity=1 ]   (558.52,49.02) -- (604.52,6.02) ;

% Text Node
\draw (441.08,19.35) node [anchor=north west][inner sep=0.75pt]  [font=\huge,rotate=-358.58]  {$x_{3}$};

\end{tikzpicture}}  
 & %  
     \resizebox{.4 \textwidth}{!}{%\tikzset{every picture/.style={line width=0.75pt}} %set default line width to 0.75pt        

\begin{tikzpicture}[x=0.75pt,y=0.75pt,yscale=-1,xscale=1]

%Shape: Circle [id:dp049478145904912285] 
\draw  [color={rgb, 255:red, 208; green, 2; blue, 27 }  ,draw opacity=1 ][fill={rgb, 255:red, 208; green, 2; blue, 27 }  ,fill opacity=1 ] (29,20) .. controls (29,17.24) and (31.24,15) .. (34,15) .. controls (36.76,15) and (39,17.24) .. (39,20) .. controls (39,22.76) and (36.76,25) .. (34,25) .. controls (31.24,25) and (29,22.76) .. (29,20) -- cycle ;
%Shape: Circle [id:dp597865605180764] 
\draw  [color={rgb, 255:red, 208; green, 2; blue, 27 }  ,draw opacity=1 ][fill={rgb, 255:red, 208; green, 2; blue, 27 }  ,fill opacity=1 ] (64,43) .. controls (64,40.24) and (66.24,38) .. (69,38) .. controls (71.76,38) and (74,40.24) .. (74,43) .. controls (74,45.76) and (71.76,48) .. (69,48) .. controls (66.24,48) and (64,45.76) .. (64,43) -- cycle ;
%Shape: Circle [id:dp03321279691471113] 
\draw  [color={rgb, 255:red, 208; green, 2; blue, 27 }  ,draw opacity=1 ][fill={rgb, 255:red, 208; green, 2; blue, 27 }  ,fill opacity=1 ] (106,65) .. controls (106,62.24) and (108.24,60) .. (111,60) .. controls (113.76,60) and (116,62.24) .. (116,65) .. controls (116,67.76) and (113.76,70) .. (111,70) .. controls (108.24,70) and (106,67.76) .. (106,65) -- cycle ;
%Straight Lines [id:da15855389453970536] 
\draw [color={rgb, 255:red, 0; green, 0; blue, 0 }  ,draw opacity=1 ]   (11,65) -- (128,65) ;
%Shape: Circle [id:dp7534753985751717] 
\draw  [color={rgb, 255:red, 208; green, 2; blue, 27 }  ,draw opacity=1 ][fill={rgb, 255:red, 208; green, 2; blue, 27 }  ,fill opacity=1 ] (160,78) .. controls (160,75.24) and (162.24,73) .. (165,73) .. controls (167.76,73) and (170,75.24) .. (170,78) .. controls (170,80.76) and (167.76,83) .. (165,83) .. controls (162.24,83) and (160,80.76) .. (160,78) -- cycle ;
%Shape: Circle [id:dp7779686155492531] 
\draw  [color={rgb, 255:red, 208; green, 2; blue, 27 }  ,draw opacity=1 ][fill={rgb, 255:red, 208; green, 2; blue, 27 }  ,fill opacity=1 ] (359,126.5) .. controls (359,123.74) and (361.24,121.5) .. (364,121.5) .. controls (366.76,121.5) and (369,123.74) .. (369,126.5) .. controls (369,129.26) and (366.76,131.5) .. (364,131.5) .. controls (361.24,131.5) and (359,129.26) .. (359,126.5) -- cycle ;
%Shape: Circle [id:dp23212820006386536] 
\draw  [color={rgb, 255:red, 208; green, 2; blue, 27 }  ,draw opacity=1 ][fill={rgb, 255:red, 208; green, 2; blue, 27 }  ,fill opacity=1 ] (454.5,135.25) .. controls (454.5,132.49) and (456.74,130.25) .. (459.5,130.25) .. controls (462.26,130.25) and (464.5,132.49) .. (464.5,135.25) .. controls (464.5,138.01) and (462.26,140.25) .. (459.5,140.25) .. controls (456.74,140.25) and (454.5,138.01) .. (454.5,135.25) -- cycle ;
%Shape: Circle [id:dp5439300296757643] 
\draw  [color={rgb, 255:red, 208; green, 2; blue, 27 }  ,draw opacity=1 ][fill={rgb, 255:red, 208; green, 2; blue, 27 }  ,fill opacity=1 ] (581.5,143) .. controls (581.5,140.24) and (583.74,138) .. (586.5,138) .. controls (589.26,138) and (591.5,140.24) .. (591.5,143) .. controls (591.5,145.76) and (589.26,148) .. (586.5,148) .. controls (583.74,148) and (581.5,145.76) .. (581.5,143) -- cycle ;
%Straight Lines [id:da9939861155158032] 
\draw [color={rgb, 255:red, 0; green, 0; blue, 0 }  ,draw opacity=1 ]   (109,20) -- (12.5,19) ;
%Straight Lines [id:da6460141075147547] 
\draw [color={rgb, 255:red, 0; green, 0; blue, 0 }  ,draw opacity=1 ]   (109,20) -- (129,1) ;
%Straight Lines [id:da644383057111223] 
\draw [color={rgb, 255:red, 0; green, 0; blue, 0 }  ,draw opacity=1 ]   (108.5,43) -- (128,65) ;
%Straight Lines [id:da796824295382208] 
\draw [color={rgb, 255:red, 0; green, 0; blue, 0 }  ,draw opacity=1 ]   (181,83) -- (449.01,82.02) ;
%Straight Lines [id:da9753553652203049] 
\draw [color={rgb, 255:red, 0; green, 0; blue, 0 }  ,draw opacity=1 ]   (181,83) -- (224,125.5) ;
%Straight Lines [id:da48627689074126645] 
\draw [color={rgb, 255:red, 0; green, 0; blue, 0 }  ,draw opacity=1 ]   (224,125.5) -- (449,125.5) ;
%Straight Lines [id:da49133266575001033] 
\draw [color={rgb, 255:red, 0; green, 0; blue, 0 }  ,draw opacity=1 ]   (449,125.5) -- (449.01,82.02) ;
%Straight Lines [id:da51467780889709] 
\draw [color={rgb, 255:red, 0; green, 0; blue, 0 }  ,draw opacity=1 ]   (224,125.5) -- (224.01,184.02) ;
%Straight Lines [id:da7007627180528562] 
\draw [color={rgb, 255:red, 0; green, 0; blue, 0 }  ,draw opacity=1 ]   (449.01,82.02) -- (521.01,0.02) ;
%Straight Lines [id:da7832199179696879] 
\draw [color={rgb, 255:red, 0; green, 0; blue, 0 }  ,draw opacity=1 ]   (449,125.5) -- (532.51,199.02) ;
%Straight Lines [id:da6859330855524519] 
\draw [color={rgb, 255:red, 0; green, 0; blue, 0 }  ,draw opacity=1 ]   (532.51,199.02) -- (535.51,233.02) ;
%Straight Lines [id:da6255332306196266] 
\draw [color={rgb, 255:red, 0; green, 0; blue, 0 }  ,draw opacity=1 ]   (532.51,199.02) -- (655.51,70.02) ;
%Straight Lines [id:da6529064472478432] 
\draw [color={rgb, 255:red, 0; green, 0; blue, 0 }  ,draw opacity=1 ]   (532.51,199.02) -- (530.51,234.02) ;
%Shape: Square [id:dp7171542865243234] 
\draw   (539,7) -- (589,7) -- (589,57) -- (539,57) -- cycle ;
%Straight Lines [id:da16444256483632191] 
\draw [color={rgb, 255:red, 0; green, 0; blue, 0 }  ,draw opacity=1 ]   (109,20) -- (108.5,43) ;
%Straight Lines [id:da0840655116527731] 
\draw [color={rgb, 255:red, 0; green, 0; blue, 0 }  ,draw opacity=1 ]   (108.5,43) -- (11,43) ;
%Straight Lines [id:da46996664213146677] 
\draw [color={rgb, 255:red, 0; green, 0; blue, 0 }  ,draw opacity=1 ]   (128,65) -- (181,83) ;
%Shape: Circle [id:dp09440932328926133] 
\draw  [color={rgb, 255:red, 208; green, 2; blue, 27 }  ,draw opacity=1 ][fill={rgb, 255:red, 208; green, 2; blue, 27 }  ,fill opacity=1 ] (258,84.5) .. controls (258,81.74) and (260.24,79.5) .. (263,79.5) .. controls (265.76,79.5) and (268,81.74) .. (268,84.5) .. controls (268,87.26) and (265.76,89.5) .. (263,89.5) .. controls (260.24,89.5) and (258,87.26) .. (258,84.5) -- cycle ;

% Text Node
\draw (541.08,10.35) node [anchor=north west][inner sep=0.75pt]  [font=\huge,rotate=-358.58]  {$x_{4}$};

\end{tikzpicture}}  
 \\ \hline
     \resizebox{.4 \textwidth}{!}{%\tikzset{every picture/.style={line width=0.75pt}} %set default line width to 0.75pt        

\begin{tikzpicture}[x=0.75pt,y=0.75pt,yscale=-1,xscale=1]

%Shape: Circle [id:dp049478145904912285] 
\draw  [color={rgb, 255:red, 208; green, 2; blue, 27 }  ,draw opacity=1 ][fill={rgb, 255:red, 208; green, 2; blue, 27 }  ,fill opacity=1 ] (29,20) .. controls (29,17.24) and (31.24,15) .. (34,15) .. controls (36.76,15) and (39,17.24) .. (39,20) .. controls (39,22.76) and (36.76,25) .. (34,25) .. controls (31.24,25) and (29,22.76) .. (29,20) -- cycle ;
%Shape: Circle [id:dp597865605180764] 
\draw  [color={rgb, 255:red, 208; green, 2; blue, 27 }  ,draw opacity=1 ][fill={rgb, 255:red, 208; green, 2; blue, 27 }  ,fill opacity=1 ] (64,43) .. controls (64,40.24) and (66.24,38) .. (69,38) .. controls (71.76,38) and (74,40.24) .. (74,43) .. controls (74,45.76) and (71.76,48) .. (69,48) .. controls (66.24,48) and (64,45.76) .. (64,43) -- cycle ;
%Shape: Circle [id:dp03321279691471113] 
\draw  [color={rgb, 255:red, 208; green, 2; blue, 27 }  ,draw opacity=1 ][fill={rgb, 255:red, 208; green, 2; blue, 27 }  ,fill opacity=1 ] (106,65) .. controls (106,62.24) and (108.24,60) .. (111,60) .. controls (113.76,60) and (116,62.24) .. (116,65) .. controls (116,67.76) and (113.76,70) .. (111,70) .. controls (108.24,70) and (106,67.76) .. (106,65) -- cycle ;
%Straight Lines [id:da15855389453970536] 
\draw [color={rgb, 255:red, 0; green, 0; blue, 0 }  ,draw opacity=1 ]   (11,65) -- (128,65) ;
%Shape: Circle [id:dp7534753985751717] 
\draw  [color={rgb, 255:red, 208; green, 2; blue, 27 }  ,draw opacity=1 ][fill={rgb, 255:red, 208; green, 2; blue, 27 }  ,fill opacity=1 ] (160,78) .. controls (160,75.24) and (162.24,73) .. (165,73) .. controls (167.76,73) and (170,75.24) .. (170,78) .. controls (170,80.76) and (167.76,83) .. (165,83) .. controls (162.24,83) and (160,80.76) .. (160,78) -- cycle ;
%Shape: Circle [id:dp5270532448174468] 
\draw  [color={rgb, 255:red, 208; green, 2; blue, 27 }  ,draw opacity=1 ][fill={rgb, 255:red, 208; green, 2; blue, 27 }  ,fill opacity=1 ] (318,104) .. controls (318,101.24) and (320.24,99) .. (323,99) .. controls (325.76,99) and (328,101.24) .. (328,104) .. controls (328,106.76) and (325.76,109) .. (323,109) .. controls (320.24,109) and (318,106.76) .. (318,104) -- cycle ;
%Shape: Circle [id:dp7779686155492531] 
\draw  [color={rgb, 255:red, 208; green, 2; blue, 27 }  ,draw opacity=1 ][fill={rgb, 255:red, 208; green, 2; blue, 27 }  ,fill opacity=1 ] (375,126.5) .. controls (375,123.74) and (377.24,121.5) .. (380,121.5) .. controls (382.76,121.5) and (385,123.74) .. (385,126.5) .. controls (385,129.26) and (382.76,131.5) .. (380,131.5) .. controls (377.24,131.5) and (375,129.26) .. (375,126.5) -- cycle ;
%Shape: Circle [id:dp23212820006386536] 
\draw  [color={rgb, 255:red, 208; green, 2; blue, 27 }  ,draw opacity=1 ][fill={rgb, 255:red, 208; green, 2; blue, 27 }  ,fill opacity=1 ] (466.5,146) .. controls (466.5,143.24) and (468.74,141) .. (471.5,141) .. controls (474.26,141) and (476.5,143.24) .. (476.5,146) .. controls (476.5,148.76) and (474.26,151) .. (471.5,151) .. controls (468.74,151) and (466.5,148.76) .. (466.5,146) -- cycle ;
%Shape: Circle [id:dp5439300296757643] 
\draw  [color={rgb, 255:red, 208; green, 2; blue, 27 }  ,draw opacity=1 ][fill={rgb, 255:red, 208; green, 2; blue, 27 }  ,fill opacity=1 ] (599,166) .. controls (599,163.24) and (601.24,161) .. (604,161) .. controls (606.76,161) and (609,163.24) .. (609,166) .. controls (609,168.76) and (606.76,171) .. (604,171) .. controls (601.24,171) and (599,168.76) .. (599,166) -- cycle ;
%Straight Lines [id:da9939861155158032] 
\draw [color={rgb, 255:red, 0; green, 0; blue, 0 }  ,draw opacity=1 ]   (109,20) -- (12.5,19) ;
%Straight Lines [id:da6460141075147547] 
\draw [color={rgb, 255:red, 0; green, 0; blue, 0 }  ,draw opacity=1 ]   (109,20) -- (129,1) ;
%Straight Lines [id:da644383057111223] 
\draw [color={rgb, 255:red, 0; green, 0; blue, 0 }  ,draw opacity=1 ]   (108.5,43) -- (128,65) ;
%Straight Lines [id:da51467780889709] 
\draw [color={rgb, 255:red, 0; green, 0; blue, 0 }  ,draw opacity=1 ]   (246.5,107) -- (246.5,189) ;
%Straight Lines [id:da7007627180528562] 
\draw    (359,107) -- (445.5,3) ;
%Straight Lines [id:da7832199179696879] 
\draw    (359,107) -- (401,146) ;
%Straight Lines [id:da44753667002377306] 
\draw    (401,146) -- (604,147) ;
%Straight Lines [id:da6859330855524519] 
\draw    (604,147) -- (604,184) ;
%Straight Lines [id:da6255332306196266] 
\draw    (604,147) -- (654,84) ;
%Straight Lines [id:da6529064472478432] 
\draw    (401,146) -- (401,186) ;
%Shape: Square [id:dp7171542865243234] 
\draw   (539,7) -- (589,7) -- (589,57) -- (539,57) -- cycle ;
%Straight Lines [id:da16444256483632191] 
\draw [color={rgb, 255:red, 0; green, 0; blue, 0 }  ,draw opacity=1 ]   (109,20) -- (108.5,43) ;
%Straight Lines [id:da0840655116527731] 
\draw [color={rgb, 255:red, 0; green, 0; blue, 0 }  ,draw opacity=1 ]   (108.5,43) -- (11,43) ;
%Straight Lines [id:da46996664213146677] 
\draw [color={rgb, 255:red, 0; green, 0; blue, 0 }  ,draw opacity=1 ]   (128,65) -- (246.5,107) ;
%Shape: Ellipse [id:dp35430366760200704] 
\draw   (246.5,107) .. controls (246.5,105.62) and (271.68,104.5) .. (302.75,104.5) .. controls (333.82,104.5) and (359,105.62) .. (359,107) .. controls (359,108.38) and (333.82,109.5) .. (302.75,109.5) .. controls (271.68,109.5) and (246.5,108.38) .. (246.5,107) -- cycle ;

% Text Node
\draw (553.08,13.35) node [anchor=north west][inner sep=0.75pt]  [font=\huge,rotate=-358.58]  {$y$};

\end{tikzpicture}}  
 & %  
     \resizebox{.4 \textwidth}{!}{%\tikzset{every picture/.style={line width=0.75pt}} %set default line width to 0.75pt        

\begin{tikzpicture}[x=0.75pt,y=0.75pt,yscale=-1,xscale=1]

%Shape: Circle [id:dp049478145904912285] 
\draw  [color={rgb, 255:red, 208; green, 2; blue, 27 }  ,draw opacity=1 ][fill={rgb, 255:red, 208; green, 2; blue, 27 }  ,fill opacity=1 ] (29,20) .. controls (29,17.24) and (31.24,15) .. (34,15) .. controls (36.76,15) and (39,17.24) .. (39,20) .. controls (39,22.76) and (36.76,25) .. (34,25) .. controls (31.24,25) and (29,22.76) .. (29,20) -- cycle ;
%Shape: Circle [id:dp597865605180764] 
\draw  [color={rgb, 255:red, 208; green, 2; blue, 27 }  ,draw opacity=1 ][fill={rgb, 255:red, 208; green, 2; blue, 27 }  ,fill opacity=1 ] (64,43) .. controls (64,40.24) and (66.24,38) .. (69,38) .. controls (71.76,38) and (74,40.24) .. (74,43) .. controls (74,45.76) and (71.76,48) .. (69,48) .. controls (66.24,48) and (64,45.76) .. (64,43) -- cycle ;
%Shape: Circle [id:dp03321279691471113] 
\draw  [color={rgb, 255:red, 208; green, 2; blue, 27 }  ,draw opacity=1 ][fill={rgb, 255:red, 208; green, 2; blue, 27 }  ,fill opacity=1 ] (106,65) .. controls (106,62.24) and (108.24,60) .. (111,60) .. controls (113.76,60) and (116,62.24) .. (116,65) .. controls (116,67.76) and (113.76,70) .. (111,70) .. controls (108.24,70) and (106,67.76) .. (106,65) -- cycle ;
%Straight Lines [id:da15855389453970536] 
\draw [color={rgb, 255:red, 0; green, 0; blue, 0 }  ,draw opacity=1 ]   (11,65) -- (128,65) ;
%Shape: Circle [id:dp7534753985751717] 
\draw  [color={rgb, 255:red, 208; green, 2; blue, 27 }  ,draw opacity=1 ][fill={rgb, 255:red, 208; green, 2; blue, 27 }  ,fill opacity=1 ] (160,78) .. controls (160,75.24) and (162.24,73) .. (165,73) .. controls (167.76,73) and (170,75.24) .. (170,78) .. controls (170,80.76) and (167.76,83) .. (165,83) .. controls (162.24,83) and (160,80.76) .. (160,78) -- cycle ;
%Shape: Circle [id:dp5270532448174468] 
\draw  [color={rgb, 255:red, 208; green, 2; blue, 27 }  ,draw opacity=1 ][fill={rgb, 255:red, 208; green, 2; blue, 27 }  ,fill opacity=1 ] (333,107) .. controls (333,104.24) and (335.24,102) .. (338,102) .. controls (340.76,102) and (343,104.24) .. (343,107) .. controls (343,109.76) and (340.76,112) .. (338,112) .. controls (335.24,112) and (333,109.76) .. (333,107) -- cycle ;
%Shape: Circle [id:dp7779686155492531] 
\draw  [color={rgb, 255:red, 208; green, 2; blue, 27 }  ,draw opacity=1 ][fill={rgb, 255:red, 208; green, 2; blue, 27 }  ,fill opacity=1 ] (375,126.5) .. controls (375,123.74) and (377.24,121.5) .. (380,121.5) .. controls (382.76,121.5) and (385,123.74) .. (385,126.5) .. controls (385,129.26) and (382.76,131.5) .. (380,131.5) .. controls (377.24,131.5) and (375,129.26) .. (375,126.5) -- cycle ;
%Shape: Circle [id:dp23212820006386536] 
\draw  [color={rgb, 255:red, 208; green, 2; blue, 27 }  ,draw opacity=1 ][fill={rgb, 255:red, 208; green, 2; blue, 27 }  ,fill opacity=1 ] (466.5,146) .. controls (466.5,143.24) and (468.74,141) .. (471.5,141) .. controls (474.26,141) and (476.5,143.24) .. (476.5,146) .. controls (476.5,148.76) and (474.26,151) .. (471.5,151) .. controls (468.74,151) and (466.5,148.76) .. (466.5,146) -- cycle ;
%Shape: Circle [id:dp5439300296757643] 
\draw  [color={rgb, 255:red, 208; green, 2; blue, 27 }  ,draw opacity=1 ][fill={rgb, 255:red, 208; green, 2; blue, 27 }  ,fill opacity=1 ] (599,166) .. controls (599,163.24) and (601.24,161) .. (604,161) .. controls (606.76,161) and (609,163.24) .. (609,166) .. controls (609,168.76) and (606.76,171) .. (604,171) .. controls (601.24,171) and (599,168.76) .. (599,166) -- cycle ;
%Straight Lines [id:da9939861155158032] 
\draw [color={rgb, 255:red, 0; green, 0; blue, 0 }  ,draw opacity=1 ]   (109,20) -- (12.5,19) ;
%Straight Lines [id:da6460141075147547] 
\draw [color={rgb, 255:red, 0; green, 0; blue, 0 }  ,draw opacity=1 ]   (109,20) -- (129,1) ;
%Straight Lines [id:da644383057111223] 
\draw [color={rgb, 255:red, 0; green, 0; blue, 0 }  ,draw opacity=1 ]   (108.5,43) -- (128,65) ;
%Straight Lines [id:da51467780889709] 
\draw [color={rgb, 255:red, 0; green, 0; blue, 0 }  ,draw opacity=1 ]   (246.5,107) -- (246.5,189) ;
%Straight Lines [id:da7007627180528562] 
\draw [color={rgb, 255:red, 0; green, 0; blue, 0 }  ,draw opacity=1 ]   (359,107) -- (445.5,3) ;
%Straight Lines [id:da7832199179696879] 
\draw [color={rgb, 255:red, 0; green, 0; blue, 0 }  ,draw opacity=1 ]   (359,107) -- (401,146) ;
%Straight Lines [id:da44753667002377306] 
\draw    (401,146) -- (604,147) ;
%Straight Lines [id:da6859330855524519] 
\draw    (604,147) -- (604,184) ;
%Straight Lines [id:da6255332306196266] 
\draw    (604,147) -- (654,84) ;
%Straight Lines [id:da6529064472478432] 
\draw    (401,146) -- (401,186) ;
%Shape: Square [id:dp7171542865243234] 
\draw   (539,29) -- (582,29) -- (582,72) -- (539,72) -- cycle ;
%Straight Lines [id:da16444256483632191] 
\draw [color={rgb, 255:red, 0; green, 0; blue, 0 }  ,draw opacity=1 ]   (109,20) -- (108.5,43) ;
%Straight Lines [id:da0840655116527731] 
\draw [color={rgb, 255:red, 0; green, 0; blue, 0 }  ,draw opacity=1 ]   (108.5,43) -- (11,43) ;
%Straight Lines [id:da46996664213146677] 
\draw [color={rgb, 255:red, 0; green, 0; blue, 0 }  ,draw opacity=1 ]   (128,65) -- (246.5,107) ;
%Straight Lines [id:da9000303830718615] 
\draw [color={rgb, 255:red, 0; green, 0; blue, 0 }  ,draw opacity=1 ]   (246.5,107) -- (327,107) ;
%Straight Lines [id:da9862549752568031] 
\draw [color={rgb, 255:red, 0; green, 0; blue, 0 }  ,draw opacity=1 ]   (327,107) -- (359,107) ;
%Shape: Circle [id:dp6135959863432625] 
\draw  [color={rgb, 255:red, 0; green, 0; blue, 0 }  ,draw opacity=1 ][fill={rgb, 255:red, 0; green, 0; blue, 0 }  ,fill opacity=1 ] (301,107) .. controls (301,104.24) and (303.24,102) .. (306,102) .. controls (308.76,102) and (311,104.24) .. (311,107) .. controls (311,109.76) and (308.76,112) .. (306,112) .. controls (303.24,112) and (301,109.76) .. (301,107) -- cycle ;

% Text Node
\draw (550.08,30.35) node [anchor=north west][inner sep=0.75pt]  [font=\huge,rotate=-358.58]  {$z$};
% Text Node
\draw (269,83.4) node [anchor=north west][inner sep=0.75pt]    {$2$};
% Text Node
\draw (287,117.4) node [anchor=north west][inner sep=0.75pt]    {$g=1$};
% Text Node
\draw (332,82.4) node [anchor=north west][inner sep=0.75pt]    {$2$};

\end{tikzpicture}}  
 
 \end{tabular}
\caption{The remaining vertices of $B$.}
\label{fig:remver}
\end{figure}

The weights of the cycle $\alpha_B$ are computed using the techniques of \cite[Remark 4.8]{KerberMarkwig}, \cite[Corollary 2.27]{GathmannKerberMarkwig}. Given an edge $e$ of $B$, corresponding to a given topological type of maps, choose a tree containing all eight markings of the source curve; $x,y$ and the edge lengths  $l_1, \ldots ,l_{15}$ of its compact edges  may be used as integral linear coordinates for the cone of $WS_{8}(\mathbb{TP}^2,3)$ containing $e$. Then $\alpha_B$ assigns to $e$ the gcd of the $16\times 16$ minors of the matrix representing the linear map $\prod_{i=1}^8 ev_i$. In Figure \ref{fig:matrixmult}, we show as an example the case of $e$ corresponding to the edge labeled $5$ in Figure \ref{fig:fam}. For the choice of  tree of $\Gamma$ and labeling of its compact edges  depicted in the figure, the resulting matrix is explicitly computed. Deleting the fifth column, we shaded a block triangular decomposition of the remaining matrix that is easily seen to have determinant one. It follows that the weight of the edge labeled $5$ in $B$ is equal to one.

\usetikzlibrary{fit}
\tikzset{%
  highlight/.style={rectangle,rounded corners,fill=red!15,draw = white,
    fill opacity=0.4,thick,inner sep=0pt}
}
\newcommand{\tikzmark}[2]{\tikz[overlay,remember picture,
  baseline=(#1.base)] \node (#1) {#2};}
\newcommand{\Highlight}[1][submatrix]{%
    \tikz[overlay,remember picture]{
    \node[highlight,fit=(left.north west) (right.south east)] (#1) {};}
}
\begin{figure}
    \centering

{

     \resizebox{.6 \textwidth}{!}{%\tikzset{every picture/.style={line width=0.75pt}} %set default line width to 0.75pt        

\begin{tikzpicture}[x=0.75pt,y=0.75pt,yscale=-1,xscale=1]
%uncomment if require: \path (0,481); %set diagram left start at 0, and has height of 481

%Shape: Circle [id:dp049478145904912285] 
\draw  [color={rgb, 255:red, 208; green, 2; blue, 27 }  ,draw opacity=1 ][fill={rgb, 255:red, 208; green, 2; blue, 27 }  ,fill opacity=1 ] (29,20) .. controls (29,17.24) and (31.24,15) .. (34,15) .. controls (36.76,15) and (39,17.24) .. (39,20) .. controls (39,22.76) and (36.76,25) .. (34,25) .. controls (31.24,25) and (29,22.76) .. (29,20) -- cycle ;
%Shape: Circle [id:dp597865605180764] 
\draw  [color={rgb, 255:red, 208; green, 2; blue, 27 }  ,draw opacity=1 ][fill={rgb, 255:red, 208; green, 2; blue, 27 }  ,fill opacity=1 ] (64,43) .. controls (64,40.24) and (66.24,38) .. (69,38) .. controls (71.76,38) and (74,40.24) .. (74,43) .. controls (74,45.76) and (71.76,48) .. (69,48) .. controls (66.24,48) and (64,45.76) .. (64,43) -- cycle ;
%Shape: Circle [id:dp03321279691471113] 
\draw  [color={rgb, 255:red, 208; green, 2; blue, 27 }  ,draw opacity=1 ][fill={rgb, 255:red, 208; green, 2; blue, 27 }  ,fill opacity=1 ] (145,72) .. controls (145,69.24) and (147.24,67) .. (150,67) .. controls (152.76,67) and (155,69.24) .. (155,72) .. controls (155,74.76) and (152.76,77) .. (150,77) .. controls (147.24,77) and (145,74.76) .. (145,72) -- cycle ;
%Straight Lines [id:da15855389453970536] 
\draw [color={rgb, 255:red, 0; green, 0; blue, 0 }  ,draw opacity=1 ]   (14.5,61) -- (120.5,61) ;
%Shape: Circle [id:dp7534753985751717] 
\draw  [color={rgb, 255:red, 208; green, 2; blue, 27 }  ,draw opacity=1 ][fill={rgb, 255:red, 208; green, 2; blue, 27 }  ,fill opacity=1 ] (212,84) .. controls (212,81.24) and (214.24,79) .. (217,79) .. controls (219.76,79) and (222,81.24) .. (222,84) .. controls (222,86.76) and (219.76,89) .. (217,89) .. controls (214.24,89) and (212,86.76) .. (212,84) -- cycle ;
%Shape: Circle [id:dp5270532448174468] 
\draw  [color={rgb, 255:red, 208; green, 2; blue, 27 }  ,draw opacity=1 ][fill={rgb, 255:red, 208; green, 2; blue, 27 }  ,fill opacity=1 ] (294,107) .. controls (294,104.24) and (296.24,102) .. (299,102) .. controls (301.76,102) and (304,104.24) .. (304,107) .. controls (304,109.76) and (301.76,112) .. (299,112) .. controls (296.24,112) and (294,109.76) .. (294,107) -- cycle ;
%Shape: Circle [id:dp7779686155492531] 
\draw  [color={rgb, 255:red, 208; green, 2; blue, 27 }  ,draw opacity=1 ][fill={rgb, 255:red, 208; green, 2; blue, 27 }  ,fill opacity=1 ] (375,126.5) .. controls (375,123.74) and (377.24,121.5) .. (380,121.5) .. controls (382.76,121.5) and (385,123.74) .. (385,126.5) .. controls (385,129.26) and (382.76,131.5) .. (380,131.5) .. controls (377.24,131.5) and (375,129.26) .. (375,126.5) -- cycle ;
%Shape: Circle [id:dp23212820006386536] 
\draw  [color={rgb, 255:red, 208; green, 2; blue, 27 }  ,draw opacity=1 ][fill={rgb, 255:red, 208; green, 2; blue, 27 }  ,fill opacity=1 ] (466.5,146) .. controls (466.5,143.24) and (468.74,141) .. (471.5,141) .. controls (474.26,141) and (476.5,143.24) .. (476.5,146) .. controls (476.5,148.76) and (474.26,151) .. (471.5,151) .. controls (468.74,151) and (466.5,148.76) .. (466.5,146) -- cycle ;
%Shape: Circle [id:dp5439300296757643] 
\draw  [color={rgb, 255:red, 208; green, 2; blue, 27 }  ,draw opacity=1 ][fill={rgb, 255:red, 208; green, 2; blue, 27 }  ,fill opacity=1 ] (599,166) .. controls (599,163.24) and (601.24,161) .. (604,161) .. controls (606.76,161) and (609,163.24) .. (609,166) .. controls (609,168.76) and (606.76,171) .. (604,171) .. controls (601.24,171) and (599,168.76) .. (599,166) -- cycle ;
%Straight Lines [id:da9939861155158032] 
\draw [color={rgb, 255:red, 0; green, 0; blue, 0 }  ,draw opacity=1 ]   (100.5,19) -- (11.5,19) ;
%Straight Lines [id:da6460141075147547] 
\draw [color={rgb, 255:red, 0; green, 0; blue, 0 }  ,draw opacity=1 ]   (100.5,19) -- (120.5,0) ;
%Straight Lines [id:da644383057111223] 
\draw [color={rgb, 255:red, 0; green, 0; blue, 0 }  ,draw opacity=1 ]   (120.5,61) -- (181,83) ;
%Straight Lines [id:da796824295382208] 
\draw    (181,83) -- (360,83) ;
%Straight Lines [id:da9753553652203049] 
\draw    (181,83) -- (203,106) ;
%Straight Lines [id:da48627689074126645] 
\draw    (203,106) -- (359,107) ;
%Straight Lines [id:da49133266575001033] 
\draw    (360,83) -- (359,107) ;
%Straight Lines [id:da51467780889709] 
\draw    (203,106) -- (203,188) ;
%Straight Lines [id:da7007627180528562] 
\draw    (360,83) -- (432,1) ;
%Straight Lines [id:da7832199179696879] 
\draw    (359,107) -- (401,146) ;
%Straight Lines [id:da44753667002377306] 
\draw    (401,146) -- (604,147) ;
%Straight Lines [id:da6859330855524519] 
\draw    (604,147) -- (604,184) ;
%Straight Lines [id:da6255332306196266] 
\draw    (604,147) -- (654,84) ;
%Straight Lines [id:da6529064472478432] 
\draw    (401,146) -- (401,186) ;
%Straight Lines [id:da16444256483632191] 
\draw [color={rgb, 255:red, 0; green, 0; blue, 0 }  ,draw opacity=1 ][fill={rgb, 255:red, 74; green, 144; blue, 226 }  ,fill opacity=1 ]   (100.5,19) -- (101.5,42) ;
%Straight Lines [id:da666703974482969] 
\draw [color={rgb, 255:red, 0; green, 0; blue, 0 }  ,draw opacity=1 ]   (101.5,42) -- (120.5,61) ;
%Straight Lines [id:da0840655116527731] 
\draw    (101.5,42) -- (12.5,41) ;

% Text Node
\draw (60.5,-2.6) node [anchor=north west][inner sep=0.75pt]    {$l_{1}$};
% Text Node
\draw (102.5,22.4) node [anchor=north west][inner sep=0.75pt]    {$l_{2}$};
% Text Node
\draw (79,22.4) node [anchor=north west][inner sep=0.75pt]    {$l_{3}$};
% Text Node
\draw (112,37.9) node [anchor=north west][inner sep=0.75pt]    {$l_{4}$};
% Text Node
\draw (132,46.9) node [anchor=north west][inner sep=0.75pt]    {$l_{5}$};
% Text Node
\draw (165,60.4) node [anchor=north west][inner sep=0.75pt]    {$l_{6}$};
% Text Node
\draw (190.5,65.4) node [anchor=north west][inner sep=0.75pt]    {$l_{7}$};
% Text Node
\draw (177,96.9) node [anchor=north west][inner sep=0.75pt]    {$l_{8}$};
% Text Node
\draw (323.06,105.4) node [anchor=north west][inner sep=0.75pt]  [xslant=0.12]  {$l_{10}$};
% Text Node
\draw (367.5,99.4) node [anchor=north west][inner sep=0.75pt]    {$l_{11}$};
% Text Node
\draw (396,122.9) node [anchor=north west][inner sep=0.75pt]    {$l_{12}$};
% Text Node
\draw (432,126.9) node [anchor=north west][inner sep=0.75pt]    {$l_{13}$};
% Text Node
\draw (526.5,129.4) node [anchor=north west][inner sep=0.75pt]    {$l_{14}$};
% Text Node
\draw (608,140.9) node [anchor=north west][inner sep=0.75pt]    {$l_{15}$};
% Text Node
\draw (239,107.9) node [anchor=north west][inner sep=0.75pt]    {$l_{9}$};

\end{tikzpicture}}  

}

\tiny{
$$
\left[
\begin{array}{ccccccccccccccccc}
%e_1 & e_2 & l_1 & l_2 & l_3 & l_4 & l_5 & l_6 & l_7 & l_8 & l_9 & l_{10} & l_{11} & l_{12} & l_{13} & l_{14} & l_{15} \\
%ev_1

\tikzmark{left}{1} & 0 & 0 & 0 & 0 & 0 & 0& 0 & 0 & 0 & 0 & 0 & 0 & 0 & 0 & 0 & 0 
\\
0 & \tikzmark{right}{1} & 0 & 0 & 0 & 0 & 0 & 0 & 0 & 0 & 0 & 0 & 0 & 0 & 0 & 0 & 0  \\ 
\Highlight[first]
%ev_2
1 & 0 & \tikzmark{left}{1} & 0 & -1 & 0 & 0& 0 & 0 & 0 & 0 & 0 & 0 & 0 & 0 & 0 & 0 
\\
0 & 1 & 0 & \tikzmark{right}{-1} & 0 & 0 & 0 & 0 & 0 & 0 & 0 & 0 & 0 & 0 & 0 & 0 & 0 \\ \Highlight[first]
%ev_3
1 & 0 & 1 & 0 & 0 & \tikzmark{left}{1} & {2}& 0 & 0 & 0 & 0 & 0 & 0 & 0 & 0 & 0 & 0 
\\
0 & 1 & 0 & -1 & 0 & -1 & \tikzmark{right}{-1} & 0 & 0 & 0 & 0 & 0 & 0 & 0 & 0 & 0 & 0  \\ \Highlight[first]
%ev_4
1 & 0 & 1 & 0 & 0 & 1 & 2& \tikzmark{left}{2} & 1 & 0 & 0 & 0 & 0 & 0 & 0 & 0 & 0 
\\
0 & 1 & 0 & -1 & 0 & -1 & -1 & -1 & \tikzmark{right}{0} & 0 & 0 & 0 & 0 & 0 & 0 & 0 & 0  \\ \Highlight[first]
%ev_5
1 & 0 & 1 & 0 & 0 & 1 & 2& 2 & 0 & \tikzmark{left}{1} & 1 & 0 & 0 & 0 & 0 & 0 & 0 
\\
0 & 1 & 0 & -1 & 0 & -1 & -1 & -1 & 0 & -1 & \tikzmark{right}{0} & 0 & 0 & 0 & 0 & 0 & 0  \\\Highlight[first]
%ev_6
1 & 0 & 1 & 0 & 0 & 1 & 2& 2 & 0 & {1} & 1 & \tikzmark{left}{1} & 1 & 0 & 0 & 0 & 0
\\
0 & 1 & 0 & -1 & 0 & -1 & -1 & -1 & 0 & -1 & 0 & 0 & \tikzmark{right}{-1} & {0} & 0 & 0 & 0  \\ \Highlight[first]

%ev_7
1 & 0 & 1 & 0 & 0 & 1 & 2& 2 & 0 & {1} & 1 & 1 & 1 & \tikzmark{left}{1} & 1 & 0 & 0 
\\
0 & 1 & 0 & -1 & 0 & -1 & -1 & -1 & 0 & -1 & 0 & 0 & -1 & {-1} & \tikzmark{right}{0} & 0 & 0  \\ \Highlight[first]
%ev_8
1 & 0 & 1 & 0 & 0 & 1 & 2& 2 & 0 & {1} & 1 & 1 & 1 & 1 & 1  & \tikzmark{left}{1} & 0
\\
0 & 1 & 0 & -1 & 0 & -1 & -1 & -1 & 0 & -1 & 0 & 0 & -1 & -1 & 0 & 0  & \tikzmark{right}{-1} \\ \Highlight[first]
% closing the loop
%0 & 0 & 0 & 0 & 0 & 0 & 0& 0 &  1 & 1 & 0 & -1 & 0 & 0 & 0 & 0 & 0 & \tikzmark{left}{-1} & -1
%\\
%0 & 0 & 0 & 0 & 0 & 0 & 0 & 0 & 0 & 0 & -1 & 0 & 0 & 0 & 0 & 0 & 0 & 1 & \tikzmark{right}{0} \\ \Highlight[first] 
\end{array}
\right]
$$}
    \caption{In the upper part of the Figure, we choose a  tree containing all eight markings for a curve $\Gamma$ in the topological type  of maps denoted 5 in Figure \ref{fig:fam}. The matrix below represents the function $\prod_{i=1}^{8} ev_i $ in the coordinates $x,y, l_1, \ldots, l_{15}$ for the cone parameterizing maps of this topological type.}
    \label{fig:matrixmult}
\end{figure}
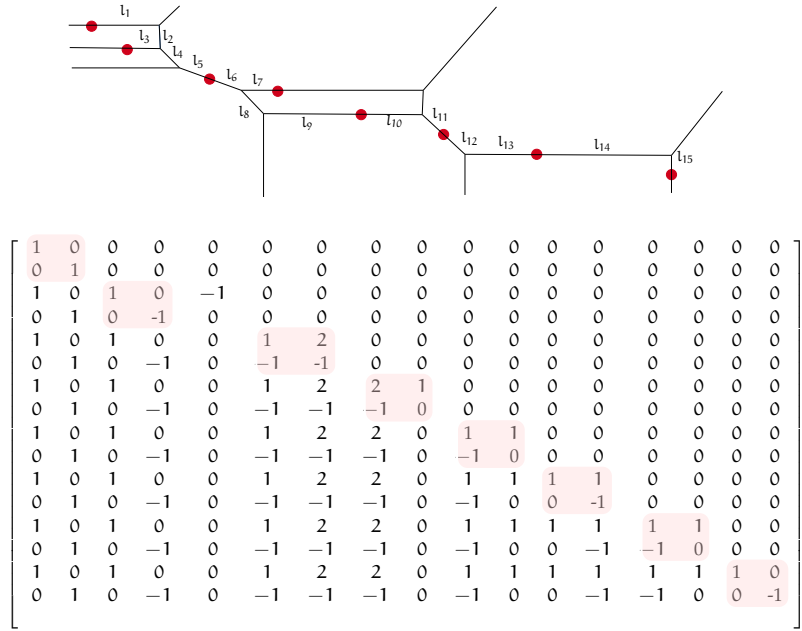

\bibliographystyle{amstest2}
\bibliography{lib}

\end{document}